\documentclass[11pt]{amsart}
\usepackage{amscd,amsmath,amssymb,amsfonts,verbatim}
\usepackage[cmtip, all]{xy}
\usepackage{MnSymbol}
\usepackage[a4 paper, left=3cm, right=3cm, top = 2.5cm, bottom = 2.5cm ]{geometry}
\usepackage{comment}
\usepackage[colorlinks, 
linkcolor=black,anchorcolor=black, citecolor=black, 
linkcolor=black]{hyperref}

\newtheorem{thm}{Theorem}[section]
\newtheorem{prop}[thm]{Proposition}
\newtheorem{lem}[thm]{Lemma}
\newtheorem{cor}[thm]{Corollary}

\theoremstyle{definition}

\newtheorem{defn}[thm]{Definition}
\theoremstyle{remark}
\newtheorem{remk}[thm]{Remark}
\newtheorem{remks}[thm]{Remarks}

\newtheorem{exm}[thm]{Example}
\newtheorem{exms}[thm]{Examples}
\newtheorem{notat}[thm]{Notation}
\numberwithin{equation}{section}

{\hfill$\square$\end{defn}}
{\hfill$\square$\end{remk}}
{\hfill$\square$\end{remks}}
{\hfill$\square$\end{exm}}
{\hfill$\square$\end{exms}}
{\hfill$\square$\end{notat}}

\newcommand{\thmref}{Theorem~\ref}
\newcommand{\propref}{Proposition~\ref}
\newcommand{\corref}{Corollary~\ref}
\newcommand{\defref}{Definition~\ref}
\newcommand{\lemref}{Lemma~\ref}

\newcommand{\remref}{Remark~\ref}

\newcommand{\sD}{{\mathcal D}}
\newcommand{\sE}{{\mathcal E}}
\newcommand{\sF}{{\mathcal F}}
\newcommand{\sG}{{\mathcal G}}
\newcommand{\sH}{{\mathcal H}}
\newcommand{\sI}{{\mathcal I}}

\newcommand{\sK}{{\mathcal K}}
\newcommand{\sL}{{\mathcal L}}
\newcommand{\sM}{{\mathcal M}}
\newcommand{\sN}{{\mathcal N}}
\newcommand{\sO}{{\mathcal O}}

\newcommand{\sX}{{\mathcal X}}

\newcommand{\A}{{\mathbb A}}

\newcommand{\F}{{\mathbb F}}

\renewcommand{\H}{{\mathbb H}}

\newcommand{\N}{{\mathbb N}}
\renewcommand{\P}{{\mathbb P}}
\newcommand{\Q}{{\mathbb Q}}

\newcommand{\Z}{{\mathbb Z}}

\newcommand{\fm}{{\mathfrak m}}

\newcommand{\fM}{{\mathfrak M}}

\newcommand{\ff}{{\mathfrak f}}

\newcommand{\Ker}{{\rm Ker}}
\newcommand{\gr}{{\rm gr}}

\newcommand{\surj}{\twoheadrightarrow}
\newcommand{\inj}{\hookrightarrow}
\newcommand{\red}{{\rm red}}

\newcommand{\rank}{{\rm rank}}
\newcommand{\Pic}{{\rm Pic}}
\newcommand{\Div}{{\rm Div}}
\newcommand{\Hom}{{\rm Hom}}

\newcommand{\Spec}{{\rm Spec \,}}
\newcommand{\sing}{{\rm sing}}
\newcommand{\Char}{{\rm char}}

\newcommand{\bk}{{\rm bk}}

\newcommand{\ab}{\rm ab}

\newcommand{\Et}{{\rm {\bf{Et}}}}

\newcommand{\id}{{\operatorname{id}}}

\newcommand{\Sch}{{\operatorname{\mathbf{Sch}}}}

\newcommand{\Sm}{{\mathbf{Sm}}}

\newcommand{\can}{{\operatorname{\rm can}}}

\newcommand{\Wedge}{{\Lambda}}
\newcommand{\eff}{{\operatorname{\rm eff}}}

\newcommand{\et}{{\textnormal{\'et}}}

\newcommand{\ds}{{/\kern-3pt/}}

\newcommand{\res}{{\operatorname{res}}}

\renewcommand{\log}{{\operatorname{log}}}

\newcommand{\Tor}{{\operatorname{Tor}}}

\newcommand{\Br}{{\operatorname{Br}}}

\newcommand{\un}{\underline}
\newcommand{\ov}{\overline}

\renewcommand{\dim}{\text{\rm dim}}

\newcommand{\tuborg}{\left\{\begin{array}{ll}}
\newcommand{\sluttuborg}{\end{array}\right.}

\newcommand{\zar}{{\rm zar}}
\newcommand{\nis}{{\rm nis}}

\newcommand{\reg}{{\rm reg}}

\newcommand{\dlog}{{\rm dlog}}
\newcommand{\dR}{{\rm dR}}

\newcommand{\divv}{{\rm div}}

\newcommand{\rsw}{{\rm rsw}}
\newcommand{\Rsw}{{\rm Rsw}}

\newcommand{\Sw}{{\rm sw}}

\newcommand{\wt}{\widetilde}
\newcommand{\wh}{\widehat}
\newcommand{\cont}{{\rm cont}}

\newcommand{\coker}{{\rm Coker}}
\newcommand{\Fil}{{\rm fil}}

\newcommand{\n}{{\un{n}}}

\newcommand{\etl}{{\acute{e}t}}

\newcounter{elno}

\newcounter{elno-abc}   

\newcounter{elno-abc-prime}   

\begin{document}
\title[Ramification filtration via de Rham--Witt complex]{Kato's Ramification filtration
via de Rham--Witt complex and applications}
\author{Amalendu Krishna, Subhadip Majumder}
\address{Department of Mathematics, South Hall, Room 6607, University of California
  Santa Barbara, CA, 93106-3080, USA.}
\email{amalenduk@math.ucsb.edu}
\address{Laboratoire de Math{\'e}matiques d'Orsay, 307, rue Michel Magat,
Facult{\'e} des Sciences d'Orsay, Universit{\'e} Paris--Saclay,
F-91405 Orsay Cedex, France}
\email{subhamajumdar.sxc@gmail.com}




\maketitle
\vskip .4cm

\begin{quote}\emph{Abstract.}
 Given an $F$-finite regular scheme $X$ of positive characteristic and a
 simple normal crossing divisor $E$ on $X$,
 we introduce a filtration on the de Rham--Witt
  complex $W_m\Omega^\bullet_{X\setminus E}$. When $X$ is the spectrum of a
  henselian discrete valuation ring $A$ with quotient field $K$,
  this extends the classical filtration on
  $W_m(K)$ due to Brylinski.  We show that Kato's ramification filtration on
$H^q_\et(X \setminus E, {\Q}/{\Z}(q-1))$ for $q \ge 1$ admits an explicit
  description in terms of the above filtration on the de Rham--Witt complex
  of $X \setminus E$. When $q =1$, this specializes to the results of Kato and
  Kerz--Saito.

  As applications, we construct a generalized version of Kato's refined Swan conductor
  for regular schemes, prove Lefschetz theorems for Kato's ramification filtrations
  for smooth projective varieties over $F$-finite fields, and prove
  Lefschetz and vanishing theorems for the Brauer group.
\end{quote}

\setcounter{tocdepth}{1}
\tableofcontents

\vskip .4cm

\section{Introduction}\label{sec:Intro}
\subsection{Background}\label{sec:Background}
To study the wild ramification in the $p$-adic {\'e}tale motivic cohomology classes
of a henselian
discrete valuation field $K$ with imperfect residue field, Kato \cite{Kato-89}
introduced a ramification filtration
$\{\Fil_nH^{q}_\et(K, {\Q_p}/{\Z_p}(q-1))\}_{n \ge 0}$
on the cohomology groups $H^{q}_\et(K, {\Q_p}/{\Z_p}(q-1))$ for $q \ge 1$.
This generalizes the classical upper numbering filtration on the
character groups of the Galois groups of complete discrete valuation fields with
perfect residue fields (cf. \cite[Chap.~IV, \S~3, Chap.~XV, \S~2]{Serre-LF}).

Kato showed that when $K$ has characteristic
$p > 0$, his ramification filtration on $H^1_{\et}(K, {\Q_p}/{\Z_p})$
has a simple description as the direct limit (over $m$) of the images of
a filtration on $W_m(K)$ defined earlier by Brylinski \cite[\S~1]{Brylinski}
under the canonical maps $\partial_{K,m} \colon W_m(K) \to
H^1_{\et}(K, {\Q_p}/{\Z_p})$ induced by the Artin-Schreier-Witt
exact sequence
\begin{equation}\label{eqn:AS-0}
  0 \to {\Z}/{p^m} \to W_m(K) \xrightarrow{1-\ov{F}} W_m(K)
  \xrightarrow{\partial_{K,m}} H^1_{\et}(K, {\Q_p}/{\Z_p}).
  \end{equation}
The description of Kato's ramification filtration on $H^1_{\et}(K, {\Q_p}/{\Z_p})$
in terms of the Brylinski's filtration on $W_m(K)$ has many consequences
(e.g., see \cite{Daichi}, \cite{Kato-89}, \cite{Kerz-Saito-Duke}, \cite{Leal},
\cite{GK-Adv}, \cite{GK-Nis}).

If $X$ is a Noetherian normal and integral scheme of characteristic $p > 0$
and $D \subset X$ is an effective Weil divisor with complement $U$, Kato
(cf. \cite[\S~7,8]{Kato-89}) defined a filtration on
$H^{q}_{\et}(U, {\Q_p}/{\Z_p}(q-1))$ (see \S~\ref{sec:Kato-comp} for the
  definition of this group) for all $q \ge 1$ by 
\[
\Fil_D^\log H^{q}_{\et}(U, {\Q_p}/{\Z_p}(q-1)) =
  \{\chi \in H^{q}_\et(U, {\Q_p}/{\Z_p}(q-1))| \chi_i \in
  \Fil_{n_i} H^{q}_\et(K_i, {\Q_p}/{\Z_p}(q-1)) \ \forall i\},
  \]
where $\{x_1, \cdots , x_r\}$ is the set of generic points of $D$,
  $n_i$ is the multiplicity of $D$ at $x_i$, $K_i$ is the henselization of
  the function field of $X$ at $x_i$ and $\chi_i$ is the image of $\chi$
  under the canonical pull-back map $H^{q}_\et(U, {\Q_p}/{\Z_p}(q-1)) \to
  H^{q}_\et(K_i, {\Q_p}/{\Z_p}(q-1))$. The group
  $\Fil_D^\log H^{q}_{\et}(U, {\Z}/{p^m}(q-1))$ is defined analogously
  (cf. \S~\ref{sec:Kato-global}). 

Using Kato's description of his filtration on
$H^1_{\et}(K, {\Q_p}/{\Z_p})$ in terms of Brylinski's filtration on $W_m(K)$ for a
henselian field $K$,
Kerz--Saito \cite{Kerz-Saito-ANT} showed that if $X$ is a smooth projective variety
over a perfect field of characteristic $p > 0$ and $D_\red \subset X$ is a simple
normal crossing divisor with complement $U$, then $\Fil_D H^1_{\et}(U, {\Q_p}/{\Z_p})$
can be described as an {\'e}tale cohomology group (see \S~\ref{sec:GKF} for
details). As a major application of this cohomological description, they
\cite{Kerz-Saito-Duke} solved the class
field theory problem for the abelianized {\'e}tale fundamental group with modulus
over finite fields (see also \cite{BKS}).

The principal objective of this paper is to
extend Brylinski's filtration on $W_m(K)$ to a filtration on the full de Rham--Witt
complex $W_m\Omega^\bullet_K$, and prove that Kato's filtration
on $H^{q}_{\et}(K, {\Q_p}/{\Z_p}(q-1))$ can be described in terms of this new
filtration on $W_m\Omega^\bullet_K$. This 
extends \cite[Thm.~3.2]{Kato-89} to all integers $q \ge 1$.
Using this extension and its generalization to schemes, we show that
Kato's global filtration $\Fil_D^\log H^{q}_{\et}(U, {\Q_p}/{\Z_p}(q-1))$
can be described as an {\'e}tale cohomology group,
extending the result of Kerz--Saito for $H^1$ to $H^q$ for all $q \ge 1$.
We derive several applications of these generalizations some of which are included in
this paper while the remaining ones are given in \cite{KM-3}, \cite{KM-2} and
\cite{KM-1}.

{\bf Standing assumption:}
We fix once and for all a prime $p > 0$ and assume (except in \S~\ref{sec:Lef-Br})
all fields to be of characteristic $p$, and all schemes to be $\F_p$-schemes.

\subsection{Extensions of the filtrations of Brylinski and Kato} 
\label{sec:BKF}
Let $X$ be a equidimensional
Noetherian regular $F$-finite $\F_p$-scheme and let $E \subset X$ be a simple
normal crossing divisor with irreducible components $E_1, \ldots , E_r$.
We shall refer to the pair $(X,E)$ as an snc-pair.
Let $j \colon U = X \setminus E \inj X$ be the inclusion of the complement of $E$.
Let $\Div_E(X)$ denote the subgroup of the Weil divisor group $\Div(X)$ consisting
of divisors supported on $E$. Recall that $\Div_E(X)$ is a partially
ordered group where $D \le D'$ if $D'-D$ is effective.
We refer to Definitions~\ref{defn:Filtered-Witt-complex} and
~\ref{defn:Filtered-Witt-complex-0} for the definition of good filtered
Witt-complexes. 

Our first result is the following broad generalization of Brylinski's filtration on
the rings of Witt vectors of discrete valuation fields.

\begin{thm}\label{thm:Main-0}
  There is a $\Div_E(X)$-filtered Witt-complex
  $\{\Fil_D W_m\Omega^\bullet_U\}$ together with an inclusion of
  filtered Witt-complexes $\{\Fil_D W_m\Omega^\bullet_U\} \inj
  \{j_* W_m\Omega^\bullet_U\}$ over $X$ such that the following hold.
  \begin{enumerate}
  \item
    $\varinjlim_D \Fil_D W_m\Omega^\bullet_U \xrightarrow{\cong}
    j_* W_m\Omega^\bullet_U$.
    \item
    $\Fil_D W_m\sO_U$ coincides with Brylinski's filtration on $W_m(K)$
    when $X = \Spec(A)$,
    where $A$ is a discrete valuation ring with quotient field $K$,
    residue field $\ff$ and $E = \Spec(\ff)$.
  \item
    $\{\Fil_D W_m\Omega^\bullet_U\}$ is good.
  \item
    $\{\Fil_D W_m\Omega^\bullet_U\}$ is contravariantly functorial in the pair
    $(X,E)$.
  \end{enumerate}
  \end{thm}

\vskip.2cm

\begin{remk}\label{remk:Comparison}
  The following are some remarks about \thmref{thm:Main-0}.
  \begin{enumerate}
  \item
    In \cite{JSZ}, Jannsen--Saito--Zhao defined a filtration
    $\{\Fil^{\rm JSZ}_D W_m\Omega^q_U\}$ on $j_* W_m\Omega^q_U$.
    Their filtration is related to the one in \thmref{thm:Main-0} via the isomorphism $$\Fil^{\rm JSZ}_D W_m\Omega^q_U \cong \Fil_{p^{m-1}D} W_m\Omega^q_U$$ (cf. \lemref{lem:Log-fil-4}(1)). However, the filtration defined in \cite{JSZ} does not, in general, recover its
    relative logarithmic Hodge--Witt sheaf \(W_m\Omega^q_{X\mid D,\log}\), nor does it yield an appropriate dualizing object for \(W_m\Omega^q_{X\mid D,\log}\) for arbitrary effective \(D\). The filtration
    $\{\Fil_D W_m\Omega^\bullet_U\}$ rectifies these issues (cf. \lemref{lem:Complex-6}).
 \item
      In \cite{Ren-Rulling}, Ren--R{\"u}lling introduced subsheaves
      $W_m\Omega^q_{(X, \pm D)}$
      of $j_* W_m\Omega^q_U$ when $X$ is a smooth over a perfect field and
      $D \ge 0$.
      Their subsheaf $W_m\Omega^q_{(X, D)}$ (resp. $W_m\Omega^q_{(X, -D)}$)
      coincides with $\Fil_{D} W_m\Omega^q_U$
      (resp. $\Fil_{-D} W_m\Omega^q_U$) given by \thmref{thm:Main-1}
    when $D =0$ (resp. $D = E$), see \lemref{lem:Log-fil-4}.
    In other cases, (for $D$ effective) $W_m\Omega^q_{(X, D)}$ and $\Fil_{D} W_m\Omega^q_U$
    (resp. $W_m\Omega^q_{(X, -D)}$ and $\Fil_{-D} W_m\Omega^q_U$) are 
    different from each other, even when $m =1$ and $q \ge 1$
    (or $q = 0$ and $m \ge 2$), as is evident
    from \lemref{lem:Log-fil-4} and \cite[(6.2.4), Lem.~8.4]{Ren-Rulling}
    (see also Remark~\ref{remk:Log-fil-4-ext}).
  \item
    In \cite{GK-Duality-1}, Gupta--Krishna introduced a filtration
    $\{W_m\Omega^q_{X|D}\}_{D \ge 0}$ on $W_m\Omega^q_X$.
    We show that $W_m\Omega^q_{X|D} \subset \Fil_{-D}W_m\Omega^q_U$ (cf.
    \lemref{lem:Log-fil-4}(5)). This is an equality only if $D = E$. One can not expect
    this inclusion to be an equality in general
    because $\{W_m\Omega^q_{X|D}\}$ is a non-logarithmic
    Matsuda type filtration while $\{\Fil_{D}W_m\Omega^q_U\}$ is a logarithmic one.
  \item
    In \cite{Shiho-1}, Shiho defined a filtration on the de Rham--Witt complex
    which is indexed by effective $\Q$-divisors, and whose purpose seems to be
    different than we do here.
\end{enumerate}
\end{remk}


We now proceed to give applications of \thmref{thm:Main-0}. 

\vskip.2cm

\subsection{Ramification subgroups as sheaf cohomology and their functoriality}
\label{sec:Ramif-sheaf}
Let $X$ be as in \thmref{thm:Main-0}. We assume that $D \in \Div_E(X)$ is effective.
For $q \ge 0$, we let $W_m\sF^{q, \bullet}_{X|D}$ denote the two term complex
$[Z_1\Fil_D W_m\Omega^q_U \xrightarrow{1-C} \Fil_D W_m\Omega^q_U]$ on the
{\'e}tale site of $X$, where
$Z_1\Fil_D W_m\Omega^q_U = \Fil_D W_m\Omega^q_U  \bigcap j_* Z_1W_m\Omega^q_U$ and
$C$ is the Cartier operator.
We refer to \S~\ref{sec:Global} for the proof that
$W_m\sF^{q, \bullet}_{X|D}$ is well-defined. We now have
the following extension of the theorems of Kato (cf. ~\eqref{eqn:AS-0}) and
Kerz--Saito \cite{Kerz-Saito-ANT}.

\begin{thm}\label{thm:Main-1}
  For $m, q \ge 1$, there is a canonical isomorphism
  \[
  \Fil_D^\log H^{q}_\et(U, {\Z}/{p^m}(q-1)) \xrightarrow{\cong}
  \H^1_\et(X, W_m\sF^{q-1, \bullet}_{X|D}).
  \]
\end{thm}

\vskip .3cm

One immediate corollary of \thmref{thm:Main-1} is the following functoriality
property of Kato's ramification filtration (cf. \corref{cor:Fil-functor}) which
can not be deduced from its definition. For morphisms of finite type schemes
over $\F_p$ and $q =1$, this result was shown by Kato--Leal--Saito
\cite[Thm.~4.3.1]{KLS}.
We refer the reader to \S~\ref{sec:de-log-0}
for the definition of morphisms between snc-pairs.

\begin{cor}\label{cor:Main-1-0}
  For $m, q \ge 1$, the ramification filtration
  $\Fil_D^\log H^{q}_\et(U, {\Z}/{p^m}(q-1))$
  is functorial with respect to morphisms of snc-pairs.
  \end{cor}



\subsection{Global refined Swan conductor}\label{sec:GRSC}
Let $X$ be as in \thmref{thm:Main-0} and let $D \ge D' \ge D/p \ge 0$ be two divisors
in $\Div_E(X)$ (see \S~\ref{sec:FDRWC} for the definition of $D/p$).
If $X = \Spec(A)$, where $A$ is an $F$-finite henselian discrete
valuation ring with maximal ideal $(\pi)$, quotient field $K$ and residue field $\ff$,
then one of the main results of
Kato in \cite{Kato-89} is the construction of an injective refined Swan
conductor map
\begin{equation}\label{eqn:Kato-SC-0}
  {\rsw}^q_{K,n} \colon
  \frac{\Fil^{\log}_nH^{q}_\et(K, {\Q_p}/{\Z_p}(q-1))}
       {\Fil^{\log}_{n-1}H^{q}_\et(K, {\Q_p}/{\Z_p}(q-1))} \inj 
  \pi^{-n} \Omega^{q}_A(\log \pi)\otimes_A \ff
  \end{equation}
for $n,q \ge 1$.
When $X$ is of finite type over $\F_p$ and $q =1$, Kato--Leal--Saito
\cite[Thm.~4.2.2]{KLS} showed that Kato's Swan conductor at the
generic points of $E$ glue together to give rise to an injective map
${\Rsw} \colon \frac{\Fil^{\log}_{D} H^{1}_\et(U, {\Q_p}/{\Z_p})}
{\Fil^{\log}_{D/p}H^{1}_\et(U,{\Q_p}/{\Z_p} )} \to
H^0_\zar(X, \Omega^1_X(\log E) \otimes_{\sO_X} \frac{\sO_X(D)}{\sO_{X}(D/p)})$.

We apply \thmref{thm:Main-1} to prove the following 
generalization of these results. 

\begin{thm}\label{thm:Main-9}
Let $X$ be as in \thmref{thm:Main-0} and let $D \ge D' \ge D/p \ge 0$ be two divisors
in $\Div_E(X)$. Let $q \ge 1$ be an integer and let $F=D-D'$.
Then there exists a monomorphism of ind-abelian groups 
  \begin{equation}\label{eqn:Main-9-0}
    \Rsw^{\bullet,q}_{X|(D,D')} \colon
    \left\{ \frac{\Fil^\log_DH^{q}_\et(U, {\Z}/{p^m}(q-1))}
            {\Fil^\log_{D'}H^{q}_\et(U, {\Z}/{p^m}(q-1))}\right\}_{m \ge 1}
      \longrightarrow
H^0_\zar\left(F, \Omega^{q}_X(\log E)(D) \otimes_{\sO_X} \sO_F\right).
  \end{equation}
  Moreover, $\Rsw^{\bullet,q}_{X|(D,D')}$ is functorial with respect to morphisms of
  snc-pairs.
\end{thm}

\thmref{thm:Main-9} has several applications. For instance, Kato's integrality theorem
(cf. \cite[Thm.~7.1]{Kato-89}) follows immediately from it.
One can also define Matsuda's non-log version of Kato's
filtration (cf. \cite{Matsuda}) on $H^q_\et(X, {\Q}/{\Z}(q-1))$ for all $q \ge 1$ (see \cite[\S~ 4.3]{KM-2}). Using \thmref{thm:Main-9}, we can show that Matsuda's filtration is also functorial with respect to morphisms of
  snc-pairs (cf. \cite[\S~4.3]{KM-2}). Most importantly, 
 \thmref{thm:Main-9} is indispensable in the proofs of the main
  results of \cite{KM-2}.

\subsection{Lefschetz theorem for Kato filtration}
\label{sec:Lef**}
Let us now assume that $X \subset \P^n_k$ is a smooth projective connected scheme
of dimension $d$ over an $F$-finite field $k$ and
$E \subset X$ is a simple normal crossing divisor with complement $U$.
We let $D \in \Div_E(X)$ be effective and let $m,q \ge 1$. As another application of
Theorem~\ref{thm:Main-1}, we prove the following
Lefschetz hyperplane theorem for the logarithmic Hodge--Witt cohomology with
modulus and the ramification subgroups $\Fil_D^\log H^{q}_\et(U, {\Z}/{p^m}(q-1))$.
When $q = 1$, part (1) of this theorem was shown earlier by
Kerz--Saito (cf. \cite[Thm.~3.8]{Kerz-Saito-ANT}).
We refer to Definition~\ref{def:ample**} for the definition of sufficient ampleness
(where it is defined under the name Diff$(r,D)$ condition) for a 
hypersurface section.

\begin{thm}\label{thm:Main-5}
  Let $X'$ be a smooth hypersurface section of $X$ in $\P^n_k$
  which intersects $E$ transversely. Let $D' = D \times_X X'$ and
  $U' = X' \setminus D'$. Assume also that $X'$ is sufficiently ample
  (relative to $D$ in $(1)$ and $(2)$). Then the pull-back maps
 \begin{enumerate}
 \item  \hspace{.2cm} $\Fil_D^\log H^{q}_\et(U, {\Z}/{p^m}(q-1)) \to
   \Fil^\log_{D'} H^{q}_\et(U', {\Z}/{p^m}(q-1))$
    \item \hspace{.2cm}
    $H^i_\et(X,W_m\Omega^{q-1}_{X|D,\log}) \to H^i_\et(X', W_m\Omega^{q-1}_{X'|D',\log})$
    \item \hspace{.2cm}
    $H^i_\et(X,W_m\Omega^{q-1}_{X,\log}) \to H^i_\et(X',W_m\Omega^{q-1}_{X',\log})$
   \end{enumerate}
 are isomorphisms (resp. injective) if $q \le d-1$ (resp. $q = d$) in $(1)$ and
 if $i+q \le d-1$ (resp. $i+q=d$) in $(2)$ and $(3)$.
\end{thm}

\subsection{Lefschetz theorems for Brauer group}\label{sec:Lef***}
The Lefschetz theorem for the {\'e}tale fundamental group of
smooth projective schemes over a field is a fundamental result of
Grothendieck (cf. \cite[Exp.~XII, Cor.~3.5]{SGA2}).
To our knowledge, similar results for the Brauer group are not known.
As a consequence of \thmref{thm:Main-5}, we prove the following.

\begin{thm}\label{thm:Main-6}
  Assume in the setting of \thmref{thm:Main-5} that $d \ge 4$.
  Then, for sufficiently ample $X'$, the pull-back map $\Br(X)\{p\} \to \Br(X')\{p\}$
  is an isomorphism.
  If $k$ is either finite, local or separably closed, then the pull-back map
  $\Br(X) \to \Br(X')$ is an isomorphism.
  \end{thm}

We show that the second part of \thmref{thm:Main-6} is valid even when $k$ is a local or an algebraically closed field of characteristic zero. As a consequence of \thmref{thm:Main-6}, we get the following result whose prime-to-$\Char(k)$ part (in particular, characteristic 0 analogue) was shown by Colliot-Th{\'e}l{\`e}ne--Skorobogatov
(cf. \cite[Cor.~5.5.4]{CTS}). This can be considered as a version of Grothendieck--Lefschetz theorem for the Brauer group. 

\begin{thm}\label{thm:Main-7}
  Let $k$ be an arbitrary field and let $X$ be a smooth complete
  intersection in $\P^n_k$ of dimension at least three.
  Then the pull-back map $\Br(k) \to \Br(X)$ is an isomorphism.
\end{thm}

Combining \thmref{thm:Main-7} with \cite[Prop.~3.11]{KM-2}, we obtain the
following vanishing theorem for Brauer group.

\begin{thm}\label{thm:Main-8}
  Let $R$ be a henselian discrete valuation ring with finite residue field.
  Let $k$ be the quotient field of $R$.
  Let $\sX$ be a smooth proper $R$-scheme of relative dimension $d$
  such that its generic fiber is a complete intersection in $\P^n_k$.
  Then $\Br(\sX) =0$ if $d \neq 2$.
\end{thm}

\vskip .3cm

\subsection{Other applications}\label{sec:FAapp}
We now list some applications of Theorems~\ref{thm:Main-0} and ~\ref{thm:Main-1}
which are not included in this paper but are available elsewhere.

Jannsen--Saito-Zhao \cite{JSZ} and Zhao \cite{Zhao} for a fixed modulus.
In particular, we obtain an extension of Milne's duality theorem
\cite{Milne-Zeta} to the world of modulus pairs over finite fields.
These applications will appear in a separate paper. 

(2) In \cite{KM-2}, we give several applications of
Theorems~\ref{thm:Main-0} and ~\ref{thm:Main-1} to the study of wild Brauer classes and
the Brauer--Manin pairing.
The results obtained in op. cit. extend those
of Bright--Newton \cite{Bright-Newton}, Saito--Sato \cite{Saito-Sato-ENS} and
Kai \cite{Kai} to positive characteristic.
  
(3) In \cite{KM-1}, we prove a duality theorem for
the {\'e}tale cohomology of the logarithmic Hodge--Witt sheaves with modulus
on smooth projective curves over local fields. This is first result of its kind and
was not known before. In op. cit., we apply this duality theorem to establish a
class field theory for open curves over local fields (see also \cite{KM-0}).

\subsection{Layout of the paper}\label{sec:Layout}
The proofs of Theorems~\ref{thm:Main-0} and~\ref{thm:Main-1} constitute the heart of
this paper. To prove it, we begin by introducing the filtered de Rham--Witt complex and establishing
its basic properties, except for its goodness, in \S~\ref{sec:FDRWC}. Proving goodness
is the most challenging part of the proofs of Theorems~\ref{thm:Main-0} and
~\ref{thm:Main-1}, and we divide this task into several steps.

In \S~\ref{sec:FPSR}, we first prove the goodness for a power series algebra in one
variable. Our main tool is a presentation of the usual de Rham--Witt complex of such
algebras due to Geisser--Hesselholt \cite{Geisser-Hesselholt-Top}. In
\S~\ref{sec:FDW-Reg}, we extend this to all regular local $F$-finite
$\F_p$-algebras. We do so by reducing to the case of complete algebras and then to
power series algebras in several variables, which we treat separately using the ideas
developed in \S~\ref{sec:FPSR}.

In \S~\ref{sec:CFDRW}, we introduce the Cartier operator on filtered de Rham--Witt
complex. We prove several properties
of the Cartier and Frobenius operators on filtered de Rham--Witt
complex which are classically known for the ordinary de Rham--Witt complex.
We then extend the proof of the goodness of filtered de Rham--Witt
complex to all $F$-finite regular $\F_p$-schemes. This completes the proof
of Theorem~\ref{thm:Main-0}.

In \S~\ref{sec:Kato-comp}, we prove Theorem~\ref{thm:Main-1} for henselian
discrete valuation fields which are $F$-finite. We use this special case to
finish the proof of this theorem in \S~\ref{sec:Kato-global}.
In \S~\ref{sec:RSC}, we apply Theorem~\ref{thm:Main-1} to construct a very
general version of Kato's refined Swan conductor for all regular and $F$-finite
$\F_p$-schemes. This generalized  refined Swan conductor map has many applications.
These applications will be given in \cite{KM-2}. Kato's integrality theorem
\cite{Kato-89} is an immediate consequence of the generalized  refined Swan
conductor. 
In \S~\ref{sec:Lef} and \S~\ref{sec:Lef-Br},
we apply Theorem~\ref{thm:Main-1} to prove Lefschetz
theorems for Kato's ramification subgroups and Brauer groups.

\subsection{Notation}\label{sec:Notn}
For a commutative ring $A$, an $A$-scheme will mean a scheme over $\Spec(A)$. 
We fix a prime $p$ for all of this paper and shall assume all schemes to be Noetherian
and separated $\F_p$-schemes. If we need to talk about a scheme of other
type (e.g., in \S~\ref{sec:Lef}), we shall mention separately.
If $A$ is a commutative ring,
we let $\Sch_A$ denote the category of separated Noetherian $A$-schemes. The product
$X \times_{\Spec(A)} Y$ will be written as $X \times_A Y$. 
For a scheme $X$, we let $X^{(q)}$ (resp. $X_{(q)}$) denote the set of points on $X$
having codimension (resp. dimension) $q$. We let $X_\reg$ (resp. $X_\sing$)
denote the regular (resp. singular) locus of $X$.
We let $\Sm_X$ denote the category of smooth schemes over $X$.

Given a commutative ring $A$, we let $Q(A)$ denote the total ring of fractions of
$A$. We let $\Sch_{A/\zar}$ (resp. $\Sch_{A/\nis}$, resp. $\Sch_{A/ \etl})$
denote the Zariski (resp. Nisnevich, resp. {\'e}tale) site of $\Sch_{A}$.
We let $\epsilon \colon
\Sch_{A/ \et} \to \Sch_{A/ \nis}$ denote the canonical morphism of sites.
If $\sF$ is a sheaf on $\Sch_{A/\nis}$, 
we shall write $\epsilon^* \sF$ also as $\sF$ as long as the usage of
the {\'e}tale topology is clear in a context.
For an $\F_p$-scheme $X$, $\psi \colon X \to X$ will denote the
absolute Frobenius morphism.

For an abelian group $A$, we shall write $\Tor^1_{\Z}(A, {\Z}/n)$ as
$A[n]$ and $A/{nA}$ as $A/n$.
The tensor product $A \otimes_{\Z} B$ will be written as $A \otimes B$.
We shall let $A\{p\}$ (resp. $A\{p'\}$) denote the
subgroup of elements of $A$ which are annihilated by a power of (resp.
an integer prime to) $p$. If $\sF$ is a sheaf on some site of a scheme, we shall write
the kernel (resp. cokernel) of the map $\sF \xrightarrow{n} \sF$ as $_n\sF$
(resp. ${\sF}/n$).
Given a surjective ring homomorphism $A \surj A/I$, we shall denote the image of
an element $a \in A$ in $A/I$ typically by $\ov{a}$ if no confusion arises.
We shall let $\N_0$ denote the set of non-negative integers.

\section{de Rham--Witt complex with log poles}\label{sec:FDRWC**}
In this section, we shall recall the de Rham--Witt complex with log poles on which
our construction of the filtered de Rham--Witt complex will be based.
We shall identify this complex with the de Rham--Witt
complex of a log scheme ({\`a} la Hyodo-Kato) obtained from a normal crossing divisor
on a regular scheme. We shall use this identification to derive some key
properties of the de Rham--Witt complex with log poles which will subsequently be used
in our study of the filtered de Rham--Witt complex.
We refer the reader to \cite{Kato-log} and \cite{Ogus} for the
basic results that we shall use about log schemes in this paper.

\subsection{Definition and Witt complex structure}\label{sec:de-log-0}
Let $X$ be a connected regular Noetherian $\F_p$-scheme
and let $E \subset X$ be a simple normal crossing divisor (snc divisor)
with irreducible
components $E_1, \ldots , E_r$. Let $j \colon U \to X$ be the inclusion of the
complement of $E$ and let $K$ denote the function field of $X$.
We shall refer $(X,E)$ as an {\sl snc-pair} in the sequel.
We let $\Div(X)$ denote the group of Weil divisors on $X$ and let $\Div_Z(X)$
denote the subgroup of $\Div(X)$ consisting of Weil divisors which are supported on
a closed subset $Z \subset X$.
The divisor $E$ endows $X$ with a unique log structure $(X, \sM_X, \alpha_X)$, where
$\alpha_X \colon \sM_X = \sO_X \cap j_*\sO^\times_U \inj \sO_X$ is the usual
inclusion (cf. \cite{Kato-log}). 
We shall denote the resulting log scheme by $X_\log$.
A morphism $f \colon (X',E') \to (X,E)$ of snc-pairs is, by definition, a morphism of
log schemes $f \colon X'_\log \to X_\log$. Under such a morphism, one checks that
$f^*(E_i) \in \Div_{E'}(X')$ for every $i$ and $f(U') \subset U$.

We let $W_m\Omega^{\bullet}_X$ denote the de Rham--Witt complex. We shall denote the
pro-de Rham--Witt complex $\{W_m\Omega^{\bullet}_X\}_{m \ge 1}$
by $W_\star\Omega^{\bullet}_X$. A similar notation will be used for any variant of the
pro-de Rham--Witt complex. We let $[\cdot]_m \colon \sO_X \to W_m\sO_X$ be the
Teichm{\"u}ller map and recall that this is a morphism of monoids under the product
operations on its two sides.
We also recall that the pull-back map $W_m\Omega^q_X \to j_* W_m\Omega^q_U$ is injective
for $q \ge 0$
(cf. \cite[Prop.~2.8]{KP-Comp}), a fact which will be implicit throughout our
discussion.

To define the de Rham--Witt complex with log poles, we suppose first that
$X = \Spec(A)$, where $A$ is a regular $\F_p$-algebra and
$E_i = \Spec({A}/{(x_i)})$, we let $\pi = \stackrel{r}{\underset{i =1}\prod} x_i$
and let $A_\pi = A[\pi^{-1}]$.
We let $S_0 = \{0\}$ and for $1 \le j \le r$, let $S_j$ be the collection of ordered
subsets $J = \{i_1 < \cdots < i_j\}$ of $\{1, \ldots , r\}$.
For $J = \{i_1 < \cdots < i_j\} \in S_j$, we let $W_m\Omega^q_A(\log\pi)_J =
W_m\Omega^{q-j}_A\dlog[x_{i_1}]_m \wedge \cdots \wedge \dlog[x_{i_j}]_m$
as a $W_m(A)$-submodule of $W_m\Omega^q_{A_\pi}$. For $J = S_0$, we let
$W_m\Omega^q_A(\log\pi)_J = W_m\Omega^q_A$. We let $W_m\Omega^q_{X}(\log E)$ be the
image $W_m\Omega^q_A(\log \pi)$ of the canonical map
$\stackrel{{min\{q,r\}}}{\underset{j=0}\bigoplus}
{\underset{J \in S_j}\bigoplus} W_m\Omega^q_A(\log\pi)_J$ $\to W_m\Omega^q_{A_\pi}$.
It is easy to check that the map $d \colon W_m\Omega^q_{A_\pi} \to W_m\Omega^{q+1}_{A_\pi}$
takes $W_m\Omega^q_A(\log \pi)$ to $W_m\Omega^{q+1}_A(\log \pi)$ which yields a
subcomplex $W_m\Omega^\bullet_A(\log \pi) \subset W_m\Omega^\bullet_{A_\pi}$.

 Let now $(X,E)$ be an arbitrary snc-pair.
Since the assignment $X' \mapsto W_m\Omega^\bullet_{X'}(\log E')$ is functorial 
on the category of affine snc-pairs $(X',E')$, it defines a unique subcomplex 
whose term at the $q$-th level is a Zariski subsheaf of $j_*W_m\Omega^q_U$, denoted by
$W_m\Omega^q_X(\log E)$. The stalk of this subsheaf at a point $x \in X$ is
  $W_m\Omega^q_{X_x}(\log  E_x)$, where $X_x = \Spec(\sO_{X,x})$ and
$E_x= X_x \times_X E$. The complex $W_m\Omega^\bullet_{X}(\log E)$ is the desired
de Rham--Witt complex of $X$ with log poles along $E$.
There are functorial inclusions $W_*\Omega^\bullet_X \inj W_*\Omega^\bullet_X(\log E)
\inj j_* W_\star \Omega^\bullet_U$ of pro-Witt complexes on $X_\zar$.
Furthermore, $W_1\Omega^\bullet_X(\log E)$ coincides with the classical de Rham
complex with log poles $\Omega^\bullet_X(\log E)$.

We let $W_\star\Omega^\bullet_{X_\log}$ denote the pro-de Rham--Witt complex of the
log scheme $X_\log$, associated to the snc-pair $(X,E)$ (cf. \cite[\S~4]{Hyodo-Kato},
\cite[\S~1.2]{Geisser-Hesselholt-JAMS}, \cite[\S~3.4]{Matsue}).
Note that $W_1\Omega^\bullet_{X_\log}$ coincides with the de Rham complex
$\Omega^\bullet_{X_\log}$ of the log scheme $X_\log$. 
For $m \ge 1$, the map $\alpha_m = [\cdot]_m \circ \alpha_X \colon \sM_X \to W_m\sO_X$
makes the latter into a sheaf of log rings on $X$. This log structure
on $W_m\sO_X$ allows one to define the notion of Witt complexes and the
universal Witt complex on $X_\log$.  
We refer the reader to \cite[\S~1.2]{Geisser-Hesselholt-JAMS} (see also
\cite[\S~3]{HM-Annals}) for the definitions of Witt complexes and the
universal Witt complex on $X_\log$. We only note that the pro-de Rham--Witt
complex $W_\star\Omega^\bullet_{X_\log}$ is the universal Witt complex on $X_\log$. 
In this paper, a log Witt-complex on a scheme with a log structure (e.g., $X$)
will be another name of a Witt-complex on the resulting log scheme
(e.g., $X_\log$). 

\begin{lem}\label{lem:LWC-1}
  $W_\star\Omega^\bullet_X(\log E)$ is a Witt complex on $X_\log$. In particular,
  there is a canonical (strict) map of Witt complexes
  $\theta_X \colon W_\star\Omega^\bullet_{X_\log} \to W_\star\Omega^\bullet_X(\log E)$.
  \end{lem}
\begin{proof}
  To endow  $W_\star\Omega^\bullet_X(\log E)$ with the structure of a Witt complex
 on $X_\log$, we let $\dlog \circ [\cdot]_m \colon \sM_X \to W_m\Omega^1_X(\log E)$ be
  the map locally defined by $a \mapsto \dlog[a]_m$.
  We take $\lambda \colon (W_\star \sO_X, \sM_X) \to
  (W_\star\Omega^0_X(\log E), \sM_X)$ to
be the identity map. Since $j_* W_\star\Omega^\bullet_U$ is a Witt complex over
$X$ and $W_\star\Omega^\bullet_{X}(\log E)$ is its subcomplex, the lemma will follow
if we verify
that the latter is preserved by the Frobenius and Verchiebung operators of
$j_*W_\star\Omega^\bullet_U$ (cf. \cite[\S~1]{Geisser-Hesselholt-JAMS}).
Since $F$ is multiplicative and $F(\dlog[a]_m) = \dlog[a]_{m-1}$,
the claim is clear for $F$.
Since $V(w\dlog[a]_m) = V(wF(\dlog[a]_{m+1})) = V(w)\dlog[a]_{m+1}$, we 
see that $V$ also preserves $W_\star\Omega^\bullet_{X}(\log E)$.
\end{proof}

\subsection{Cartier isomorphism for $\Omega^\bullet_{X_{\log}}$}\label{sec:C-log}
Before we proceed further, we need to recall the notion of $F$-finiteness.
An $\F_p$-scheme $Y$ is called $F$-finite if the absolute Frobenius map
$\psi \colon Y \to Y$ is a finite morphism. The following result about $F$-finite
rings and schemes is well known. For the definitions of excellent and quasi-excellent
rings, we refer to \cite[Chap.~32. p.~260]{Matsumura}.

\begin{prop}\label{prop:F-fin}
  The following are true for all $q \ge 0, m \ge 1$.
  \begin{enumerate}
  \item
    If $A$ is an $F$-finite $\F_p$-algebra, then every essentially of finite type
    $A$-scheme is $F$-finite. The henselization and completion of $A$ along any ideal
    are $F$-finite. The strict henselization of $A$ at any maximal ideal is $F$-finite.
    \item
      An $F$-finite Noetherian ring is excellent. 
    \item
      If $A$ is a quasi-excellent Noetherian local $\F_p$-algebra, then it is
      $F$-finite if and only if its residue field is $F$-finite.
    \item
      If $A$ is a complete Noetherian local $\F_p$-algebra, then it is
      $F$-finite if and only if its residue field is $F$-finite.
    \item
      If $A$ is Noetherian and $F$-finite, then so is $W_m(A)$ and
      $W_m\Omega^q_A$ is a finitely generated $W_m(A)$-module.
    \item
      If $A$ is Noetherian local $F$-finite $\F_p$-algebra with completion
      $\wh{A}$, then the canonical map $W_m\Omega^q_A \otimes_{W_m(A)} W_m(\wh{A}) \to
      W_m\Omega^q_{\wh{A}}$ is an isomorphism.
      \item
      If $A$ is a regular local $F$-finite $\F_p$-algebra with maximal ideal
      $\fm = (x_1, \ldots , x_n)$ and residue field $\ff$, then $\Omega^1_A$ is a free
      $A$-module of finite rank. If $y_1, \ldots , y_s \in A$ are such that
      $\{d\ov{y}_1, \ldots , d\ov{y}_s\}$ forms a basis of $\Omega^1_{\ff}$
      as an $\ff$-vector space, then $\{dy_1, \ldots , dy_s, dx_1, \ldots , dx_n\}$ is
      a basis for $\Omega^1_A$ as an $A$-module.
\end{enumerate}
\end{prop}
\begin{proof}
Items (1) and (5) follow from \cite[Thm.~3.6, Lem.~3.8, 3.10]{Dundas-Morrow}
  using the surjection $\Omega^q_{W_m(A)} \surj W_m\Omega^q_A$. Item (2) follows
  from \cite[Thm.~2.5]{Kunz-AJM} and (4) follows from the Cohen structure theorem.
  Item (3) follows from (4) and \cite[Cor.~22]{Hashimoto}. Item (6) follows from
  \cite[Lem.~2.11]{Morrow-ENS} and item (7) follows from \cite[Lem.~7.2]{Kato-89}. 
\end{proof}

We now return to the set-up of ~\ref{sec:de-log-0}.
We assume that the snc-pair $(X,E)$ is local in that 
$X = \Spec(A)$, where $A$ is a
regular local $F$-finite $\F_p$-algebra of Krull dimension $n$. Let $\fm =
(x_1, \ldots , x_n)$ denote the maximal ideal of $A$. We let
$E_i = \Spec({A}/{(x_i)})$ and let $\pi = \stackrel{r}{\underset{i =1}\prod} x_i$,
where $1 \le r \le n$. We let $A_\log$ denote the log ring structure on $A$
corresponding to the log scheme $X_\log$.

\begin{lem}\label{lem:LWC-2}
The canonical maps of log differential graded $A$-algebras
$\Omega^\bullet_{A_\log} \xrightarrow{\theta_A}\Omega^\bullet_A(\log \pi)$ \
$\xleftarrow{\phi_A}\Wedge^\bullet_A (\Omega^1_A(\log\pi))$
are isomorphisms.  Moreover, each $\Omega^q_A(\log\pi)$ is a free $A$-module of
finite rank.
\end{lem}
\begin{proof}
  This is standard using \propref{prop:F-fin}(7) and the structure of
  $\Omega^\bullet_{A_\log}$ (cf. \cite[Thm.~1.2.4]{Ogus}).
\end{proof}

To prove the next result, we let $\F_p^{[m]} = \F_p[X_1, \ldots , X_m]$ for
$m \in \N$. There are then homomorphisms of log rings
\begin{equation}\label{eqn:Log-smooth}
\F_{p,\log} \to (\F_p^{[n]})_\log \xrightarrow{\gamma} A_\log, \ \mbox{where} \
\gamma(X_i) = x_i \ \forall \ 1 \le i \le n.
\end{equation}
The log structure of $\F_p$ is trivial while 
that of $\F_p^{[n]}$ (resp. $A$) is given by the snc
divisor $V((X))$ (resp. $V((\pi))$) on $\A^n_{\F_p}$ (resp. $\Spec(A)$), where
$X = X_1\cdots X_r$. The first arrow in ~\eqref{eqn:Log-smooth} is log smooth
while the second arrow is a morphism of log rings such that the underlying morphism of
rings is regular. To see the latter claim, first note that $A$ is formally
$\fm$-smooth over $(\F_p^{[n]})_{(X_1,\cdots,X_n)}$, as $\F_p$ is perfect and $A$ is
regular (cf. \cite[Lem.~1.3]{Tanimoto-2}). Since $A$ is excellent
by \propref{prop:F-fin}(2), it follows now from Andr{\'e}'s theorem
(cf. \cite[p.~260]{Matsumura}) that $\F_p^{[n]} \to A$ is regular.
We let $H_{\dR}^i(A_\log)$ denote the $i$-th cohomology of the log de Rham complex
$\Omega^\bullet_{A_\log}$.

\begin{lem}\label{lem:Log-Cartier}
    For $q \ge 0$, there is an isomorphism of $A$-modules
    $C^{-1}_q \colon \Omega^q_{A_\log} \xrightarrow{\cong} H_{\dR}^q(A_\log)$.
  \end{lem}
 \begin{proof}
   Since the map $\gamma \colon \F_p^{[n]} \to A$ is regular, it follows from
   \cite[Thm.~2.5]{Popescu} that there is a direct system $\{A(\lambda)\}$ of 
    smooth (in particular, finitely generated)
    $\F_p^{[n]}$-algebras such that $A = \varinjlim_\lambda A(\lambda)$.
    We endow each $A(\lambda)$ with the log structure given by the pull-back of the
    log structure of $(\F_p^{[n]})_\log$. Since the log structure of $A$ also has this
    property, it follows that we have a direct system of
    log smooth $(\F_p^{[n]})_{\log}$-algebras $\{A(\lambda)_\log\}$ such that
    $A_\log = \varinjlim_\lambda A(\lambda)_\log$ in the category of log rings.
    As $\F_{p,\log} \to (\F_p^{[n]})_\log$ is log smooth, we get that each
    $A(\lambda)_\log$ is log smooth over $\F_{p,\log}$.

    It is easy to check directly that each map
    $F_{p,\log} \to A(\lambda)_\log$ is integral and saturated, in particular
    of Cartier type (cf. \cite[\S~2.12]{Hyodo-Kato}).
    We conclude from op. cit. that there is an isomorphism
    $C^{-1}_q \colon \Omega^q_{A(\lambda)_\log} \xrightarrow{\cong}
    H^q_\dR(A(\lambda)_{\log})$ given by $C^{-1}_q(a\dlog(b_1)\wedge \cdots \wedge
    \dlog(b_q)) = a^p\dlog(b_1)\wedge \cdots \wedge \dlog(b_q)$.
Since this is clearly natural with respect to the homomorphisms of log rings, we
    get after passing to the limit that there is a similar inverse Cartier isomorphism
    for $A$ as well. It is clear from the description that $C^{-1}_q$ is $A$-linear
    if we endow $\Ker(d)$ with the $A$-module structure via latter's Frobenius
    endomorphism.
\end{proof}

We let $Z_1\Omega^q_{A_\log} = \Ker(d \colon \Omega^q_{A_\log} \to \Omega^{q+1}_{A_\log})$
  and $B_1\Omega^q_{A_\log} = {\rm Image}(d \colon \Omega^{q-1}_{A_\log} \to
  \Omega^{q}_{A_\log})$. Taking the inverse of $C^{-1}_\bullet$, we get the following.
  
\begin{cor}\label{cor:Log-Cartier-0}
  For $q \ge 0$, there is an $A$-linear Cartier homomorphism
  $C_q \colon Z_1\Omega^q_{A_\log} \to \Omega^q_{A_\log}$ which induces an
  isomorphism of $A$-modules $C_q \colon H^q_\dR({A_\log}) \xrightarrow{\cong}
  \Omega^q_{A_\log}$.
\end{cor}

One can now define $A$-modules $Z_i\Omega^q_{A_\log}$ and $B_i\Omega^q_{A_\log}$ as
in the non-log case (cf. \cite[Chap.~0, (2.2.3)]{Illusie}) so that there is a
filtration 
\[
  0 = B_0\Omega^q_{A_\log} \subseteq B_1\Omega^q_{A_\log} \subseteq \cdots \subseteq \
  \bigcup_i B_i\Omega^q_{A_\log} \subseteq \ \bigcap_i Z_{i}\Omega^q_{A_\log}  \subseteq
  \cdots \subseteq 
 Z_{1}\Omega^q_{A_\log} \subseteq Z_{0}\Omega^q_{A_\log} = \Omega^q_{A_\log}
\]
on $\Omega^q_{A_\log}$ together with an isomorphism
$C^{-i}_q \colon \Omega^q_{A_\log} \xrightarrow{\cong}
{Z_{i}\Omega^q_{A_\log}}/{B_i\Omega^q_{A_\log}}$ for each $i \ge 0$. Using this
isomorphism and mimicking the proof of \cite[Lem.~1.14]{Lorenzon}, one gets the
following.

\begin{cor}\label{cor:Log-cartier-1}
For $q \ge 0$ and $m  \ge 1$, there are exact sequences
  \begin{equation}\label{eqn:log-cartier-1.2}
      0 \to B_{m}\Omega^q_{A_\log} \to B_{m+1}\Omega^q_{A_\log} \xrightarrow{C^m}
      B_1\Omega^{q}_{A_\log} \to 0;
    \end{equation}
    \begin{equation}\label{eqn:log-cartier-1.1}
      0 \to Z_{m+1}\Omega^q_{A_\log} \to Z_m\Omega^q_{A_\log} \xrightarrow{dC^m}
      B_1\Omega^{q+1}_{A_\log} \to 0.
    \end{equation}
    \end{cor}

\subsection{Identification with $W_\star\Omega^\bullet_{X_{\log}}$}\label{sec:Comparison}
The key fact we need to identify $W_\star\Omega^\bullet_A(\log \pi)$ with
$W_\star\Omega^\bullet_{A_{\log}}$ is the following.

\begin{lem}\label{lem:LWC-4}
  For every $q \ge 0$ and $m \ge 1$, the canonical map
  $W_m\Omega^q_{A_\log} \to W_m\Omega^q_{A_\pi}$ is injective.
\end{lem}
 \begin{proof}
  We shall prove the lemma by induction on $m$. The base case $m =1$ follows from
  \lemref{lem:LWC-2}. 
For $q, r \ge 0$ and $m \ge 1$, we let
$\Fil^r_V W_m\Omega^q_{A_\log} = V^r(W_{m-r}\Omega^q_{A_\log}) +
dV^r(W_{m-r}\Omega^{q-1}_{A_\log}) \subset W_m\Omega^q_{A_\log}$.
This yields the (decreasing) V-filtration on $W_m\Omega^q_{A_\log}$ as a
$W_m(A)$-module. Furthermore, we have an exact sequence
(cf. \cite[Lem.~3.2.4]{HM-Annals}) 
\begin{equation}\label{eqn:LWC-4-0}
  0 \to \Fil^r_V W_m\Omega^q_{A_\log} \to W_m\Omega^q_{A_\log}
  \xrightarrow{R^{m-r}} W_r\Omega^q_{A_\log} \to 0,
  \end{equation}
which is natural in $A_\log$. 

We let $R^q_{r}(A_\log) = \Ker((V^{r}, dV^{r}): \Omega^q_{A_\log} \oplus
\Omega^{q-1}_{A_\log} \to W_{r+1}\Omega^q_{A_\log})$.
By comparing the above exact sequence with the analogous exact sequence for $A_\pi$
and using induction on $m$, it suffices to show that
the square
\begin{equation}\label{eqn:LWC-4-1}
  \xymatrix@C1pc{
    R^q_{m-1}(A_\log) \ar[r] \ar[d] & \Omega^q_{A_\log} \oplus \Omega^{q-1}_{A_\log}
    \ar[d] \\
    R^q_{m-1}(A_{\pi}) \ar[r]  & \Omega^q_{A_\pi} \oplus \Omega^{q-1}_{A_\pi}}
  \end{equation}
is Cartesian for $m \ge 2$.

Suppose now that there is an element $\alpha \in
(\Omega^q_{A_\log} \oplus \Omega^{q-1}_{A_\log}) \bigcap R^q_{m-1}(A_{\pi})
\subset (\Omega^q_{A_\pi} \oplus \Omega^{q-1}_{A_\pi})$. Then
$\alpha \in  B_{m}\Omega^q_{A_\pi} \oplus Z_{m-1}\Omega^{q-1}_{A_\pi}$
by \cite[(1.15.3)]{Lorenzon}. We write $\alpha = (\alpha_0, \alpha_1)$ with
$\alpha_0 \in B_{m}\Omega^q_{A_\pi} \bigcap \Omega^q_{A_\log}$ and
$\alpha_1 \in  Z_{m-1}\Omega^{q-1}_{A_\pi} \bigcap \Omega^{q-1}_{A_\log}$.
We claim that $\alpha \in B_{m}\Omega^1_{A_\log} \oplus Z_{m-1}\Omega^1_{A_\log}$.

To prove the claim, we first note that $B_1\Omega^q_{A_\log} \subset \Omega^q_{A_\log} \inj
\Omega^q_{A_\pi}$  by \lemref{lem:LWC-2}, and this implies that
the map $B_1\Omega^q_{A_\log} \to B_1\Omega^q_{A_\pi}$ is injective. 
We next note that as $\Omega^{\bullet}_{A_\log} \inj \Omega^{\bullet}_{A_\pi}$,
our assumption already implies that
$\alpha \in Z_1\Omega^q_{A_\log} \oplus Z_{1}\Omega^{q-1}_{A_\log}$.
Using \cite[(1.14.2)]{Lorenzon},  \eqref{eqn:log-cartier-1.1} and induction on $m$,
we conclude that
$\alpha_1 \in Z_{m-1}\Omega^1_{A_\log}$. Since $\alpha_0 \in B_{m}\Omega^q_{A_\pi}
\bigcap \Omega^q_{A_\log} \subset Z_{i}\Omega^q_{A_\pi}$
for every $i \ge 1$, the same argument shows that
$\alpha_0 \in Z_m\Omega^q_{A_\log}$.

In the commutative diagram of short exact sequences
\begin{equation}\label{eqn:LWC-4-2}
  \xymatrix@C.8pc{
    0 \ar[r] & B_{m}\Omega^q_{A_\log} \ar[r] \ar[d] & Z_m\Omega^q_{A_\log} \ar[r] \ar[d] &
    H^q_\dR(A_\log) \ar[r] \ar@{^{(}->}[d] & 0 \\
     0 \ar[r] & B_{m}\Omega^q_{A_\pi} \ar[r] & Z_m\Omega^q_{A_\pi} \ar[r] &
    H^q_\dR(A_\pi) \ar[r] & 0,}
\end{equation}
the right vertical arrow is injective by the $m =1$ case of the lemma and
\corref{cor:Log-Cartier-0}. A diagram chase implies that
$\alpha_0 \in B_{m}\Omega^q_{A_\log}$. This proves the claim.

We now use \cite[(1.15.3)]{Lorenzon} for $A_\log$ and $A_\pi$
(together with a limit argument as in \lemref{lem:Log-Cartier}), giving us a
commutative diagram of short exact sequences
\begin{equation}\label{eqn:LWC-4-3}
  \xymatrix@C.8pc{
    0 \ar[r] & R^q_{m-1}(A_\log) \ar[r] \ar[d] &
    B_{m}\Omega^1_{A_\log} \oplus Z_{m-1}\Omega^1_{A_\log} \ar[r] \ar[d] & B_1\Omega^q_{A_\log}
    \ar[r] \ar[d] & 0 \\
   0 \ar[r] & R^q_{m-1}(A_\pi) \ar[r] &
    B_{m}\Omega^1_{A_\pi} \oplus Z_{m-1}\Omega^1_{A_\pi} \ar[r] & B_1\Omega^q_{A_\pi}
    \ar[r] & 0.}
  \end{equation}
Since the right vertical arrow is injective, we conclude from the above claim 
and a diagram chase that $\alpha \in R^q_{m-1}(A_\log)$. This proves that
~\eqref{eqn:LWC-4-1} is Cartesian and completes the proof of the lemma.
\end{proof}

We can now prove the main result of this section.

\begin{thm}\label{thm:Log-DRW}
  Let $X$ be a regular $F$-finite $\F_p$-scheme and let $X_\log$ denote the
  log scheme associated to an snc divisor $E \subset X$.
Then the canonical map
of Witt complexes $\theta_X \colon W_\star\Omega^\bullet_{X_\log} \to
W_\star\Omega^\bullet(\log E)$ on $X_\log$ is a levelwise (wrt $m \ge 1$) isomorphism.
\end{thm}
\begin{proof}
 We can assume $X = \Spec(A)$ with $A$ as in \lemref{lem:LWC-2}.
  We now look at the commutative diagram of Witt complexes on $A_\log$:
  \begin{equation}\label{eqn:Log-DRW-0}
    \xymatrix@C1pc{
      W_m\Omega^\bullet_A(\log\pi) \ar@{.>}[rr] \ar@{^{(}->}[dr] & &
      W_m\Omega^\bullet_{A_\log} \ar@{^{(}->}[dl] \\
      & W_m\Omega^\bullet_{A_\pi}. &}
    \end{equation}

The diagonal arrow on the right is injective by \lemref{lem:LWC-4}.
  It is clear from the definition of $W_m\Omega^\bullet_A(\log\pi)$ that it lies in
  $W_m\Omega^\bullet_{A_\log}$ as a Witt subcomplex of $W_m\Omega^\bullet_{A_\pi}$.
  In particular, the inclusion
  $W_m\Omega^\bullet_A(\log\pi) \inj W_m\Omega^\bullet_{A_\pi}$ uniquely factors through
  an inclusion of Witt complexes
  $\phi_A \colon W_m\Omega^\bullet_A(\log\pi) \to W_m\Omega^\bullet_{A_\log}$.
  Since $W_m\Omega^\bullet_{A_\log}$ is the universal Witt-complex over $A_\log$,
  this forces $\phi_A$ and $\theta_A$ to be the inverses of each other.
\end{proof}

\begin{cor}\label{cor:Log-DRW-0}
  Let $X$ be as in \thmref{thm:Log-DRW}. Then for every $q, r,s \ge 0$ with
  $r+s = m \ge 1$, there is a short exact sequence
  of $W_m\sO_X$-modules
  \[
  0 \to V^r(W_{s}\Omega^q_{X}(\log E)) + dV^r(W_{s}\Omega^{q-1}_X(\log E)) \to
    W_m\Omega^q_X(\log E) \xrightarrow{R^{s}} W_r\Omega^q_X(\log E) \to 0.
    \]
\end{cor}
\begin{proof}
  Combine \cite[\S~3]{HM-Annals} and \thmref{thm:Log-DRW}.
  \end{proof}

Let $(X,E)$ be an snc-pair with $X$ $F$-finite. Suppose that $u \colon X' \inj X$ is the inclusion of a regular connected divisor such that $E' := X' \times_X E$ is an snc divisor on  $X'$. We shall use the following well-known lemma later in the paper.

\begin{lem}\label{lem:Kahler-seq}
The following sequences are exact for every $q \ge 0$.
  \begin{enumerate}
 \item
$0 \to \Omega^q_X(\log E) \to \Omega^q_X(\log (E + X')) \xrightarrow{\res}
    u_* \Omega^{q-1}_{X'}(\log E') \to 0$.
  \item
$0 \to \Omega^q_X(\log (E + X'))(-X') \to \Omega^q_X(\log E) \to
    u_* \Omega^q_{X'}(\log E') \to 0$.
  \end{enumerate}
\end{lem}
\begin{proof}
  We assume $q \ge 1$ as the lemma is otherwise obvious.
  By \propref{prop:F-fin}, we first note that $\Omega^1_X$ (resp. $\Omega^1_{X'}$)
  is a locally free sheaf on $X$ (resp. $X'$) and that $\rank(\Omega^1_{X}) =
  \rank(\Omega^1_{X'}) + 1 < \infty$.
  It is now easy to define the residue map
  $\Omega^q_X(\log (E + X')) \xrightarrow{\res} u_* \Omega^{q-1}_{X'}(\log E')$ locally by requiring $\res (\omega \dlog(f))=\bar\omega$ and $\res (\omega')=0$, where $\omega$ (resp. $\omega'$) is a local section of $\Omega^{q-1}_X(\log E)$ (resp. $\Omega^{q}_X(\log E)$), $\bar\omega$ denotes its image in $\Omega^{q-1}_{X'}(\log E')$, and $f$ is a local equation of $X'$ in $X$. This glues to a global definition (cf. \cite[Tag 0FMU]{SP}).
The exactness of (1) is now an easy consequence of \propref{prop:F-fin}(7).
 
 To prove (2), we use the fundamental exact sequence 
 \[
  \sO_{X'}(-X') \xrightarrow{\ov{d}} u^* \Omega^1_X(\log E) \xrightarrow{u^*}
  \Omega^1_{X'}(\log E') \to 0
  \]
  of log differentials (cf. \cite[Prop.~IV.2.3.2]{Ogus}).
  All terms of this sequence are locally free sheaves
  on $X'$ such that the rank of the middle term is the sum of the ranks of the other
  two terms. This implies that $\ov{d}$ must be injective.
Taking the higher exterior powers, we thus get a short exact sequence
  \begin{equation}\label{eqn:Kahler-seq-0}
    0 \to \Omega^{q-1}_{X'}(\log E')(-X') \xrightarrow{\ov{d}}
    u^* \Omega^q_X(\log E)
    \to \Omega^q_{X'}(\log E') \to 0.
  \end{equation}
   From the above exact sequence, it follows that the map $\Omega^q_X(\log E) \to
    u_* \Omega^q_{X'}(\log E')$ is surjective, and that its kernel is $\Omega^q_X(\log E)(-X')+d( \Omega^{q-1}_{X}(\log E)(-X'))=\Omega^q_X(\log (E + X'))(-X')$.
\end{proof}

\section{The filtered de Rham--Witt complex}\label{sec:FDRWC}
In this section, we define the filtered de Rham--Witt complex of an $F$-finite
regular $\F_p$-scheme with respect to an snc divisor. We 
also derive the basic properties of this complex. We begin by considering
a special case.

Given an $\F_p$-scheme $X$ and
$D  = \stackrel{r}{\underset{i =1}\sum} n_i D_i \in
\Div(X)$, we let $|D|$ denote the support of $D$ and
let $D/p = \stackrel{r}{\underset{i =1}\sum} \lfloor{{n_i}/p}\rfloor D_i$,
where $\lfloor{\cdot}\rfloor$ is the greatest integer function.
For $m \ge 1$, we let
$W_m(X) = W_m(\Gamma(X, \sO_X))$. The elements of $W_m(X)$ will be denoted by
$m$-tuples $\un{a} = (a_{m-1}, \dots , a_0)$ with $a_i \in \Gamma(X, \sO_X)$.
If $I \subset A$ is an ideal in a commutative ring, we let $W_m(I) =
\{\un{a} \in W_m(A)|a_i \in I \ \forall \ i\}$.

\enlargethispage{30pt}

\subsection{Filtration on $W_m \sO_U$}\label{sec:Fil-O}
Let $X$ be an integral and locally factorial $\F_p$-scheme and 
$D = \stackrel{r}{\underset{i =1}\sum} n_i D_i \in \Div(X)$. Let
$j \colon U = X \setminus E \inj X$ be the inclusion.
The following definition is inspired by (though slightly different from)
\cite[\S~3, p.~695]{Kerz-Saito-ANT}. When $X$ is the spectrum of a henselian discrete
valuation ring, this definition is originally due to Brylinski \cite{Brylinski}.

\begin{defn}\label{defn:Log-fil-0}
We let
      \[
        \Fil_D W_m({U}) = \left\{\underline{a}=(a_{m-1},\ldots, a_0) \in W_m({U})| \
        a_i^{p^i} \in \Gamma({X}, \sO_{X}({D})) \ \forall \ i\right\}.
      \]
      We let $\Fil_D W_m \sO_U$ be the presheaf on $\Sm_X$ such that
      $\Fil_D W_m \sO_U({X'}) = \Fil_{D'} W_m(U')$, where
      $D' = X' \times_X D \in \Div(X')$ and $U' = X' \times_X U$.
\end{defn}

One easily checks that $\Fil_D W_m \sO_U$ actually restricts to a sheaf on $X_\zar$.
If $X$ is integral, we moreover have the inclusions of Zariski sheaves
      $\Fil_D W_m \sO_U \subset \Fil_{D'} W_m \sO_W \subset h_* W_m \sO_W$
      whenever $D \le D'$ and $h \colon W = X \setminus |D'| \inj X$ is the
      inclusion. Note also that the canonical map
$j^*(\Fil_D W_m \sO_U) \to W_m \sO_U$ is an isomorphism. 

      We let $\iota \colon \{\eta\} \inj X$ denote the
inclusion of the generic point and let $K = k(\eta)$. We
shall write the constant sheaf $\iota_* W_m\Omega^\bullet_{\eta}$ on $X$
simply as $W_m\Omega^\bullet_K$.
      If $P \in X$ and $x_i$ is a generator of the ideal of
$D_i$ in $\sO_{X,P}$, then one checks that
  $a_i^{p^i} \in (\sO_{X}(D))_P=\sO_{X,P}.(\prod_j
          x^{-n_j}_j)$ iff $a_i^{p^i} (\prod_j
          x^{n_j}_j) \in \sO_{X,P}$. Letting  $\pi = \prod_j x_j$, we therefore
         obtain 
\begin{equation}\label{eqn:Log-fil-1*}
  (\Fil_D W_m \sO_U)_P
          =
          \left\{\un{a} \in W_m(\sO_{X,x}[\pi^{-1}])| \ a_i^{p^i} ({\prod}_j
          x^{n_j}_j) \in \sO_{X,P} \ \forall \ i\right\}.
\end{equation}

If $X = \Spec(A)$, where $A$ is a local ring and $x_1, \ldots , x_r$ are
generators of the ideals of the irreducible components of $D$, we shall write
$\Fil_D W_m \sO_U$ as $\Fil_I W_m(A_\pi)$ or $\Fil_{\un{n}} W_m(A_\pi)$, where
$\pi = \prod_i x_i$ and
$I \subset Q(A)$ is the invertible ideal $(x)$, where $x = \prod_i x^{n_i}_i$ and
$\un{n}=(n_1, \ldots, n_r) \in \Z^r$. We shall write $\sO_X(D)$ as $A{x^{-1}}$.

\begin{lem}\label{lem:Log-fil-1}
    $\Fil_DW_m\sO_U$ is a sheaf of abelian groups. There is an exact sequence
   \begin{equation}\label{eqn:Log-fil-1.1}
     0 \to \Fil_D\sO_U \xrightarrow{V^{m-1}} \Fil_DW_m\sO_U \xrightarrow{R}
     \Fil_{D/p}W_{m-1}\sO_U\to 0.
   \end{equation}
\end{lem}
\begin{proof}
  We may assume $X= \Spec A$ and, following the same notation as above, may let $D$ be
  given by the ideal $(x)$, where $x= \prod_i x^{n_i}_i$.
  We only need to show that the map $R$ is well-defined and surjective and
  $\Fil_DW_m\sO_U$ is a sheaf of abelian groups. The remaining assertions follow from
  \eqref{eqn:Log-fil-1*}.  We let $\un{a} \in {\Fil_IW_{m}(A_\pi)}$
  and let $I^{1/p}= (\prod_i x^{\lfloor{{n_i}/{p}}\rfloor}_i)$.
  Then $R(\un{a}) = \un{b} = (b_{m-1}, \ldots , b_0)$, where $b_i = a_{i+1}$. We can
  write (because $A$ is a UFD) $a_{i+1} = a'_{i+1}\prod_jx^{t_j}_j$, where $t_j \in \Z$ for
  every $j$ and ${a'_{i+1}} \in A$ is not divisible by any $x_j$.

  By our hypothesis, $(a_{i+1})^{p^{i+1}} x\in A$. This implies that
  $p^{i+1}t_j+n_j \ge 0$, which is equivalent to
  $p^{i}t_j + \lfloor n_j/p \rfloor \ge 0$. So we get $(b_i)^{p^i}x' \in A$ for
$0 \le i \le m-1$, where $x' = \prod_j x^{\lfloor{{n_j}/{p}}\rfloor}_j$. This implies that
  $\un{b} \in \Fil_{I^{1/p}}W_{m}(A_\pi)$. Hence, $R$ is well-defined. Also, given any
  $(b_{m-2}, {\ldots},b_0) \in \Fil_{I^{1/p}}W_{m}(A_\pi)$, we can 
  reverse the above argument to get
  $(b_{m-2}, {\ldots}, b_0, 0) \in \Fil_IW_m(A_\pi)$. This shows that $R$ is surjective.
  
Now we show that $\Fil_DW_m\sO_U$ is a Zariski sheaf of abelian groups by
  induction on $m$.
  If $m=1$, then $\Fil_D\sO_U=\sO_X(D)$ is clearly a sheaf of  abelian groups
  (actually of $\sO_X$-modules). If $m >1$, we let
  $\un{a}=(a_{m-1}, {\ldots}, a_0) \ \mbox{and} \ \un{a}'=
  (a'_{m-1}, \ldots, a'_0) \ \mbox{be \ in} \ \Fil_DW_m\sO_U$. Enough to show
  $\un{a}-\un{a}' \in \Fil_DW_m\sO_U$. Since $R$ is additive, by induction and the
  above argument, we have $R(\un{a}-\un{a}') \in \Fil_{D/p}W_{m-1}\sO_U$. So, we can
  write $R(\un{a}-\un{a}')= (c_{m-2}, {\ldots}, c_0)$ satisfying the relation
  \eqref{eqn:Log-fil-1*}. Then the previous argument shows
  $(c_{m-2}, \ldots, c_0, 0) \in \Fil_DW_m\sO_U$. Now note that $\un{a}-\un{a}'=
  (c_{m-2}, \ldots, c_0, a_0-a'_0)$ which is clearly in $\Fil_DW_m\sO_U$ (because,
  $a_0-a'_0 \in \Fil_D\sO_U$).
\end{proof}

\begin{lem}\label{lem:Log-fil-2}
  For two Weil divisors $D$ and $D'$ with complements $U$ and $U'$, respectively, we
  have $\Fil_D W_m \sO_{U} + \Fil_{D'} W_m \sO_{U'} \subset \Fil_{D+D'} W_m \sO_{U\cap U'}$
  and $\Fil_D W_m \sO_{U} \cdot \Fil_{D'} W_m \sO_{U'} \subset
  \Fil_{D+D'} W_m \sO_{U\cap U'}$ under the addition and multiplication operations of
  $W_m(K)$. In particular, $\Fil_D W_m \sO_U$ is a subsheaf of $W_m\sO_X$-modules inside
  $W_m(K)$.
  \end{lem}
\begin{proof}
Since we can check this locally, we assume that $X = \Spec(A)$, where $A$ is
  a local ring. We let $\{x_1, \ldots , x_r\}$ and $\{y_1, \ldots , y_s\}$ be the
  generators of the ideals of the irreducible components of $D$ and $D'$,
  respectively.
We let $x = \prod_i x^{m_i}_i$ and $y = {\prod}_j y^{n_j}_j$, where
  $D = \sum_i m_iD_i$ and $D' = \sum_j n_jD'_j$. Write $I = (x)$ and $I' = (y)$.
  Suppose $\un{a} \in \Fil_IW_m(A_x)$ and $\un{b} \in \Fil_{I'}W_m(A_y)$.
  We can write $\un{a} = \stackrel{m-1}{\underset{i =0}\sum}V^{i}([a_{m-1-i}]_{m-i})$ and
  $\un{b} = \stackrel{m-1}{\underset{i = 0}\sum}V^i([b_{m-1-i}]_{m-i})$.
  This yields  ${\un{a}}\cdot {\un{b}} =
  \stackrel{m-1}{\underset{i,j = 0}\sum}V^i([a_{m-1-i}]_{m-i})V^j([b_{m-1-j}]_{m-j})$.

  Using the identity $xV(y)=V(Fx.y)$, we have for $0 \le i,j \le m-1$,
  \begin{equation}\label{eqn:Log-fil-2-0}
    V^i([a_{m-1-i}]_{m-i})V^j([b_{m-1-j}]_{m-j})=\left\{ \begin{array}{ll}
                            0 & \mbox{if $i+j \ge m$} \\
                            V^{i+j}([a_{m-1-i}^{p^j}b_{m-1-j}^{p^i}]_{m-i-j}) & \mbox{if $i+j <m$}.
                          \end{array}
                        \right.
    \end{equation}
By ~\eqref{eqn:Log-fil-1*}, the condition $\un{a} \in \Fil_IW_m(A_x)$ and
  $\un{b} \in \Fil_{I'} W_m(A_y)$ means that
$xa^{p^i}_i \in A$ and $yb^{p^i}_i \in A$ for each $i$.
From this, we get for $i+j<m$ that
$ (a_{m-1-i}^{p^j}b_{m-1-j}^{p^i})^{p^{m-1-i-j}} xy= a^{p^{m-1-i}}_{m-1-i}b^{p^{m-1-j}}_{m-1-j}xy\in A$.
This implies by ~\eqref{eqn:Log-fil-2-0} that
$V^i([a_{m-1-i}]_{m-i})V^j([b_{m-1-j}]_{m-j})\in \Fil_{II'}W_m(A_{xy})$. By
\lemref{lem:Log-fil-1}, we conclude that
${\un{a}}\cdot {\un{b}} \in \Fil_{II'}W_m(A_{xy})$. Also, the claim about the addition
immediately follows from \lemref{lem:Log-fil-1} by observing that
$\Fil_D W_m \sO_{U}$ and $\Fil_{D'} W_m \sO_{U'})$ are contained in
$\Fil_{D+D'} W_m \sO_{U \cap U'}$.  
This concludes the proof of the first part of the lemma. The second part follows from
the first by taking $D' = 0$ and noting that $\Fil_0W_m\sO_X = W_m\sO_X$.
\end{proof}

\enlargethispage{20pt}

For $D$ as in Definition~\ref{defn:Log-fil-0}, let $W_m \sO_X(D)$ be the 
sheaf on $X_\zar$ which is locally defined as
$W_m \sO_X(D)  = [x^{-1}]_m \cdot W_m \sO_X \subset j_* W_m\sO_U$,
where $[-]_m \colon j_*\sO_U \to j_*W_m\sO_U$ is the Teichm{\"u}ller map,
and $D$ is locally defined by the invertible ideal $(x) \subset K$. One knows that
$W_m \sO_X(D)$ is a sheaf of invertible $W_m\sO_X$-modules. For a sheaf of
$W_m\sO_X$-modules $\sM$, we let $\sM(D) = \sM \otimes_{W_m\sO_X} W_m\sO_X(D)$. If $A$
is a local ring and $D$ is defined by $I = (x)$, we also write $\sM(D)$ as $M(x^{-1})$
if $\sM$ is defined by a $W_m(A)$-module $M$.

\begin{lem}\label{lem:Log-fil-2-1}
  One has the identity $\Fil_{p^{m-1}D}W_m\sO_U = W_m\sO_X(D)$.
\end{lem}
\begin{proof}
Since we can check this locally, we assume that $X = \Spec(A)$ is as in the proof of
  \lemref{lem:Log-fil-2}. We let $I = (x^{p^{m-1}})$.
 Following our notations, we see that
 $\un{a} = (a_{m-1}, \ldots, a_0) \in \Fil_{I}W_m(A_x)$ if and only if
 $a_i^{p^i}x^{p^{m-1}} \in A $ for all $0\le i\le m-1$. Since $A$ is a UFD, this is
 equivalent to $a_i x^{p^{m-1-i}} \in A$ for all $0\le i\le m-1$, or equivalently,
 $\un{a}\cdot [x]_m = {(a_{m-1}x, \dots, a_ix^{p^{m-1-i}}, \dots, a_0x^{p^{m-1}}) \in W_m(A)}$.
This finishes the proof.
\end{proof}

\subsection{Filtration on $W_m\Omega^\bullet_U$}\label{sec:Fil-DRW}
We now assume that $X$ is a connected regular $F$-finite $\F_p$-scheme and let 
$E \subset X$ be an snc divisor
with irreducible components $E_1, \ldots , E_r$. We let $j \colon U \inj X$ be the
inclusion of the complement of $E$. We fix $m \ge 1$.

\begin{defn}\label{defn:Log-fil-3}
  For $D = \stackrel{r}{\underset{i =1}\sum} n_i E_i \in \Div_E(X)$
  and an integer $q \ge 0$, we let
  \[
  \Fil_D W_m \Omega^q_U =  \Fil_D W_m \sO_U \cdot W_m\Omega^q_X(\log E) +
  d(\Fil_D W_m \sO_U) \cdot W_m\Omega^{q-1}_X(\log E),
  \]
  considered as a subsheaf of $j_*W_m\Omega^q_U$. We shall write
  $\Fil_D W_1 \Omega^q_U$ as $\Fil_D \Omega^q_U$.
\end{defn}

For $D = \stackrel{r}{\underset{i =1}\sum} n_i E_i$ effective, we let
$W_m\Omega^q_{X|D} = \Ker(W_m\Omega^q_X \surj W_m\Omega^q_D)$
(cf. \cite[\S~3]{GK-Duality-1}).
If $X = \Spec(A)$ and $x_1, \ldots , x_r$ are generators of the ideals of the
irreducible components of $E$, then we shall write $\Fil_D W_m \Omega^q_U$ as
$\Fil_I W_m\Omega^q_{A_\pi}$, where $\pi = \prod_i x_i$ and $I = (\prod_i x^{-n_i}_i)$. 
We let $W_m\Omega^q_{A|I} = \Ker(W_m\Omega^q_{A} \surj W_m\Omega^q_{A/I})$ if
$I \subset A$ is an ideal.

\enlargethispage{20pt}

\begin{lem}\label{lem:Log-fil-4}
  For $q \ge 0$, we have the following.
  \begin{enumerate}
    \item
      $\Fil_{p^{m-1}D} W_m \Omega^q_U \xrightarrow{\cong} W_m\Omega^q_X(\log E)(D)$.
  \item
    $\Fil_D W_m \Omega^q_U$ is a sheaf of finitely generated $W_m\sO_X$-modules.
  \item
    $\Fil_D\Omega^q_U = \Omega^q_X(\log E)(D)$. In particular, $\Fil_{-E}\Omega^d_U =
    \Omega^d_X$, where $d$ is the rank of $\Omega^1_X$.
    \item
    $\varinjlim_n \Fil_{nE} W_m \Omega^q_U \xrightarrow{\cong} j_*W_m\Omega^{q}_U$.
    \item
      $\Fil_0W_m\Omega^q_U = W_m\Omega^q_X(\log E)$ and
      $W_m\Omega^q_{X|-D} \subset \Fil_DW_m\Omega^q_U \subset W_m\Omega^q_{X|E}$ if
      $n_i < 0$ for all $1 \le i \le r$.
    \end{enumerate}
\end{lem}
\begin{proof}
 Since all claims of the lemma can be checked locally, 
  we shall assume (whenever necessary in the proof) 
  that $X = \Spec(A)$ and $D \subset X$ are as in \defref{defn:Log-fil-3}. 
To prove (1), we write $I' = I^{p^{m-1}}$ and note that
 $\Fil_{I'}W_m\Omega^q_{A_\pi} = \Fil_{I'}W_m(A_\pi) \cdot W_m\Omega^q_{A}(\log\pi) +
 d(\Fil_{I'}W_m(A_\pi))\cdot W_m\Omega^{q-1}_A(\log\pi) =
 W_m\Omega^q_A(\log\pi)(x^{-1}) + d([x^{-1}]_mW_m(A))\cdot W_m\Omega^{q-1}_A(\log\pi)$
 by \lemref{lem:Log-fil-2-1}. On the other hand,

 \[
 \begin{array}{lll}
   d([x^{-1}]_m W_m(A))\cdot
 W_m\Omega^{q-1}_A(\log\pi) & \subset & [x^{-1}]_md(W_m(A))\cdot
 W_m\Omega^{q-1}_A(\log\pi) \\
 & & +  \ [x^{-1}]_m W_m(A)\dlog([x]_m)\cdot
 W_m\Omega^{q-1}_A(\log\pi) \\
 \end{array}
 \]
 As the term on the right is part of $W_m\Omega^{q}_A(\log\pi)(x^{-1})$,
 the desired isomorphism follows.

 It is clear that $\Fil_D W_m\Omega^q_U$ is a sheaf on $X_\zar$.
 To prove (2), we thus have to show that $\Fil_D W_m \Omega^q_U$ is a sheaf of finite
 type $W_m\sO_X$-modules. To that end, we choose an effective Weil divisor
 $D' \in \Div_E(X)$ such that $D'':= D + D'$ is effective.
 In particular, $\Fil_D W_m \Omega^q_U \subset \Fil_{D''} W_m \Omega^q_U \subset
  \Fil_{p^{m-1}D''} W_m \Omega^q_U$. Since $W_m\sO_X$ is a sheaf of Noetherian rings
  by \propref{prop:F-fin}(5), it suffices to show that
  $\Fil_{p^{m-1}D''} W_m \Omega^q_U$ is a sheaf of finitely generated $W_m\sO_X$-modules.
  But this follows from (1) because $W_m\Omega^q_X(\log E)$ is a sheaf of
  finite type $W_m\sO_X$-modules, by \propref{prop:F-fin}(5), and $W_m\sO_X(D'')$ is a
  sheaf of invertible $W_m\sO_X$-modules. Item (3) is a special case of (1) where one
  takes $m =1$.

For (4), we recall that $j_*W_m\Omega^{q}_U \cong j_*W_m\sO_U \otimes_{W_m\sO_X}
W_m\Omega^q_X = j_*W_m\sO_U \cdot W_m\Omega^q_X$
(cf. \cite[Chap.~I, Prop.~1.11]{Illusie}).
Meanwhile, one checks directly from ~\eqref{eqn:Log-fil-1*} that
$j_*W_m\sO_U  = \varinjlim_n \Fil_{nE} W_m\sO_U$. This proves the desired isomorphism.

The first part of (5) is clear. To prove the first inclusion in the second part,
it suffices to show, using \cite[Lem.~2.4]{Geisser-Hesselholt-Top} and 
Definition~\ref{defn:Log-fil-3}, that $W_m(\sI_{-D}) \subset \Fil_DW_m\sO_U$. But this is
immediate from ~\eqref{eqn:Log-fil-1*}.
To prove the second inclusion in the second part of (5), it suffices to show that
$\Fil_{-E}W_m\Omega^q_U \subset W_m\Omega^q_{X|E}$.
For this, we let $\un{a} = (a_{m-1}, \ldots , a_0) \in \Fil_IW_m(A_\pi)$
and $w \in W_m\Omega^q_A(\log\pi)$, where $I = (\pi)$ .
Letting $\un{a} = \sum_i V^i([a_{m-1-i}]_{m-i})$,
 we have that $a_i^{p^i}\pi^{-1} \in A$.
 This implies that $\pi \mid a_i^{p^i}$. Equivalently, $\pi \mid a_i$ (because $A$ is a
 UFD). We write $a_i = \pi \alpha_i$, where $\alpha_i \in A$.
 We can find $b_i \in A$ and write $w$ as the sum of elements of the form
 \[
 V^{j_0}([b_0]_{m-j_0})dV^{j_1}([b_1]_{m-j_1}) \cdots dV^{j_t}([b_t]_{m-j_t})\dlog([x_{i_1}]_m)\cdots
 \dlog([x_{i_s}]_m),
 \]
where $1 \le i_1 < \cdots < i_s \le r, \ t = q-s$.
We have
$V^i([a_{m-1-i}]_{m-i})\dlog([x_{i_1}]_m)\cdots \dlog([x_{i_s}]_m) =
V^i([\alpha_{m-i-1}]_{m-i} [\pi]_{m-i} \dlog([x_{i_1}]_{m-i}) \cdots
\dlog([x_{i_s}]_{m-i}))$.
But this last term lies in $W_m\Omega^q_{A|I}$. Taking the sum over
$i \in \{0, \ldots , m-1\}$,
we get $\un{a}\dlog([x_{i_1}]_m)\cdots \dlog([x_{i_s}]_m)
  \in W_m\Omega^q_{A|I}$. This implies that $\un{a}w \in W_m\Omega^q_{A|I}$.
  We have thus shown that
$\Fil_IW_m(A_\pi)W_m\Omega^q_{A}(\log\pi) \subset W_m\Omega^q_{A|I}$.

If $\un{a} \in \Fil_IW_m(A_\pi)$ and $w \in W_m\Omega^{q-1}_A(\log\pi)$, then
$d(\un{a})w = d(\un{a}w) - \un{a}d(w)$. We showed above that
$\un{a}w \in  W_m\Omega^{q-1}_{A|I}$ and $\un{a}d(w) \in W_m\Omega^q_{A|I}$.
It follows that $d(\un{a})w \in W_m\Omega^q_{A|I}$.
This shows $\Fil_IW_m\Omega^q_{A_\pi}\subset W_m\Omega^q_{A|I}$.
\end{proof}

\begin{remk}\label{remk:Log-fil-4-ext}
  (1) \ We remark that the equality $\Fil_{-E}W_m\Omega^d_U = W_m\Omega^d_X$ in item (3) of
\lemref{lem:Log-fil-4} when $m =1$ and $d$ is the rank of $\Omega^1_X$
is actually true for every $m \ge 1$ (cf. \thmref{thm:Global-version} (14))
but the proof of the general case requires more work.
\end{remk}

Before we proceed further, we make the following
general definition. 

\begin{defn}\label{defn:Filtered-Witt-complex}
  Let $X$ be an $\F_p$-scheme and let $Z \subset X$ be a closed subset.
A $\Div_Z(X)$-filtered Witt complex on $X$ is a sheaf of
  Witt pro-complexes (in the sense of Deligne-Illusie, see \cite{Hes-Mad})
  $(\{E^\bullet_m\}_{m \ge 1}, d, F, V, R)$ on $X_\zar$ together with a pro-subcomplex of
  $W_m\sO_X$-modules
  $\{\Fil_D E^\bullet_m\}_{m \ge 1} \subset \{E^\bullet_m\}_{m \ge 1}$ for every
  $D \in \Div_Z(X)$ 
  such that the following hold for every $m \ge 1$.
  \begin{enumerate}
  \item
    $\Fil_D E^\bullet_m \cdot \Fil_{D'} E^\bullet_m \subset \Fil_{D+D'} E^\bullet_m$.
  \item
    $d(\Fil_D E^\bullet_m) \subset \Fil_{D} E^\bullet_m$; \ \
    $R(\Fil_D E^\bullet_m) \subset \Fil_{D/p} E^\bullet_m$.
  \item
    $F(\Fil_D E^\bullet_{m+1}) \subset \Fil_{D} E^\bullet_m$; \ \
    $V(\Fil_D E^\bullet_m) \subset \Fil_{D} E^\bullet_{m+1}$.
  \item
    $\Fil_D E^\bullet_m  \subset \Fil_{D'} E^\bullet_m$ if $D \le D'$ in $\Div_Z(X)$.
  \end{enumerate}
\end{defn}

In the sequel, we shall use the term `filtered Witt complex' for an
object such as the one defined in Definition~\ref{defn:Filtered-Witt-complex} and
write it simply as $\Fil_D E^\bullet_\star$, when the operators $d,F,V, R$ and the
index set $\Div_Z(X)$ are fixed.
If $X = \Spec(A)$ is affine, then $\{\Gamma(X, \Fil_DE^\bullet_m)\}$ is a
filtered Witt complex on $A$. We shall identify this with
$\{\Fil_D E^\bullet_m\}$ in this case.
The following definition will be convenient in this paper.

\begin{defn}\label{defn:Filtered-Witt-complex-0}
  We shall say that $(\{E^\bullet_m\}_{m \ge 1}, d, F, V, R)$ is a {\sl good} filtered
  Witt complex if the sequence
  \begin{equation}\label{eqn:Good-complex}
    0 \to V^{m}(\Fil_D E^q_1) + dV^{m}(\Fil_D E^{q-1}_1) \to \Fil_D E^q_{m+1}
    \xrightarrow{R} \Fil_{D/p} E^q_m \to 0
  \end{equation}
  is exact for every $q \ge 0, \ m \ge 1$ and $D \in \Div_Z(X)$.
\end{defn}

Letting $Z = E$, we now have the following.

\begin{lem}\label{lem:Log-fil-Wcom}
  The Witt complex $\left(\{j_*W_m \Omega^\bullet_U\}_{m\ge 1},
  d, F, V, R\right)$ together with the subcomplexes
  $\{\Fil_D W_m \Omega^\bullet_U\}_{m \ge 1, D \in \Div_E(X)}$
  is a $\Div_E(X)$-filtered Witt complex on $X$. 
\end{lem}
\begin{proof}
  {\bf Step~1:}
  To show that $\Fil_D W_m \Omega^\bullet_U$ is preserved under $d$, note that 
  $d(\Fil_D W_m \sO_U \cdot W_m\Omega^q_X(\log E))$ \
  $\subset d(\Fil_D W_m \sO_U) \cdot W_m\Omega^q_X(\log E) + \Fil_D W_m \sO_U \cdot
  W_m\Omega^{q+1}_X(\log E) = \Fil_D W_m\Omega^{q+1}_U$.
  We also have
  \[
  d\left(d(\Fil_D W_m \sO_U) \cdot W_m\Omega^{q-1}_X(\log E)\right) \subset
  d(\Fil_D W_m \sO_U) \cdot W_m\Omega^{q}_X(\log E) \subset
  \Fil_D W_m\Omega^{q+1}_U.
  \]
  
  {\bf Step~2:} We show that $F(\Fil_D W_{m+1} \Omega^q_U) \subset
  \Fil_D W_{m} \Omega^q_U$ for every $q \ge 0, m \ge 1$.
   When $q = 0$, the claim is clear using ~\eqref{eqn:Log-fil-1*} and the fact that
   $F = \ov{F}\circ R$, where $\ov{F}(a_{m-2},\ldots, a_0) = (a_{m-2}^p,\ldots, a_0^p)$.
   We now assume $q > 0$ and let
  $a \in \Fil_IW_{m+1}(A_\pi), \ w \in W_{m+1}\Omega^q_A(\log\pi)$. We then have
  $F(aw) = F(a)F(w) \in \Fil_I W_m(A_\pi)W_{m}\Omega^q_A(\log\pi)$ by the
  $q = 0$ case and \thmref{thm:Log-DRW}. 

Suppose next that $\un{a} \in \Fil_I W_{m+1}(A_\pi)$ and
$w \in W_{m+1}\Omega^{q-1}_A(\log\pi)$. We can then write 
  $\un{a} = \stackrel{m}{\underset{i =0}\sum}V^{i}([a_{m-i}]_{m+1-i})$.
  This yields
  $Fd(\un{a}) = Fd[a_m]_{m+1} + \stackrel{m-1}{\underset{i=0}\sum} FdV^{i+1}[a_{m-1-i}]_{m-i} =
  Fd[a_m]_{m+1} + d\un{a}'$, where $\un{a}' = (a_{m-1}, \ldots , a_0) \in \Fil_IW_m(A_\pi)$.
  In particular, $d\un{a}' \cdot F(w) \in d(\Fil_IW_m(A_\pi)) \cdot
  W_m\Omega^{q-1}_A(\log\pi) \subset \Fil_I W_m\Omega^{q}_{A_\pi}$.
We also have
  \begin{equation}\label{eqn:Log-fil-4-0}
    \begin{array}{lll}
    \Fil_I W_m\Omega^q_{A_\pi}\cdot W_m\Omega^{q'}_A(\log\pi)
    & \subset & \Fil_IW_m(A_\pi)W_m\Omega^{q+q'}_A(\log\pi) \\
    & & + \ d(\Fil_I W_m(A_\pi))W_m\Omega^{q+q'-1}_A(\log\pi) = \Fil_I W_m\Omega^{q+q'}_{A_\pi}.
  \end{array}
    \end{equation}
Since $F(d[a_m]_{m+1} w) = Fd[a_m]_{m+1} \cdot F(w)$
  and $F(w) \in W_m\Omega^{q-1}_A(\log\pi)$
  by \thmref{thm:Log-DRW}, it remains to show
  using ~\eqref{eqn:Log-fil-4-0} that $Fd[a_m]_{m+1} \in \Fil_IW_m\Omega^1_{A_\pi}$.

To that end, we write $a_m = a'(\prod_i x^{n_i}_i)$, where $n_i \in \Z$ for every $i$  and $a' \in A$ is not divisible by any $x_i$. Then the condition $\un{a} \in \Fil_I W_{m+1}(A_\pi)$ implies that $[\prod_i x^{n_i}_i]_{m+1} \in \Fil_I W_{m+1}(A_\pi)$. Letting  $\gamma = \prod_i x^{n_i}_i$, we get 
\[ Fd[a_m]_{m+1} = Fd[a']_{m+1} \cdot F([\gamma]_{m+1}) + F([a_m]_{m+1})\dlog([\gamma]_{m}).
\]
  We now note that $\dlog([\gamma]_m) = \sum\limits_{i}n_i.\dlog([x_i]_m) \in  W_m\Omega^1_A(\log \pi)$ by \thmref{thm:Log-DRW} which implies that  $F([a_m]_{m+1})\dlog([\gamma]_m) \in  \Fil_I W_m\Omega^1_{A_\pi}$. On the other hand, we have $Fd[a']_{m+1} \cdot F([\gamma]_{m+1}) \in \Fil_I W_m(A_\pi) \cdot W_m\Omega^1_A \subset \Fil_I W_m\Omega^1_{A_\pi}$. This proves our claim.

 {\bf Step~3:} We show that $V(\Fil_D W_{m} \Omega^q_U) \subset
 \Fil_D W_{m+1} \Omega^q_U$ for $q \ge 0, m \ge 1$. The $q = 0$ case is clear using
 ~\eqref{eqn:Log-fil-1*}. For $q > 0$,  we let $\un{a} \in \Fil_IW_m(A_\pi)$
 and $w \in W_m\Omega^q_A(\log\pi)$. We write $w$ as the sum of elements of the form
 \[
 V^{j_0}([b_0]_{m-j_0})dV^{j_1}([b_1]_{m-j_1}) \cdots
 dV^{j_t}([b_t]_{m-j_t})\dlog[x_{i_1}]_m\cdots \dlog[x_{i_s}]_m,
 \]
where $1 \le i_1 < \cdots < i_s \le r, \ t = q-s$ and $b_i \in A$. Letting
$\un{b} = \un{a}V^{j_0}[b_0]_{m-j_0}$ and $w' = dV^{j_1}([b_1]_{m-j_1}) \cdots dV^{j_t}([b_t]_{m-j_t})$, we get
$V(\un{b}w') = V(\un{b})dV^{j_1+1}([b_1]_{m-j_1})\cdots dV^{j_t+1}([b_t]_{m-j_t})
\in \Fil_I W_{m+1}(A_\pi) \cdot W_{m+1}\Omega^t_A$.
This yields
\[
\begin{array}{lll}
V(\un{a}w) & = & V(\un{b}w'\dlog[x_{i_1}]_m\cdots \dlog[x_{i_s}]_m) = 
V(\un{b}w')\cdot \dlog[x_{i_1}]_{m+1}\cdots \dlog[x_{i_s}]_{m+1} \\
& \in &
\Fil_I W_{m+1}(A_\pi) \cdot W_{m+1}\Omega^t_A \cdot W_{m+1}\Omega^s_A(\log\pi) \\
& \subset & \Fil_IW_{m+1}\Omega^t_A(\log\pi) \cdot  W_{m+1}\Omega^s_A(\log\pi) \\
& \subset & \Fil_{I}W_{m+1}\Omega^q_{A_\pi}.
\end{array}
\]
If $\un{a} \in \Fil_I W_m(A_\pi)$ and $w \in W_m\Omega^{q-1}_A(\log\pi)$, then
$V(d\un{a} \cdot w) = dV(\un{a})\cdot V(w) \in d(\Fil_IW_{m+1}(A_\pi))\cdot
W_{m+1}\Omega^{q-1}_A(\log\pi) 
\subset \Fil_I W_{m+1}\Omega^{q}_{A_\pi}$.
This proves the claim.

{\bf Step~4:} We now show that $R(\Fil_D W_m\Omega^q_U) \subset
\Fil_{D/p} W_m\Omega^q_U$. We shall in fact prove a stronger claim, namely,
$R(\Fil_D W_m\Omega^q_U) = \Fil_{D/p} W_m\Omega^q_U$.
Now, the $q = 0$ case of this stronger claim follows already from
\lemref{lem:Log-fil-1}. On the other hand, the claim for $W_m\Omega^q_X(\log E)$ is immediate from its definition and the surjectivity of $R$ for the classical de Rham-Witt sheaves. The general case then follows easily from the cases $q=0$ and $D=0$ because $R$ is multiplicative and commutes with $d$.

{\bf Step~5:} In this last step, we show that $\Fil_DW_m\Omega^\bullet_U$ satisfies
condition (1) of Definition~\ref{defn:Filtered-Witt-complex}. 
Using \lemref{lem:Log-fil-2} and an elementary computation, it suffices to
check that $\Fil_IW_m(A_\pi)\cdot d(\Fil_{I'}W_m(A_\pi)) \subset
\Fil_{II'}W_m\Omega^1_{A_\pi}$, where $x = \prod_ix^{n_i}_i$ (resp. $y = \prod_i x^{m_i}_i$)
and $I = (x)$ (resp. $I' = (y)$).
We let $\un{a} = \stackrel{m-1}{\underset{i =0}\sum}V^{i}([a_{m-1-i}]_{m-i})
\in \Fil_I W_m(A_\pi)$ and
$\un{b} = \stackrel{m-1}{\underset{i = 0}\sum}V^i([b_{m-1-i}]_{m-i}) \in
\Fil_{I'} W_m(A_\pi)$. This yields ${\un{a}}\cdot d{\un{b}} =
\sum_{i,j} V^i([a_{m-1-i}]_{m-i})dV^j([b_{m-1-j}]_{m-j})$.
It suffices now to show that $V^i([a_{m-1-i}]_{m-i})dV^j([b_{m-1-j}]_{m-j}) \in
\Fil_{II'}W_m\Omega^1_{A_\pi}$ for every $0 \le i,j \le m-1$.

We note that $[b_{m-1-j}]_{m-j}\in \Fil_{I'}W_{m-j}(A_\pi)$. If we write $b_{m-1-j} =b'\prod_l x^{r_l}_l$, where ${b'} \in A, \ r_l \in \Z$ and no
$x_l$ divides $b'$, then we get that $\prod_l [x^{r_l}_l]_{m-j} \in \Fil_{I'}W_{m-j}({A_\pi})$
(see the proof of Step~2). We write $\gamma' = \prod_l x^{r_l}_l$ .
We first assume $i \ge j$ and write 
$$V^i([a_{m-1-i}]_{m-i})dV^j([b_{m-1-j}]_{m-j}) =
V^j(V^{i-j}[a_{m-1-i}]_{m-i} d[b_{m-1-j}]_{m-j}).$$
It is enough to show using Step~3 that
$V^{i-j}([a_{m-1-i}]_{m-i}) d([b_{m-1-j}]_{m-j}) \in \Fil_{II'}W_{m-j}\Omega^1_{A_\pi}$.

We now compute
\[
d[b_{m-1-j}]_{m-j}V^{i-j}([a_{m-1-i}]_{m-i}) = d[b']_{m-j}[\gamma']_{m-j}V^{i-j}([a_{m-1-i}]_{m-i}) \hspace{1cm}\]
\[\hspace{5cm}+ [b'\gamma']_{m-j}
V^{i-j}([a_{m-1-i}]_{m-i}) \dlog([\gamma']_{m-j}).
\]
The terms $[\gamma']_{m-j}V^{i-j}([a_{m-1-i}]_{m-i})$ and $[b'\gamma']_{m-j}
V^{i-j}([a_{m-1-i}]_{m-i})$ lie $\Fil_{II'}W_{m-j}({A_\pi})$ by \lemref{lem:Log-fil-2}.
Since $d[b']_{m-j} \in W_{m-j}\Omega^1_{A}$ and $\dlog([\gamma']_{m-j}) \in W_{m-j}\Omega^1_A(\log\pi)$
by \thmref{thm:Log-DRW}, we conclude that
$d([b_{m-1-j}]_{m-j})V^{i-j}([a_{m-1-i}]_{m-i}) \in \Fil_{II'}W_{m-j}\Omega^1_{A_\pi}$.

If $i < j$, we apply the Leibniz rule, \lemref{lem:Log-fil-2} and Step~1 to reduce
to the case $i\ge j$. The condition (4) Definition~\ref{defn:Filtered-Witt-complex} is
clear for $\{\Fil_D W_m\Omega^\bullet_U\}_{D \in \Div_E(X)}$. This concludes the proof. 
\end{proof}

\begin{defn}\label{defn:DRW-complex-filt**}
  Let $(X,E)$ be the snc-pair as above. By the filtered de Rham--Witt complex on $X$,
  we shall mean the Witt complex $j_* W_m\Omega^\bullet_U$ endowed with the
  filtration $\{\Fil_D W_m\Omega^\bullet_U\}_{D \in \Div_E(X)}$. We shall usually
    write it simply as $\Fil_DW_m\Omega^\bullet_U$.
\end{defn}

\begin{defn}\label{defn:Log-fil-3-0}
  Let $(X,E)$ be the snc-pair as above.
  A Zariski (or {\'e}tale) sheaf on $X$ of the form $\Fil_D W_m \Omega^q_U$ will
be called a filtered Hodge--Witt sheaf or a Hodge--Witt sheaf with modulus ($D$).
The cohomology group $H^i_\zar(X, \Fil_D W_m \Omega^q_U)$
(for any $m \ge 1, i, q \ge 0$ and $D \in \Div_E(X)$) will be called a
Hodge--Witt cohomology with modulus.
\end{defn}

\section{Filtered de Rham--Witt complex of power series algebras}\label{sec:FPSR}
In the previous section, we showed that the filtered de Rham--Witt complex is an
example of a filtered Witt complex on a regular $F$-finite $\F_p$-scheme with an snc
divisor. Our goal now is to show that it is a good filtered
Witt complex (cf. Definition~\ref{defn:Filtered-Witt-complex-0}).
This part is very challenging and before we get to its final step,
we need to prove several
other properties of the filtered de Rham--Witt complex which are classical for the
ordinary de Rham--Witt complex.

In this section, we prove the goodness of the filtered de Rham--Witt complex
in a special case when $X$ is the spectrum of a power series ring in one variable.
We begin by providing a concrete description of the filtered de Rham--Witt
complex in this case.

\subsection{Description of $\Fil_DW_m\Omega^\bullet_K$ using a theorem of
  Geisser--Hesselholt}\label{sec:GSPS}
We let $S$ be an $F$-finite Noetherian regular $\F_p$-algebra. Let
$A = S[[\pi]]$ and $K = A_\pi$. For $n \in \Z$ and
the invertible ideal $I = (\pi^n) \subset K$,
we shall denote $\Fil_IW_m\Omega^q_{K}$ by $\Fil_nW_m\Omega^q_K$ and for any
$W_m(A)$-module $M$, write $M \otimes_{W_m(A)} [\pi^{n}]_m W_m(A)$ as $M(n)$.
We let $\gr_nW_m\Omega^q_K = \frac{\Fil_nW_m\Omega^q_K}{\Fil_{n-1}W_m\Omega^q_K}$.

We shall use the following theorem of
Geisser--Hesselholt (cf. \cite[Thm.~B]{Geisser-Hesselholt-Top}) in order to describe
$\Fil_nW_m\Omega^q_K$. Let $I_p$ denote the set of positive integers prime to $p$.
We fix integers  $q \ge 0$ and $m \ge 1$.

\begin{thm}$($Geisser--Hesselholt$)$\label{thm:GH-Top}
Every element $\omega \in W_m\Omega^q_A$ can be uniquely written as an infinite
  series
  \[
  w = {\underset{i \ge 0}\sum} a^{(m)}_{0,i} [\pi]^i_m +
  {\underset{i \ge 0}\sum} b^{(m)}_{0,i} [\pi]^i_m d[\pi]_m +
  {\sum_{s=1}^{m-1}}{\underset{j \in I_p}\sum}
  \left(V^s(a^{{(m-s)}}_{s,j}[\pi]^j_{m-s}) +
  dV^s(b^{{(m-s)}}_{s,j}[\pi]^{j}_{m-s})\right),
  \]
  where the components $a^{(n)}_{s,j} \in W_n\Omega^q_S$ and
    $b^{(n)}_{s,j} \in W_n\Omega^{q-1}_S$.
\end{thm}

For $s \ge 0$ and $i \in \Z$, we let
\begin{equation}\label{eqn:GH-0}
  A^{m,q}_i(S) = W_m\Omega^q_S [\pi]^i_m+ W_m\Omega^{q-1}_S [\pi]^i_m\dlog([\pi]_m);
\end{equation}
\[
B^{m,q}_{i,s}(S) = V^s\left(W_{{m-s}}\Omega^q_S [\pi]^i_{m-s}\right) +
dV^s\left(W_{m-s}\Omega^{q-1}_S [\pi]^i_{m-s}\right).
\]
Note that $A^{m,q}_i(S), B^{m,q}_{i,s}(S) \subset W_m\Omega^q_K$.

\begin{defn}\label{defn:GH-1}
  For $n \in \Z$, we let
  \[
  F^{m,q}_n(S) =
 \left\{\begin{array}{ll} A^{m,q}_{{n}/{p^{m-1}}}(S) &  \mbox{if $p^{m-1} \mid n$} \\
    B^{m,q}_{{n}/{p^{m-1-s}}, s}(S) & \mbox{if $s = {\rm min}\{j > 0|p^{m-1-j} \mid n\}$}.
  \end{array}\right.
  \]
\end{defn}

\begin{cor}\label{cor:GH-2}
  There is a canonical decomposition
  \begin{equation}\label{eqn:GH-2-0*}
    \theta^{m,q}_A \colon {\underset{n \ge 0}\prod} F^{m,q}_n(S)
    \xrightarrow{\cong} W_m\Omega^q_A(\log\pi).
  \end{equation}
  \end{cor}
\begin{proof}
  \thmref{thm:GH-Top} says that the canonical maps from $W_m\Omega^q_S,\ F^{m,q}_n(S)$
  (for $n \ge 1$) to $W_m\Omega^q_A$ induce an isomorphism
  \begin{equation}\label{eqn:GH-2-0}
  \ov \theta_A^{m,q} \colon  W_m\Omega^q_S \bigoplus {\underset{n \ge 1}\prod}
  F^{m,q}_n(S) \xrightarrow{\cong} W_m\Omega^q_A.
  \end{equation}
  We can extend $\ov \theta_A^{m,q}$ uniquely to a map 
  \[
  \theta_A^{m,q} \colon {\underset{n \ge 0}\prod} F^{m,q}_n(S) = W_m\Omega^{q-1}_S
  \dlog[\pi]_m \bigoplus W_m\Omega^q_S \bigoplus {\underset{n \ge 1}\prod}
  F^{m,q}_n(S)
  \]
  \[
  \xrightarrow{\id \ \bigoplus \ \ov \theta_A^{m,q}}W_m\Omega^{q-1}_S \dlog[\pi]_m
  \bigoplus W_m\Omega^q_A \xrightarrow{\phi}  W_m\Omega^q_A(\log \pi).
  \]

  We only have to show that the canonical map $\phi$
  is an isomorphism. The injectivity of this
  map can be checked directly using \thmref{thm:GH-Top}
  while for the surjectivity, note that
  $W_m\Omega^q_A(\log \pi) =  W_m\Omega^q_A + W_m\Omega^{q-1}_A \dlog[\pi]_m$
  (cf. \S~\ref{sec:de-log-0}). Therefore, it is enough to show that
  $W_m\Omega^{q-1}_A \dlog[\pi]_m$ lies in the image of $\phi$. Using the
  identity $W_m\Omega^{q-1}_A \dlog[\pi]_m = W_m\Omega^{q-1}_S \dlog[\pi]_m +
  \left(\prod_{n \ge 1} F^{m,q-1}_n(S)  \dlog[\pi]_m\right)$ (cf. \thmref{thm:GH-Top}),
  one is left with checking the inclusion
  $F^{m,q-1}_n(S)\dlog[\pi]_m \subseteq F^{m,q}_n(S)$.
  But this is a routine exercise using the identity $V(\alpha)\dlog\beta = V(\alpha \dlog\beta)$.
  \end{proof}

In order to extend this decomposition to $W_m\Omega^q_K$, we need the following.

\begin{lem}\label{lem:Positive-F-2}
  For all $n,l \in \Z$, 
  we have
  \[
[\pi]_m^{l}F^{m,q}_n(S) = F^{m,q}_{n+ p^{m-1} l}(S).\]
\end{lem}
\begin{proof}
We first consider the
  case $n = p^{m-1}i$ with $i \in \Z$. In this case, one checks that
  $F^{m,q}_n(S) = [\pi]^i_mF^{m,q}_0(S)$. This yields
  $[\pi]^{l}_m F^{m,q}_n(S) =  [\pi]^{i+l}_mF^{m,q}_0(S) =
  F^{m,q}_{p^{m-1}(i+l)}(S)=F^{m,q}_{n+p^{m-1}l}$. 
  
 Suppose now that $n = p^{m-1-s}i$ with $0 <s \le m-1$ and $|i| \in I_p$.
  Let $t = n+p^{m-1}l = p^{m-1-s}i + p^{m-1}l = p^{m-1-s}(i +p^sl) = p^{m-1-s}i'$,
  where we let $i' = i+p^sl$. Note that $|i'| \in I_p$. Letting $r = m-s$, we then have
\[
  \begin{array}{lllr}
    [\pi]^{l}_mF^{m,q}_n(S) & = & [\pi]^{l}_m \left(V^s(W_r\Omega^q_S[\pi]^i_{r}) +
    dV^s(W_r\Omega^{q-1}_S[\pi]^i_r)\right) \\
    & = & V^s(W_r\Omega^q_S[\pi]^{i'}_r) +  [\pi]^{l}_mdV^s(W_r\Omega^{q-1}_S[\pi]^i_r) 
    \\
    & {\subset}^\dagger & V^s(W_r\Omega^q_S[\pi]^{i'}_r) +
    dV^s(W_r\Omega^{q-1}_S[\pi]^{i'}_r) \\
    && +  V^s(W_r\Omega^{q-1}_S[\pi]^i_rF^sd[\pi]^{l}_m)\\
    & \subset & F^{m,q}_{t}(S) +  V^s(W_r\Omega^{q-1}_S[\pi]^i_rF^sd[\pi]^{l}_r),
  \end{array}
  \]
  where ${\subset}^\dagger$ is obtained by applying the Leibniz rule to
  $[\pi]^{l}_m V^s(W_r\Omega^{q-1}_S[\pi]^i_r)$.
  We also used the relation $aV(b)=V(F(a).b)$.

 Meanwhile, using the relation $F(d[t]_r)=[t]_{r-1}^{p-1}d[t]_{r-1}$, we have
  \[
  \begin{array}{lll}
    V^s(W_r\Omega^{q-1}_S[\pi]^i_rF^sd[\pi]^{l}_r) & \subset &
    V^s(W_r\Omega^{q-1}_Sd([\pi]^{i'}_r)) \\
    & \subset & V^sd(W_r\Omega^{q-1}_S[\pi]^{i'}_r) +
    V^s(d(W_r\Omega^{q-1}_S)[\pi]^{i'}_r) \\
    & \subset & dV^s(W_r\Omega^{q-1}_S[\pi]^{i'}_r) + V^s(W_r\Omega^q_S[\pi]^{i'}_r) =
    F^{m,q}_{t}(S).
   \end{array}
  \]
  This shows that $[\pi]^l_mF^{m,q}_n(S) \subset F^{m,q}_{n+p^{m-1}l}(S)$ for all
  $l,n \in \Z$.
  But this also implies $[\pi]^{-l}_mF^{m,q}_{n+p^{m-1}l}(S) \subset F^{m,q}_{n}(S)$,
  and we are done. 
\end{proof}

Combining \corref{cor:GH-2} and \lemref{lem:Positive-F-2}, we get

\begin{cor}\label{cor:GH-3}
 There is a canonical decomposition
  \[
  \theta^{m,q}_K \colon
  \left(\bigoplus_{n <0} F^{m,q}_{n}(S)\right)
 \bigoplus \left(\prod_{n\ge 0} F^{m,q}_n(S)\right) \xrightarrow{\cong}  W_m\Omega^q_K.
  \]
\end{cor}
\begin{proof}
We first note that the canonical map
  \begin{equation}\label{eqn:Base-change-0}
  W_m\Omega^q_A(\log\pi) \otimes_{W_m(A)} W_m(K) \to W_m\Omega^q_K
  \end{equation}
  is an isomorphism of $W_m(K)$-modules. This is easily deduced from
  \cite[Chap.~0, (1.5.3) and Chap.~I, Prop.~1.11]{Illusie}.
  Using this isomorphism, we get
  \[
  \begin{array}{lll}
    W_m\Omega^q_K & \cong & W_m\Omega^q_A(\log\pi)[\pi^{-1}]_m \ \ \cong \ \
    {\varinjlim}_{l \ge 0}
      [\pi]^{-l}_m W_m\Omega^q_A(\log\pi) \\
      & {\cong}^1 &  {\varinjlim}_{l \ge 0} [\pi]^{-l}_m \prod_{n \ge 0} F^{m,q}_n(S) 
      \ \ {\cong} \ \ {\varinjlim}_{l \ge 0}  \prod_{n \ge 0} [\pi]^{-l}_m  F^{m,q}_n(S) \\ 
      & {\cong}^2 & {\varinjlim}_{l \ge 0} \prod_{n \ge 0} F^{m,q}_{n-p^{m-1}l}(S)
      \ \ \cong \ \ {\varinjlim}_{l \ge 0} \prod_{n \ge -p^{m-1}l}  F^{m,q}_n(S) \\
      & \cong &  \left[{\varinjlim}_{l \ge 0}
      \left({\bigoplus}_{-p^{m-1}l \le n \le -1} F^{m,q}_n(S)\right)\right]
      \bigoplus \left[{\varinjlim}_{l \ge 0}
        \left(\prod_{n \ge 0} F^{m,q}_n(S)\right)\right] \\
      & \cong & \left[{\varinjlim}_{l \ge 0}
        \left({\bigoplus}_{-l \le n \le -1} F^{m,q}_n(S)\right)\right] \bigoplus
      \left(\prod_{n \ge 0} F^{m,q}_n(S)\right)
       \\
       & \cong & \left(\bigoplus_{n < 0} F^{m,q}_n(S)\right) \bigoplus
       \left(\prod_{n \ge 0} F^{m,q}_n(S)\right),
\end{array}
  \]
  where the isomorphisms ${\cong}^1$ and ${\cong}^2$ hold by \corref{cor:GH-2}
  and \lemref{lem:Positive-F-2}, respectively. 
\end{proof}

Our task now is to describe the pre-image of $\Fil_nW_m\Omega^q_K$ under
$\theta^{m,q}_K$. We shall complete it in several steps.
We begin by checking some further properties of the groups $F^{m,q}_n(S)$.

\begin{lem}\label{lem:F-group-1}
  For all $n \in \Z$, we have the following.
  \begin{enumerate}
    \item
      $d(F^{m,q}_n(S)) \subset F^{m,q+1}_n(S)$.
    \item
      $V(F^{m,q}_n(S)) \subset F^{m+1,q}_n(S)$.
    \item
      $F(F^{m,q}_n(S)) \subset F^{m-1,q}_n(S)$.
    \item
      $R(F^{m,q}_n(S)) = \left\{\begin{array}{ll}
F^{m-1,q}_{n/p}(S) & \mbox{if $p|n$} \\
0  & \mbox{if $p\nmid n$}.
\end{array}\right.$

      \end{enumerate}
\end{lem}
\begin{proof}
 We divide the proof into two cases: (1) $p^{m-1}|n$ and (2) $p^{m-1}\nmid n$.
  \\
  $Case~1:$ $n = p^{m-1}i$. \\
  In this case, 
  $F^{m,q}_n(S) = (W_m\Omega^q_S + W_m\Omega^{q-1}_S\dlog([\pi]_m))[\pi]^i_m$.
  Since
  $d(W_m\Omega^q_S[\pi]^i_m) \subset  W_m\Omega^{q+1}_S[\pi]^i_m +
    W_m\Omega^q_S[\pi]^i_m \dlog([\pi]_m) = F^{m,q+1}_n(S)$,
    and $d(W_m\Omega^{q-1}_S[\pi]^i_m\dlog([\pi]_m)) \subset
    W_m\Omega^q_S[\pi]^i_m\dlog([\pi]_m)
     \subset F^{m,q+1}_n(S)$, item (1) follows.

 For (2), note that if $i \in I_p$, then
    $V(W_m\Omega^q_S[\pi]^i_m) \in F^{m+1,q}_n(S)$.
    Next, we have
    \[
    \begin{array}{lll}
      V(W_m\Omega^{q-1}_S\dlog([\pi]_m)[\pi]^i_m) & = &
      V(W_m\Omega^{q-1}_S d([\pi]^i_m)) \\
      & \subset & Vd(W_m\Omega^{q-1}_S [\pi]^i_m) + V(W_{m}\Omega^{q}_S[\pi]^i_m)
      \subset  F^{m+1,q}_n(S).
       \end{array}
    \]
    If $i = pi'$, then $n = p^mi'$ and
    $V(W_m\Omega^q_S[\pi]^i_m)
    \subset W_{m+1}\Omega^q_S[\pi]^{i'}_{m+1} \subset F^{m+1,q}_n(S)$.
   Also, we have $V(W_m\Omega^{q-1}_S\dlog[\pi]_m[\pi]^i_m)
    \subset  W_{m+1}\Omega^{q-1}_S\dlog[\pi]_{m+1}[\pi]^{i'}_{m+1} \subset F^{m+1,q}_n(S)$.

 For (3), write $n = p^{m-2}i'$, where $i' = pi$. We then get
    $F(W_m\Omega^q_S[\pi]^i_m) \subset W_{m-1}\Omega^q_S[\pi]^{i'}_{m-1}$ 
    which lies in $F^{m-1,q}_n(S)$. Similarly,
    \[
    F(W_m\Omega^{q-1}_S\dlog([\pi]_m)[\pi]^i_m) \subset
    W_{m-1}\Omega^{q-1}_S\dlog([\pi]_{m-1})[\pi]^{i'}_{m-1} \subset F^{m-1,q}_n(S).
    \]
    For (4), note that $R(W_m\Omega^q_S[\pi]^i_m) = W_{m-1}\Omega^q_S[\pi]^{i}_{m-1}$ and
    $R(W_m\Omega^{q-1}_S\dlog[\pi]_m[\pi]^i_m)
    = W_{m-1}\Omega^{q-1}_S\dlog[\pi]_{m-1}[\pi]^i_{m-1}$. But this implies that
    $R(F^{m,q}_{n}(S))=F^{m-1,q}_{n'}(S)$, where $n' = n/p = p^{m-2}i$.
 \\
    \noindent
    \\
    $Case~2:$ $n = p^{m-1-s}i$ with $s \ge 1$ and $|i| \in I_p$. \\
    In this case, we have
    $F^{m,q}_n(S) = V^s(W_{r}\Omega^q_S[\pi]^i_r) + dV^s(W_r\Omega^{q-1}_S[\pi]^i_r)$,
    where $r = m-s$ $= (m+1)-(s+1)$.  But this implies that
    $d(F^{m,q}_n(S)) = dV^s(W_{r}\Omega^q_S[\pi]^i_r) \subset F^{m,q+1}_n(S)$.
    This proves (1). For (2), note that
    $V(F^{m,q}_n(S)) = V^{s'}(W_{r}\Omega^q_S[\pi]^i_r) +
    VdV^s(W_r\Omega^{q-1}_S[\pi]^i_r) \subset$ \\
    $V^{s'}(W_{r}\Omega^q_S[\pi]^i_r) + dV^{s'}(W_r\Omega^{q-1}_S[\pi]^i_r) =
    F^{m+1,q}_n(S)$ if we let $s' = s+1$.
   
For (3), suppose first that $s = 1$ so that $n = p^{m-2}i$.
    Then
    \[
    \begin{array}{lll}
     F(F^{m,q}_n(S)) & = & FV(W_{m-1}\Omega^q_S[\pi]^i_{m-1}) +
     FdV(W_{m-1}\Omega^{q-1}_S[\pi]^i_{m-1}) \\
     & \subset & W_{m-1}\Omega^q_S[\pi]^i_{m-1} + d(W_{m-1}\Omega^{q-1}_S[\pi]^i_{m-1}) \\
     & \subset & \left(W_{m-1}\Omega^q_S + 
     W_{m-1}\Omega^{q-1}_S\dlog([\pi]_{m-1})\right)[\pi]^i_{m-1} = F^{m-1,q}_n(S).
\end{array}
    \]
If $s \ge 2$, we let $s' = s-1 \ge 1$ so that $n = p^{m-2-s'}i$ and
    $r = (m-1) - s'$.
    This yields
    \[
    \begin{array}{lll}
      F(F^{m,q}_n(S)) & = & FV^s(W_{r}\Omega^q_S[\pi]^i_{r}) +
      FdV^s(W_{r}\Omega^{q-1}_S[\pi]^i_{r}) \\
      & \subset &
    V^{s'}(W_{r}\Omega^q_S[\pi]^i_{r}) +
    dV^{s'}(W_{r}\Omega^{q-1}_S[\pi]^i_{r}) = F^{m-1,q}_n(S).
    \end{array}
    \]

 Finally, for (4), note that
    $$RV^s(W_r\Omega^q_S[\pi]^i_r) + RdV^s(W_r\Omega^{q-1}_S[\pi]^i_r)
    = \hspace{3cm} $$
    $$ \hspace{4cm} V^s(W_{r-1}\Omega^q_S[\pi]^{{n'}/{p^{m-2-s}}}_{r-1})+
    dV^s(W_{r-1}\Omega^{q-1}_S
    [\pi]^{{n'}/{p^{m-2-s}}}_{r-1}),$$ where $n' = n/p$. The latter group is in
    $F^{m-1,q}_{n'}(S)$
    if $s < m-1$ and is zero if $s = m-1$. This completes the proof.
\end{proof}

\begin{lem}\label{lem:F-group-0}
  For all $n, n' \ge 0$, we have $F^{m,0}_n(S) \cdot F^{m,q}_{n'}(S) \subset
  F^{m, q}_{n+n'}(S)$. 
\end{lem}
\begin{proof}   
  First note that $W_m(S)\cdot F^{m,q}_{n'}(S) \subset F^{m,q}_{n'}(S)$. Indeed, if
  $p^{m-1} \mid n'$, then this is obvious. If $p^{m-1} \nmid n'$, then only need to show
  $W_m(S)\cdot dV^s(W_r\Omega^{q-1}_S[\pi]^i_r) \subset F^{m,q}_{n'}(S)$, where
  $n'=p^{m-1-s}i$, $|i| \in I_p$, $r=m-s$.
  But this follows easily by applying the Leibniz rule.
We now assume $1 \le s \le m-1$ and write $r = m-s$. If $n = p^{m-1}i$, 
 then $F^{m,0}_n(S) = W_m(S)[\pi]^i_m$ and we get $F^{m,0}_n(S) \cdot F^{m,q}_{n'}(S) =
  W_m(S)[\pi]^i_m \cdot F^{m,q}_{n'}(S) \subset F^{m, q}_{n+n'}(S)$
  by the above observation and \lemref{lem:Positive-F-2}.

  If $|i| \in I_p$ and $n = p^{m-1-s}i$, then $F^{m,0}_n(S) = V^s(W_{r}(S)[\pi]^i_{r})$
  and we get
  \[
\begin{array}{lll}
  F^{m,0}_n(S) \cdot F^{m,q}_{n'}(S) & = & V^s(W_{r}(S)[\pi]^i_{r}) \cdot 
 F^{m,q}_{n'}(S) \\
 &\subset^1 & V^s([\pi]^i_{r}.F^{m-s,q}_{n'}(S)) \subset^2 V^s(F^{m-s,q}_{n'+n}(S))
 \subset^3  F^{m-s,q}_{n'+n}(S),
\end{array}
\]
where $\subset^1$ uses \lemref{lem:F-group-1}(3),
$\subset^2$ uses  \lemref{lem:Positive-F-2}, while
$\subset^3$ uses \lemref{lem:F-group-1}(2). 
\end{proof}

\begin{lem}\label{lem:F-group-6}
  For all $n \ge 0$, the map $\theta^{m,q}_A$ induces an inclusion 
  $F^{m,q}_n(S) \hookrightarrow \Fil_{-n}W_m\Omega^q_K$.
\end{lem}
\begin{proof}
If $n = p^{m-1}i$, we get
  \[
  \begin{array}{lll}
    F^{m,q}_n(S) & = & A^{m,q}_i(S) =
    \left(W_m\Omega^q_S+ W_m\Omega^{q-1}_S\dlog([\pi]_m)\right)[\pi]^i_m \\
    & \inj & (W_m\Omega^q_A+ W_m\Omega^{q-1}_A\dlog([\pi]_m))[\pi]^i_m \\
    & = & \Fil_{-n}W_m\Omega^q_K,
  \end{array}
  \]
  where the middle inclusion is clear (cf. \thmref{thm:GH-Top}) and
  last equality uses \lemref{lem:Log-fil-4}(1).

 If $n = p^{m-1-s}i$ with $1 \le s \le m-1$ and $|i| \in I_p$, then
   \[
  \begin{array}{lll}
    F^{m,q}_n(S) & = & B^{m,q}_{i,s} = V^s(W_{r}\Omega^q_S[\pi]^i_r) +
    dV^s(W_r\Omega^{q-1}_S[\pi]^i_r) \\
    & {\inj} & V^s(\Fil_{-p^{r-1}i}W_r\Omega^q_K) +
    dV^s(\Fil_{-p^{r-1}i}W_r\Omega^{q-1}_K)
    \\
    & = & V^s(\Fil_{-n}W_r\Omega^q_K) + dV^s(\Fil_{-n}W_r\Omega^{q-1}_K) {\subset}
    \Fil_{-n}W_m\Omega^q_K,
    \end{array}
  \]
  where $r = m-s$ and the last inclusion holds by \lemref{lem:Log-fil-Wcom}.
\end{proof}

\subsection{Decomposition of $\Fil_nW_m\Omega^q_K$ into $\{F^{m,q}_n\}$}
\label{sec:Decom-fil-F}
In order to describe $\Fil_nW_m\Omega^q_K$ in terms of $\{F^{m,q}_n\}$ via
$\theta^{m,q}_K$, we need to dispose of the case $q = 0$. To do this, we
observe that $F^{m,0}_n(S) \subset \Fil_{-n}W_m(K)$ for $n \ge 0$, as one easily
deduces from \lemref{lem:F-group-6} for $q=0$.

\begin{lem}\label{lem:F-group-2}
  For $n \in \Z$, the composite map $\lambda^{m,0}_n \colon
  F^{m,0}_{n}(S) \inj \Fil_{-n}W_m(K) \surj \gr_{-n}W_m(K)$ is an isomorphism.
\end{lem}
\begin{proof}
  The assertion of the lemma is obvious if $m =1$ because $F^{1,0}_n(S)= S \pi^n$ and
  $\Fil_{-n}W_1(K)= \pi^n A$. 
For $m > 1$,
 \lemref{lem:Log-fil-1} implies that the sequence
  \begin{equation}\label{eqn:F-group-2-0}
0 \to \Fil_{n}W_1(K) \xrightarrow{V^{m-1}}  \Fil_{n}W_m(K) \xrightarrow{R}
\Fil_{\lfloor{n}/p\rfloor}W_{m-1}(K) \to 0
  \end{equation}
  is exact for every $n \in \Z$. Comparing this sequence for $-n$ and $-n-1$ and
  noting that $\Fil_{-n-1}W_m(K)\subset\Fil_{-n}W_m(K)$, we get an exact sequence
\begin{equation}\label{eqn:F-group-2-1}  
0 \to \gr_{-n}W_1(K) \xrightarrow{V^{m-1}} \gr_{-n}W_m(K) \xrightarrow{R} 
\frac{\Fil_{-\lceil{n/p}\rceil}W_{m-1}(K)}{\Fil_{-\lceil{{n+1}/p}\rceil}W_{m-1}(K)} \to 0.
\end{equation}

If we write $n+1 = pt + q$ with $0 \le q < p$, we see that $V^{m-1}$ is an isomorphism
if $q \neq 1$. Since $s = m-1$ (cf. \defref{defn:GH-1}) in this case, we find also
that the image of $V^{m-1}$ is $F^{m,0}_n(S)$.
If $q =1$, we can rewrite ~\eqref{eqn:F-group-2-1} as
\[
0 \to \gr_{-n}W_1(K) \xrightarrow{V^{m-1}} \gr_{-n}W_m(K) \xrightarrow{R} 
\gr_{-t}W_{m-1}(K) \to 0.
\]

We now look at the diagram
\begin{equation}\label{eqn:F-group-2-2} 
  \xymatrix@C1.8pc{
    0 \ar[r] & F^{1,0}_n(S) \ar[r]^-{V^{m-1}} \ar[d]_-{\lambda^{1,0}_n} &
    F^{m,0}_n(S) \ar[r]^-{R} \ar[d]^-{\lambda^{m,0}_n} & F^{m-1,0}_t(S) \ar[r]
    \ar[d]^-{\lambda^{m-1,0}_t} & 0 \\
    0 \ar[r] & \gr_{-n}W_1(K) \ar[r]^-{V^{m-1}} & \gr_{-n}W_m(K) \ar[r]^-{R} &
    \gr_{-t}W_{m-1}(K) \ar[r] & 0.}
\end{equation}
It is easy to check that this diagram is commutative and the top row is exact.
Now, the left and the right vertical arrows are bijective by induction.
We conclude that the same is true for the middle vertical arrow too.
\end{proof}

Before we proceed further, we need to compare the topologies of $W_m\Omega^q_A(\log\pi)$
induced by various filtrations. We let $\tau_1$ be the topology on
$W_m\Omega^q_A(\log\pi)$ induced by the filtration $
\{W_m\Omega^q_A(\log\pi)(n)\}_{n\ge 0}$ (cf. the beginning of \S~\ref{sec:GSPS}).
We let $\tau_2$ be the topology on
$W_m\Omega^q_A(\log\pi)$ induced by the filtration
$\{W_m((\pi^n))W_m\Omega^q_A(\log\pi)\}_{n \ge 0}$. We let $\tau_3$ be the
topology on $W_m\Omega^q_A(\log\pi)$ induced by the filtration
$\{I^nW_m\Omega^q_A(\log\pi)\}_{n \ge 0}$, where $I = W_m((\pi))$. Finally, we let
$\tau_4$ be the topology on $W_m\Omega^q_A(\log\pi)$ induced by the filtration 
$\{\Fil_{-n}W_m\Omega^q_K\}_{n \ge 0}$.

\begin{lem}\label{lem:Topology-0}
  The topologies defined above all agree and $W_m(A)$ is complete with respect
  to these topologies.
\end{lem}
\begin{proof}
  The topologies $\tau_1, \ \tau_2$ and $\tau_3$ will
  agree on $W_m\Omega^q_A(\log\pi)$ if
  they do so on $W_m(A)$. The agreement between $\tau_1$ and $\tau_2$ in this special
  case is an elementary exercise and the the same between $\tau_2$ and $\tau_3$
  follows from \cite[Cor.~2.3]{Geisser-Hesselholt-Top}. The agreement between
  $\tau_1$ and $\tau_4$ on $W_m\Omega^q_A(\log\pi)$ follows from
  \lemref{lem:Log-fil-4}(1).
  To prove the second part, it is enough to check that $W_m(A)$ is
  $\tau_2$-complete.
  But this follows from \lemref{lem:W-nil}.
  \end{proof}

\begin{cor}\label{cor:F-group-3}
For $n \ge 0$, the map $\theta^{m,q}_A$ of ~\eqref{eqn:GH-2-0}
induces an isomorphism
\[
\theta^{m,0}_n \colon {\underset{n' \ge n}\prod}
F^{m,0}_{n'}(S) \xrightarrow{\cong} \Fil_{-n}W_m(K).
\]
\end{cor}
\begin{proof}
  The ring $W_m(A)$ is Noetherian and $I$-adically  (where $I = W_m((\pi))$)
  complete by \propref{prop:F-fin}(5) and \lemref{lem:Topology-0}. Since
  $\Fil_{-n}W_m(K)$ is an
  ideal of $W_m(A)$ by \lemref{lem:Log-fil-4}(1), it follows that $\Fil_{-n}W_m(K)$
    is $I$-adically complete and hence $\tau_4$- complete by \lemref{lem:Topology-0}.
  That is, the canonical map
$\Fil_{-n}W_m(K) \to {\varprojlim}_{n' \ge 0} \frac{\Fil_{-n}W_m(K)}{\Fil_{-n-n'}W_m(K)}$
is an isomorphism of $W_m(A)$-modules. Meanwhile, one deduces from
\lemref{lem:F-group-2} that the canonical map
$\stackrel{n'-1}{\underset{i = 0}\oplus} F^{m,0}_{n+i}(S) \to
\frac{\Fil_{-n}W_m(K)}{\Fil_{-n-n'}W_m(K)}$ is an isomorphism.
Passing to the limit as $n' \to \infty$, we get an isomorphism
${\underset{n' \ge n}\prod}
F^{m,0}_{n'}(S) \xrightarrow{\cong} \Fil_{-n}W_m(K)$,
which is induced by the canonical maps
$\theta^{m,0}_A \colon F^{m,0}_{n'}(S) \to W_m(A)$. 
\end{proof}

In order to generalize \corref{cor:F-group-3} to the case when $q > 0$, we let
$\Fil'_{-n}W_m\Omega^q_K = \theta^{m,q}_A({\prod}_{j \ge n}
F^{m,q}_{j}(S))$ for $n \ge 0$. By \corref{cor:GH-2}, this is a
descending filtration on $W_m\Omega^q_A(\log\pi) = \Fil'_{0}W_m\Omega^q_K$.

\begin{lem}\label{lem:fil'-mod}
  For $n \ge 0$, $\Fil'_{-n}W_m\Omega^q_K$ is a $W_m(A)$-submodule of
  $W_m\Omega^q_A(\log \pi)$.
\end{lem}
\begin{proof}
  We only need to show that $\Fil'_{-n}W_m\Omega^q_K$ is closed under the action of
  $W_m(A)$. We let $\omega = \sum_{j \ge n} \alpha_j \in \Fil'_{-n}W_m\Omega^q_K$,
  where $\alpha_j \in \theta^{m,q}_A(F^{m,q}_{j}(S))$. Write
  $\omega_i = \stackrel{i}{\underset{j = n}\sum} \alpha_j$. Since
  $W_m\Omega^q_A(\log \pi)$ is finitely generated $W_m(A)$-module
  (cf. \lemref{lem:Log-fil-4}(2)), it is complete with respect to the
  $I$-adic topology.
  So, we can write $\omega = \lim_{i \to \infty} \omega_i$, where the limit is taken
  with respect to the
  $\tau_3$ topology. Now let $\gamma = \sum_{j \ge 0} \beta_j \in W_m(A)$, where
  $\beta_j \in \theta^{m,0}_A(
  F^{m,0}_{j}(S))$. Similarly to above, we can write $\gamma = \lim_{i \to \infty}
  \gamma_i$, where $\gamma_i = \stackrel{i}{\underset{j = 0}\sum} \beta_j$ and the
  limit is taken in the $\tau_3$ topology. Then
  $\lim_{i \to \infty} \gamma_i \cdot \omega_i =
  \gamma \cdot \omega$. But $\gamma_i \cdot \omega_i \subset
  \theta^{m,q}_A(\bigoplus\limits_{j=n}^{2i} F^{m,q}_j(S))$ by \lemref{lem:F-group-0}.
  Hence, $\gamma \cdot \omega \in \sum_{j \ge n}\theta^{m,q}_A ( F^{m,q}_j(S))=
  \Fil'_{-n}W_m\Omega^q_K$.
\end{proof}

\begin{lem}\label{lem:Topology-1}
  $d(\Fil'_{-n}W_m\Omega^q_K) \subset \Fil'_{-n}W_m\Omega^{q+1}_K$ for every
  $n \ge 0$.
\end{lem}
\begin{proof}
  We continue using the notations from the proof of \lemref{lem:fil'-mod}. For
  $\omega \in \Fil'_{-n}W_m\Omega^q_K$, we can write $\omega = \lim_{i \to \infty}
  \omega_i$, where the limit is taken with respect to
  $\tau_3$ and hence with respect to $\tau_4$ by \lemref{lem:Topology-0}.
  We now let $n' \ge 0$ be any integer. We can then find $i \gg 0$ such that
  $\omega - \omega_i \in \Fil_{-n'}W_m\Omega^q_K$ for all $i \gg n$. 
  This implies that $d(\omega) - d(\omega_i) = d(\omega- \omega_i) \in
  \Fil_{-n'}W_m\Omega^{q+1}_K$ for all $i \gg n$ by \lemref{lem:Log-fil-Wcom}. 
But this means that $d(\omega) = \lim_{i \to \infty} d(\omega_i) =
\sum_{j \ge n} d(\alpha_j)$, where $d(\alpha_j)\subset
d(\theta^{m,q}_A(F^{m,q}_{j}(S)))=^1
\theta^{m,q+1}_A(d(F^{m,q}_{j}(S))) \subset^2
\theta^{m,q+1}_A(F^{m,q+1}_{j}(S))$.
Here, the equality $=^1$ holds because the map $\theta^{m,q}_A$ is induced by the
natural map $F^{m,q}_{j}(S) \to W_m\Omega^q_A(\log \pi)$, which commutes with $d$.
The inclusion $\subset^2$ holds by \lemref{lem:F-group-1}(1).
\end{proof}

\begin{lem}\label{lem:F-group-4}
  We have $\Fil_{-n}W_m\Omega^q_K \subseteq \Fil'_{-n}W_m\Omega^q_K$ for every
  $n \ge 0$.
\end{lem}
\begin{proof}
By definition of $\Fil_{-n}W_m\Omega^q_K$ and Corollaries ~\ref{cor:GH-2} and
  ~\ref{cor:F-group-3}, we have
  \[
  \begin{array}{lll}
  \Fil_{-n}W_m\Omega^q_K & = & \Fil_{-n}W_m(K)W_m\Omega^q_A(\log\pi) +
  d(\Fil_{-n}W_m(K))W_m\Omega^{q-1}_A(\log\pi) \\
  & = &  \Fil'_{-n}W_m(K)W_m\Omega^q_A(\log\pi) +
  d(\Fil'_{-n}W_m(K))W_m\Omega^{q-1}_A(\log\pi) \\
  & \subset &  \Fil'_{-n}W_m(K)W_m\Omega^q_A(\log\pi) +
  d(\Fil'_{-n}W_m(K)W_m\Omega^{q-1}_A(\log\pi)) \\
  & & + \ \Fil'_{-n}W_m(K)d(W_m\Omega^{q-1}_A(\log\pi)).
  \end{array}
  \]

  Next, note that every element of $\Fil'_{-n}W_m(K)W_m\Omega^q_A(\log\pi)$ is a
  finite sum of elements of the form $\gamma \cdot \omega$, where
  $\gamma \in \Fil'_{-n}W_m(K)$ and $\omega \in W_m\Omega^q_A(\log\pi)=
  \Fil'_0W_m\Omega^q_K$. Following the proof of
  \lemref{lem:fil'-mod}, we write  $\omega = \lim_{i \to \infty} \omega_i$, where
  $\omega_i = \stackrel{i}{\underset{j = 0}\sum} \alpha_j$ and $\gamma =
  \lim_{i \to \infty} \gamma_i$, where $\gamma_i = \stackrel{i}{\underset{j = n}\sum}
  \beta_j$ (note the change in indices of $\alpha_j$ and $\beta_j$). These limits are
  taken in the $\tau_3$ topology. This yields
  $\gamma \cdot \omega = \lim_{i \to \infty} \gamma_i \cdot \omega_i \in
  \Fil'_{-n}W_m\Omega^q_K$.
  This implies $\Fil'_{-n}W_m(K)W_m\Omega^q_A(\log\pi) \subset
  \Fil'_{-n}W_m\Omega^q_K$.

Similarly,
$\Fil'_{-n}W_m(K)d(W_m\Omega^{q-1}_A(\log\pi)) \subset 
\Fil'_{-n}W_m(K)W_m\Omega^{q}_A(\log\pi) \subset \Fil'_{-n}W_m\Omega^q_K$.
Finally, $d(\Fil'_{-n}W_m(K)W_m\Omega^{q-1}_A(\log\pi)) \subset
d(\theta^{m,q-1}_A({\prod}_{j \ge n} F^{m,q-1}_{j}(S))) = d(\Fil'_{-n}W_m\Omega^{q-1}_K)$
by \lemref{lem:F-group-0}.
We are done since the latter group lies in $\Fil'_{-n}W_m\Omega^{q}_K$ by
\lemref{lem:Topology-1}.
\end{proof}

\begin{cor}\label{cor:F-group-5}
  For $n \ge 0$, $\Fil'_{-n}W_m\Omega^q_K$ is open (hence closed) in
  $W_m\Omega^q_A(\log\pi)$ with
  respect to the topology $\tau_i$ for $1 \le i \le 4$.
\end{cor}
\begin{proof}
  Combining Lemmas~\ref{lem:fil'-mod} and ~\ref{lem:F-group-4}, we get that
  $\Fil'_{-n}W_m\Omega^q_K$ is $\tau_4$-open. We now apply \lemref{lem:Topology-0}. 
\end{proof}
 
\begin{lem}\label{lem:F-group-7}
We have $\Fil'_{-n}W_m\Omega^q_K = \Fil_{-n}W_m\Omega^q_K$ for every $n \ge 0$.
\end{lem}
\begin{proof}
In view of \lemref{lem:F-group-4}, we only need to show that
  $\Fil'_{-n}W_m\Omega^q_K \subset \Fil_{-n}W_m\Omega^q_K$.
  The latter assertion is obvious if $n =0$ because $\Fil_{0}W_m\Omega^q_K =
  W_m\Omega^q_A(\log\pi) = \Fil'_{0}W_m\Omega^q_K$ by
  \corref{cor:GH-2}. We assume therefore that $n \ge 1$.
  In particular, $\Fil'_{-n}W_m\Omega^q_K \subset W_m\Omega^q_A$.
  We now let $\omega \in \Fil'_{-n}W_m\Omega^q_K$.
We write $\omega = \lim_{i \to \infty} \omega_i$ as in the proof of
\lemref{lem:fil'-mod}, where the limit is taken with respect to $\tau_3$.
\lemref{lem:F-group-6} then says that $\omega_i \in \Fil_{-n}W_m\Omega^q_K$
    for each $i \ge n$. Since $\Fil_{-n}W_m\Omega^q_K$ is $\tau_4$-open (and hence
    closed) in $W_m\Omega^q_A(\log\pi)$ by \corref{cor:F-group-5}, we conclude that
    $\omega \in \Fil_{-n}W_m\Omega^q_K$.
\end{proof}

We shall now describe $\Fil_nW_m\Omega^q_K$ for arbitrary
$n \in \Z$. For $l \in \Z$, we let $t = n-p^{m-1}l $.

\begin{lem}\label{lem:Positive-F-0}
  We have $\Fil_tW_m(K) = [\pi]^{l}_m \Fil_{n}W_m(K)$ for every $n \in \Z$.
\end{lem}
\begin{proof}
Let $x = (a_{m-1}, \ldots , a_0) \in \Fil_nW_m(K)$, that is,
  $a_i^{p^i} \pi^n \in A$ for each $i$.
  Then $[\pi]^{l}_mx = (a_{m-1}\pi^{l}, a_{m-2}\pi^{pl}, \ldots , a_0\pi^{p^{m-1}l})$.
  Letting $b_i = a_i\pi^{p^{m-1-i}l}$, we get $(b_i)^{p^i}\pi^{t}
  = a_i^{p^i}\pi^{p^{m-1}l+t} = a_i^{p^i}\pi^n \in A$. This implies that
  $[\pi]^{l}_m\Fil_nW_m(K) \subset \Fil_{t}W_m(K)$ for all $l,n \in \Z$. Then,
  we also get $[\pi]^{-l}_m\Fil_tW_m(K) \subset \Fil_{n}W_m(K)$, as desired.
\end{proof}

\begin{lem}\label{lem:Positive-F-1}
  We have $\Fil_tW_m\Omega^q_K = [\pi]^{l}_m \Fil_{n}W_m\Omega^q_K$ for every $n \in \Z$.
\end{lem}
\begin{proof}
By definition, $\Fil_tW_m\Omega^q_K = \Fil_{t}W_m(K)W_m\Omega^q_A(\log\pi) +
d(\Fil_{t}W_m(K))W_m\Omega^{q-1}_A(\log\pi)$.
By \lemref{lem:Positive-F-0}, $\Fil_{t}W_m(K)W_m\Omega^q_A(\log\pi) 
=  [\pi]^{l}_m \Fil_{n}W_m(K)W_m\Omega^q_A(\log\pi)$ and
\[
\begin{array}{lll}
d(\Fil_{t}W_m(K))W_m\Omega^{q-1}_A(\log\pi) & = &
d([\pi]^{l}_m \Fil_{n}W_m(K))W_m\Omega^{q-1}_A(\log\pi) \\
& \subset & [\pi]^{l}_m d(\Fil_{n}W_m(K)) W_m\Omega^{q-1}_A(\log\pi) \\
& & + \ \Fil_{n}W_m(K)[\pi]^{l}_mW_m\Omega^{q-1}_A(\log\pi) \dlog([\pi]_m) \\
& \subset &  [\pi]^{l}_m \big( d(\Fil_{n}W_m(K)) W_m\Omega^{q-1}_A(\log\pi)  \\
& & + \ \Fil_{n}W_m(K)W_m\Omega^{q-1}_A \dlog([\pi]_m)\big) \subset
[\pi]^{l}_m\Fil_{n}W_m\Omega^q_K.
\end{array}
\]
This shows that $\Fil_tW_m\Omega^q_K \subseteq [\pi]^{l}_m \Fil_{n}W_m\Omega^q_K$ for
all $l \in \Z$. The reverse inclusion also follows in the same way we did before
(cf. \lemref{lem:Positive-F-0}).
\end{proof}

We can finally describe the pre-image of $\Fil_\bullet W_m\Omega^q_K$ under the
isomorphism $\theta^{m,q}_K$ of \corref{cor:GH-3} as follows.

\begin{prop}\label{prop:Fil-decom}
  For $n \in \Z$, there is a canonical decomposition
\[
\theta^{m,q}_K \colon
 \prod_{j \ge -n} F^{m,q}_j(S) \xrightarrow{\cong} \Fil_nW_m\Omega^q_K.
  \]
  \end{prop}
\begin{proof}
 For $n \le 0$, this follows from \lemref{lem:F-group-7}.
Suppose now $n \ge 0$. Choose $l \gg 0$ such that $t =n- p^{m-1}l <0$.
 \lemref{lem:Positive-F-1} then says that
  $\Fil_nW_m\Omega^q_K = [\pi]^{-l}_m\Fil_{t}W_m\Omega^q_K$.
  On the other hand, we have
  \[
    [\pi]^{-l}_m\Fil_{t}W_m\Omega^q_K \ {\cong^1} \ 
    [\pi]^{-l}_m\left(\prod_{j \ge- t}F^{m,q}_j(S)\right) \ \cong \
    \prod_{j \ge -t} [\pi]^{-l}_mF^{m,q}_j(S) \ \cong \ \prod_{j \ge -t}
    F^{m,q}_{j - p^{m-1}l}(S),
    \]
    where the first and the last isomorphisms are by \lemref{lem:F-group-7}
    (induced by $(\theta^{m,q}_A)^{-1}$) and \lemref{lem:Positive-F-2}, respectively.  
By a change of variables $j'=j-p^{m-1}l$, we get
$\prod_{j \ge -t} F^{m,q}_{j - p^{m-1}l}(S) = \prod_{j' \ge -n} F^{m,q}_{j'}(S)$.
Hence, we get an isomorphism $\Fil_nW_m\Omega^q_K \cong
\prod_{j' \ge -n} F^{m,q}_{j'}(S)$, whose inverse is induced by the canonical map
$F^{m,q}_{j'}(S) \to \Fil_nW_m\Omega^q_K$ for $j' \ge -n$ (because $\cong^1$ is
induced by $(\theta^{m,q}_A)^{-1}$). We conclude that the resulting (inverse)
isomorphism is $\theta^{m,q}_K$.
\end{proof}

\subsection{Goodness of $\Fil_n\Omega^{\bullet}_K$}\label{sec:Good-GH}
We apply \propref{prop:Fil-decom} to get the following
result which proves the goodness of the filtered de Rham--Witt complex in a
special case.

\begin{prop}\label{prop:V-R-Fil}
  For every $n \in \Z$, there is a short exact sequence
  \[
  0 \to V^{m-1}(\Fil_n\Omega^q_K) + dV^{m-1}(\Fil_n\Omega^{q-1}_K) \to
  \Fil_nW_m\Omega^q_K \xrightarrow{R} \Fil_{\lfloor{n/p}\rfloor}W_{m-1}\Omega^q_K \to 0.
  \]
\end{prop}
\begin{proof}
  The surjectivity of $R$ is already shown in the proof of Step~4 of
  \lemref{lem:Log-fil-Wcom}. To prove the exactness in the middle, we let
  $M = V^{m-1}(\Fil_n\Omega^q_K) + dV^{m-1}(\Fil_n\Omega^{q-1}_K)$ and
  $N = \Ker(R)$. It follows then from \cite[Prop.~2.3]{Shiho} that
  $M \subset N =
  (V^{m-1}(\Omega^q_K) + dV^{m-1}(\Omega^{q-1}_K)) \bigcap \Fil_nW_m\Omega^q_K$.
  
We now let $x \in N$. By \propref{prop:Fil-decom}, we can uniquely write
  $x = \sum_{i \ge -n} a_i$ with $a_i \in F^{m,q}_i(S)$.
  By \corref{cor:GH-3}, we can also write 
  $x = V^{m-1}(\sum_{i \ge -l} b_i) + dV^{m-1}(\sum_{i \ge -l} c_i)$ for some $l \gg |n|$
and $b_i \in F^{1,q}_i(S), \ c_i \in F^{1,q-1}_i(S)$.
Hence, we get $\sum_{i \ge -n} a_i = \sum_{i \ge -l}(V^{m-1}(b_i) + dV^{m-1}(c_i))$.
Letting $a'_i = (V^{m-1}(b_i) + dV^{m-1}(c_i))$, we get from \lemref{lem:F-group-1}
that $a'_i \in F^{m,q}_i(S)$. We conclude from \corref{cor:GH-3} that
$a_i = a'_i$ for each $i \ge -l$. In particular, $a'_i = 0$ for $i < -n$.

We thus get 
$x = \sum_{i \ge -n} V^{m-1}(b_i) + \sum_{i \ge -n} dV^{m-1}(c_i)
= V^{m-1}\left(\sum_{i \ge -n} b_i\right) + dV^{m-1}\left(\sum_{i \ge -n} c_i\right)$.
Note here that $V$ and $d$ are continuous with respect to the topology $\tau_2$
by virtue of Lemmas~\ref{lem:Log-fil-Wcom} and ~\ref{lem:Topology-0}.
Since $b_i \in F^{1,q}_i(S)$ and $c_i \in F^{1,q-1}_i(S)$, \propref{prop:Fil-decom}
implies that $\sum_{i \ge -n} b_i \in \Fil_{n}\Omega^q_K$ and
$\sum_{i \ge -n}c_i \in \Fil_{n}\Omega^{q-1}_K$. Letting
$b = \sum_{i \ge -n} b_i$ and $c = \sum_{i \ge -n}c_i$, we get that
$x = V^{m-1}(b) + dV^{m-1}(c)$, where $b \in  \Fil_{n}\Omega^q_K$ and
$c \in \Fil_{n}\Omega^{q-1}_K$. It follows that $x \in M$.
\end{proof}

\section{Filtered de Rham--Witt complex of regular \texorpdfstring{$\F_p$}{Fp}-algebras}
\label{sec:FDW-Reg}
In this section, we prove the goodness of (i.e.,
extend \propref{prop:V-R-Fil} to) the filtered de Rham--Witt complex of
arbitrary regular local $F$-finite $\F_p$-algebras.
We begin with a description of the filtered
de Rham complex of a multivariate power series ring.

\subsection{de Rham complex of multivariate power series ring}
\label{sec:GH-multi-var}
We let $S$ be an $F$-finite Noetherian regular $\F_p$-algebra and 
let $A = S[[x_1,\ldots , x_d]]$. We let $ 1\le r \le d$ and $K = A_\pi$, where
$\pi = \prod_{1 \le i \le r} x_i$. For $\un{n} = (n_1, \ldots , n_r) \in \Z^r$ and the
invertible ideal $I = (x^{n_1}_1\cdots x^{n_r}_r) \subset K$, we 
denote $\Fil_I W_m\Omega^q_K$ by $\Fil_{\un{n}}W_m\Omega^q_K$ and write
$M(x^{n_1}_1\cdots x^{n_r}_r)$ as $M(\un{n})$ for any $W_m(A)$-module $M$
(cf. notation in the beginning of \S~\ref{sec:GSPS}). Also, $\un{n}/p$ will denote
$(\lfloor n_1/p \rfloor, \ldots ,\lfloor n_r/p \rfloor) \in \Z^r$.
We let $q \ge 0$ and $\Sigma$ be the set of 3-tuples of non-negative integers
$(u, s, t)$ such $u + s + t = q$, either $s = 0$ or
$1 \le i_1 < \cdots < i_s \le r$, and either $t = 0$ or 
$r+1 \le j_1 < \cdots < j_t \le d$.
We let $F^{1,q}_0(S) = {\underset{(u, s,t) \in \Sigma}\sum}
\Omega^{u}_S \dlog(x_{i_1}) \cdots \dlog(x_{i_s}) d(x_{j_1}) \cdots d(x_{j_t})$.
Here, the term $\dlog(x_{i_1}) \cdots \dlog(x_{i_s})$ appears in $F^{1,q}_0(S)$ only
when $s \neq 0$. The same rules applies to $d(x_{j_1}) \cdots d(x_{j_t})$.
One notes that $F^{1,q}_0(S) \subset \Omega^q_K$. For
$\un{m} = (m_1, \ldots , m_d) \in \Z^d$,
we let $F^{1,q}_{\un{m}}(S) = (\prod_i x^{m_i}_i)F^{1,q}_0(S)$.

\begin{lem}\label{lem:F-function-1}
  For any $\un{n} = (n_1, \ldots , n_r) \in \Z^r$, there is a canonical isomorphism
  of $S$-modules
  \[
  \Fil_{\un{n}}\Omega^q_K \xrightarrow{\cong} {\underset{\un{m} \in \N_0^d}\prod}
  (x^{-n_1}_1\cdots x^{-n_r}_r)F^{1,q}_{\un{m}}(S).
  \]
\end{lem}
\begin{proof}
  By \lemref{lem:Log-fil-4}, it suffices to prove this when $\un{n} = \un{0}$.
  Now, it is an elementary computation using the description of
  $\Omega^1_A(\log \pi)$ in \S~\ref{sec:de-log-0}, \lemref{lem:Kahler-seq}(1) (which
  splits in the present case) and induction on $r$ that
  $\Omega^1_A \cong \left(\Omega^1_S \bigoplus (\stackrel{r}{\underset{i =1}\oplus}
 S \dlog(x_i)) \bigoplus  (\stackrel{d}{\underset{j = r+1}\oplus}
 S d(x_j))\right) \otimes_S A$. 
 Form this, we get
 \[
 \Fil_{{0}}\Omega^q_K \cong \Omega^q_A(\log \pi) =
   \bigwedge^q(\Omega^1_A(\log \pi)) \cong
{\underset{(u, s,t) \in \Sigma}\sum}
A\Omega^{u}_S \dlog(x_{i_1}) \cdots \dlog(x_{i_s}) d(x_{j_1}) \cdots d(x_{j_t}).
\]
As $A = {\underset{\un{m} \in \N_0^d}\prod} (x^{m_1}_1\cdots x^{m_d}_d)S$, we see that the
right end term coincides with
${\underset{\un{m} \in \N_0^d}\prod}F^{1,q}_{\un{m}}(S)$.
 This completes the proof.
\end{proof}

For $\un{m} = ((m_1, \ldots , m_r), (m_{r+1}, \ldots , m_d)) \in \Z^r \times
\N_0^{d-r}$,
we let $x^{\un{m}} = \prod_{1 \le i \le d} x^{m_i}_i$.
Taking the colimit of
$\Fil_{\un{n}} \Omega^q_K$ as $\un{n} \to \infty$, we get
\begin{cor}\label{cor:F-function-2}
  Every element $\omega \in \Omega^q_K$ can be written uniquely as an infinite
  series
  $\omega = {\underset{\un{m} \in \Z^r \times \N_0^{d-r}}\sum} a_{\un{m}}x^{\un{m}}$,
  where $a_{\un{m}} \in F^{1,q}_0(S)$ for all $\un{m}$ and $a_{\un{m}} = 0$ if
  $m_i \ll 0$ for some $1 \le i \le r$. Furthermore, $\omega \in
  \Omega^q_A(\log\pi)$ if
  and only if $a_{\un{m}} = 0$ for every $\un{m}$ such that $m_i < 0$ for some
  $1 \le i \le r$, and $\omega \in \Fil_{\un{n}}\Omega^q_K$
  if and only if $a_{\un{m}} =0 $ for every $\un{m}$ such that
  $m_i < -n_i$ for some $1 \le i \le r$.
  \end{cor}

The following result is standard.

  \begin{lem}\label{lem:W-nil}
    Let $m \ge 1$ and $S$ be an $\F_p$-algebra. Let $I \subset S$ be an ideal and
    let $\wh{S}$ denote the $I$-adic completion of $S$. Then the canonical
    map $W_m(\wh{S}) \to \varprojlim_n {W_m(S)}/{W_m(I^n)}$ is an isomorphism.
    If $S$ is a local ring with maximal ideal $\fm$, then $W_m(S)$ is a local
    ring with maximal ideal $\fM_m = \{\un{a} =
    {(a_{m-1}, \ldots, a_0)}| \ a_{m-1} \in \fm\}$ and ${\fM_m}/{W_m(\fm)}$ is a
    nilpotent ideal in $W_m({S}/{\fm})$. In particular, $(\fM_m)^n \subseteq W_m(\fm)$
    for all $n \gg 0$ if $S$ is moreover Noetherian.
   \end{lem}
  \begin{proof}
    The first claim follows from \cite[Lem.~2.1]{Geisser-Hesselholt-Top} since the
    canonical map $W_m(\wh{S}) \to \varprojlim_n W_m({S}/{I^{n}})$ is clearly bijective
    as one can check by induction on $m$. For the second claim, we recall that
    the kernel of the surjection $R \colon W_m(S) \surj W_{m-1}(S)$ is nilpotent
    because $V^{m-1}({a}) V^{m-1}({a}) = V(V^{m-2}({a}) \cdot FV^{m-1}({a})) =
    V(V^{m-2}({a}) \cdot V^{m-2}(p{a})) =0$ for any $a \in S$. This fact and
    induction on $m$ together show that $W_m(S)$ is local with maximal ideal
    $\fM_m$. Applying this fact to ${S}/{\fm}$ instead of $S$ and using induction on
    $m$, we get that ${\fM_m}/{W_m(\fm)}$ is nilpotent. The last claim follows from
    \propref{prop:F-fin} which says that $W_m(S)$ is Noetherian if $S$ is so.
    \end{proof}

\subsection{From complete to non-complete case}\label{sec:com-to-noncom}
We let $A$ be a regular local $F$-finite $\F_p$-algebra with 
maximal ideal $\fm = (x_1, \ldots , x_d)$. We let $\pi = x_1\cdots x_r$ and
$K = A_\pi$, where $1 \le r \le d$ and $L= Q(A)$. We let $\wh{A}$ denote the
$\fm$-adic completion of
$A$ with the maximal ideal $\wh{\fm}$ and let $\wh{K} = \wh{A}_\pi$.
For $1 \le i \le r$, we let $A_i = A_{(x_i)}, \ \wh{A}_i = \wh{A_{(x_i)}}$ and
$\wh{K}_i = Q(\wh{A}_i)$. We fix integers $q \ge 0$ and $m \ge 1$.
For $\un{n} = (n_1, \ldots , n_r) \in \Z^r$ and the invertible ideal
$I = (x^{n_1}_1\cdots x^{n_r}_r) \subset K$, we let $\Fil_{\un{n}}W_m\Omega^q_K =
\Fil_I W_m\Omega^q_K$.  We also let
$t\un{n}=(tn_1, {\ldots},tn_r)$ for any $t \in \Z$ and
$-\un{1} = (-1, \ldots , -1) \in \Z^r$. We let
$X = \Spec(A)$ and $X_f = \Spec(A[f^{-1}])$ for any $f \in A \setminus \{0\}$.
Let $\gamma_i \colon W_m\Omega^q_K \to W_m\Omega^q_{\wh{K}_i}$ denote the  map
induced by the inclusion $K \inj \wh{K}_i$.

\begin{lem}\label{lem:Complete-1}
  For $1 \le i \le r$, the map
  $\gamma_i \colon W_m\Omega^q_K \to  W_m\Omega^q_{\wh{K}_i}$ is injective.
  \end{lem}
  \begin{proof}
We look at the commutative diagram
    \begin{equation}\label{eqn:Complete-1-0}
      \xymatrix@C.5pc{
        W_m\Omega^q_A \ar[r]^-{\alpha_1} \ar[d] & W_m\Omega^q_{A_i} \ar[r]^-{\alpha_2}
        \ar[d] & W_m\Omega^q_{\wh{A}_i} \ar[d] \\
        W_m\Omega^q_A {\underset{W_m(A)}\otimes} W_m(A)_{[\pi]_m} \ar[r]^-{\beta_1}
        \ar[d]_-{\cong} &
        W_m\Omega^q_{A_i} {\underset{W_m(A)}\otimes} W_m(A)_{[\pi]_m}
        \ar[r]^-{\beta_2} \ar[d]^-{\cong} & 
        W_m\Omega^q_{\wh{A}_i} {\underset{W_m(A)}\otimes} W_m(A)_{[\pi]_m}
        \ar[d]^-{\cong} \\
        W_m\Omega^q_A {\underset{W_m(A)}\otimes} W_m(A)_{[\pi]_m} \ar[r]
        \ar[d]_-{\delta_1} &
       W_m\Omega^q_{A_i} {\underset{W_m(A_i)}\otimes} W_m(A_i)_{[x_i]_m} \ar[r]
       \ar[d]^-{\delta_2} & W_m\Omega^q_{\wh{A}_i} {\underset{W_m(\wh{A}_i)}\otimes}
          W_m(\wh{A}_i)_{[x_i]_m} 
       \ar[d]^-{\delta_3} \\
       W_m\Omega^q_{K} \ar[r]^-{\eta_1} & W_m\Omega^q_{L} \ar[r]^-{\eta_2} &
       W_m\Omega^q_{\wh{K}_i}.}
    \end{equation}

Since $A$ is a regular local $\F_p$-algebra, the composition
    $W_m\Omega^q_A \to W_m\Omega^q_L$ of all left vertical
    arrows with $\eta_1$ is injective (cf. \cite[Prop.~2.8]{KP-Comp}). It follows
    that $\alpha_1$ is injective. Next, \lemref{lem:W-nil} implies that 
    $W_m(\wh{A}_i)$ is the $W_m((x_i))$-adic completion of $W_m(A_i)$, and therefore, it
    is flat over $W_m(A_i)$. We can therefore use  \propref{prop:F-fin} to deduce that
    the arrow $\alpha_2$ is injective. As the vertical arrows
    on the top level are the localizations, it follows that $\beta_1$ and
    $\beta_2$ are injective. 
    Since the vertical arrows on the middle and the bottom levels are
    isomorphisms by \cite[Chap.~I, Prop.~1.11]{Illusie},
    it follows that $\eta_1$ and $\eta_2$ are injective. This concludes the
    proof.
\end{proof}

  If $S$ is a local ring with maximal ideal $\fm$ and $M$ is any $W_m(S)$-module,
  we let $M^{\wedge}$ denote the $\fM_m$-adic completion of $M$. From
  \lemref{lem:W-nil}, it follows that the canonical map from the
  $W_m(\fm)$-adic completion of $M$ to $M^{\wedge}$ is an isomorphism.
    
  Returning back to our set-up, we conclude
  from \propref{prop:F-fin} and \lemref{lem:W-nil}
that $W_m\Omega^q_A$ is a finitely generated $W_m(A)$-module and
\begin{equation}\label{eqn:Non-complete-0}
  W_m\Omega^q_A \otimes_{W_m(A)} W_m(\wh{A}) \xrightarrow{\cong} {(W_m\Omega^q_A)}^{\wedge}
  \xrightarrow{\cong} W_m\Omega^q_{\wh{A}}.
\end{equation}

\begin{lem}\label{lem:Non-complete-1}
  There are canonical isomorphisms of $W_m(\wh{A})$-modules
  \begin{equation}\label{eqn:Non-complete-1-0}
  W_m\Omega^q_A(\log\pi) \otimes_{W_m(A)} W_m(\wh{A}) \xrightarrow{\cong}
  {W_m\Omega^q_A(\log\pi)}^{\wedge} \xrightarrow{\cong} W_m\Omega^q_{\wh{A}}(\log\pi).
    \end{equation}
\end{lem}
\begin{proof}
  Since $W_m\Omega^q_A(\log\pi)$ is a finitely generated $W_m(A)$-module as we observed
  earlier, the first isomorphism in ~\eqref{eqn:Non-complete-1-0}
  is clear. By the definition of $W_m\Omega^q_{\wh{A}}(\log\pi)$ and
  ~\eqref{eqn:Non-complete-0}, we see that the composite arrow in
  ~\eqref{eqn:Non-complete-1-0} is surjective. To show it is injective, we look at
  the commutative diagram
  \begin{equation}\label{eqn:Non-complete-1-1}
    \xymatrix@C1pc{
      W_m\Omega^q_A(\log\pi) \otimes_{W_m(A)} W_m(\wh{A}) \ar[r] \ar[d] &
      W_m\Omega^q_{A_\pi} \otimes_{W_m(A)} W_m(\wh{A}) \ar[d] \\
      W_m\Omega^q_{\wh{A}}(\log\pi) \ar[r] & W_m\Omega^q_{\wh{K}}.}
    \end{equation}

  The top horizontal arrow in ~\eqref{eqn:Non-complete-1-1}
  is injective because $W_m(\wh{A})$ is
  a flat $W_m(A)$-module. Furthermore,
  \begin{equation}\label{eqn:Non-complete-1-2}
  \begin{array}{lll}
  W_m\Omega^q_{K} \otimes_{W_m(A)} W_m(\wh{A}) & \cong & W_m\Omega^q_A \otimes_{W_m(A)}
  W_m(A_\pi) \otimes_{W_m(A)}  W_m(\wh{A}) \\
  & \cong & W_m\Omega^q_{\wh{A}} \otimes_{W_m(\wh{A})}
  (W_m(\wh{A}) \otimes_{W_m(A)} W_m(K)) \\
  & \cong &  W_m\Omega^q_{\wh{A}} \otimes_{W_m(\wh{A})} W_m(\wh{A})_{[\pi]_m} \\
  & \cong & W_m\Omega^q_{\wh{A}} \otimes_{W_m(\wh{A})} W_m(\wh{K}) \cong
  W_m\Omega^q_{\wh{K}}
  \end{array}
  \end{equation}
  as a $W_m(\wh{A})$-module.
This shows that the right vertical arrow in ~\eqref{eqn:Non-complete-1-1} is
bijective (as it is same as the composition of the isomorphisms in
\eqref{eqn:Non-complete-1-2}).
It follows that the left vertical arrow is injective. This completes
the proof.
\end{proof}

\begin{lem}\label{lem:Non-complete-4}
  The maps
  \begin{enumerate}
    \item
      $d \colon \Fil_{\un{n}}W_m\Omega^q_K \to \Fil_{\un{n}}W_m\Omega^{q+1}_K$;
    \item
      $V \colon \Fil_{\un{n}}W_m\Omega^q_K \to \Fil_{\un{n}}W_{m+1}\Omega^{q}_K$;
    \item
      $F \colon \Fil_{\un{n}}W_{m+1}\Omega^q_K \to \Fil_{\un{n}}W_{m}\Omega^{q}_K$;
    \item
      $R \colon \Fil_{\un{n}}W_{m+1}\Omega^q_K \to\Fil_{\un{n}/p} W_{m}\Omega^{q}_K$
\end{enumerate}
  are continuous with respect to the $W_\star(\fm)$-adic
  topology of their sources and targets.
\end{lem}
\begin{proof}
 The claims (3) and (4) are easy to verify. To prove (1) and (2), we note that
  $ \Fil_{\un{n}}W_m\Omega^q_K$ is a finitely generated $W_m(A)$-module by
  \lemref{lem:Log-fil-4}. We let
  $J_s \subset W_m(A)$ be the ideal $([x_1]^s_m, \ldots , [x_d]^s_m)$.
  Then the $W_m(\fm)$-adic topology of $\Fil_{\un{n}}W_m\Omega^q_K$
  coincides with the topology given by the filtration
  $\{J_s\Fil_{\un{n}}W_m\Omega^q_K\}$ as shown in the proof of
  \cite[Cor.~2.3]{Geisser-Hesselholt-Top}.
Hence, it suffices to note that
  given $s \ge 1$, one has $d([x_i]^{p^mt}_m\Fil_{\un{n}}W_m\Omega^q_K)\subset
  J_s\Fil_{\un{n}}W_m\Omega^{q+1}_K$ and $V([x_i]^{pt}_m\Fil_{\un{n}}W_m\Omega^q_K) \subset
  J_s\Fil_{\un{n}}W_{m+1}\Omega^q_K$ for all $t \ge s$ and all $i$.
\end{proof}

\begin{lem}\label{lem:Non-complete-7}
  For $m \ge 1$, we have the following.
  \begin{enumerate}
    \item
      The map $F \colon W_{m+1}(A) \to W_m(A)$ is a finite ring homomorphism.
    \item
      The $W_m(\fm)$-adic topology of $W_m(A)$ coincides with its $W_{m+1}(\fm)$-adic
      topology via $F$.
    \item
      For any $A$-submodule $M$ of $\Omega^{q-1}_K$ (resp. $\Omega^q_K$), 
  $dV^{m-1}(M)$ (resp. $V^{m-1}(M)$) is a $W_{m+1}(A)$-submodule of
  $W_m\Omega^q_K$, where $W_{m+1}(A)$ acts on the latter via $F$.
  \end{enumerate}
\end{lem}
\begin{proof}
The item (1) is a well-known result of Langer-Zink (cf. \cite[Thm.~2.6]{Morrow-ENS}).
For (2), we can replace (for any $m$) the
$W_m(\fm)$-adic topology by the one given by the
filtration $\{J_s\}$ defined in the proof of \lemref{lem:Non-complete-4}.
Our claim now follows because $F([x_i]_{m+1}) = [x_i]^p_m$ for
every $i$.
To prove (3), we let $a \in W_{m+1}(A)$ and $x \in M$.
We get 
\begin{equation}\label{eqn:Non-complete-7-0}
\begin{array}{lll}  
  F(a)dV^{m-1}(x) & =  & F(a) FdV^m(x) = F(adV^m(x))  =  FdV^m(F^m(a)x) -F(daV^m(x))\\
  & = & FdV^m(F^m(a)x) - FV^m(F^md(a)x) =  FdV^m(F^m(a)x) \\
& = & dV^{m-1}(F^m(a)x),
\end{array}
\end{equation}
because $FV^m = 0$. This finishes the proof as $F^m(a)x \in M$. The other case is
easier  as $F(a)V^{m-1}(x) = V^{m-1}(F^m(a)x)$.
\end{proof}

We shall need the following well-known commutative algebra fact to prove the
next result. We omit the proof.

\begin{lem}\label{lem:Non-complete-2}
  Let $S$ be a Noetherian local ring with maximal ideal $\fm$ and let
  $M$ be a finitely generated $S$-module. Assume that $N' \subset M^{\wedge}$ is
  an $\wh{S}$-submodule. Assume further that $N \subset M \cap N'$ is an $S$-submodule
  which is dense in $N'$ in the $\fm$-adic topology. Then $N' = N^{\wedge}$ and
  $N = M \cap N'$.
  \end{lem}
  
\begin{lem}\label{lem:Non-complete-3*}
 The canonical map $\Fil_{\un{n}} W_m\Omega^q_K \to \Fil_{\un{n}} W_m\Omega^q_{\wh{K}}$
  induces an isomorphism of $W_m(\wh{A})$-modules
  \[
  \phi \colon (\Fil_{\un{n}} W_m\Omega^q_K)^{\wedge} \xrightarrow{\cong}
  \Fil_{\un{n}} W_m\Omega^q_{\wh{K}}.
  \]
\end{lem}
\begin{proof}
  Since $\Fil_{\un{n}} W_m\Omega^q_K$ is a finitely generated $W_m(A)$-module
  (cf. \lemref{lem:Log-fil-4}), we know that the canonical
  map $\Fil_{\un{n}} W_m\Omega^q_K \otimes_{W_m(A)} W_m(\wh{A}) \to
  (\Fil_{\un{n}} W_m\Omega^q_K)^{\wedge}$
  is an isomorphism. On the other hand, we have an injection
  $\Fil_{\un{n}} W_m\Omega^q_K \otimes_{W_m(A)} W_m(\wh{A}) \inj
  W_m\Omega^q_K \otimes_{W_m(A)} W_m(\wh{A}) \cong W_m\Omega^q_{\wh{K}}$,
where the last isomorphism is by ~\eqref{eqn:Non-complete-1-2}.
  It follows that $\phi$ is injective.
  For surjectivity, we first consider the case $q=0$. If $m=1$, the isomorphism is
  clear by \lemref{lem:Log-fil-4}(3). For $m >1$, consider the commutative diagram
  $W_m(\wh A)$ modules.
 \begin{equation}\label{eqn:Non-complete-8-0}
     \xymatrix@C1.4pc{
       0 \ar[r] & \Fil_{\un{n}}K\otimes_{W}\wh W  \ar[r]^-{V^{m-1}} \ar[d] &
       \Fil_{\un{n}}W_m(K)\otimes_{W}\wh W  \ar[r]^-{R} \ar[d] &
       \Fil_{\un{n}/p}W_{m-1}(K)\otimes_{W}\wh W  \ar[r] \ar[d] & 0 \\
       0 \ar[r] & \Fil_{\un{n}}\wh{K} \ar[r]^-{V^{m-1}} &  \Fil_{\un{n}}W_m(\wh{K})
       \ar[r]^-{R} & \Fil_{\un{n}/p}W_{m-1}(\wh{K}) \ar[r]& 0,
     }
 \end{equation}
 where $W=W_m(A)$ and $\wh W= W_m(\wh A)$.

 The top row is obtained by tensoring
\eqref{eqn:Log-fil-1.1} (with $U =\Spec(K),
D = V((\prod\limits_{1 \le i \le r} x_i^{n_i}))$)
 with $W_m(\wh A)$ (over $W_m(A)$), and noting that $W_m(\wh A)$ is the
 $W_m(\fm)$-adic completion of $W_m(A)$ (cf. \eqref{eqn:Non-complete-0}),
 and hence flat over $W_m(A)$. Here, the module structure on $\Fil_{\un{n}}K$ is given
 via $F^{m-1}$. By \lemref{lem:Non-complete-7}(2), we see that the $W_m(\fm)$-adic
 topology on $\Fil_{\un{n}}K$ (via $F^{m-1}$) is the
 same as the $\fm$-adic topology. Hence,
 $\Fil_{\un{n}}K \otimes_{W_m(A)} W_m(\wh A) = (\Fil_{\un{n}} K)^\wedge$, the $\fm$-adic
 completion of $\Fil_{\un{n}}K$ as an $A$-module.

 We also note that the $W_{m-1}(\fm)$-adic topology of $W_{m-1}(A)$ coincides with
  its $W_m(\fm)$-topology via the surjection $R \colon W_m(A) \surj W_{m-1}(A)$.
  It follows that for any finitely generated $W_{m-1}(A)$-module $M$, one has
  isomorphisms
  \begin{equation}\label{eqn:Non-complete-6-0}
  M \otimes_{W_m(A)} W_{m}(\wh{A})
  \cong {M}^{\wedge}_m \cong {M}^{\wedge}_{m-1} \cong M \otimes_{W_{m-1}(A)} W_{m-1}(\wh{A}),  
\end{equation}
  where $M^{\wedge}_m$ (resp. ${M}^{\wedge}_{m-1}$) is the $W_{m}(\fm)$
  (resp. $W_{m-1}(\fm)$)-adic completion of $M$.
  Thus, the left and right vertical arrows in ~\eqref{eqn:Non-complete-8-0} are
  isomorphisms by $m=1$ case and
  induction, respectively. Hence, the middle vertical arrow is also an isomorphism.
  In particular, we get that $\Fil_{\un{n}}W_m(K)$ is dense in $\Fil_{\un{n}}W_m(\wh K)$.

We now let $q >0$ and recall that
  $\Fil_{\un{n}} W_m\Omega^q_{\wh{K}}$ is the sum of
  $\Fil_{\un{n}}W_m(\wh{K}) W_m\Omega^q_{\wh{A}}(\log\pi)$ and
$d(\Fil_{\un{n}}W_m(\wh{K}))W_m\Omega^{q-1}_{\wh{A}}(\log\pi)$. Hence, by
Lemmas~\ref{lem:Non-complete-1} and \ref{lem:Non-complete-4} and the $q=0$ case,  we
get that $\Fil_{\un{n}} W_m\Omega^q_{K}$ is dense in $\Fil_{\un{n}} W_m\Omega^q_{\wh{K}}$.
On the other hand, the lemma is clear if $\un{n}=p^{m-1}\un{n}''$ for some
$\un{n}''$, by Lemmas~\ref{lem:Log-fil-4}(1) and ~\ref{lem:Non-complete-1}. We
therefore choose some $\un{n}'$ such that $p^{m-1}\un{n}' > \un{n}$ so that
$\Fil_{\un{n}} W_m\Omega^q_{S} \subset \Fil_{p^{m-1}\un{n'}} W_m\Omega^q_{S}$ for
$S \in \{K,\wh K \}$, and apply \lemref{lem:Non-complete-2}
with $S = W_m(A)$ (which is a local ring with maximal ideal $\fM_m$ by
\lemref{lem:W-nil}), $M=\Fil_{p^{m-1}\un{n'}} W_m\Omega^q_{K}$, $N=\Fil_{\un{n}} W_m\Omega^q_{K}$ and
$N'=\Fil_{\un{n}} W_m\Omega^q_{\wh{K}}$ to conclude the proof.
\end{proof}

\begin{cor}\label{cor:Non-complete-3}
  Consider $F(\Fil_{\n}W_{m+1}\Omega^q_{K})$ as a $W_{m+1}(A)$-module (via $F$). Then
  the canonical map $F(\Fil_{\n}W_{m+1}\Omega^q_{K}) \to
  F(\Fil_{\n}W_{m+1}\Omega^q_{\wh K})$ induces an isomorphism of
  $W_{m+1}(\wh A)$-modules (via $F$)
  $\phi: \left(F(\Fil_{\n}W_{m+1}\Omega^q_{K})\right)^\wedge \xrightarrow{\cong}
  F(\Fil_{\n}W_{m+1}\Omega^q_{\wh K})$.
\end{cor}
\begin{proof}
  Consider $\Fil_{\n}W_m\Omega^q_K$ as a finitely generated $W_{m+1}(A)$-module via $F$
  (cf. \lemref{lem:Non-complete-7}(1)) so that  $F(\Fil_{\n}W_{m+1}\Omega^q_{K})
  \subset \Fil_{\n}W_m\Omega^q_K$ is a $W_{m+1}(A)$-submodule. By
  Lemmas~\ref{lem:Non-complete-4}(3), ~\ref{lem:Non-complete-7}(2) and ~\ref{lem:Non-complete-3*}, we know that
  $F(\Fil_{\n}W_{m+1}\Omega^q_{K})$ is dense in $F(\Fil_{\n}W_{m+1}\Omega^q_{\wh K})$ with
  respect to $W_{m+1}(\wh \fm)$-adic topology. Also, by
  Lemmas~\ref{lem:Non-complete-7}(2) and ~\ref{lem:Non-complete-3*}, we have
  \[
  \Fil_{\n}W_{m}\Omega^q_{\wh K} \cong \Fil_{\n}W_{m}\Omega^q_{K}
  \otimes_{W_{m+1}(A)}W_{m+1}(\wh A),
  \]
  the $W_{m+1}(\fm)$-adic completion (via $F$) of $\Fil_{\n}W_{m}\Omega^q_{K}$.
  Now, we can apply \lemref{lem:Non-complete-2} with $S=W_{m+1}(A)$, $M=\Fil_{\n}W_{m}\Omega^q_{K}$, $N'=F(\Fil_{\n}W_{m+1}\Omega^q_{\wh{K}})$ and
  $N=F(\Fil_{\n}W_{m+1}\Omega^q_{K})$ to conclude the proof.
\end{proof}

We let $\Fil''_{\un{n}}W_m\Omega^q_K$ denote the kernel of the canonical map
\[
\gamma = (\ov{\gamma_i}) \colon W_m\Omega^q_K \to {\underset{1 \le i \le r}\bigoplus}
\frac{W_m\Omega^q_{\wh{K}_i}}{\Fil_{n_i}W_m\Omega^q_{\wh{K}_i}}.
\]
It is clear that $\Fil_{\un{n}}W_m\Omega^q_K \subset  \Fil''_{\un{n}}W_m\Omega^q_K$. 

 Assume now that $A$ is complete. For any $1 \le i \le r$ and $n \in \Z$, we let
$E^{1,q}_{i, n}(K) = \left\{\omega = {\underset{\un{m} \in \Z^r \times \N^{d-r}}\sum}
 a_{\un{m}}x^{\un{m}} | \ a_{\un{m}} = 0 \ \mbox{if} \ m_i < - n\right\}$.
We let
$\Fil^i_{\un{n}}\Omega^q_K = {\underset{1 \le j \le i}\bigcap} E^{1,q}_{j, n_j}(K)$.
We then get a filtration
\begin{equation}\label{eqn:F-function-2-0}
  \Omega^q_K = \Fil^0_{\un{n}}\Omega^q_K(K) \supset \Fil^1_{\un{n}}\Omega^q_K
  \supset \cdots \supset \Fil^r_{\un{n}}\Omega^q_K = \Fil_{\un{n}}\Omega^q_K.
  \end{equation}

\begin{lem}\label{lem:Complete-3}
 If $A$ is complete, then one has a short exact sequence of $W_m(A)$-modules
  \[
 0 \to V^{m-1}(\Fil_{\un{n}} \Omega^q_K) + dV^{m-1}(\Fil_{\un{n}}\Omega^{q-1}_K) \to
  \Fil''_{\un{n}}W_m\Omega^q_K \xrightarrow{R} \Fil''_{{\un{n}}/p}W_{m-1}\Omega^{q}_K.
  \]
\end{lem}
 \begin{proof}
We let ${T} = V^{m-1}(\Fil_{\un{n}} \Omega^q_K) + dV^{m-1}(\Fil_{\un{n}}\Omega^{q-1}_K)$.
   Using \cite[Prop.~2.3]{Shiho}, we only need to show that
   every element $\omega \in
   \left( V^{m-1}(\Omega^q_K) + dV^{m-1}(\Omega^{q-1}_K) \right) \bigcap
   \Fil''_{\un{n}}W_m\Omega^q_K$ lies in ${T}$.
   We shall show more generally that for every $0 \le i \le r$, we can write
   $\omega = V^{m-1}(a) + dV^{m-1}(b)$
   with $a \in \Fil^i_{\un{n}}\Omega^q_K$ and $b \in \Fil^i_{\un{n}}\Omega^{q-1}_K$
   (cf. ~\eqref{eqn:F-function-2-0}).

   We write $A = k[[x_1, \ldots , x_d]]$, where $k = {A}/{\fm}$.
   If $d =1$, the lemma already follows from \propref{prop:V-R-Fil}.
   We therefore assume $d \ge 2$ and write
   $A = B[[x_i]]$, where $B = k[[x_1, \ldots , x_{i-1}, x_{i+1}, \ldots , x_d]]$.
   We then have a canonical isomorphism $\wh{A}_i  \cong Q(B)[[x_i]]$ under which the
   inclusion $A \inj \wh{A}_i$ coincides with the map
   \begin{equation}\label{eqn:Complete-4}
     \theta_i \colon  B[[x_i]] \to Q(B)[[x_i]],
   \end{equation}
   which is identity on $B$ and sends $x_i$ to itself. In particular, the
   inclusion $K \inj \wh{K}_i$ coincides with the map induced by
$\theta_i$ on the quotient fields of its source and target.
Furthermore, $\gamma_i$ is induced by $\theta_i$.

We now write $\omega = V^{m-1}(a_0) + dV^{m-1}(b_0)$ for some
   $a_0 \in \Omega^q_K = \Fil^0_{\un{n}}\Omega^q_K$ and
   $b_0 \in \Omega^{q-1}_K = \Fil^0_{\un{n}}\Omega^{q-1}_K$.
   We suppose next that $i \ge 1$ and that we have found
   $a \in \Fil^{i-1}_{\un{n}}\Omega^q_K$ and
   $b \in \Fil^{i-1}_{\un{n}}\Omega^{q-1}_K$ such that
   $\omega = V^{m-1}(a) + dV^{m-1}(b)$.
In the notations of \corref{cor:F-function-2}, we can uniquely write
   $a = \sum_{\un{m}} a_{\un{m}}x^{\un{m}}$ and $b = \sum_{\un{m}} b_{\un{m}}x^{\un{m}}$
such that $a_{\un{m}} = 0 = b_{\un{m}}$ for any $\un{m}$ such that $m_j <- n_j$ for
some $j < i$.  We can uniquely write $a = a_1 + a_2$ and $b = b_1 + b_2$ such that
   \begin{enumerate}
     \item
       $a_1, a_2 \in \Fil^{i-1}_{\un{n}}\Omega^q_K$ and
       $b_1, b_2 \in \Fil^{i-1}_{\un{n}}\Omega^{q-1}_K$. 
     \item
       $a_1 \in  \Fil^{i}_{\un{n}}\Omega^q_K, \ b_1 \in  \Fil^{i}_{\un{n}}\Omega^{q-1}_K$.
       In particular,
       $\gamma_i\left(V^{m-1}(a_1) + dV^{m-1}(b_1)\right) \in
       \Fil_{n_i}W_m\Omega^q_{\wh K_i}$ by \lemref{lem:Log-fil-Wcom} (since $\gamma_i$
       commutes with all operators and $\gamma_i(a_1) \in \Fil_{n_i}\Omega^q_{K_i}$
       and $\gamma_i(b_1) \in \Fil_{n_i}\Omega^{q-1}_{K_i}$).
     \item
       No term of the infinite series $a_2$ (resp. $b_2$) lies in
       $\Fil^{i}_{\un{n}}\Omega^q_K$ (resp. $\Fil^{i}_{\un{n}}\Omega^{q-1}_K$).
\end{enumerate}

The above conditions imply that 
   \[
   \gamma_i\left(V^{m-1}(a_2) + dV^{m-1}(b_2)\right) =
   \gamma_i\left(\omega - (V^{m-1}(a_1) + dV^{m-1}(b_1))\right)
   \in \Fil_{n_i}W_m\Omega^q_{\wh K_i}.
   \]

   On the other hand, \propref{prop:Fil-decom} says that
   $\gamma_i(a_2) = \theta^*_i(a_2) \in {\underset{j < -n_i}\prod} F^{1,q}_j(Q(B))$
   and $\gamma_i(b_2) = \theta^*_i(b_2) \in
   {\underset{j < -n_i}\prod} F^{1,q-1}_j(Q(B))$. By \lemref{lem:F-group-1},
   this implies that
   \[
   \gamma_i \left(V^{m-1}(a_2) + dV^{m-1}(b_2)\right)
   = V^{m-1}(\gamma_i(a_2)) + dV^{m-1}(\gamma_i(b_2))
   \in  {\underset{j < -n_i}\prod} F^{m,q}_j(Q(B)).
   \]
   We thus get 
   $\gamma_i\left(V^{m-1}(a_2) + dV^{m-1}(b_2)\right) \in
   \left[{\underset{j < -n_i}\prod}F^{m,q}_{j}(Q(B)) \right] \bigcap
   \Fil_{n_i}W_m\Omega^q_{\wh{K}_i}$.
     But the latter intersection is trivial by \corref{cor:GH-3} and
     \propref{prop:Fil-decom}. It follows by
     \lemref{lem:Complete-1} that $V^{m-1}(a_2) + dV^{m-1}(b_2) = 0$.
Equivalently, $\omega =  V^{m-1}(a_1) + dV^{m-1}(b_1)$. We are now done by the
   condition (2) above and induction on $i$.  
 \end{proof}

   \vskip .3cm

We now prove the result we were after. This result generalizes \propref{prop:V-R-Fil}.

\begin{prop}\label{prop:Complete-0}
There is a short exact sequence
of $W_m(A)$-modules
  \[
  0 \to V^{m-1}(\Fil_{\un{n}} \Omega^q_K) + dV^{m-1}(\Fil_{\un{n}}\Omega^{q-1}_K) \to
  \Fil_{\un{n}}W_m\Omega^q_K \xrightarrow{R} \Fil_{{\un{n}}/p}W_{m-1}\Omega^{q}_K \to 0.
  \]
\end{prop}
\begin{proof}
  We only need to show that $\Ker(R) \subset V^{m-1}(\Fil_{\un{n}} \Omega^q_K) +
  dV^{m-1}(\Fil_{\un{n}}\Omega^{q-1}_K)$ as the surjectivity of $R$ was already shown in
  the proof of \lemref{lem:Log-fil-Wcom}.
  If $A$ is complete, we are done by \lemref{lem:Complete-3} since
  $\Fil_{\un{n}}W_m\Omega^q_K \subset  \Fil''_{\un{n}}W_m\Omega^q_K$.
  Suppose now that $A$ is not complete and let $E = V^{m-1}(\Fil_{\un{n}} \Omega^q_K) +
  dV^{m-1}(\Fil_{\un{n}}\Omega^{q-1}_K)$.

  One easily verifies that $E$ is a $W_m(A)$-submodule of
  $\Fil_{\un{n}}W_m\Omega^q_K$. Indeed, \lemref{lem:Log-fil-4} implies that
  $E \subset \Fil_{\un{n}}W_m\Omega^q_K$. Moreover, one has
  $a(V^{m-1}(b) + dV^{m-1}(c)) = V^{m-1}(F^{m-1}(a)b-F^{m-1}d(a)c) + dV^{m-1}(F^{m-1}(a)c)$
  for any $a \in W_m(A), \ b \in \Fil_{\un{n}} \Omega^q_K$ and
  $c \in \Fil_{\un{n}}\Omega^{q-1}_K$. \lemref{lem:Log-fil-4} implies once again that
  $a(V^{m-1}(b) + dV^{m-1}(c)) \in E$.
 Using the complete case, it remains therefore to show that 
$E = \left[V^{m-1}(\Fil_{\un{n}} \Omega^q_{\wh{K}}) +
  dV^{m-1}(\Fil_{\un{n}}\Omega^{q-1}_{\wh{K}})\right] \bigcap \Fil_{\un{n}}W_m\Omega^q_K$.
But this follows from \lemref{lem:Non-complete-2}
because $E$ is dense in
$V^{m-1}(\Fil_{\un{n}}\Omega^q_{\wh{K}}) + dV^{m-1}(\Fil_{\un{n}}\Omega^{q-1}_{\wh{K}})$ by
Lemma~\ref{lem:Non-complete-4} and \corref{cor:Non-complete-3}.
\end{proof}

\begin{cor}\label{cor:Complete-2}
 The inclusions $K \inj \wh{K}_i$
  induce an exact sequence
  \[
  0 \to  \Fil_{\un{n}}W_m\Omega^q_K \to W_m\Omega^q_K \xrightarrow{\gamma}
    {\underset{1 \le i \le r}\bigoplus}
  \frac{W_m\Omega^q_{\wh{K}_i}}{\Fil_{n_i}W_m\Omega^q_{\wh{K}_i}}.
  \]
\end{cor}
\begin{proof}
The corollary is equivalent to the claim that the inclusion
 $\Fil_{\un{n}}W_m\Omega^q_K  \subset \Fil''_{\un{n}}W_m\Omega^q_K$ is an equality.
 To prove this claim, we first assume that $A$ is complete and prove 
by induction on $m$.
  The case $m =1$ follows directly from \corref{cor:F-function-2}.
  For $m \ge 2$, we use the commutative diagram
  \begin{equation}\label{eqn:Complete-2-0}
    \xymatrix@C.8pc{
      0 \ar[r] & V^{m-1}(\Fil_{\un{n}}\Omega^q_K) + dV^{m-1}(\Fil_{\un{n}}\Omega^{q-1}_K)
      \ar[r] \ar[d] & \Fil_{\un{n}}W_m\Omega^q_K \ar[r]^-{R} \ar[d] &
      \Fil_{\un{n}}W_{m-1}\Omega^q_K \ar[r] \ar[d] &
      0 \\
    0 \ar[r] & V^{m-1}(\Fil''_{\un{n}}\Omega^q_K) + dV^{m-1}(\Fil''_{\un{n}}\Omega^{q-1}_K)
    \ar[r] & \Fil''_{\un{n}}W_m\Omega^q_K \ar[r]^-{R}  &
    \Fil''_{\un{n}}W_{m-1}\Omega^q_K. & }
    \end{equation}
 \propref{prop:Complete-0} says that the top row  in
  ~\eqref{eqn:Complete-2-0} is exact and \lemref{lem:Complete-3}
  implies that the bottom row is exact at $\Fil''_{\un{n}}W_m\Omega^q_K$.
  The left and the right vertical arrows are bijective by induction.
  A diagram chase shows that the middle vertical arrow is also bijective.

 For $A$ not necessarily complete, it suffices to show that
  $\Fil_{\un{n}}W_m\Omega^q_K = \Fil_{\un{n}}W_m\Omega^q_{\wh{K}} \bigcap W_m\Omega^q_K$.
  But this follows because $W_m(\wh{A})$ is a faithfully flat $W_m(A)$-algebra
  (cf. \propref{prop:F-fin}, \lemref{lem:W-nil}) and there are isomorphisms
  \begin{equation}\label{eqn:Complete-2-1}
    \frac{W_m\Omega^q_K}{\Fil_{\un{n}}W_m\Omega^q_K} \otimes_{W_m(A)} W_m(\wh{A})
  \cong \frac{W_m\Omega^q_K \otimes_{W_m(A)} W_m(\wh{A})}
        {\Fil_{\un{n}}W_m\Omega^q_K \otimes_{W_m(A)} W_m(\wh{A})}
        \cong \frac{W_m\Omega^q_{\wh{K}}}{\Fil_{\un{n}}W_m\Omega^q_{\wh{K}}},
  \end{equation}
  where the second isomorphism follows from ~\eqref{eqn:Non-complete-1-2} and
  \lemref{lem:Non-complete-3*}.
\end{proof}

\begin{cor}\label{cor:Ext-3}
  The inclusion $\Fil_{-\un{1}} W_m\Omega^d_K \inj W_m\Omega^d_A$ is a bijection if
  $d$ is the rank of $\Omega^1_A$.
\end{cor}
\begin{proof}
  For $m = 1$, this follows from \lemref{lem:Log-fil-4}(3).
  We now use the commutative diagram
  \begin{equation}\label{eqn:Ext-3-0}
    \xymatrix@C1.5pc{
      0 \to V^{m-1}(\Fil_{-\un{1}}\Omega^d_K) + dV^{m-1}(\Fil_{-\un{1}} \Omega^{d-1}_K) \ar[r]
      \ar[d] &
      \Fil_{-\un{1}} W_m\Omega^d_K \ar[r]^-{R} \ar[d] &  \Fil_{-\un{1}} W_{m-1}\Omega^d_K
      \ar[r]
      \ar[d] & 0 \\
      0 \to V^{m-1}(\Omega^d_A) + dV^{m-1}(\Omega^{d-1}_A) \ar[r]  &
      W_m\Omega^d_A  \ar[r]^-{R} & W_{m-1}\Omega^d_A \ar[r] & 0}
  \end{equation}
  (where the vertical arrows are the canonical inclusions),
  \propref{prop:Complete-0}
  and induction on $m$ to finish the proof in the general case.
\end{proof}

We shall need the following result in the next section.

\begin{lem}\label{lem:Complete-5}
  The canonical maps
  \[
  \frac{\Fil_{\un{n}}W_m\Omega^q_K}{dV^{m-1}(\Fil_{\un{n}}\Omega^{q-1}_K)}
  \to \frac{W_m\Omega^q_K}{dV^{m-1}(\Omega^{q-1}_K)} \ \ \mbox{and} \ \ 
\frac{\Fil_{\un{n}}W_{m}\Omega^q_K}{V^{m-1}(\Fil_{\un{n}}\Omega^{q}_K)}
  \to \frac{W_{m}\Omega^q_K}{V^{m-1}(\Omega^{q}_K)}
  \]
are injective.
\end{lem}
\begin{proof}
 We shall prove the injectivity of the first map as the other one is completely
  analogous.
We only need to show that every element
  $\omega \in \Fil_{\un{n}}W_m\Omega^q_K \bigcap dV^{m-1}(\Omega^{q-1}_K)$ lies in
$dV^{m-1}(\Fil_{\un{n}}\Omega^{q-1}_K)$. If $A$ is complete, this is proven
exactly as we proved \lemref{lem:Complete-3} (with $a = 0$). If $A$ is not 
complete, it suffices to show that
\[
\Fil_{\un{n}}W_m\Omega^q_K \bigcap
dV^{m-1}(\Fil_{\un{n}}\Omega^{q-1}_{\wh{K}}) = dV^{m-1}(\Fil_{\un{n}}\Omega^{q-1}_K).
\]

To that end, we note that $\Fil_{\un{n}}W_m\Omega^q_{\wh{K}} =
(\Fil_{\un{n}}W_m\Omega^q_K)^\wedge$ by \lemref{lem:Non-complete-3*}. In particular,
$\Fil_{\un{n}}\Omega^{q-1}_K$ is dense in $\Fil_{\un{n}}\Omega^{q-1}_{\wh{K}}$ with respect
to the $\wh{\fm}$-adic topology.  It follows from \lemref{lem:Non-complete-4} that
$dV^{m-1}(\Fil_{\un{n}}\Omega^{q-1}_K)$ is a dense subgroup of 
$dV^{m-1}(\Fil_{\un{n}}\Omega^{q-1}_{\wh{K}})$ with respect to the $W_m(\fm)$-adic
topology. We conclude from \lemref{lem:Non-complete-7} that
$\Fil_{\un{n}}W_m\Omega^q_{\wh{K}}$ is
the $W_{m+1}(\fm)$-adic completion of the $W_{m+1}(A)$-module $\Fil_{\un{n}}W_m\Omega^q_K$
via $F \colon W_{m+1}(A) \to W_m(A)$ and
$dV^{m-1}(\Fil_{\un{n}}\Omega^{q-1}_K)$ is a dense subgroup of 
$dV^{m-1}(\Fil_{\un{n}}\Omega^{q-1}_{\wh{K}})$ with respect to the $W_{m+1}(\fm)$-adic
topology. By \lemref{lem:Non-complete-2}, it remains
to show that $dV^{m-1}(\Fil_{\un{n}}\Omega^{q-1}_K)$ is a $W_{m+1}(A)$-submodule of
$\Fil_{\un{n}}W_m\Omega^q_K$. But this follows again from \lemref{lem:Non-complete-7}.
\end{proof}

\section{Cartier operator on filtered de Rham--Witt complex}\label{sec:CFDRW}
Recall from \cite{Kato-Duality} (see also \cite[\S~10]{GK-Duality})
that for any regular $\F_p$-scheme $X$, there
is a unique Cartier homomorphism $C \colon Z_1W_m\Omega^q_X \to W_m\Omega^q_X$
such that $W_{m+1}\Omega^q_X \xrightarrow{R} W_m\Omega^q_X$ is the composition
$W_{m+1}\Omega^q_X \stackrel{F}{\surj} Z_1W_m\Omega^q_X  \xrightarrow{C} W_m\Omega^q_X$
and $Z_1W_m\Omega^q_X := {\rm Image}(F) = \Ker(W_{m}\Omega^q_X \xrightarrow{F^{m-1}d}
\Omega^{q+1}_X)$ is a sheaf of $W_m\sO_X$-modules via $\ov{F}$.
Our goal in this section is to explain the interaction of the Cartier
homomorphism with the $\Div_E(X)$-filtration on de Rham--Witt complex, where
$E \subset X$ is an sncdivisor.
We also prove some additional properties of the filtered de Rham--Witt
complex which are known when one ignores the filtration. In the end of the section,
we state a global version of our results for $F$-finite
regular $\F_p$-schemes.

We let $A$ be a regular local $F$-finite $\F_p$-algebra with 
maximal ideal $\fm = (x_1, \ldots , x_d)$. We let $\pi = x_1\cdots x_r$ and
$K = A_\pi$, where $1 \le r \le d$. We let $\wh{A}$ denote the
$\fm$-adic completion of
$A$ with the maximal ideal $\wh{\fm}$ and let $\wh{K} = \wh{A}_\pi$.
We fix $q \ge 0, m \ge 1$ and let $Z_1\Fil_{\n}W_m\Omega^q_K :=
\Fil_{\n}W_m\Omega^q_K \bigcap Z_1W_m\Omega^q_K$ for $\n \in \Z^r$.
Equivalently, $Z_1\Fil_{\un{n}}W_m\Omega^q_K=\Ker (F^{m-1}d:
\Fil_{\un{n}}W_m\Omega^q_K \to \Fil_{\un{n}}\Omega^{q+1}_K)$. 
We let ${\un{n}}/{p^m} = (\lfloor{{n_1}/{p^m}}\rfloor, \ldots ,
\lfloor{{n_r}/{p^m}}\rfloor)$.
We shall say $\un{n} \ge \un{n'}$ if $n_i \ge n'_i$ for all $1 \le i \le r$.
We let $-\un{1} = (-1, \ldots , -1) \in \Z^r$.
We shall continue to use the notations of \S~\ref{sec:com-to-noncom}.

\subsection{The Cartier homomorphism}\label{sec:Cartier}
We begin with the following.

\begin{lem}\label{lem:Complete-6}
  The Cartier map $C \colon Z_1W_m\Omega^q_K \to W_m\Omega^q_K$ restricts to a
  $W_{m}(A)$-linear map
  \[
  C \colon Z_1\Fil_{\un{n}}W_m\Omega^q_K \to \Fil_{{\un{n}}/p}W_m\Omega^q_K
  \]
  such that $C(dV^{m-1}(\Fil_{\un{n}}\Omega^{q-1}_K)) = 0$.
\end{lem}
\begin{proof}
  Let $\wh{K}_i$ denote the quotient field of $\wh{A_{(x_i)}}$.
  By \corref{cor:Complete-2}, we only need to show that
  $C \circ \gamma_i(\omega) = \gamma_i \circ C(\omega)$ lies in
  $\Fil_{\lfloor{{n_i}/p}\rfloor}W_m\Omega^q_{\wh{K}_i}$ for every $\omega \in 
  Z_1W_m\Omega^q_K \bigcap \Fil_{\un{n}}W_m\Omega^q_K$ and $1 \le i \le r$.
  It suffices therefore to show that
  $Z_1W_m\Omega^q_K  \bigcap \Fil_nW_m\Omega^q_K \subset F(\Fil_nW_{m+1}\Omega^q_K)$
  when $A = S[[\pi]]$, where $S$ is a regular local $F$-finite $\F_p$-algebra.

To that end, we let $x \in \Fil_nW_m\Omega^q_K \bigcap Z_1W_m\Omega^q_K = 
\Fil_nW_m\Omega^q_K \bigcap F(W_{m+1}\Omega^q_K)$.
By \corref{cor:GH-3}, \lemref{lem:F-group-1} and \propref{prop:Fil-decom},
we can write $x = \sum_{i \ge -n} a_i = \sum_{i \ge -l} a'_i = F( \sum_{i \ge -l} b_i)$
with $a_i \in F^{m,q}_i(S)$
and $a'_i = F(b_i) \in F^{m,q}_i(S)$ for some $b_i \in F^{m+1,q}_i(S)$ and for some
$l \gg |n|$.
 By the uniqueness of the decomposition in \corref{cor:GH-3},
we must have $a_i = a'_i$ for every $i \ge -l$. In particular, $a'_i = 0$ for
$i < -n$. This implies that $x = \sum_{i \ge -n} a_i = F(\sum_{i \ge -n} b_i)$.
Since $\sum_{i \ge -n} b_i \in \Fil_nW_{m+1}\Omega^q_K$, the desired claim follows.
\end{proof}

The following key proposition will be used in \S~\ref{sec:katofil}
(cf. \corref{cor:VR-5}) to study Kato's ramification filtration.

\begin{prop}\label{prop:Cartier-fil-1}
If $A$ is a DVR and $n \ge 0$, then every element $\omega \in Z_1W_m\Omega^q_K$
  has the property that $\omega \in Z_1\Fil_nW_m\Omega^q_K$ if and only if
  $(1-C)(\omega) \in \Fil_n W_m\Omega^q_K$.
\end{prop}
\begin{proof}
If $\omega \in  Z_1\Fil_nW_m\Omega^q_K$, then $C(\omega) = CF(\omega') =
  R(\omega')$ for some $\omega' \in \Fil_nW_{m+1}\Omega^q_K$
  by \lemref{lem:Complete-6}. It follows from \lemref{lem:F-group-1} that $C(\omega)
  \in \Fil_{\lfloor{n/p}\rfloor}W_m\Omega^q_K \subset \Fil_nW_m\Omega^q_K$.  
  In particular, the map
  $Z_1\Fil_nW_m\Omega^q_K \xrightarrow{1-C} \Fil_nW_m\Omega^q_K$ is defined.

Suppose now that $\omega \in Z_1W_m\Omega^q_K$ is such that
$(1-C)(\omega) \in \Fil_nW_m\Omega^q_K$. To show that $\omega \in \Fil_nW_m\Omega^q_K$
(and hence lies in $Z_1\Fil_nW_m\Omega^q_K$), we can use \corref{cor:Complete-2} to
replace $A$ by $\wh{A}$. It suffices therefore to prove our assertion when $A$ is
more generally of the form $S[[\pi]]$, where $S$ is a regular local $F$-finite
$\F_p$-algebra. We assume this to be the case in the rest of the proof.

We can write $\omega = F(\alpha)$ for
  some $\alpha \in W_{m+1}\Omega^q_K$. This yields $(1-C)(\omega) = (1-C)F(\alpha)
  = F(\alpha) - R(\alpha)$. By \corref{cor:GH-3}, we can write $\alpha$ uniquely
  as $\alpha = \sum_{i \ge -l} a_i$ for some $a_i \in F^{m+1,q}_i(S)$ and $l \gg 0$.
  By \lemref{lem:F-group-1}, we get
  \begin{equation}\label{eqn:Cartier-fil-1-0}
    F(\alpha) - R(\alpha) = \sum_{i \ge -l} F(a_i) - \sum_{i \ge -l}R(a_i)
  \in \Fil_nW_m\Omega^q_K.
  \end{equation}

 If $l \le n$, then $-n \le -l$ and hence $\alpha \in \Fil_nW_{m+1}\Omega^q_K$
  so that $\omega \in Z_1\Fil_nW_m\Omega^q_K$.
If $l > n$, then \lemref{lem:F-group-1} says that
$R(a_{-l}) \in F^{m,q}_{-l/p}(S)$ if $p|l$ and $R(a_{-l}) = 0$ if $|l| \in I_p$. Since
$F(a_{-l}) \in F^{m,q}_{-l}$ and $ \sum_{i \ge -l+1} F(a_i) - \sum_{i \ge -l}R(a_i) \in
\Fil_{l-1}W_m\Omega^q_K$, it follows from ~\eqref{eqn:Cartier-fil-1-0} that
  $F(a_{-l}) = 0$. An induction argument says that
  $F(a_i) = 0$ for all $i < -n$. Thus, we get $\omega = F(\alpha) =
  \sum_{i \ge -n} F(a_i) \in \Fil_nW_m\Omega^q_K$. 
\end{proof}

The next few results explain the interaction of the Cartier homomorphism with
other operators of the filtered de Rham--Witt complex.

\begin{lem}\label{lem:Complete-4}
We have the following.
  \begin{enumerate}
  \item
    $Z_1\Fil_{\un{n}}W_m\Omega^q_K = F(\Fil_{\un{n}}W_{m+1}\Omega^q_K)$.
  \item
    $F \colon \Fil_{\un{n}}W_{m+1}\Omega^q_K \to  \Fil_{\un{n}}W_{m}\Omega^q_K$
    induces isomorphisms
    \[
    \ov{F} \colon \Fil_{{\un{n}}/p}W_{m}\Omega^q_K \xrightarrow{\cong}
    \frac{Z_1\Fil_{\un{n}}W_m\Omega^q_K}{dV^{m-1}(\Fil_{\un{n}}\Omega^{q-1}_K)}; \ \
    \ov{C} \colon
    \frac{Z_1\Fil_{\un{n}}W_m\Omega^q_K}{dV^{m-1}(\Fil_{\un{n}}\Omega^{q-1}_K)}
     \xrightarrow{\cong} \Fil_{{\un{n}}/p}W_{m}\Omega^q_K.
     \]
\end{enumerate}
\end{lem}
\begin{proof}
We first note that the lemma is classical for $W_\star\Omega^\bullet_K$
  (cf. \cite{Illusie} and \cite{Kato-Duality}). 
To prove (1) and (2), we use \propref{prop:Complete-0} and the identity $FV^m = 0$
to get a commutative diagram
  \begin{equation}\label{eqn:Complete-4-0}
    \xymatrix@C1pc{
     \Fil_{\un{n}}W_{m+1}\Omega^q_K \ar@{->>}[d]_-{R} \ar[r]^-F &  
     \Fil_{\un{n}}W_{m}\Omega^q_K \ar@{->>}[d] \\
     \Fil_{{\un{n}}/p}W_{m}\Omega^q_K \ar[r]^-{\ov{F}} &
     \frac{\Fil_{\un{n}}W_{m}\Omega^q_K}{dV^{m-1}(\Fil_{\un{n}}\Omega^{q-1}_K)}.}
    \end{equation}
  Since $dV^{m-1}(\Fil_{\un{n}}\Omega^{q-1}_K) = FdV^m(\Fil_{\un{n}}\Omega^{q-1}_K)
  \subset F(\Fil_{\un{n}}W_{m+1}\Omega^q_K)$, we thus get maps
  \begin{equation}\label{eqn:Complete-4-1}
\Fil_{{\un{n}}/p}W_{m}\Omega^q_K \stackrel{\ov{F}}{\surj} 
\frac{F(\Fil_{\un{n}}W_{m+1}\Omega^q_K)}{dV^{m-1}(\Fil_{\un{n}}\Omega^{q-1}_K)}
\inj \frac{Z_1\Fil_{\un{n}}W_m\Omega^q_K}{dV^{m-1}(\Fil_{\un{n}}\Omega^{q-1}_K)}
\xrightarrow{\ov{C}} \Fil_{{\un{n}}/p}W_{m}\Omega^q_K.
\end{equation}

We now look at the commutative diagram
\begin{equation}\label{eqn:Complete-4-2}
    \xymatrix@C1pc{
  \Fil_{{\un{n}}/p}W_{m}\Omega^q_K  \ar[r]^-{\ov{F}} \ar[d] &
      \frac{Z_1\Fil_{\un{n}}W_m\Omega^q_K}{dV^{m-1}(\Fil_{\un{n}}\Omega^{q-1}_K)}
      \ar[d] \ar[r]^-{\ov{C}}  & \Fil_{{\un{n}}/p}W_{m}\Omega^q_K \ar[d]
      \ar[r]^-{\ov{F}} &
      \frac{Z_1\Fil_{\un{n}}W_m\Omega^q_K}{dV^{m-1}(\Fil_{\un{n}}\Omega^{q-1}_K)} \ar[d] \\
W_{m}\Omega^q_K \ar[r]^-{\ov{F}} &
      \frac{Z_1W_m\Omega^q_K}{dV^{m-1}(\Omega^{q-1}_K)}
      \ar[r]^-{\ov{C}} & W_{m}\Omega^q_K \ar[r]^-{\ov{F}} &
      \frac{Z_1\Fil_{\un{n}}W_m\Omega^q_K}{dV^{m-1}(\Fil_{\un{n}}\Omega^{q-1}_K)}.}
\end{equation}
One knows classically that the composition of any two adjacent maps in the bottom row
is identity. Since the vertical arrows are injective by \lemref{lem:Complete-5},
we conclude that the same holds in the top row too. This proves the lemma.
\end{proof}

We let $B_0\Fil_{\un{n}}\Omega^q_K =0$ and $B_1\Fil_{\un{n}}\Omega^q_K =
d(\Fil_{\un{n}}\Omega^{q-1}_K)$.
We let $Z_0\Fil_{\un{n}}\Omega^q_K = \Fil_{\un{n}}\Omega^q_K$.
For $i \ge 2$, we let $Z_i\Fil_{\un{n}}\Omega^q_K$ (resp. $B_i\Fil_{\un{n}}\Omega^q_K$)
be the inverse image of $Z_{i-1}\Fil_{\un{n}}\Omega^q_K$
(resp. $B_{i-1}\Fil_{\un{n}}\Omega^q_K$) under the composite map
$Z_{i-1}\Fil_{\un{n}}\Omega^q_K \surj
\frac{Z_{i-1}\Fil_{\un{n}}\Omega^q_K}{B_1\Fil_{\un{n}}\Omega^q_K}
\xrightarrow{\ov{C}} \Fil_{{\un{n}}/p}\Omega^q_K$. 
The following is easily deduced from ~\eqref{eqn:Complete-4-0} and
~\eqref{eqn:Complete-4-2}.

\begin{lem}\label{lem:Complete-9}
\begin{enumerate}
  \item
$Z_i\Fil_{\un{n}}\Omega^q_K = F^i(\Fil_{\un{n}}W_{i+1}\Omega^q_K), \ \
  B_i\Fil_{\un{n}}\Omega^q_K = F^{i-1}d(\Fil_{\un{n}}W_{i}\Omega^{q-1}_K)$.
\item
  $Z_i\Fil_{\un{n}}\Omega^q_K \supset Z_{i+1}\Fil_{\un{n}}\Omega^q_K \supset
  B_{j+1}\Fil_{\un{n}}\Omega^q_K \supset B_{j}\Fil_{\un{n}}\Omega^q_K \ \
  \forall \ \ i, j \ge 0$.
\item
  The maps
$\ov{C} \colon \frac{Z_i\Fil_{\un{n}}\Omega^q_K}{B_i\Fil_{\un{n}}\Omega^q_K}
  \xrightarrow{\cong}
  \frac{Z_{i-1}\Fil_{{\un{n}}/p}\Omega^q_K}{B_{i-1}\Fil_{{\un{n}}/p}\Omega^q_K}$
  and $\ov{F} \colon
  \frac{Z_{i-1}\Fil_{{\un{n}}/p}\Omega^q_K}{B_{i-1}\Fil_{{\un{n}}/p}\Omega^q_K}
  \xrightarrow{\cong} \frac{Z_i\Fil_{\un{n}}\Omega^q_K}{B_i\Fil_{\un{n}}\Omega^q_K}$
are inverses of each other.
\end{enumerate}
\end{lem}

\begin{lem}\label{lem:Complete-8}
  We have $Z_i\Fil_{\un{n}}\Omega^q_K = Z_i\Omega^q_K \bigcap \Fil_{\un{n}}\Omega^q_K$
  and $B_i\Fil_{\un{n}}\Omega^q_K = B_i\Omega^q_K \bigcap \Fil_{\un{n}}\Omega^q_K$.
\end{lem}
\begin{proof}
  If $i = 1$, the first claim follows from the definition of
  $Z_1\Fil_{\un{n}}\Omega^q_K$ while the second
  claim follows from \lemref{lem:Complete-5}.
  To prove the first claim for $i \ge 2$, we look at the commutative
  diagram of exact sequences
  \[
  \xymatrix@C.8pc{
  0 \ar[r] &  B_1\Fil_{\un{n}}\Omega^q_K  \ar[r] \ar[d] & Z_{i}\Fil_{\un{n}}\Omega^q_K 
  \ar[r]^-{C} \ar[d] &  Z_{i-1}\Fil_{{\un{n}}/p}\Omega^q_K  \ar[r] \ar[d] & 0 \\
  0 \ar[r] &  B_1\Omega^q_K \bigcap \Fil_{\un{n}}\Omega^q_K  \ar[r] &
  Z_{i}\Omega^q_K  \bigcap \Fil_{\un{n}}\Omega^q_K 
  \ar[r]^-{C} &  Z_{i-1}\Omega^q_K \bigcap \Fil_{{\un{n}}/p}\Omega^q_K &}
  \]
  The right vertical arrows is bijective by induction, and we showed above that the
  left vertical arrow is bijective. This implies the same for the middle vertical
  arrow. The second claim for $i \ge 2$ is proven similarly by replacing
  the $Z_i$-groups (resp. $Z_{i-1}$-groups) by the $B_i$-groups (resp. $B_{i-1}$-groups)
  in the above diagram.
\end{proof}

\begin{cor}\label{cor:Complete-8-00}
  For $i \ge 1$, there is an exact sequence
  \[
  0 \to Z_i\Fil_{\un{n}}\Omega^q_K \to Z_{i-1}\Fil_{\un{n}}\Omega^q_K
  \xrightarrow{(-1)^{q+1}dC^{i-1}} B_1 \Fil_{{\un{n}}/{p^{i-1}}}\Omega^{q+1}_K \to 0.
  \]
\end{cor}
\begin{proof}
  The surjectivity of $(-1)^{q+1}dC^{i-1}$ follows from \lemref{lem:Complete-9}(1).
  The rest of the corollary is easily reduced by \lemref{lem:Complete-8} to showing
  that
\[
0 \to Z_i\Omega^q_K \to Z_{i-1}\Omega^q_K
\xrightarrow{(-1)^{q+1}dC^{i-1}} \Omega^{q+1}_K 
\]
is exact, which is classical (cf. \cite[Lem.~1.14]{Lorenzon}).
\end{proof}

\begin{lem}\label{lem:VR-0}
  There exists a short exact sequence
  \begin{equation}\label{eqn:VR-0-0}
  0 \to V^{m-1}(Z_1\Fil_{\un{n}}\Omega^q_K) + dV^{m-1}(\Fil_{\un{n}}\Omega^{q-1}_K) \to
  Z_1\Fil_{\un{n}}W_m\Omega^q_K \xrightarrow{R} Z_1\Fil_{{\un{n}}/p}W_{m-1}\Omega^q_K \to 0.
  \end{equation}
\end{lem}
\begin{proof}
 By comparing with the classical case and using \lemref{lem:Log-fil-4},
  one checks that ~\eqref{eqn:VR-0-0} is a well-defined complex.
  Moreover, Proposition~\ref{prop:Complete-0} and \lemref{lem:Complete-4}(1)
  together imply that the map $R$ is surjective. We now let $\omega \in \Ker(R)$.
  We then know from Proposition~\ref{prop:Complete-0} that $\omega =
  V^{m-1}(a) + dV^{m-1}(b)$ for some $a \in \Fil_{\un{n}}\Omega^q_K$ and
  $b \in  \Fil_{\un{n}}W_m\Omega^{q-1}_K$. By definition of $Z_1\Fil_{\un{n}}\Omega^q_K$,
  we then get $da = F^{m-1}d(V^{m-1}(a) + dV^{m-1}(b)) = 0$. In other words,
  $\omega \in V^{m-1}(Z_1\Fil_{\un{n}}\Omega^q_K) + dV^{m-1}(\Fil_{\un{n}}\Omega^{q-1}_K)$.
\end{proof}

\begin{defn}\label{def:two-term*}
For $\n \in \Z^r$, we let
  \[
  Z_1R^q_{m,\un{n}}(K) =
\Ker(Z_1\Fil_{\un{n}}\Omega^q_K \oplus \Fil_{\un{n}}\Omega^{q-1}_K
\xrightarrow{V^{m},\ dV^{m}} Z_1\Fil_{\un{n}}W_{m+1}\Omega^q_K) \text{ and}
\]
\[
R^q_{m,\un{n}}(K) = \Ker(\Fil_{\un{n}}\Omega^q_K \oplus \Fil_{\un{n}}\Omega^{q-1}_K
\xrightarrow{V^{m}, \ dV^{m}} \Fil_{\un{n}}W_{m+1}\Omega^q_K).
\]
\end{defn}
One notes that $Z_1R^q_{m,\un{n}}(K) = R^q_{m,\un{n}}(K)$. Indeed, any element
$(a,b) \in  R^q_{m,\un{n}}(K)$ has the property that $a \in Z_1\Omega^q_K$.
This is easily deduced from \cite[Chap.~I, Thm.~3.8]{Illusie} using 
N{\'e}ron-Popescu approximation. 

\begin{defn}\label{def:two-term}
   For $\un{n} \ge -\un{1}$, we
let $C^{m,q}_{\un{n}, \bullet}(A)$ denote the 2-term complex
\begin{equation}\label{eqn:two-term-0}
  \left(Z_1\Fil_{\un{n}}W_m\Omega^q_K \xrightarrow{1-C}
  \Fil_{\un{n}}W_m\Omega^q_K\right).
\end{equation}
\end{defn}

We shall use the following property of $C^{m,q}_{\un{n}, \bullet}(A)$.

\begin{lem}\label{lem:VR-2}
  In the derived category of abelian groups, there exists a distinguished triangle
\[
C^{1,q}_{\un{n}, \bullet}(A) \xrightarrow{V^{m-1}} C^{m,q}_{\un{n}, \bullet}(A)
\xrightarrow{R}  C^{m-1,q}_{{\un{n}}/p, \bullet}(A) \xrightarrow{+}
C^{1,q}_{\un{n}, \bullet}(A) [1] .
\]
\end{lem}
\begin{proof}
We consider the diagram (for $\un{n} \ge -\un{1}$)
\begin{equation}\label{eqn:VR-1}
    \xymatrix@C2.5pc{
      0 \ar[r] & R^q_{m-1,\un{n}}(K) \ar[d]_-{\phi} \ar[r]^-{\alpha} &  
  Z_1\Fil_{\un{n}}\Omega^q_K \oplus \Fil_{\un{n}}\Omega^{q-1}_K \ar[rr]^-{V^{m-1},\ dV^{m-1}} 
  \ar[d]^-{\phi}& & Z_1\Fil_{\un{n}}W_m\Omega^q_K \ar[d]^-{(1-C)} \\
   0 \ar[r] & R^q_{m-1,\un{n}}(K) \ar[r]^-{\alpha} &  
 \Fil_{\un{n}}\Omega^q_K \oplus \Fil_{\un{n}}\Omega^{q-1}_K \ar[rr]^-{V^{m-1},\ dV^{m-1}} 
 && \Fil_{\un{n}}W_m\Omega^q_K,}
\end{equation}
where we let $\wt{C}(a, b) = C(a)$ and $\phi = 1-\wt{C}$. The arrow $\alpha$ is
the canonical inclusion (cf. Definition~\ref{def:two-term*}).
The rows are exact since $Z_1R^q_{m-1,\un{n}}(K) =  R^q_{m-1,\un{n}}(K)$ as we observed
before. All vertical arrows are defined since
  $\Fil_{{\un{n}}/p}W_m\Omega^q_K \subset \Fil_{\un{n}}W_m\Omega^q_K$ under our
assumption on $\un{n}$. Since $CdV^{m-1} = 0$ by \lemref{lem:Complete-6}, this
diagram is commutative.

We now define $F^{1,q}_{\un{n},_\bullet}(A)$ (resp. $\wt{F}^{1,q}_{\un{n},_\bullet}(A)$)
    to be the 2-term complex consisting of the middle (resp. left)
    column of ~\eqref{eqn:VR-1}. Using \lemref{lem:VR-0},
   our claim then is reduced to showing that the composite map
    \begin{equation}\label{eqn:VR-1-0}
    C^{1,q}_{\un{n}, \bullet}(A) \inj F^{1,q}_{\un{n},_\bullet}(A) \surj
    \frac{F^{1,q}_{\un{n},_\bullet}(A)}{\wt{F}^{1,q}_{\un{n},_\bullet}(A)}
    \end{equation}
    is a quasi-isomorphism.

    To prove this quasi-isomorphism, we note that the injective arrow
    in ~\eqref{eqn:VR-1-0} is a quasi-isomorphism
    by the definition of $\phi$. It suffices therefore to show that the surjective
    arrow in  ~\eqref{eqn:VR-1-0} is a quasi-isomorphism. Equivalently, we need to
    show that the left vertical arrow in ~\eqref{eqn:VR-1} is bijective.
To prove this latter claim, it suffices to show that  $\wt{C} = (C,0)$ is nilpotent on
$R^q_{m-1,\un{n}}(K)$. To see this, note that the composite inclusion
\[
R^q_{m-1,\un{n}}(K) \xrightarrow{\alpha} \Fil_{\un{n}}\Omega^q_K \oplus
\Fil_{\un{n}}\Omega^{q-1}_K \inj \Omega^q_K \oplus \Omega^{q-1}_K
\]
actually factors through the
inclusion $R^q_{m-1,\un{n}}(K) \inj  B_m\Omega^q_K \oplus \Omega^{q-1}_K$ by
\cite[Chap.~I, Thm.~3.8]{Illusie}. Furthermore, $C^{m} = 0$ on $B_m\Omega^q_K$.
Since $\wt{C} = (C,0)$ is zero on $\Omega^{q-1}_K$, we conclude that
${(\wt{C})}^{m} = 0$ on $R^q_{m-1,\un{n}}(K)$. This finishes the proof.
\end{proof}

\subsection{Some more properties of $\Fil_{\un{n}} W_m\Omega^\bullet_K$}
\label{sec:More-prop}
In this subsection, we shall prove some more properties of various operators on
the filtered de Rham--Witt complex that will be used in this text. The most of these
are contained in the following.

\begin{lem}\label{lem:Complete-10}
  We have the following.
  \begin{enumerate}
\item
    $\Ker(F^{m-1}d \colon \Fil_{\un{n}}W_m\Omega^q_K \to \Fil_{\un{n}}\Omega^{q+1}_K)
    = F(\Fil_{\un{n}}W_{m+1}\Omega^q_K)$.
  \item
    $\Ker(F^{m-1}\colon \Fil_{\un{n}}W_m\Omega^q_K \to \Fil_{\un{n}}\Omega^q_K) =
    V(\Fil_{\un{n}}W_{m-1}\Omega^q_K)$.
  \item
    $\Ker(dV^{m-1}\colon \Fil_{\un{n}}\Omega^q_K \to \Fil_{\un{n}}W_m\Omega^{q+1}_K)
    = F^m(\Fil_{\un{n}}W_{m+1}\Omega^q_K)$.
  \item
    $\Ker(V \colon \Fil_{\un{n}}W_{m}\Omega^q_K \to \Fil_{\un{n}}W_{m+1}\Omega^{q}_K)
    = dV^{m-1}(\Fil_{\un{n}}\Omega^{q-1}_K)$.
  \item
    $\Ker(V^{m-1} \colon \Fil_{\un{n}}\Omega^q_K \to \Fil_{\un{n}}W_{m}\Omega^{q}_K)
    = F^{m-1}dV(\Fil_{\un{n}}W_{m-1}\Omega^{q-1}_K)$.
  \item 
    $\Ker(dV^{m-1}\colon \Fil_{\un{n}}\Omega^q_K \to
    \Fil_{\un{n}}W_m\Omega^{q+1}_K/V^{m-1}\Fil_\n \Omega^{q+1}_K)
    = F^{m-1}(\Fil_{\un{n}}W_{m}\Omega^q_K)$.
   \item
     $\Ker(V^{m-1} \colon \Fil_{\un{n}}\Omega^q_K \to
     \Fil_{\un{n}}W_{m}\Omega^{q}_K/dV^{m-1}\Fil_{\un{n}}\Omega^{q-1}_K )
    = F^{m}dV(\Fil_{\un{n}}W_{m}\Omega^{q-1}_K)$. 
    \end{enumerate}
\end{lem}
\begin{proof}
  We first note that all statements are classical for $W_\star\Omega^\bullet_K$ and
  can be easily deduced from \cite[Chap.~I, \S~3]{Illusie} using N{\'e}ron-Popescu
  approximation. Now, item (1) follows directly from \lemref{lem:Complete-4}(1).
Item (2) is a special case of the more general claim that
\begin{equation}\label{eqn:ker-F-l}
    \Ker(F^l \colon \Fil_{\un{n}}W_{m+l}\Omega^q_K \to  \Fil_{\un{n}}W_{m}\Omega^q_K)
= V^m(\Fil_{\un{n}}W_{l}\Omega^q_K)
\end{equation}
for all $m \ge 0$ and $l \ge 1$.
Since this general claim is known for $W_\star\Omega^\bullet_K$
(e.g., use \cite[Chap.~I, (3.21.1.3)]{Illusie} and N{\'e}ron-Popescu approximation), one
gets that $V^m(\Fil_{\un{n}}W_{l}\Omega^q_K) \subseteq
\Ker(F^l \colon \Fil_{\un{n}}W_{m+l}\Omega^q_K \to  \Fil_{\un{n}}W_{m}\Omega^q_K)$.
We shall prove the reverse inclusion by induction on $l$. We first consider the
case $l = 1$. It suffices to show in this case 
that $V^m(\Omega^q_{K})
\bigcap \Fil_{\un{n}}W_{m+1}\Omega^q_K = V^m(\Fil_{\un{n}}\Omega^q_K)$. But this follows
from \lemref{lem:Complete-5}.

We now assume $l \ge 2$ and $m \ge 1$. 
We let $\omega \in \Fil_{\un{n}}W_{m+l}\Omega^q_K$ be such that
$F^l(\omega) = 0$. We let $\omega' = F(\omega)$ so that $\omega' \in
\Fil_{\un{n}}W_{m+l-1}\Omega^q_K$ and $F^{l-1}(\omega') = 0$.
The induction hypothesis
implies that $F(\omega) = \omega' = V^{m}(x)$ for some $x \in
\Fil_{\un{n}}W_{l-1}\Omega^q_K$. This implies that
$F^{l-2}d(x) = F^{m+l-2}dV^m(x) = F^{m+l-2}dF(\omega) = 0$ in $\Omega^{q+1}_K$
because $dF=pFd$.
We conclude from item (1) that $x = F(y)$ from some
$y \in \Fil_{\un{n}}W_{l}\Omega^q_K$.
This yields
\[
F(\omega - V^m(y)) = F(\omega) - FV^m(y) = F(\omega) - V^mF(y) =
F(\omega) - V^m(x) =0.
\]
Using the $l =1$ case, we get
$\omega - V^m(y) \in V^m(\Fil_{\un{n}}W_{l}\Omega^q_K)$.
In particular, $\omega \in V^m(\Fil_{\un{n}}W_{l}\Omega^q_K)$.

To prove (3), it suffices to show that $$F^m(\Fil_{\un{n}}W_{m+1}\Omega^q_K) =
F^m(W_{m+1}\Omega^q_K) \bigcap \Fil_{\un{n}}W_{m+1}\Omega^q_K.$$ But this follows from
Lemmas~\ref{lem:Complete-9} and ~\ref{lem:Complete-8} because of the
classical identity $Z_m\Omega^q_K = F^m(W_{m+1}\Omega^q_K)$. The proof of (4) is identical to that of the $l =1$ case of \eqref{eqn:ker-F-l}
using \lemref{lem:Complete-5} and the classical case.

To prove (5), we can assume $m \ge 2$. In the latter case, we only need to show
that $\Ker(V^{m-1}) \subset
F^{m-2}d(\Fil_{\un{n}}W_{m-1}\Omega^{q-1}_K)$. By \lemref{lem:Complete-9}, we can replace
$F^{m-2}d(\Fil_{\un{n}}W_{m-1}\Omega^{q-1}_K)$ by $B_{m-1}\Fil_{\un{n}}\Omega^q_K$.
By the classical case, we are thus reduced to showing that
$B_{m-1}\Omega^q_K \bigcap \Fil_{\un{n}}\Omega^q_K \\ = B_{m-1}\Fil_{\un{n}}\Omega^q_K$.
But this follows from \lemref{lem:Complete-8}.
The proofs of (6) and (7) are analogous to the proofs of (3) and (5), respectively,
using \lemref{lem:Complete-5}.
\end{proof}

\begin{lem}\label{lem:VR-3}
  Assuming $\un{n} \ge -\un{1}$, we have
  \[
  d(\Fil_{\un{n}}W_m\Omega^{q}_K) \subset {\rm Image}(Z_1\Fil_{\un{n}}W_m\Omega^{q+1}_K
  \xrightarrow{1-C} \Fil_{\un{n}}W_m\Omega^{q+1}_K).
  \]
\end{lem}
\begin{proof}
  For $\omega \in \Fil_{\un{n}}W_m\Omega^{q}_K$, we have
  $d(\omega) = FdV(\omega) - RdV(\omega) + FdV^2R(\omega) - RdV^2R(\omega) +
  FdV^3R^2(\omega) - RdV^3R^2(\omega) + \cdots + FdV^mR^{m-1}(\omega) -
  RdV^{m}R^{m-1}(\omega)$, as $RdV^mR^{m-1}(\omega) = 0$. Hence, we get
  $$d(\omega) =
  (F-R)((dV + dV^2R + \cdots + dV^mR^{m-1})(\omega)) = (1-C)Fd(x),$$ where
  we let $x = (V + V^2R + \cdots + V^mR^{m-1})(\omega)$. It suffices therefore to show
  that $Fd(x) \in Z_1\Fil_{\un{n}}W_m\Omega^{q+1}_K$.
  To this end, we first note that $F^{m-1}d(Fd(x)) = pF^{m}d^2(x) = 0$,
  and this implies that $Fd(x) \in Z_1W_m\Omega^{q+1}_K$. Secondly,
  \lemref{lem:Log-fil-4} implies that $Fd(x) \in \Fil_{\un{n}}W_m\Omega^{q+1}_K$
  because $\omega \in \Fil_{\un{n}}W_m\Omega^{q}_K$. This completes
  the proof.
\end{proof}

\begin{remk}\label{remk:VR-3-0}
  The same proof also shows that $$d(\Fil_{-\un{1}}W_m\Omega^{q}_K) \subset
  {\rm Image}(\Fil_{-\un{1}}W_{m+1}\Omega^{q+1}_K
  \xrightarrow{R-F} \Fil_{-\un{1}}W_m\Omega^{q+1}_K).$$
\end{remk}

\begin{lem}\label{lem:VR-4}
  The abelian group $\frac{\Fil_\n W_m\Omega^q_K}{d(\Fil_\n W_m\Omega^{q-1}_K)}$ is
  generated by the images of elements of the form
  $V^i([x]_{m-i})\dlog([x_1]_m)\wedge \cdots \wedge \dlog([x_q]_m)$, where
  $[x]_{m-i} \in \Fil_\n W_{m-i}(K), \ x_i \in K^\times$ and $0 \le i \le m-1$.
\end{lem}
\begin{proof}
  The $q =0$ case of the lemma is clear and so we assume $q \ge 1$.
 We know from \lemref{lem:Log-fil-4}(2) that
 ${\Fil_\n\Omega^q_K} = \pi^{-\n}\Omega^q_A(\log\pi)$
 ($\pi^{-\n}=x_1^{-n_1}\cdots x_r^{-{n_r}}$),
  which is clearly generated as an abelian group by
  elements of the desired kind. This proves our assertion when $m =1$.
  Using this case and induction on $m$, the general case follows at once by
  \propref{prop:Complete-0}, combined with the
  observation that $[x]_{m-i-1} \in \Fil_{{\un{n}}/p}W_{m-i-1}(K)$ implies
  $[x]_{m-i} \in \Fil_{{\un{n}}}W_{m-i}(K)$ and
  $R(V^i([x]_{m-i})w_1) =  V^i([x]_{m-i-1})w_2$ if we let
  $w_1 = \dlog([x_1]_m)\wedge \cdots \wedge \dlog([x_q]_m))$ and
  $w_2 = \dlog([x_1]_{m-1})\wedge \cdots \wedge \dlog([x_q]_{m-1})$.
  \end{proof}

Recall from \cite[Chap.~I]{Illusie} that the multiplication map
$p \colon W_m\Omega^q_K \to W_{m}\Omega^q_K$ has a factorization
$W_m\Omega^q_K \stackrel{R}{\surj}  W_{m-1}\Omega^q_K \stackrel{\ov{p}}{\inj}
W_{m}\Omega^q_K$.

\begin{lem}$(${\rm cf}. \cite[Cor.~2.2.7]{JSZ}$)$\label{lem:Complete-7}
Let $\omega \in W_{m-1}\Omega^q_K$.
Then $\omega \in \Fil_{{\un{n}}/p}W_{m-1}\Omega^q_K$ if and only if
$\ov{p}(\omega) \in \Fil_{\un{n}}W_m\Omega^q_K$.
\end{lem}
\begin{proof}
  The `only if' part is clear because $R 
 \colon \Fil_{\un{n}}W_{m}\Omega^q_K \to
  \Fil_{{\un{n}}/p}W_{m-1}\Omega^q_K$ is surjective. For the reverse implication,
we let $\wt{\omega} \in W_m\Omega^q_K$ be such that $R(\wt{\omega}) = \omega$.
This yields
$VF(\wt{\omega}) = p\wt{\omega} = \ov{p}(w) \in \Fil_{\un{n}}W_{m}\Omega^q_K$.
Since $F^{m-1}VF(\wt{\omega}) = pF^{m-1}(\wt{\omega}) = 0$,  
it follows from \lemref{lem:Complete-10}(2) that $VF(\wt{\omega}) = V(y')$ for some
$y' \in \Fil_{\un{n}}W_{m-1}\Omega^q_K$. We now apply \lemref{lem:Complete-10}(4)
(for $W_*\Omega^\bullet$) to conclude that $F(\wt{\omega}) - y' = dV^{m-2}(z') =
FdV^{m-1}(z')$ for some $z' \in \Omega^{q-1}_K$.
That is, $F(\wt{\omega} - dV^{m-1}(z')) = y' \in \Fil_{\un{n}}W_{m-1}\Omega^q_K$.
\lemref{lem:Complete-10}(1) (for $W_*\Omega^\bullet$) then implies that
$F^{m-1}d(y') = 0$.

As $y' \in \Fil_{\un{n}}W_{m-1}\Omega^q_K$,
we apply \lemref{lem:Complete-10}(1) one more time to conclude that
$F(\wt{\omega} - dV^{m-1}(z')) = y' = F(y'')$ for some
$y'' \in \Fil_{\un{n}}W_{m}\Omega^q_K$.
In particular, $F(\wt{\omega} - dV^{m-1}(z') - y'') =0$.
We now apply ~\eqref{eqn:ker-F-l} (for $W_*\Omega^\bullet$ and $l =1$)
to find $z'' \in \Omega^q_K$ such that
$\wt{\omega} - dV^{m-1}(z') - y'' = V^{m-1}(z'')$.
That is, $\wt{\omega} - y'' = V^{m-1}(z'') + dV^{m-1}(z')$.
Equivalently, $R(\wt{\omega} - y'') = 0$ so that
$\omega = R(\wt{\omega}) = R(y'') \in \Fil_{{\un{n}}/p}W_{m}\Omega^q_K$.
This concludes the proof. 
\end{proof}

Recall that $\sK^M_{q, (-)}$ denotes the improved Milnor $K$-theory sheaf
(cf. \cite{Kerz-JAG}) on the big Zariski site of the category of
  $\F_p$-schemes.
The following result will be used at a later stage in this text.

\begin{lem}\label{lem:Gersten-0}
  Let $X=\Spec(A), \ X_\pi = \Spec(A[1/{\pi}])$ and $Z \in \{X, X_\pi\}$.
  Let $q \ge 0$ and $i > 0$ be any integers. Then 
  $H^i_\zar(Z, \sK^M_{q, Z}) = 0$.
  In particular, there is a short exact sequence
  \begin{equation}\label{eqn:Gersten-0-0}
  0 \to H^0_\zar(Z, \sK^M_{q, Z}) \xrightarrow{p^m} H^0_\zar(Z, \sK^M_{q, Z})
  \to H^0_\zar(Z, {\sK^M_{q, Z}}/{p^m}) \to 0
  \end{equation}
  for every $m \ge 1$.
\end{lem}
\begin{proof}
  The second part of the lemma follows directly from
  its first part using the exact sequence of Zariski sheaves
    \begin{equation}\label{eqn:Gersten-0-1}
    0 \to \sK^M_{q, Z} \xrightarrow{p^m} \sK^M_{q, Z} 
    \to {\sK^M_{q, Z}}/{p^m} \to 0.
    \end{equation}

    The first part of the lemma for $X$ is well-known, and follows
    from \cite[Prop.~10(8)]{Kerz-JAG}.
    We shall prove the same for $X_\pi$ by induction on $r$.
  Suppose first that $r =1$ and let $Y = \Spec({A}/{(\pi)})$.
  For any regular scheme $W$ and $n \ge 1$, we let $G_n(W)$ denote the Gersten
  complex
  \[
    {\underset{w \in W^{(0)}}\oplus} K^M_{n}(k(w)) \xrightarrow{\partial}
    {\underset{w \in W^{(1)}}\oplus} K^M_{n-1}(k(w)) \xrightarrow{\partial} \cdots .
    \]  
 We then have an exact sequence
    \[
    0 \to G_{q-1}(Y)[-1] \to G_q(X) \to G_q(X_\pi) \to 0.
    \]
    Since $H^i(G_q(X)) = H^i(G_{q-1}(Y)) = 0$ for $i > 0$ by
    \cite[Prop.~10(8)]{Kerz-JAG}), it follows that $H^i_\zar(X_\pi, \sK^M_{q, X_\pi})
    = H^i(G_q(X_\pi)) = 0$ for $i > 0$.

    Suppose now that $r \ge 2$ and set $\pi_1=x_2 \cdots x_r$. We let
    $X_1 = \Spec(A[{1}/\pi_1])$, $Y_1 = \Spec({A[{1}/{\pi_1}]}/{(x_1)})$, 
    and $U = X_1 \setminus Y_1 = X_\pi$. Then we have an exact sequence
    \[
    0 \to G_{q-1}(Y_1)[-1] \to G_q(X_1) \to G_q(U) \to 0.
    \]
    We see by induction on $r$ that $H^i(G_{q-1}(Y_1)) = 0= H^i(G_q(X_1))$ for
    $i > 0$. It follows that $H^i_\zar(X_\pi, \sK^M_{q, X_\pi}) =
    H^i(G_{q}(U)) = 0$ for $i > 0$.
    \end{proof}

\subsection{Global results}\label{sec:Global}
Let $X$ be a Noetherian regular $F$-finite $\F_p$-scheme and let
$E$ be an snc divisor on  $X$
with irreducible components $E_1, \ldots , E_r$.  Let $j \colon U \inj X$ be the
inclusion of the complement of $E$ in $X$. Let $\eta_i$ denote the generic
point of $E_i$ and let $\wh{K}_i$ be the quotient field of $\wh{\sO_{X, \eta_i}}$. Let
$j_i \colon \Spec(\wh{\sO_{X, \eta_i}}) \to X$ denote the canonical map.
We fix integers $q \ge 0$ and $m \ge 1$. Let $D = \sum_i n_iE_i \in \Div_E(X)$
and $Z_1\Fil_{D}W_m\Omega^q_U :=\Fil_{D}W_m\Omega^q_U \bigcap j_* (Z_1W_m\Omega^q_U)$.
\lemref{lem:Complete-6} immediately implies the following.

\begin{cor}\label{cor:Complete-6-global}
The Cartier map $C \colon Z_1W_m\Omega^q_U \to
  W_m\Omega^q_U$ restricts to a $W_{m+1}\sO_X$-linear map
  \[
  C \colon Z_1\Fil_D W_m\Omega^q_U \to \Fil_{D/p}W_m\Omega^q_U
  \]
  such that $C(dV^{m-1}(\Fil_{D}\Omega^{q-1}_U)) = 0$.
\end{cor}

\begin{defn}\label{defn:C-complex-global}
If $D + E \ge 0$, we let
$W_m\sF^{q,\bullet}_{X|D}$ denote the 2-term complex of Zariski sheaves
\begin{equation}\label{eqn:Fil-D-complex}
    \left(Z_1\Fil_{D}W_m\Omega^q_U \xrightarrow{1-C} \Fil_{D}W_m\Omega^q_U\right).
\end{equation}
\end{defn}

It is classical that $W_m\sF^{q,\bullet}_{X|D}$ is quasi-isomorphic to
$W_m\Omega^q_{X, \log}$ when $E = \emptyset$ (cf. \cite[Lem.~4.1.5]{Kato-Comp}).
  The local results of Sections~\ref{sec:com-to-noncom},
~\ref{sec:Cartier} and ~\ref{sec:More-prop} immediately imply the following
(the item (13) of the theorem follows from \cite[Prop.~2.10]{Shiho}).

\begin{thm}\label{thm:Global-version}
Given any $D \in \Div_E(X)$, we have the following.
  \begin{enumerate}
    \item
    There exists an  exact sequence of Zariski sheaves of $W_m\sO_X$-modules
    \[
    0 \to \Fil_DW_m\Omega^q_U \to j_* W_m\Omega^q_U \to \ 
    \stackrel{r}{\underset{i =1}\bigoplus} 
    \frac{(j_i)_*(W_m\Omega^q_{\wh{K}_i})}{(j_i)_*(\Fil_{n_i}W_m\Omega^q_{\wh{K}_i})}.
    \]
  \item
    There is a short exact sequence of Zariski sheaves of $W_m\sO_X$-modules
    \[
    0 \to V^{m-1}(\Fil_D\Omega^q_U) + dV^{m-1}(\Fil_D\Omega^{q-1}_U) \to
    \Fil_DW_m\Omega^q_U \xrightarrow{R}  \Fil_{D/p}W_{m-1}\Omega^q_U \to 0.
    \]
    \item
      $\Ker(F^{m-1}d \colon \Fil_{D}W_m\Omega^q_U \to \Fil_{D}\Omega^{q+1}_U)=
      Z_1\Fil_{D}W_m\Omega^q_U = F(\Fil_{D}W_{m+1}\Omega^q_U) $.
  \item
    There exist isomorphisms of Zariski sheaves of $W_m\sO_X$-modules
   \[
    \ov{F} \colon \Fil_{{D}/p}W_{m}\Omega^q_U \xrightarrow{\cong}
    \frac{Z_1\Fil_{D}W_m\Omega^q_U}{dV^{m-1}(\Fil_{D}\Omega^{q-1}_U)}; \ \
    \ov{C} \colon  \frac{Z_1\Fil_{D}W_m\Omega^q_U}{dV^{m-1}(\Fil_{D}\Omega^{q-1}_U)}
     \xrightarrow{\cong} \Fil_{{D}/p}W_{m}\Omega^q_U.
    \]
  \item
    $\Ker(F^{m-1}\colon \Fil_{D}W_m\Omega^q_U \to \Fil_{D}\Omega^q_U) =
    V(\Fil_{D}W_{m-1}\Omega^q_U)$.
  \item
    $\Ker(dV^{m-1}\colon \Fil_{D}\Omega^q_U \to \Fil_{D}W_m\Omega^{q+1}_U)
    = F^m(\Fil_{D}W_{m+1}\Omega^q_U)= Z_m\Fil_D\Omega^q_U$.
  \item
    $\Ker(V \colon \Fil_{D}W_{m}\Omega^q_U \to \Fil_{D}W_{m+1}\Omega^{q}_U)
    = dV^{m-1}(\Fil_{D}\Omega^{q-1}_U)$.
  \item
$\Ker(V^m \colon \Fil_{D}\Omega^q_U \to \Fil_{D}W_{m+1}\Omega^{q}_U)
    = F^mdV(\Fil_{D}W_m\Omega^{q-1}_U)= B_m\Fil_D\Omega^q_U$.
  \item
  If $D + E \ge 0$, there exists a distinguished triangle
    \[
     W_1\sF^{q,\bullet}_{X|D} \xrightarrow{V^{m-1}} W_m\sF^{q,\bullet}_{X|D}
     \xrightarrow{R} W_{m-1}\sF^{q,\bullet}_{X|{D/p}} \xrightarrow{+}
     W_1\sF^{q,\bullet}_{X|D} [1]
    \]
    in the derived category of sheaves of abelian groups on
    $X_\zar$.
  \item $\varinjlim\limits_{D\ge -E} W_m\sF^{q,\bullet}_{X|D}
    \xrightarrow{\simeq} \left(j_* Z_1W_m\Omega^q_U
    \xrightarrow{1-C} j_*W_m\Omega^q_U\right) = {\bf R}j_* (W_m\Omega^q_{U,\log}).$
     \item 
       $\Ker(dV^{m-1}\colon \Fil_{D}\Omega^q_U \to \Fil_{D}W_m\Omega^{q+1}_U/V^{m-1}\Fil_D
       \Omega^{q+1}_U) = F^{m-1}(\Fil_{D}W_{m}\Omega^q_U)$.
   \item
     $\Ker(V^{m-1} \colon \Fil_{D}\Omega^q_U \to
     \Fil_{D}W_{m}\Omega^{q}_U/dV^{m-1}\Fil_{D}\Omega^{q-1}_U )
     = F^{m}dV(\Fil_{D}W_{m}\Omega^{q-1}_U)$.
   \item
     $j^*(W_m\sF^{q,\bullet}_{X|D})$ is quasi-isomorphic to
     $W_m\Omega^q_{U, \log}$.
   \item
     The inclusion $\Fil_{-E}W_m\Omega^d_U \inj W_m\Omega^d_X$ is a bijection if
     $d$ is the rank of $\Omega^1_X$.
  \end{enumerate}
\end{thm}

\vskip .2cm

{\bf{Proof of \thmref{thm:Main-0}:}}
The first three items follow directly from \lemref{lem:Log-fil-Wcom}
and \thmref{thm:Global-version}(2). If $f \colon (X', E') \to (X,E)$ is a morphism
between two snc-pairs (cf. \S~\ref{sec:de-log-0}), then it is clear from the
construction of $\Fil_DW_m\Omega^q_U$ that the pull-back map
$f^* \colon W_m\Omega^q_{X \setminus E} \to f_*(W_m\Omega^q_{X'\setminus E'})$ restricts
to a map $f^* \colon \Fil_D W_m\Omega^q_{X \setminus E} \to f_*(\Fil_{D'}
W_m\Omega^q_{X'\setminus E'})$ whenever $f^*(D) \le D'$.
\qed

\section{Kato's filtration via filtered de Rham--Witt}\label{sec:Kato-comp}
The goal of this section is to present Kato's ramification filtration
\cite{Kato-89} on the $p$-adic {\'e}tale cohomologies of henselian discrete valuation
fields in terms of the hypercohomology of the complexes
$W_m\sF^{q,\bullet}_{X|D}$. This will extend \cite[Thm.~3.2]{Kato-89} to
$H^q(K)\{p\}$ for all $q \ge 1$.
We let $k=\F_p$ and let ${\Et}_k$ denote the big {\'e}tale
site of the category of all $k$-schemes. For a local $k$-algebra $R$, we let
$K^M_*(R)$ denote the Milnor $K$-theory of $R$. We let $\sK^M_{\star, (-)}$ denote the
Milnor $K$-theory sheaf on ${\Et}_k$. In this paper, all Milnor $K$-groups and
Milnor $K$-theory sheaves will be the one introduced in \cite{Kerz-JAG} (cf.
\S~\ref{sec:More-prop}).

\subsection{Cohomology of $W_m\sF^{q, \bullet}_{X|D}$ and Kato filtration}
\label{sec:katofil}
We begin by setting up our notations.
Given any positive integer $n$ prime to $p$,
and any integer $q$, we let ${\Z}/n(q)$ denote the classical {\'e}tale sheaf
of $n$-th roots of unity $\mu_n$ and its (positive and negative) powers on ${\Et}_k$
(cf. \cite[p.~163]{Milne-EC}).
For $X \in \Sch_k$ and $q, m \ge 1$, we let $W_m\Omega^q_{X, \log}$ be the image of the
map of {\'e}tale sheaves $\dlog \colon (\sO^{\times}_X)^q \to W_m\Omega^q_X$, given by
$\dlog(\{x_1, \ldots , x_q\}) = \dlog[x_1]_m \wedge \cdots
\wedge \dlog[x_q]_m$ (cf. \cite[Chap.~I, \S~5.7]{Illusie}).
This map uniquely factors through the composite quotient map
$(\sO^{\times}_X)^q \surj \sK^M_{q, X} \surj  {\sK^M_{q,X}}/{p^m}$
(cf. \cite[\S~1.2]{Morrow-ENS}, \cite[Lem.~3.2.8]{Zhao}).
We shall denote the resulting map ${\sK^M_{q,X}}/{p^m} \to W_m\Omega^q_X$ also by
$\dlog$. This map is bijective if $X$ is regular (cf. \cite[Thm.~5.1]{Morrow-ENS}).
We let $W_m\Omega^0_{\log} = {\Z}/{p^m}$ and $W_m\Omega^q_{\log} = 0$ if
$m \le 0$.
For $q \in \Z$ and $n = p^mr \in \N$ with $(p,r) =1$, we let
${\Z}/n(q) := {\Z}/r(q) \oplus W_m\Omega^q_{(-,\log)}[-q]$
and consider it as a complex of sheaves on ${\Et}_k$.
We let $H^i_\et(X, \Z/n(j)) := \H^i_\et(X, \Z/n(j))$.

If $r' \in \N$ such that $r \mid r'$ and $p \nmid r'$, we have a natural map
$\Z/r(q) \to \Z/r'(q)$, induced by the maps ${\Z}/r \cong \frac{1}{r}\Z/\Z \inj
\frac{1}{r'}\Z/\Z \cong \Z/r'$ and $\mu_r \inj \mu_{r'}$. On the other hand, the
map $\ov{p} \colon W_m\Omega^q_{(-,\log)} \to W_{m+1}\Omega^q_{(-,\log)}$
(cf. \lemref{lem:Complete-7}) induces a
canonical map $\Z/p^m(q) \to \Z/p^{m+1}(q)$. As a result, we have canonical maps
$\Z/n(q) \to \Z/n'(q)$, whenever $n \mid n'$. In particular,
$\{\Z/n(q)\}_{n \ge 1}$ form an ind-sheaf on ${\Et}_k$ for any given $q \ge 0$. 
Given $X \in \Sch_k$, we let $H^i_\et(X, {\Q}/{\Z}(j)) = \varinjlim_n
H^i_\et(X, \Z/n(j))$ and $H^q(X) =  H^q_\et(X, {\Q}/{\Z}(q-1))$ so that
$H^q(X)\{p\} = H^q_\et(X, {\Q_p}/{\Z_p}(q-1))$. We let 
$H^q_{n}(X) = H^q_\et(X, {\Z}/{n}(q-1))$ so that
$H^q_{p^m}(X) = H^1_\et(X, W_m\Omega^{q-1}_{X, \log})$.
We write $H^q_{n}(X)$ (resp. $H^q(X)$)
as $H^q_{n}(A)$ (resp. $H^q(A)$) if $X = \Spec(A)$.
We shall use the following result a few times in this text.

\begin{lem}\label{lem:Inj-GL}
  Assume that $X = \Spec(R)$ with $R$ regular local.
  Let $\fm = (x_1, \ldots , x_d)$ be the maximal ideal of $R$ and let $\pi =
  x_1 \cdots x_r$ for some $1 \le r \le d$. We let $X_\pi = \Spec(A[1/{\pi}])$.
  Then the canonical map
$H^q_{p^m}(Z) \to H^q(Z)$ is injective and
  identifies $H^q_{p^m}(Z)$ as $H^q(Z)[p^m]$ for $Z \in \{X, X_\pi\}$.
  If $X$ is any regular $k$-scheme, then $H^1_{p^m}(X) = H^1(X)[p^m]$.
  If $X$ is any regular $k$-scheme with
  $\Pic(X)[p] = 0$, then $H^2_{p^m}(X) = H^2(X)[p^m]$.
\end{lem}
\begin{proof}
Recall that we have an exact sequence
of {\'e}tale sheaves (cf. \cite[Thm.~8.1]{Geisser-Levine})
\begin{equation}\label{eqn:GL}
  0 \to {\sK^M_{q,X}}/{p^m} \to {\sK^M_{q,X}}/{p^{m+n}} \xrightarrow{\alpha_n}
  {\sK^M_{q,X}}/{p^n} \to 0
\end{equation}
for every $n \ge 1$ and $q \ge 0$, where $\alpha_n$ is such that the composition
${\sK^M_{q,X}}/{p^{m+n}} \xrightarrow{\alpha_n} {\sK^M_{q,X}}/{p^n}
\cong \frac{{p^m}{\sK^M_{q,X}}}{p^{m+n}\sK^M_{q,X}} \inj
      {\sK^M_{q,X}}/{p^{m+n}}$ is multiplication by $p^m$.
      It suffices therefore to show that the canonical map
      $H^0_\et(Z,{\sK^M_{q,Z}}/{p^{m+n}}) \to
      H^0_\et(Z,{\sK^M_{q,Z}}/{p^n})$ is surjective for $q \ge 0$ if $X$ is local and
      $Z \in \{X, X_\pi\}$. But this follows from \lemref{lem:Gersten-0}.

As the map $H^0_\et(X, {\sK^M_{0,X}}/{p^{m+n}}) \to
H^0_\et(X,{\sK^M_{0,X}}/{p^n})$ is always surjective, we get 
$H^1_{p^m}(X) = H^1_{m+n}(X)[p^m]$ for any regular $k$-scheme $X$.
For $q =2$, the claim follows from our assumption using the exact sequence
$0 \to {\Pic(X)}/{p^m} \to H^2_{p^m}(X) \to \Br(X)[p^m] \to 0$.
\end{proof}

If $R$ is a regular local ring, the exact sequence
\begin{equation}\label{eqn:Milnor-0}
  0 \to W_m\Omega^q_{{(-)}, \log} \to Z_1W_m\Omega^q_{(-)} \xrightarrow{1-C}
    W_m\Omega^q_{(-)} \to 0
\end{equation}
of sheaves on $\Spec(R)_\et$ yields an exact sequence
\begin{equation}\label{eqn:Milnor-0.1}
  0 \to W_m\Omega^q_{R, \log} \to Z_1W_m\Omega^q_R \xrightarrow{1-C} W_m\Omega^q_R
  \xrightarrow{\delta^q_m} H^{q+1}_{p^m}(R) \to 0.
  \end{equation}
We also have a commutative diagram (cf. \cite[\S~1.3]{Kato-89})
\begin{equation}\label{eqn:Milnor-1}
  \xymatrix@C1pc{
    {K^M_q(R)}/{p^m} \otimes W_m\Omega^{q'}_{R} \ar[r]^-{\cup}
    \ar[d]_-{\id \otimes \delta^{q'}_m} \ar[dr]^-{\lambda^{q+q'}_{p^m}} &
    W_m\Omega^{q+q'}_{R} \ar[d]^-{\delta^{q+q'}_{m}} \\
     {K^M_q(R)}/{p^m} \otimes H^{q'+1}_{p^m}(R) \ar[r]^-{\cup} &
     H^{q+q'+1}_{p^m}(R),}
\end{equation}
where the bottom row is induced by the Bloch-Kato map
$\beta^q_{R} \colon K^M_q(R) \to H^q_\et(R, {\Z}/{p^m}(q))$,
followed by the cup product on the {\'e}tale cohomology of
$W_m\Omega^\bullet_{(-), \log}$ 
on $\Spec(R)$. The top row is induced by the dlog map, followed by the wedge
product. We shall denote the composite map
$W_m\Omega^q_R \xrightarrow{\delta^q_m} H^{q+1}_{p^m}(R) \inj H^{q+1}(R)$
also by $\delta^q_m$.

\vskip .3cm

{\bf{The setting of \S~\ref{sec:Kato-comp}:}}
From now on until the end of \S~\ref{sec:Kato-comp}, we fix an $F$-finite henselian
discrete valuation ring $A$ containing $k$ and 
let $\fm = (\pi)$ denote the maximal ideal of $A$. Let $\ff = {A}/{\fm}$ and
$K = Q(A)$.  We let $X = \Spec(A), \ D_n = V((\pi^n))$ and
$U = \Spec(K)$.
We fix integers $n, q \ge 0$ and $m \ge 1$, and let $T^{m,q}_n(K)$ denote the
cokernel of the map $(1-C) \colon Z_1\Fil_nW_m\Omega^q_K \to \Fil_nW_m\Omega^q_K$.

\begin{lem}\label{lem:hyp}
 We have 
  \begin{enumerate}
      \item $H^i_{\et}(X, Z_1 \Fil_{D_{n}} W_m \Omega^q_U) =
  H^i_\et(X, \Fil_{D_{n}} W_m \Omega^q_U) = 0$  for $i \ge 1$. 
  \item \[
\H^i_{\et}(X, W_m\sF^{q, \bullet}_{X|{D_n}}) = \left\{\begin{array}{ll}
K^M_q(K)/p^m & \mbox{if $i = 0$} \\
T^{m,q}_n(K) & \mbox{if $i = 1$} \\
 0 & \mbox{if $i > 1$.}
\end{array}\right.
\]
  \end{enumerate}
   \end{lem}
\begin{proof}
  The item (1) is clear since the underlying sheaves are quasi-coherent
  sheaves of $W_m\sO_X$-modules.
   To prove (2), we  look at the hypercohomology exact sequence
    \begin{equation}\label{eqn:hyp-0}
      0 \to \H^0_\et(W_m\sF^{q, \bullet}_{X|{D_n}}) \to Z_1 \Fil_{{n}} W_m \Omega^q_K
      \xrightarrow{1-C} \Fil_{{n}} W_m \Omega^q_K \to \H^1_\et(W_m\sF^{q, \bullet}_{X|{D_n}})
      \xrightarrow{} 0,
    \end{equation}
    where $\H^i_\et(\sE)$ (resp. $H^i_\et(\sE)$) is the shorthand notation for
    $\H^i_\et(X,\sE)$ (resp. $H^i_\et(X,\sE)$) and the last term of
    \eqref{eqn:hyp-0} is
    $0$ by (1). Since $W_m\Omega^q_{A} \subset \Fil_{n}W_m\Omega^q_{K}
    \subset W_m\Omega^q_{K}$, we have 
    \begin{equation*}
        \begin{array}{cl}
          \dlog \colon K^M_q(A)/p^m   &\xrightarrow{\cong}
          \Ker(1-C: Z_1W_m\Omega^q_{A} \to
          W_m\Omega^q_{A} ) \\
          &\subset \Ker(1-C: Z_1\Fil_{n}W_m\Omega^q_{K} \to \Fil_{n}W_m\Omega^q_{K} )=
          \H^0_\et(X,W_m\sF^q_{X|{D_n}})\\
             &\subset \Ker(1-C: Z_1W_m\Omega^q_{K} \to W_m\Omega^q_{K} )= {K^M_q(K)}/p^m.
        \end{array}
    \end{equation*}
     
    On the other hand, as $A$ is a DVR, $\dlog(K^M_q(K)/p^m) \subset
    Z_1W_m\Omega^q_{K}$ is generated as a group by elements of the form
    \[
    (1) \ \dlog([a_1]_m)\wedge \cdots \wedge \dlog([a_q]_m) \ \ \mbox{and} \ \
    (2) \ \dlog([a_1]_m)\wedge \cdots \wedge \dlog([a_{q-1}]_m)\wedge
    \dlog([\pi]_m),
    \]
    where $a_i \in A^\star$. However, the elements of the form (1) and (2) also lie
    within $Z_1\Fil_{0}W_m\Omega^q_{K} \subset Z_1\Fil_{n}W_m\Omega^q_{K}$ for $n \ge 0$.
    Since such elements clearly lie in $\Ker(1-C \colon Z_1W_m\Omega^q_{K} \to
    W_m\Omega^q_{K})$, we conclude  that $K^M_q(K)/p^m \cong
    \H^0_\et(X,W_m\sF^q_{X|{D_n}})$ and
    $\H^1_\et(X,W_m\sF^q_{X|{D_n}}) =\coker(1-C: Z_1\Fil_{n}W_m\Omega^q_{K} \to
    \Fil_{n}W_m\Omega^q_{K} )=T^{m,q}_n(K)$.
    The remaining claim follows from item (1).
     \end{proof}

By \lemref{lem:hyp}, we get an exact sequence
\begin{equation}\label{eqn:hyp-1}
  0 \to {K^M_q(K)}/p^m \to Z_1\Fil_{n}W_m\Omega^q_{K} \xrightarrow{1-C}
  \Fil_{n}W_m\Omega^q_{K} \xrightarrow{\delta^{m,q}_n} T^{m,q}_n(K) \to 0.
  \end{equation}
This sequence maps canonically to the exact sequence ~\eqref{eqn:Milnor-0.1} with
$R = K$ and hence induces a map $\ov{\delta}^{m,q}_n \colon T^{m,q}_n(K) \to H^{q+1}_{p^m}(K)$.
We shall denote the composite map $T^{m,q}_n(K) \to H^{q+1}_{p^m}(K) \inj H^{q+1}(K)$ also
by $\ov{\delta}^{m,q}_n$.

By \lemref{lem:Log-fil-Wcom} (see Step 5 for $D=0$, $D'= D_n$), we see that the product map
on the top row of ~\eqref{eqn:Milnor-1} (with $R =K$) restricts to
a map
\begin{equation}\label{eqn:Milnor-3}
K^M_q(K)/p^m \otimes \Fil_nW_m\Omega^{q'}_K \xrightarrow{\cup} \Fil_nW_m\Omega^{q+q'}_K.
\end{equation}
It follows from Lemmas~\ref{lem:VR-3} and ~\ref{lem:VR-4}
that the composite map 
$$\partial_n^{m,q}: K^M_q(K) \otimes \Fil_nW_m(K) \xrightarrow{\cup}
\Fil_nW_m\Omega^q_K \xrightarrow{\delta^{m,q}_n} T^{m,q}_n(K)$$
is surjective. We therefore get the following.

\begin{cor}\label{cor:VR-5}
  The map
  \[
 \ov{\delta}^{m,q}_n \colon T^{m,q}_n(K) \to
    {\rm Image} (\Fil_nW_m(K) \otimes K^M_q(K)  \xrightarrow{\lambda^q_m} H^{q+1}(K))
    \]
    is a bijection.
\end{cor}
\begin{proof}
 We only need to show that the map
 $T^{m,q}_n(K)\to H^{q+1}_{p^m}(K)$ is injective. But this follows from
 \propref{prop:Cartier-fil-1}. 
\end{proof}

We now recall Kato's filtration on $H^{q}(K)$ from \cite[\S~2]{Kato-89}. For any
$A$-algebra $R$, let $R^h$ denote the henselization of $R$ with respect to the
ideal $\fm R$. We let
$\Spec({R}/{\fm R}) \xrightarrow{\iota} \Spec(R) \xleftarrow{j} \Spec(R[\pi^{-1}])$
be the inclusions and let
\begin{equation}\label{eqn:Kato-0}
V^q_n(R) = \H^q_\et({R}/{\fm R}, \iota^*{\bf R}j_*({\Z}/{n}(q-1)))
\xrightarrow{\cong} V^q_n(R^h) \xrightarrow{\cong} H^q_\et(R^h[\pi^{-1}], {\Z}/{n}(q-1))
\end{equation}
\[
\mbox{and} \ V^q(R) = {\varinjlim}_{n} V^q_n(R) \cong V^q(R^h) \cong H^q(R^h[\pi^{-1}]).
\]
The following definition is due to Kato \cite[Defn.~2.1]{Kato-89}.

\begin{defn}\label{defn:Kato-1}
For $n,q \ge 0$, we let $\Fil^{\bk}_nH^{q+1}(K)$ be the set of all elements
  $\chi \in H^{q+1}(K)$ such that $\{\chi, 1 + \pi^{n+1}T\} :=
\sigma_{n,q}(\chi \otimes (1 + \pi^{n+1}T)) = 0$ in $V^{q+2}(A[T])$
under the composite map $\sigma_{n,q}$:
\[
H^{q+1}(K)\otimes ((A[T])^h[\pi^{-1}])^{\times}  \to
H^{q+1}(K)\otimes H^1_\et((A[T])^h[\pi^{-1}], {\Q}/{\Z}(1)) \xrightarrow{\cup}
V^{q+2}(A[T]).
\]

We let $\gr^{\bk}_0H^{q+1}(K) = \Fil^{\bk}_0H^{q+1}(K)$ and
$\gr^{\bk}_nH^{q+1}(K) = {\Fil^{\bk}_nH^{q+1}(K)}/{\Fil^{\bk}_{n-1}H^{q+1}(K)}$ if $n \ge 1$.
  \end{defn}

The following result of Kato (cf. \cite[Thm.~3.2]{Kato-89})
provides a description of $\Fil^{\bk}_nH^{q+1}(K)$ in terms of Milnor $K$-theory and
Witt-vectors.

\begin{thm}\label{thm:Kato-2}
  One has the following.
  \begin{enumerate}
  \item
    $\Fil^{\bk}_nH^1(K)\{p\} = {\underset{m \ge 1}\bigcup} \delta^0_m(\Fil_nW_m(K))$.
  \item
 $\Fil^{\bk}_nH^{q+1}(K)\{p\}$ coincides with the image of
    \[
    K^M_{q}(K) \otimes \Fil^{\bk}_nH^1(K)\{p\} \xrightarrow{\cup} H^{q+1}(K).
    \]
    \end{enumerate}
\end{thm}

In \cite{Kerz-Saito-Eratum}, Kerz--Saito proved the following refined version
of \thmref{thm:Kato-2}(1).

\begin{thm}\label{thm:Kato-3}
  One has $\Fil^{\bk}_nH^1(K)[p^m] = \delta^0_m(\Fil_nW_m(K))$.
\end{thm}

The theorems of Kato and Kerz--Saito provide a simple description of
Kato's ramification filtration on $H^1(K)[p^m]$ in terms of Witt-vectors.
The objective of this section is to extend this description to all $q \ge 0$.

\subsection{The $V$-operator on $T^{\star,q}_n(K)$}\label{sec:ind-T}
We fix $n, q \ge 0$ and $m \ge 1$ as
 before. When $n > 0$, we shall write $n = p^rl$, where $r \ge 0$ and $p \nmid l$.
 We let $S^{m,q}_n(K) = \Fil^{\bk}_nH^{q+1}(K)[p^m]$. Since $T^{m,q}_n(K)$
 is a $p^m$-torsion group, we already deduce from
 \corref{cor:VR-5}, \thmref{thm:Kato-2} and ~\eqref{eqn:Milnor-1} that the map
 $\ov{\delta}^{m,q}_n$ has a factorization
 $\ov{\delta}^{m,q}_n\colon T^{m,q}_n(K) \inj S^{m,q}_n(K)$. 
 In the next few lemmas, we shall study the map $V \colon T^{m,q}_n(K) \to T^{m+1,q}_n(K)$
 and estimate its cokernel in terms of the differential forms on $\ff$.
We begin with the following.

\begin{lem}\label{lem:Kato-fil-0}
  There is a commutative diagram
   \begin{equation}\label{eqn:Kato-fil-0-0}
   \xymatrix@C1pc{
     Z_1\Fil_nW_m\Omega^q_K \ar[r]^-{1-C} \ar[d]_-{V} &  \Fil_nW_m\Omega^q_K
     \ar[d]^-{V} \\
     Z_1\Fil_nW_{m+1}\Omega^q_K \ar[r]^-{1-C} &  \Fil_nW_{m+1}\Omega^q_K.}
   \end{equation}
   \end{lem}
\begin{proof}
  By Lemmas~\ref{lem:Log-fil-Wcom} and ~\ref{lem:Complete-6},
we only need to show that $VC = CV$ on $Z_1W_m\Omega^q_K$
and $V(Z_1W_m\Omega^q_K) \subset Z_1W_{m+1}\Omega^q_K$.
  For showing the first claim, it is enough to
  show that $VCF = CVF$. But this is clear because
  $VCF = VR = RV = CFV = CVF$.
   For the second claim, note that $\omega \in Z_1W_m\Omega^q_K$ implies
   $F^{m-1}d(\omega) = 0$. This implies in turn that $F^mdV(\omega) =
   F^{m-1}d(\omega) = 0$. That is, $V(\omega) \in Z_1W_{m+1}\Omega^q_K$.
\end{proof}

Using \lemref{lem:Kato-fil-0}, we get a commutative diagram
 (cf. \cite[\S~1.3]{Kato-89})
 \begin{equation}\label{eqn:Kato-5}
   \xymatrix@C2.5pc{
     \Fil_nW_m\Omega^q_K \ar[d]_-{V} \ar@{->>}[r]^-{\delta^{m,q}_n} & T^{m,q}_n(K)
     \ar[d]^-{V}
     \ar@{^{(}->}[r]^-{\ov{\delta}^{m,q}_n} & S^{m,q}_n(K) \ar@{^{(}->}[d]^-{\can} \\
     \Fil_nW_{m+1}\Omega^q_K \ar@{->>}[r]^-{\delta^{m+1,q}_n} & T^{m+1,q}_n(K)
     \ar@{^{(}->}[r]^-{\ov{\delta}^{m+1,q}_n} &  S^{m+1,q}_n(K),}
   \end{equation}
 where the middle $V$ is the unique map induced by the
 commutative diagram ~\eqref{eqn:Kato-fil-0-0}, and the right vertical arrow
 is the canonical inclusion $\Fil^{\bk}_nH^{q+1}(K)[p^m] \inj
 \Fil^{\bk}_nH^{q+1}(K)[p^{m+1}]$. Since the canonical inclusion
 $H^{q+1}_{p^m}(K) \inj H^{q+1}_{p^{m+1}}(K)$ (cf. \lemref{lem:Inj-GL}) is induced by
 $\ov p = V \colon W_m\Omega^q_{(-,\log)} \to W_{m+1}\Omega^q_{(-,\log)}$, we see that the
 right square of \eqref{eqn:Kato-5} is commutative.

   It follows from the commutative diagram ~\eqref{eqn:Kato-6}
  that the map $T^{m,q}_{n-1}(K) \to T^{m,q}_{n}(K)$ is injective
  for every $n \ge 1$.
We let $\wt{T}^{m,q}_n(K) = \frac{T^{m,q}_{n}(K)}{T^{m,q}_{n-1}(K)}$.
 Since the commutative diagram ~\eqref{eqn:Kato-fil-0-0} is compatible with
 the change in values of $n \ge 0$, it induces maps
 $V \colon \wt{T}^{m,q}_n(K) \to \wt{T}^{m+1,q}_n(K)$ as $n$ varies over all
 positive integers. We let $\wt{M}^{m,q}_n(K) =
 \frac{\wt{T}^{m,q}_n(K)}{V(\wt{T}^{m-1,q}_n(K))}$.

 \vskip.3cm
 
In the rest of the discussion, we shall assume $n \ge 1$ unless stated otherwise.

\begin{lem}\label{lem:Kato-fil-1}
   The map $V \colon \wt{T}^{m,q}_n(K) \to \wt{T}^{m+1,q}_n(K)$ is surjective for every
   $m \ge r+1$.
 \end{lem}
 \begin{proof}
Since $V(a\dlog([x_1]_m)\wedge \cdots \wedge \dlog([x_q]_m) =
   V(a)\dlog([x_1]_m)\wedge \cdots \wedge \dlog([x_q]_m)$,
   ~\eqref{eqn:Kato-fil-0-0} implies that the diagram
   \begin{equation}\label{eqn:Kato-fil-1-0}
   \xymatrix@C1pc{
K^M_q(K) \otimes \Fil_nW_m(K) \ar@{->>}[rr]^-{\partial^{m,q}_n} \ar[d]_-{\id \otimes V}
&& T^{m,q}_{n}(K) \ar[d]^-{V} \\
K^M_q(K) \otimes \Fil_nW_{m+1}(K) \ar@{->>}[rr]^-{\partial^{m+1,q}_n} && T^{m+1,q}_{n}(K)}
   \end{equation}
   is commutative. Furthermore, the horizontal arrows are surjective by
   \corref{cor:VR-5}. It suffices therefore to show that
   $\frac{\Fil_nW_{m+1}(K)}{\Fil_{n-1}W_{m+1}(K) + V(\Fil_nW_{m}(K))} = 0$ for
   $m \ge r+1$.

   To that end, we look at the commutative diagram of exact sequences (cf.
   \lemref{lem:Log-fil-1})
  \begin{equation}\label{eqn:Kato-fil-1-1}
   \xymatrix@C.8pc{ 
     0 \ar[r] & \Fil_{n-1}W_{m}(K) \ar[r]^-{V} \ar[d] & \Fil_{n-1}W_{m+1}(K)
     \ar[d] \ar[r]^-{R^m} & \Fil_{\lfloor{{(n-1)}/{p^m}}\rfloor}W_1(K) \ar[d] \ar[r] & 0 \\
   0 \ar[r] & \Fil_{n}W_{m}(K) \ar[r]^-{V} & \Fil_{n}W_{m+1}(K)
   \ar[r]^-{R^{m}} & \Fil_{\lfloor{{n}/{p^m}}\rfloor}W_1(K) \ar[r] & 0.}
  \end{equation}
Since $m \ge r+1$, we see that $\lfloor{{(n-1)}/{p^m}}\rfloor =
  \lfloor{{n}/{p^m} -{1}/{p^m}}\rfloor = \lfloor{{n}/{p^m}}\rfloor$.
  This implies that the right vertical arrow in ~\eqref{eqn:Kato-fil-1-1} is
  bijective and proves what we had asserted. 
\end{proof}

For the next three lemmas, we assume that $1 \le m \le r$ and consider the map
\[
\wt{\theta}_1 \colon A \otimes
{\underset{(q \ \mbox{times})}{A^\times \otimes \cdots \otimes A^\times}}
\to \wt{T}^{m,q}_{n}(K);
\]
\[
\wt{\theta}_1(x \otimes y_1 \otimes \cdots \otimes y_q)
 = \delta^{m,q}_n([x\pi^{-np^{1-m}}]_m\dlog([y_1]_m)\wedge \cdots \wedge \dlog([y_q]_m))
 \text{ mod } T^{m,q}_{n-1}(K).
 \]
 This is defined because $np^{1-m} = p^{r+1-m}l\in \N$ and
 $[x\pi^{-np^{1-m}}]_m \in \Fil_nW_m(K)$.

 \begin{lem}\label{lem:Kato-fil-2}
   $\wt{\theta}_1$ descends to a group homomorphism
   $\theta^q_1 \colon \Omega^q_{\ff} \to \wt{M}^{m,q}_n(K)$.
 \end{lem}
 \begin{proof}
For $\{y_1, \ldots , y_q\} \subset A^\times$, we write
   $\dlog([y_1]_m)\wedge \cdots \wedge \dlog([y_q]_m) = \dlog(\un{y})$. 
   If $x = a\pi$ for some $a \in A$, then
   $\delta^{m,q}_n([x]_m\dlog(\un{y})) =
   \delta^{m,q}_n([a\pi^{-(n-p^{m-1})p^{1-m}}]_m\dlog(\un{y})) \in T^{m,q}_{n-1}(K)$,
   as $[a\pi^{-(n-p^{m-1})p^{1-m}}]_m \in \Fil_{n-1}W_m(K)$.
Suppose next that $y_i = 1 + a\pi$ for some $a \in A$ and $1 \le i \le q$.
   By \cite[Lem.~3.5]{Kato-89}, we see that
   as an element of $\Fil_nW_m\Omega^1_K$, $[x]_m\dlog([1+a\pi]_m)$ actually lies
   in $\Fil_{n-1}W_m\Omega^1_K$. It follows that
   $[x]_m\dlog(\un{y}) \in \Fil_{n-1}W_m\Omega^q_K$ and hence
   $\delta^{m,q}_n([x]_m\dlog(\un{y}))$ dies in $\wt{T}^{m,q}_n(K)$.
We have thus shown that $\wt{\theta}_1$ descends a group homomorphism
\begin{equation}\label{eqn:Kato-fil-2-0}
  \wt{\theta}_1 \colon  \ff \otimes \ff^\times \otimes \cdots \otimes \ff^\times \to
  \wt{T}^{m,q}_n(K) \surj \wt{M}^{m,q}_n(K).
\end{equation}

We now let $\{x_1, \ldots , x_l\}$ and $\{x'_1, \ldots , x'_{l'}\}$ be two sets of
elements of
$A^\times$ such that $\sum_i x_i - \sum_j x'_j = a\pi$ for some $a \in A$.
Let $y_1, \ldots , y_{q-1} \in A^\times$. To show that $\wt{\theta}_1$ factors through
$\Omega^q_{\ff}$, it suffices to show, using \cite[Lem.~4.2]{Bloch-Kato}, that
$\wt{\theta}_1$ kills elements of the form
$(\sum_i x_i \otimes x_i- \sum_j x'_i \otimes x'_i) \otimes y_1 \otimes \cdots
\otimes y_{q-1}$.

To prove this claim, we let  $w_0 = [\pi^{-np^{1-m}}]_m\dlog[y_1]_m \cdots
\dlog[y_{q-1}]_m$. In ${T}^{m,q}_n(K)$, we then get the
following:
\[
w_0\left(\sum_i [x_i]_m\dlog[x_i]_m -
\sum_j [x'_i]_m\dlog[x'_j]_m\right) = w_0d\left(\sum_i[x_i]_m - \sum_j [x'_j]_m\right)
\]
\[
=w_0d\left([\sum_i x_i - \sum_jx'_i]_m+V(\un{b})\right) = w_0d[a\pi]_m +
{w_0}dV(\un{b})
\]
for some $\un{b}=(a_{m-2}, \ldots , a_1, a_0) \in W_{m-1}(A)$.

Since $w_0 \in \Fil_nW_m\Omega^{q-1}_K$, we get that $[a\pi]_mw_0 \in
\Fil_{n-1}W_m\Omega^{q-1}_K$, and hence $w_0d[a\pi]_m=[a\pi]_mw_0 \dlog[a\pi]_m \in
\Fil_{n-1}W_m\Omega^q_K$. So, it dies in $\wt{M}^{m,q}_n(K)$. On the other hand,
\[
w_0dV(\un{b})=dV(F(w_0)\un{b})
- V(Fd(w_0)\un{b}) \in (1-C)(\Fil_nW_m\Omega^{q}_K) +V(\Fil_nW_{m-1}\Omega^{q}_K)
\]
by Lemmas~\ref{lem:Log-fil-Wcom} and ~\ref{lem:VR-3} and hence also dies in
$\wt{M}^{m,q}_n(K)$. This proves the claim and completes the proof of the lemma.
\end{proof}

An identical proof shows that $\wt{\phi} \colon A \otimes A^\times \otimes \cdots
 \otimes A^\times \to \wt{T}^{r+1,q}_{n}(K)$, given by
 $\wt{\phi}(x \otimes y_1 \otimes \cdots \otimes y_q)
 = \delta^{m,q}_n([x\pi^{-l}]_m\dlog([y_1]_m)\wedge \cdots \wedge \dlog([y_q]_m))$
 mod $T^{r+1,q}_{n-1}(K)$, induces a homomorphism
\begin{equation}\label{eqn:Kato-fil-2-2}
   \phi^q\colon \Omega^q_{\ff} \to \wt{M}^{r+1,q}_n(K).
   \end{equation}

\begin{lem}\label{lem:Kato-fil-3}
   $\theta^q_1(Z_1\Omega^q_{\ff}) = 0$.
 \end{lem}
 \begin{proof}
   By \cite[\S~3, p.~111]{Kato-89}, $Z_1\Omega^q_{\ff}$ is generated by the elements of
   the types
   \[
  (1) \ x\dlog(x)\wedge \dlog(y_1) \wedge \cdots \wedge \dlog(y_{q-1}) \ \ \mbox{and} \ \
   (2) \ x^p\dlog(y_1) \wedge \cdots \wedge \dlog(y_{q}),
   \]
  where $x, y_j \in \ff^\star, 1 \le j \le q$.
 We now let $x, y_1, \ldots y_q \in A^\times, \ w = x^p\pi^{-np^{1-m}}$ and
 $\dlog(\un{y})_l=\dlog([y_1]_l) \wedge \cdots \wedge
   \dlog([y_q]_l)$. Since $m-1 < r$ and $p^r|n$, we can write
   $w = (x\pi^{-(n/p)p^{1-m}})^p$.

   Letting $w' = x\pi^{-(n/p)p^{1-m}}$, 
   we get by \lemref{lem:Complete-4}(1) that
   $$[w']_{m+1} \in \Fil_{n}W_{m+1}(K) \text{ and }
   [w]_m\dlog(\un{y})_m= F([w']_{m+1}\dlog(\un{y})_{m+1}) \in Z_1\Fil_nW_m\Omega^q_K.$$
   In particular,
   $\delta^{m,0}_n \circ (1-C)([w]_m\dlog(\un{y})_m) = 0$ in $T^{m,q}_n(K)$.
   From this, we get
   \[
   \delta^{m,q}_n([w]_m\dlog(\un{y})_m) = \delta^{m,q}_n
   \circ C([w]_m\dlog(\un{y})_m) = \delta^{m,q}_n([w']_m\dlog(\un{y})_m).
   \]
   Since the last term is in $\Fil_{n/p}W_m\Omega^{q}_K \subset \Fil_{n-1}W_m\Omega^q_K$,
   it follows that $\delta^{m,q}_n([w]_m\dlog(\un{y})_m)$ dies in $\wt{M}^{m,q}_n(K)$.
  This shows that the type (2) elements go to zero under $\theta_1^q$.

Next, we let $a = \pi^{-np^{1-m}}$ and $w_l'' = \dlog([y_1]_l) \wedge \cdots \wedge
   \dlog([y_{q-1}]_l)$ so that
   \[
     [xa]_m\dlog([x]_m)w_m'' = [a]_md([x]_m)w_m'' = d([xa]_mw_m'') - [x]_md([a]_m)w_m''.
     \]
   \lemref{lem:VR-3} implies $\delta^{m,q}_n([xa]_m\dlog([x]_m)w''_m) =
   -\delta^{m,q}_n([x]_md([a]_m)w''_m)$. On the other hand,
   \[
   \begin{array}{lll}
     [x]_md([a]_m)w''_m & = & [x]_mw''_m(-{n}/{p^{m-1}})[a]_m\dlog[\pi]_m \\
     & = & p(-{n}/{p^{m}}) [xa]_mw''_m\dlog[\pi]_m \\
     & = & (-{n}/{p^{m}}) VF([xa]_mw''_m\dlog[\pi]_m).
   \end{array}
   \]
Since
   $F([xa]_mw''\dlog[\pi]_m )\in \Fil_nW_{m-1}\Omega^q_K$,
   we conclude that $\delta^{m,q}_n([x]_md([a]_m)w'')$ dies in
   $\wt{M}^{m,q}_n(K)$. Hence, so does $\delta^{m,q}_n([xa]_m\dlog([x]_m)w'')$
   (image of the elements of type (1)).
   This completes the proof.
\end{proof}

We next consider the map

\[
\wt{\theta}_2 \colon A \otimes
{\underset{(q-1 \ \mbox{times})}{A^\times \otimes \cdots \otimes A^\times}}
\to \wt{T}^{m,q}_{n}(K);
\]
\[
\wt{\theta}_2(x \otimes y_1 \otimes \cdots \otimes y_{q-1})
 = \delta^{m,q}_n([x\pi^{-np^{1-m}}]_m
 \dlog([y_1]_m)\wedge \cdots \wedge \dlog([y_{q-1}]_m) \wedge \dlog([\pi]_m))
 \]
 \[
 \hspace{10cm}\text{ mod }
 T^{m,q}_{n-1}(K).
 \]

\begin{lem}\label{lem:Kato-fil-4}
  $\wt{\theta}_2$ descends to a group homomorphism
  $\theta^q_2 \colon \frac{\Omega^{q-1}_{\ff}}{Z_1\Omega^{q-1}_{\ff}} \to \wt{M}^{m,q}_n(K)$.
\end{lem}
\begin{proof}
By \lemref{lem:Log-fil-Wcom}, we see that multiplication by $\dlog([\pi]_m)$
  preserves $\Fil_nW_m\Omega^\bullet_K$. Since $F^{m-1}d(a\dlog([\pi]_m) =
  F^{m-1}d(a)\dlog([\pi]_m)$ and $C(a\dlog([\pi]_m)) = C(a)\dlog([\pi]_m)$ for
  $a \in Z_1W_\star\Omega^\bullet_K$,
  it follows that the map $W_m\Omega^{q-1}_K \xrightarrow{\dlog([\pi]_m)}
  W_m\Omega^q_K$ induces maps
  \[
  T^{m,q-1}_n(K) \xrightarrow{\dlog([\pi]_m)} T^{m,q}_n(K), \ \wt{T}^{m,q-1}_n(K)
  \xrightarrow{\dlog([\pi]_m)} \wt{T}^{m,q}_n(K).
  \]
Since $V(a\dlog([\pi]_m)) = V(a)\dlog([\pi]_m)$, we get that these maps
  also induce
  \[
  \wt{M}^{m,q-1}_n(K) \xrightarrow{\dlog([\pi]_m)} \wt{M}^{m,q}_n(K).
\]
Since ${\theta}^q_2 = \dlog([\pi]_m) \circ {\theta}^{q-1}_1 = {\theta}^{q}_1 \circ
\dlog([\pi]_m)$, we are done by Lemma~\ref{lem:Kato-fil-3}.
\end{proof}

Combining the previous three lemmas, we get a homomorphism
\begin{equation}\label{eqn:Kato-fil-5}
  (\theta^q_1, \theta^q_2) \colon \frac{\Omega^{q}_{\ff}}{Z_1\Omega^{q}_{\ff}}  \bigoplus
  \frac{\Omega^{q-1}_{\ff}}{Z_1\Omega^{q-1}_{\ff}} \to
  \wt{M}^{m,q}_n(K) \ \ {\rm for} \ \ 1 \le m \le r.
  \end{equation}

\begin{lem}\label{lem:Kato-fil-6}
  We have the following.
  \begin{enumerate}
    \item
      The map $\phi^q \colon \Omega^q_{\ff} \to \wt{M}^{r+1,q}_n(K)$ is surjective.
    \item
       For $1 \le m \le r$, the map $(\theta^q_1, \theta^q_2)$
       in ~\eqref{eqn:Kato-fil-5} is surjective.
       \end{enumerate}
\end{lem}
\begin{proof}
The map $\frac{\Fil_nW_{m}(K)}{\Fil_{n-1}W_{m}(K)} \otimes K^M_q(K) \to
  \wt{T}^{m,q}_n(K)$ is surjective by \corref{cor:VR-5}, and
  $\Fil_nW_{m}(K) = [\pi^{-\alpha}]_{m}W_{m}(A)$ by \lemref{lem:Log-fil-4}(1),
  where $\alpha = np^{1-m}$ if $1 \le m \le r$, and $\alpha = l$ if $m = r+1$.
  Using Lemmas~\ref{lem:VR-3} and ~\ref{lem:VR-4}, it only remains to show for
  proving (1) that we can replace $K^M_q(K)$ by $K^M_q(A)$ on the left hand side of
  this surjection.
To show the latter statement, it is enough to consider the case $q =1$ and to show that
  $[x\pi^{-l}]_{r+1}\dlog([\pi]_{r+1})$ lies in the image of $A \otimes A^\times$
under $\phi^1$ if $x \in A^\times$. But we know that
$[x\pi^{-l}]_{r+1}\dlog([\pi]_{r+1}) =
  -l^{-1}[x]_{r+1}d([\pi^{-l}]_{r+1})$. On the other hand,
  $[x]_{r+1}d([\pi^{-l}]_{r+1}) = d([x\pi^{-l}]_{r+1}) - [x\pi^{-l}]_{r+1}\dlog([x]_{r+1})$.
  It follows from this that $\delta^{r+1,1}_n([x\pi^{-l}]_{r+1}\dlog([\pi]_{r+1}))$ \\
  $= \phi^1(l^{-1}x \otimes x)$.

To prove (2), we can repeat the above steps which reduces finally to showing that
  $\delta^{m,1}_n([x\pi^{-\alpha}]_m\dlog([y]_m))$ lies in the image of
$(\theta^1_1, \theta^1_2)$, where $y \in K^\times$. To show this, we write
$y = u\pi^{\beta}$ for some $u \in A^\times$ and $\beta \in \Z$. This yields
  $[x\pi^{-\alpha}]_m\dlog([y]_m) = [x\pi^{-\alpha}]_m\dlog([u]_m) +
  \beta[x\pi^{-\alpha}]_m\dlog([\pi]_m) = \theta^1_1(x \otimes u) + \theta^1_2(\beta x)$.
  This completes the proof.
\end{proof}

\subsection{The main theorem}\label{sec:Kato-fil-M}
  We shall now prove the main result of this section which extends
 \cite[Thm.~3.2]{Kato-89} to $H^{q+1}(K)\{p\}$ for all $q \ge 1$.
We shall assume until \thmref{thm:Kato-fil-10} that $n \ge 1$.
 We look at the commutative diagram
 \begin{equation}\label{eqn:Kato-6}
   \xymatrix@C1pc{
     T^{m,q}_{n-1}(K) \ar[r] \ar[d]_-{\ov{\delta}^{m,q}_{n-1}} & T^{m,q}_{n}(K)
     \ar[d]^-{\ov{\delta}^{m,q}_{n}} \\
     S^{m,q}_{n-1}(K) \ar@{^{(}->}[r] & S^{m,q}_n(K).}
 \end{equation}
 We let $\wt{S}^{m,q}_n(K) = \frac{S^{m,q}_{n}(K)}{S^{m,q}_{n-1}(K)}$ and let
$\wt{N}^{m,q}_n(K) = \frac{\wt{S}^{m,q}_n(K)}{\wt{S}^{m-1,q}_n(K)}$. We let 
$\psi^{m,q}_n \colon \wt{T}^{m,q}_n(K) \to \wt{S}^{m,q}_n(K)$ be the map induced
by $\{\ov{\delta}^{m,q}_{n}\}_{n \ge 0}$. It is clear that $\ov{\delta}^{m,q}_n$
 descends to a homomorphism $\ov{\delta}^{m,q}_n \colon \wt{M}^{m,q}_n(K) \to
 \wt{N}^{m,q}_n(K)$.
\begin{lem}\label{lem:Kato-fil-7}
  For $m \ge 1$, the map
  $\ov{\delta}^{m,q}_n \colon \wt{M}^{m,q}_n(K) \to \wt{N}^{m,q}_n(K)$ is bijective.
  \end{lem}
\begin{proof}
  When $m \ge r+2$, we are done by \lemref{lem:Kato-fil-1} and
  \cite[Lem.~3.6.1]{Kato-89} as both groups are zero. Otherwise, we
look at the maps
  \[
  \Omega^q_{\ff} \xrightarrow{\phi^q} \wt{M}^{r+1,q}_n(K)
  \xrightarrow{\ov{\delta}^{r+1,q}_n}
  \wt{N}^{r+1,q}_n(K)
  \]
  and (when $1 \le m \le r$)
  \[
  \frac{\Omega^{q}_{\ff}}{Z_1\Omega^{q}_{\ff}}  \bigoplus
  \frac{\Omega^{q-1}_{\ff}}{Z_1\Omega^{q-1}_{\ff}} \xrightarrow{\theta^q_1, \theta^q_2}
  \wt{M}^{m,q}_n(K) \xrightarrow{\ov{\delta}^{m,q}_n} \wt{N}^{m,q}_n(K).
  \]
  Kato shows in \cite[\S~3]{Kato-89} (see the last step in proof of Theorem~3.2
  on p.~114) that the above two composite maps are bijective. 
  We now apply \lemref{lem:Kato-fil-6} to finish the proof.
\end{proof}

\begin{lem}\label{lem:Kato-fil-8}
  For $m \ge 1$, the square ~\eqref{eqn:Kato-6} is Cartesian.
\end{lem}
\begin{proof}
  Since $T^{m,q}_{n-1}(K) \to T^{m,q}_n(K)$ is injective
  (cf. diagram ~\eqref{eqn:Kato-6}), it is sufficient to show that the map
  $\ov{\delta}^{m,q}_n \colon \wt{T}^{m,q}_n(K) \to \wt{S}^{m,q}_n(K)$ is injective.
  We shall prove this by induction on $m$.
  When $m =1$, the map $\wt{T}^{m,q}_n(K) \to \wt{M}^{m,q}_n(K)$ is bijective, and
  we are done by \lemref{lem:Kato-fil-7}.
We now assume $m \ge 2$ and
  look at the commutative diagram
  \begin{equation}\label{eqn:Kato-fil-8-0}
    \xymatrix@C1pc{
      & \wt{T}^{m-1,q}_n(K) \ar[r]^-{V} \ar[d]_-{\ov{\delta}^{m-1,q}_n} &
      \wt{T}^{m,q}_n(K) \ar[r] \ar[d]^-{\ov{\delta}^{m,q}_n} &  \wt{M}^{m,q}_n(K)
      \ar[d]^-{\ov{\delta}^{m,q}_n} \ar[r] & 0 \\
      0 \ar[r] & \wt{S}^{m-1,q}_n(K) \ar[r] &
      \wt{S}^{m,q}_n(K) \ar[r] &  \wt{N}^{m,q}_n(K) \ar[r] & 0.}
    \end{equation}

It is clear from the definition of the groups ${S}^{m,q}_n(K)$ that
  $S^{m-1,q}_n(K) \bigcap S^{m,q}_{n-1}(K) = S^{m-1,q}_{n-1}(K)$, and this implies that
  the bottom row of ~\eqref{eqn:Kato-fil-8-0} is exact.
  The left vertical arrow in this diagram is injective by induction, and the
  right vertical arrow is bijective by \lemref{lem:Kato-fil-7}.
  It follows that the middle vertical arrow is injective.
\end{proof}

\enlargethispage{20pt}

We can now finally prove the following main result of \S~\ref{sec:Kato-comp}.
Recall that $X = \Spec(A)$ and $D_n = V((\pi^n))$. 

\begin{thm}\label{thm:Kato-fil-10}
  For $m \ge 1$ and $n,q \ge 0$, the map
  \[
  \ov{\delta}^{m,q}_n \colon T^{m,q}_n(K) \to \Fil^{\bk}_n H^{q+1}(K)[p^m]
  \]
  is an isomorphism, where $T^{m,q}_n(K)=\H^1_\et(X, W_m\sF^{q,\bullet}_{X|{D_n}})$.
\end{thm}
\begin{proof}
By \corref{cor:VR-5}, we only need to show that $\ov{\delta}^{m,q}_n$ is surjective.
  For this, we let $w \in \Fil^{\bk}_n H^{q+1}(K)[p^m] = S^{m,q}_n(K)$.
  As $H^{q+1}(K)[p^m] = H^1_\et(K, W_m\Omega^q_{K, \log})$
  by \lemref{lem:Inj-GL}, we have that $w \in H^1_\et(K, W_m\Omega^q_{K, \log})$.
  Since $\delta^{m,q} \colon W_m\Omega^q_K \surj H^1_\et(K, W_m\Omega^q_{K, \log})$, it
  follows by \lemref{lem:Log-fil-4}(4)
 that ${\varinjlim}_n T^{m,q}_n(K)$ $= H^1_\et(K, W_m\Omega^q_{K, \log})$.
In particular, $w \in T^{m,q}_{n'}(K) \bigcap S^{m,q}_n(K)$ for all $n' \gg 0$.
We conclude from \lemref{lem:Kato-fil-8} that $w \in  T^{m,q}_{n}(K)$.
\end{proof}

The following easy consequence of \thmref{thm:Kato-fil-10} will be used in the
next section.

\begin{cor}\label{cor:Kato-fil-11}
  For $m \ge 1$ and $n,q \ge 0$, the canonical map
  \[
  \frac{H^{q+1}_{p^m}(K)}{\Fil^{\bk}_n H^{q+1}(K)[p^m]} \to 
\frac{H^{q+1}_{p^m}(\wh{K})}{\Fil^{\bk}_n H^{q+1}(\wh{K})[p^m]} 
\]
is bijective.
\end{cor}
\begin{proof}
  By \thmref{thm:Kato-fil-10}, the term on the left hand side is the {\'e}tale
  cohomology with support $\H^2_{\ff}(X, W_m\sF^{q,\bullet}_{X|{D_n}})$ at the closed
  point of $X$. Similarly, the term on the right hand side is
  $\H^2_{\ff}(\wh{X}, W_m\sF^{q,\bullet}_{{\wh{X}}|{D_n}})$, where $\wh{X} = \Spec(\wh{A})$.
  By comparing the long exact cohomology sequences associated to the complexes
  $W_m\sF^{q,\bullet}_{X|{D_n}}$ and $W_m\sF^{q,\bullet}_{{\wh{X}}|{D_n}}$, we see meanwhile that
  the map $\H^2_{\ff}(X, W_m\sF^{q,\bullet}_{X|{D_n}}) \to
  \H^2_{\ff}(\wh{X}, W_m\sF^{q,\bullet}_{{\wh{X}}|{D_n}})$ is bijective.
\end{proof}

\section{Global Kato filtration via filtered de Rham--Witt}\label{sec:Kato-global}
The goal of this section is to present Kato's ramification filtration
for regular $F$-finite $\F_p$-schemes in terms of the hypercohomology of the complexes
$W_m\sF^{q,\bullet}_{X|D}$. This will extend a result of Kerz--Saito
(cf. \cite[\S~3]{Kerz-Saito-ANT})
to $H^{q+1}(U)\{p\}$ for all $q \ge 0$. The main result of this section has several
applications. Some of these will be presented in the following sections of
this text.

We let $k = \F_p$ and let $X$ be a Noetherian regular $F$-finite $k$-scheme. 
Let $E$ be an snc divisor on  $X$
with irreducible components $E_1, \ldots , E_r$.  Let $j \colon U \inj X$ be the
inclusion of the complement of $E$ in $X$. Let $\eta_i$ denote the generic
point of $E_i$ and let $K_i$ (resp. $\wh{K}_i$) be the quotient field of
$\sO^h_{X, \eta_i}$ (resp. $\wh{\sO_{X, \eta_i}}$). Let
$h_i \colon \Spec({\sO^h_{X, \eta_i}}) \to X$ denote the canonical map.
We fix integers $q \ge 0$ and $m \ge 1$. Let $D = \sum_i n_iE_i \in \Div_E(X)$
with each $n_i \ge 0$.
In view of \lemref{lem:Inj-GL}, we shall make no distinction between
$H^1_\et(L,W_m\Omega^q_{L,\log})$ and $H^{q+1}(L)[p^m]$ in our exposition if
$L$ is any $F$-finite henselian discrete valuation field containing $k$.
Recall the complex of sheaves 
\[
W_m\sF^{q,\bullet}_{X|D} =\left[Z_1\Fil_{D} W_m\Omega^q_U \xrightarrow{1 -C}
  \Fil_{D}W_m\Omega^q_U\right]
\]
on the {\'e}tale site of $X$. Recall also that
$j^*W_m\sF^{q,\bullet}_{X|D} \cong W_m\Omega^q_{U,\log}$.

\subsection{Global Kato filtration}\label{sec:GKF}
For $1 \le i \le r$, let $j_i \colon H^1_\et(U,W_m\Omega^q_{U,\log}) = H^{q+1}_{p^m}(U)
\to H^{q+1}_{p^m}(K_i)$ be the canonical map. 
We recall the global Kato filtration on $H^{q+1}_{p^m}(U)$ below.
\begin{defn}(cf. \cite[\S~7, 8]{Kato-89}, \cite[Defn.~2.7]{Kerz-Saito-ANT})
  \label{defn:Log-fil-D} 
  We let $n \ge 1$ be an integer and write $n = p^mr$ where $m \ge 0$ and
$p \nmid r$. We let
  \[
  \Fil^\log_DH^{q+1}_{p^m}(U) :=
  \Ker \left(H^{q+1}_{p^m}(U) \xrightarrow{\bigoplus j_i} \
  \stackrel{r}{\underset{i =1}\bigoplus} 
  \frac{H^{q+1}_{p^m}(K_i)}{\Fil^{\bk}_{n_i} H^{q+1}(K_i)[p^m]}\right);
  \]
  \[
  \Fil^\log_DH^{q+1}_n(U) := \Fil^\log_DH^{q+1}_{p^m}(U) \bigoplus H^{q+1}_{r}(U); \ \ \
  \Fil^\log_DH^{q+1}(U) := \varinjlim_n  \Fil^\log_DH^{q+1}_{n}(U).
  \]
 \end{defn}
By \corref{cor:Kato-fil-11}, we can replace each $K_i$ by $\wh{K}_i$ in the
definition of $ \Fil^\log_DH^{q+1}_{p^m}(U)$.
Note also that each quotient on the right hand side of the definition of
$\Fil^\log_DH^{q+1}_{p^m}(U)$ makes sense by \lemref{lem:Inj-GL}.
It is clear that $\Fil^\log_DH^{q+1}(U) =
\Ker \left(H^{q+1}(U)\{p\} \xrightarrow{\bigoplus j_i} \
  \stackrel{r}{\underset{i =1}\bigoplus} 
  \frac{H^{q+1}(K_i)}{\Fil^{\bk}_{n_i} H^{q+1}(K_i)}\right)$ \
  $\bigoplus H^{q+1}(U)\{p'\}$. It is also easily checked that the canonical map
  $\varinjlim_D \Fil^\log_DH^{q+1}_{p^m}(U) \to  H^{q+1}_{p^m}(U)$ is bijective.

\begin{exm}\label{exm:Examples-fil}
  (1) When $q=0$, we have 
$\Fil^\log_DH^{1}_{p^m}(U) \xrightarrow{\cong} \Hom_\cont(\pi_1^{\ab}(X,D)/p^m, \Q/\Z)$,
where $\pi_1^{\ab}(X,D)$ is the abelianized {\'e}tale fundamental group with modulus
(\cite[Defn.~2.7]{Kerz-Saito-ANT}, \cite[Defn.~2.9]{GK-Nis}) which is the
abelianization of the automorphism group of the fiber
functor of the Galois category of finite étale covers of $U$ whose
ramifications are bounded at each generic point of $E$ by means of Kato’s Swan
conductor.

\vskip .2cm

(2) When $q =1$, there is a canonical exact sequence
\begin{equation}\label{eqn:fundamental-0}
  0 \to {\Pic(U)}/{p^m} \to \Fil^\log_DH^{2}_{p^m}(U) \to \Fil^\log_D \Br(U)[p^m] \to 0,
\end{equation}
where $\Fil^\log_D \Br(U)$ is the Brauer group with modulus (cf. \cite[Defn.~8.7]{KRS},
where it was denoted by $\Br^{\divv}(X|D)$).
It is defined as the subgroup of $\Br(U)$ consisting of elements $\chi$
  such that for every $1 \le i \le r$, the image of
$\chi$ under the canonical map $\Br(U) \to\Br(K_i) = H^2(K_i)$ 
lies in $\Fil^{\bk}_{n_i} H^2(K_i)$.
The above exact sequence is a direct consequence of the sheaf exact sequence
\[
0 \to \sO^{\times}_U \xrightarrow{p^m}  \sO^{\times}_U \xrightarrow{\dlog}
W_m\Omega^1_{U, \log} \to 0.
\]
\end{exm}

By \cite[Lem.~3.6]{Kerz-Saito-ANT},  there is a canonical isomorphism
\[
\Fil^\log_DH^{q+1}_{p^m}(U) \cong \H^{1}_\et(X, W_m\sF^{q,\bullet}_{X|D})
\text{ if }q=0.
\]
Our goal in this section is to extend this isomorphism to arbitrary $q \ge 0$.
Applications of such an isomorphism will be given in the following sections, and
in \cite{KM-3}, \cite{KM-2} and \cite{KM-1}.

\enlargethispage{20pt}

\subsection{Cartier operator on local cohomology}\label{sec:Cartier-loc}
The proof of our main theorem will be based on a local result that we shall prove
in this subsection. We begin with the following.

\begin{lem}\label{lem:Z-1-fil}
  Let $\ff$ be an $F$-finite field of characteristic $p$.
  Let $R = \ff[[Y_1,Y_2]], \ K=R[(Y_1Y_2)^{-1}]$ and $n, m \in \Z$. Let's write
  \[
  F^{1,q}_0 = F^{1,q}_0(R)=\Omega^q_{\ff} \oplus  \Omega^{q-1}_{\ff} \dlog Y_1 \oplus
  \Omega^{q-1}_{\ff} \dlog Y_2 \oplus  \Omega^{q-2}_{\ff} \dlog Y_1 \dlog Y_2;
  \]
  \[
  Z_1F^{1,q}_0=Z_1\Omega^q_{\ff} \oplus  (Z_1\Omega^{q-1}_{\ff}) \dlog Y_1 \oplus
  (Z_1\Omega^{q-1}_{\ff}) \dlog Y_2 \oplus ( Z_1\Omega^{q-2}_{\ff}) \dlog Y_1 \dlog Y_2.
  \]
    Then $Z_1\Fil_{(n,m)}\Omega^q_K$ has the following unique presentation.
    \begin{equation}\label{eqn:Z-1-fil-0}
    Z_1\Fil_{(n,m)}\Omega^q_K = \sum\limits_{\substack{ip \ge -n \\ jp \ge -m}}Y_1^{ip}Y_2^{jp}
    \left( Z_1F^{1,q}_0 \right)  \ +
    \sum\limits_{\substack{i \ge -n \\  j  \ge -m \\ p \nmid i
        \text{ or } p \nmid j}} d(Y_1^iY_2^j F^{1, {q-1}}_0).
    \end{equation}

Moreover, one has
    \begin{enumerate}
    \item $(\Fil_{(n,m)}\Omega^q_K)[(Y_1Y_2)^{-1}]= \Omega^q_K$ and
      $(Z_1\Fil_{(n,m)}\Omega^q_K)[(Y_1Y_2)^{-1}]= Z_1\Omega^q_K$, where we consider
      $Z_1\Omega^q_K$ and $Z_1\Fil_{(n,m)}\Omega^q_K$ as $R$-modules via the
      Frobenius action.
    \item Let $\omega \in \Omega^q_K$. Then
    \begin{enumerate}
    \item  $\omega \in (\Fil_{(n,m)}\Omega^q_K)[(Y_1)^{-1}]$
      iff it can be written uniquely as
     $\omega = \sum\limits_{\substack{i\ge -N_0 \\ j\ge -m}} Y_1^iY_2^j a_{i,j}$,
      for some $N_0 >0$ and $a_{i,j} \in F^{1,q}_0$.

    \item  $\omega \in (\Fil_{(n,m)}\Omega^q_K)[(Y_2)^{-1}]$ iff it can be written
      uniquely as
    $\omega = \sum\limits_{\substack{i\ge -n \\ j\ge -M_0}} Y_1^iY_2^j b_{i,j}$,
    for some $M_0 >0$ and $b_{i,j} \in F^{1,q}_0$.
    \item  $\omega \in (Z_1\Fil_{(n,m)}\Omega^q_K)[(Y_1)^{-1}]$ iff it can be written
      uniquely as
      \[
      \omega = \sum\limits_{\substack{i\ge -N'_0 \\ jp\ge -m}} Y_1^{ip}Y_2^{jp} e_{i,j} +
      \sum\limits_{\substack{i\ge -N'_0 \\ j\ge -m \\
          p\nmid i \text{ or } p \nmid j}}d(Y_1^iY_2^jf_{i,j}),
      \]
      for some $N'_0 >0$ and
$e_{i,j} \in Z_1F^{1,q}_0, \  f_{i,j} \in F^{1,{q-1}}_0.$

    \item $\omega \in (Z_1\Fil_{(n,m)}\Omega^q_K)[(Y_2)^{-1}]$ iff it can be written
      uniquely as
      \[
      \omega = \sum\limits_{\substack{ip\ge -n \\ j\ge -M'_0}} Y_1^{ip}Y_2^{jp} e'_{i,j} +
      \sum\limits_{\substack{i\ge -n \\ j\ge -M'_0 \\
          p\nmid i \text{ or } p \nmid j}}d(Y_1^iY_2^jf'_{i,j}),
      \]
      for some $M'_0 >0$ and  $e'_{i,j} \in Z_1F^{1,q}_0, \  f'_{i,j} \in
      F^{1, {q-1}}_0.$
    \end{enumerate}        
    
\end{enumerate}
    \end{lem}
\begin{proof}
  It is clear from \corref{cor:F-function-2} and \lemref{lem:Non-complete-4} that
  the right hand side of ~\eqref{eqn:Z-1-fil-0} is contained in
  $Z_1\Fil_{(n,m)}\Omega^q_K$.
  To prove the reverse inclusion, we let $\omega \in Z_1\Fil_{(n,m)}\Omega^q_K$.
  \lemref{lem:Complete-4}(2) for $m=1$ and \corref{cor:F-function-2} then say that
  there exists an element $\omega_1 =
  \sum\limits_{\substack{ip \ge -n \\ jp \ge -m}}Y_1^{ip}Y_2^{jp} e_{i,j}$ with
  $e_{i,j} \in Z_1F^{1,q}_0$ such that $\omega - \omega_1 \in
  d(\Fil_{(n,m)} \Omega^{q-1}_K)$. By \corref{cor:F-function-2} and
  \lemref{lem:Non-complete-4}, the latter group is
  of the form $\sum\limits_{\substack{i \ge -n \\ j \ge -m}} d(Y_1^iY_2^j F^{1, {q-1}}_0)$.
  We can however write this term as
  \[
  \begin{array}{lll}
  \sum\limits_{\substack{i \ge -n \\ j \ge -m}} d(Y_1^iY_2^j F^{1, {q-1}}_0) & = &
  \sum\limits_{\substack{i \ge -n \\  j  \ge -m \\ p \nmid i
      \text{ or } p \nmid j}} d(Y_1^iY_2^j F^{1, {q-1}}_0) \ + \
  \sum\limits_{\substack{i \ge -n \\  j  \ge -m \\ p \mid i,j}}
  d(Y_1^iY_2^j F^{1, {q-1}}_0) \\
  & \subset & \sum\limits_{\substack{i \ge -n \\  j  \ge -m \\ p \nmid i
      \text{ or } p \nmid j}} d(Y_1^iY_2^j F^{1, {q-1}}_0) \ + \
  \sum\limits_{\substack{ip \ge -n \\ jp \ge -m}}Y_1^{ip}Y_2^{jp}
  \left( Z_1F^{1,q}_0 \right).
  \end{array}
    \]
    It follows that $\omega = \omega'_1 + \omega_2$, where $\omega'_1$ lies in the
    first and $\omega_2$ lies in the second sum on the right hand side of
    ~\eqref{eqn:Z-1-fil-0}. The uniqueness of this presentation is clear from
    \corref{cor:F-function-2}.

We now prove the other claims. Item (1) is clear because $\Fil_{(n,m)}\Omega^q_K$ and
    $Z_1\Fil_{(n,m)}\Omega^q_K$ are $R$-modules. Items (2.a) and (2.b) follow directly
    from \corref{cor:F-function-2} using again that $\Fil_{(n,m)}\Omega^q_K$ is 
    an $R$-module. Items (2.c) and (2.d) follow 
    from ~\eqref{eqn:Z-1-fil-0} once we note
    that the $R$-module structure of $Z_1\Fil_{(n,m)}\Omega^q_K$ is via the
    canonical map $R \to R^p$ (given by $a \mapsto a^p$).
   \end{proof} 

\enlargethispage{20pt}

\begin{lem}\label{lem:1-C-inj}
    Let $R$ be as above and $\fm =(Y_1,Y_2)$. Assume $m, n \ge 0$. Then the map
    $Z_1\Fil_{(n,m)}\Omega^q_K \xrightarrow{1 -C} \Fil_{(n,m)}\Omega^q_K$ induces a
    monomorphism
    \[
    (1-C)^* \colon H^2_{\fm}(R, Z_1\Fil_{(n,m)}\Omega^q_K) \inj
    H^2_{\fm}(R, \Fil_{(n,m)}\Omega^q_K),
    \]
    where $H^*_\fm(R,-)$ is the Zariski
    (equivalently, {\'e}tale) cohomology with support at the closed point of
    $\Spec(R)$.
\end{lem}
\begin{proof}
For any $R$-module $M$ and
  non-zero element $a \in R$, we write $M[a^{-1}]$ as $M_a$. We let $\pi = Y_1Y_2$.
  By \cite[Thm.~5.1.20]{Brodman-Sharp}, there is then an exact sequence
  \begin{equation}\label{eqn:1-C-inj-0}
    M_{Y_1} \oplus M_{Y_2} \xrightarrow{\rho} M_{\pi} \xrightarrow{\phi} H^2_\fm(M) \to 0,
  \end{equation}
  where $\rho ((a,b)) = a -b$.
       Now consider the following commutative diagram.
        \begin{equation}\label{eqn:1-C-diag}
            \xymatrix{
              (Z_1\Fil_{(n,m)}\Omega^q_K)_{Y_1} \oplus
              (Z_1\Fil_{(n,m)}\Omega^q_K)_{Y_2} \ar[r]^-{\theta_1}
              \ar[d]_-{f\oplus g} & (Z_1\Fil_{(n,m)}\Omega^q_K)_{\pi} =
              Z_1\Omega^q_K \ar[d]^-{h} \\
              (\Fil_{(n,m)}\Omega^q_K)_{Y_1} \oplus (\Fil_{(n,m)}\Omega^q_K)_{Y_2}
              \ar[r]^-{\theta_2}& (\Fil_{(n,m)}\Omega^q_K)_{\pi} = \Omega^q_K, 
            }
        \end{equation}
        where each $\theta_i$ is the difference map $\rho$ of \eqref{eqn:1-C-inj-0},
        and $f,g,h$ are the maps induced by $1-C$. To show $(1-C)^*$ is injective, we
        need to show for any $\alpha \in (Z_1\Fil_{(n,m)}\Omega^q_K)_{\pi}$ that
        $h(\alpha) \in {\rm Image}(\theta_2)$ implies
        $\alpha \in {\rm Image}(\theta_1)$.

       We let $\alpha \in Z_1\Omega^q_K$
        such that $h(\alpha) \in {\rm Image}(\theta_2)$. Since $Z_1\Omega^q_K=\bigcup_{n \ge 0, m \ge 0}Z_1\Fil_{(n,m)}\Omega^q_K$, there exists $N, M \gg 0$ such that $\alpha \in Z_1\Fil_{(N,M)}\Omega^q_K$.
By \lemref{lem:Z-1-fil} (see \eqref{eqn:Z-1-fil-0}), we can write 
 
\begin{equation}\label{eqn:sum-alpha-0}
  \alpha = \sum\limits_{\substack{ip \ge -N \\ jp \ge -M}}Y_1^{ip}Y_2^{jp}c_{i,j} \ +
  \sum\limits_{\substack{i \ge -N \\ j \ge -M \\ p \nmid i \ or \ p \nmid j}} d(Y_1^iY_2^jd_{i,j}),
\end{equation}
where $c_{i,j} \in Z_1F^{1,q}_0$ and $d_{i,j} \in F^{1,{q-1}}_0$.

We may assume that $N \ge n$ and $M \ge m$. Then we can write $\alpha=\alpha_1+\alpha_2+\alpha_3$, where 
\[\alpha_1=\sum\limits_{\substack{ip \ge -n \\ jp \ge -M}}Y_1^{ip}Y_2^{jp}c_{i,j} \ +
  \sum\limits_{\substack{i \ge -n \\ j \ge -M \\ p \nmid i \ or \ p \nmid j}} d(Y_1^iY_2^jd_{i,j}), \alpha_2= \sum\limits_{\substack{-N \le ip < -n \\ jp \ge -m}}Y_1^{ip}Y_2^{jp}c_{i,j} \ +
  \sum\limits_{\substack{-N \le i < -n \\ j \ge -m \\ p \nmid i \ or \ p \nmid j}} d(Y_1^iY_2^jd_{i,j}),\]

\begin{equation}\label{eqn:alpha-3}
\begin{array}{lll}
 \alpha_3 &=& \sum\limits_{\substack{-N < ip < -n \\ -M < jp < -m}}Y_1^{ip}Y_2^{jp}c_{i,j} \ +
  \sum\limits_{\substack{-N < i < -n \\ -N < j < -m \\ p \nmid i \ or \ p \nmid j}} d(Y_1^iY_2^jd_{i,j})\\
 &=& \sum\limits_{\substack{-N < ip < -n \\ -M < jp < -m}}Y_1^{ip}Y_2^{jp}c_{i,j} \ +
  \sum\limits_{\substack{-N < i< -n \\ -N < j < -m \\ p \nmid i \ or \ p \nmid j}} Y_1^iY_2^jd'_{i,j}, 
\end{array}
\end{equation}
where we set $d'_{i,j}:=i\cdot \dlog Y_1\cdot d_{i,j} +j\cdot \dlog Y_2 \cdot d_{i,j}+d(d_{i,j})$ so that $d(Y_1^iY_2^jd_{i,j})=Y_1^iY_2^jd'_{i,j}$.

By \lemref{lem:Z-1-fil}(2.c),(2.d), we conclude that $\alpha_1$
(resp. $\alpha_2$) lies in
$(Z_1\Fil_{(n,m)}\Omega^q_K)_{Y_2}$ (resp. $(Z_1\Fil_{(n,m)}\Omega^q_K)_{Y_1}$). It follows that $\alpha_1+\alpha_2 \in {\rm Image}(\theta_1)$.

We claim that $\alpha_3=0$. This will finish the proof.

Since $h(\alpha) \in {\rm Image}(\theta_2)$, it follows from the commutativity of (\ref{eqn:1-C-diag}) that $h(\alpha_3)= h(\alpha) - h(\alpha_1+\alpha_2) \in
{\rm Image}(\theta_2)$. We shall show that this implies $\alpha_3=0$.
Indeed, suppose that $-N_0$ is the least exponent of $Y_1$ appearing in $\alpha_3$ in the last equality of \eqref{eqn:alpha-3}. Then we may write $\alpha_3$ as
\[
\alpha_3= \sum\limits_{-M < j<-m}Y_1^{-N_0}Y_2^je_{-N_0,j} +
\sum\limits_{\substack{-N_0 <i < -n \\ -M < j < -m}}Y_1^iY_2^je_{i,j},
\]

where $Y_1^iY_2^je_{i,j}$ is either of the form $d(Y_1^iY_2^jd_{i,j})$
(if $p\nmid i$ or $p \nmid j$), or it is of the form $Y_1^{pi}Y_2^{pi} c_{i,j}$
(in case $p$ divides both $i, j$), where $c_{i,j} \in Z_1F^{1,q-1}_0$. Thus we have
$$ C(Y_1^iY_2^je_{i,j})  = \left\{\begin{array}{cl}
   0 ,& \text{if } p\nmid i \text{ or } p\nmid j  \\
       Y_1^{i/p}Y_2^{j/p}C(e_{i,j}) , & \text{otherwise.}
\end{array}\right.$$
That is, $C(Y_1^iY_2^je_{i,j})$ is either zero or the exponent of $Y_1$ in
$C(Y_1^iY_2^je_{i,j})$ is strictly greater than $i$ (because $i<0$).
This implies that $(1-C)(c)$ can be written as 
\[
(1-C)(\alpha_3)=\sum\limits_{-M<j<-m}Y_1^{-N_0}Y_2^je_{-N_0,j} +
\sum\limits_{\substack{i>-N_0 \\ -M<j<-m}}Y_1^iY_2^je'_{i,j},
\]
where $e_{i,j}, e'_{i,j} \in F^{1,q}_0$.

Since $(1-C)(\alpha_3) \in {\rm Image}(\theta_2)$, it implies that every term of the above sum has the property that either the exponent of $Y_1$ is at least $-n$ or the exponent of $Y_2$ is at least $-m$. By the uniqueness statement in \lemref{lem:Z-1-fil}, we therefore conclude that $Y_1^{-N_0}Y_2^je_{-N_0,j}=0$ for all  $j<-m$. This shows that $-N_0$ is not the least exponent of $Y_1$ appearing in $\alpha_3$, which leads to a contradiction. This proves the claim and concludes the proof of the lemma.
\end{proof}

Let $R$ be an excellent $F$-finite regular local ring of dimension two
with maximal ideal $\fm = (Y_1, Y_2)$. Let
$L = R[(Y_1)^{-1}]$ and $K =  R[(Y_1Y_2)^{-1}]$. We let
$S = R/{(Y_2)}$ and let $F$ be the quotient field of $S$.
We let $W = \Spec(R)$ and $E = \Spec(S)$.
We let  $D = nY_1 + mY_2, \ D' = nY_1$ and $E' = D'|_{E}$.
We let $\Fil_{(n,m)}\Omega^q_K = \Fil_{D}\Omega^q_K, \
\Fil_n\Omega^q_L = \Fil_{D'}\Omega^q_L$ and $\Fil_n\Omega^q_{F} =
\Fil_{E'}\Omega^q_{F}$. We let
$\sF^q_{n,m}(R) = W_1\sF^{q, \bullet}_{W|D}, \ \sF^{q}_{n}(R) = W_1\sF^{q, \bullet}_{W|D'}$
and $\sF^q_{n}(S) = W_1\sF^{q, \bullet}_{E|E'}$.

\begin{lem}\label{lem:2-C-inj}
  For $n \ge 0$, there exists a short exact sequence of complexes of sheaves
  \begin{equation}\label{eqn:2-C-inj-0}
    0 \to \sF^{q}_{n}(R) \to \sF^q_{n,0}(R) \xrightarrow{\res} \sF^{q-1}_{n}(S) \to 0.
  \end{equation}
\end{lem}
\begin{proof}
By  Lemmas~\ref{lem:Log-fil-4}(3) and ~\ref{lem:Kahler-seq},
we have an exact sequence
\begin{equation}\label{eqn:2-C-inj-1}
0 \to \Fil_n\Omega^q_L \to \Fil_{(n,0)}\Omega^q_K \xrightarrow{\res} 
\Fil_n\Omega^{q-1}_{F} \to 0,
\end{equation}
where $\res$ is the Poincar{\'e} residue map (cf. proof of \lemref{lem:Kahler-seq}).
To prove that ~\eqref{eqn:2-C-inj-1} gives rise to ~\eqref{eqn:2-C-inj-0} which is
exact, we need to show the following.
\begin{enumerate}
\item
  $\res \circ d = d \circ \res$ on $\Fil_{(n,0)}\Omega^q_K$ and
  $\res \circ C = C \circ \res$ on $Z_1\Fil_{(n,0)}\Omega^q_K$.
\item
  The map 
  $\res \colon Z_1 \Fil_{(n,0)}\Omega^q_K  \to Z_1\Fil_n\Omega^{q-1}_{F}$, induced by
  the commutativity of $\res$ with $d$, is surjective.
\end{enumerate}

To prove item (1), we first note using ~\eqref{eqn:2-C-inj-1} and the
surjection $\Fil_n\Omega^{q-1}_L \surj \Fil_n\Omega^{q-1}_{F}$ that we can write any
element of $\Fil_{(n,0)}\Omega^q_K$ as $\omega = \omega_1 + \omega_2\dlog Y_2$,
where $\omega_1 \in \Fil_n\Omega^q_L$ and $\omega_2 \in \Fil_n\Omega^{q-1}_L$
with $\res(\omega_1) = 0$. The identity $(\res \circ d)(\omega) =
(d \circ \res)(\omega)$ is obvious.
To prove $\res \circ C = C \circ \res$ on $Z_1\Fil_{(n,0)}\Omega^q_K$,
we use the same argument, namely, using item (2), we can write any
element of $Z_1\Fil_{(n,0)}\Omega^q_K$ as $\omega = \omega_1 + \omega_2\dlog Y_2$,
where $\omega_1 \in Z_1\Fil_n\Omega^q_L$ and $\omega_2 \in Z_1\Fil_n\Omega^{q-1}_L$
with $\res(\omega_1) = 0$. It is straightforward
    to check that $(C \circ \res)(\omega) = (\res \circ C)(\omega)$.

    To prove item (2), let
    $\omega \in Z_1\Fil_n\Omega^{q-1}_{F}$. By  \corref{cor:fil-D-Y-2}, we
    can choose a lift $\wt{\omega}$ of $\omega$
    in $Z_1\Fil_n\Omega^{q-1}_L$. One can now check that
    $\wt{\omega}\dlog Y_2 \in Z_1\Fil_{(n,0)}\Omega^q_K$ and
    $\res(\wt{\omega}\dlog Y_2) = \omega$. This completes the proof. 
\end{proof}

\subsection{The main result}\label{sec:Main-KS}
We return to our original set-up of \S~\ref{sec:Kato-global}. 
We need the following results to prove the main theorem.
For any $x \in X$, we let $R_x = \sO_{X,x}^h, \ X_x = \Spec(R_x)$. We let $E_x$ denote
the pull-back of $E$ to $X_x$ and $U_x = X_x \setminus E_x$.
We let $\fm_x$ denote the maximal ideal of $R_x$, and let
$\wh{R}_x$ be the $\fm_x$-adic completion of $R_x$. We let $\wh{X}_x = \Spec(\wh{R}_x)$.
  We let $\wh{E}_x$ denote the pull-back of $E$ to $\wh{X}_x$ and $\wh{U}_x =
  \wh{X}_x \setminus \wh{E}_x$. We write  $\sF_m= W_m\sF^{q,\bullet}_{X|D}$ as
  abbreviation.

\begin{lem}\label{lem:H^1-fil-1}
        $\H^1_E(X, \sF_m) =0$.
\end{lem}
\begin{proof}
We first claim that
  $\H^1_x(X, \sF_m)=0$ for every $x \in  X^{(1)}\bigcap E$.
To prove the claim, we use the exact sequence
\[
\cdots \to \H^0_\et(X_x,\sF_m) \xrightarrow{f} \H^0_\et(U_x,\sF_m) \to
      \H^1_x(X_x,\sF_m) \to \H^1_\et(X_x,\sF_m) \xrightarrow{g} 
    \H^1_\et(U_x,\sF_m).
    \]
    By \lemref{lem:hyp}(2), we see that $f$ is an isomorphism.
    Also $g$ is injective by \propref{prop:Cartier-fil-1} because
    $\H^1_\et(X_x,\sF_m) =\coker(1-C \colon Z_1\Fil_{D_x}W_m\Omega^q_{K_x} \to
    \Fil_{D_x}W_m\Omega^q_{K_x} )$ by \lemref{lem:hyp}(2). Hence, we get
    $\H^1_x(X,\sF_m) \xrightarrow{\cong} \H^1_x(X_x,\sF_m) = 0$. This proves the claim.

    We now recall that for any bounded complex of abelian sheaves $\sF^\bullet$ on
    $X_\et$,
    one has the coniveau spectral sequence 
        \begin{equation}\label{eqn:spectral}
          E_1^{p,q}(\sF^\bullet) =
          \bigoplus\limits_{x \in X^{(p)} \cap E}\H^{p+q}_x(X,\sF^\bullet) \implies
          \H^{p+q}_E(X,\sF^\bullet).
        \end{equation}
The above claim shows that $E^{1,0}_\infty(\sF_m) =0$. Since
        $E^{0,1}_\infty(\sF_m)$ is clearly zero, the lemma follows.
    \end{proof}

\begin{lem}\label{lem:H^1-fil-2}
  Let $\sF_m$ be as in \lemref{lem:H^1-fil-1}. Then the restriction map
  $\H^2_E(X,\sF_m) \to \bigoplus\limits_{x \in X^{(1)}\cap E} \H^2_{x}(X, \sF_m)$ is
  injective.
\end{lem}
\begin{proof}
  We begin by claiming that $\H^2_x(X,\sF_m)=0$ for $x \in X^{(2)}\bigcap E$.
  To prove this claim, we first reduce to proving it when $m =1$ (and $D \ge 0$
  arbitrary) using \thmref{thm:Global-version}(9).

  We first assume that $x$ meets two distinct components of $E$.
  We look at the hypercohomology long exact
  sequence 
  \begin{equation}\label{eqn:H^1-fil-2-0}
  H^1_x(X_x, \Fil_D\Omega^q_{U_x}) \to \H^2_x(X,\sF_1) \to
  H^2_x(X_x, Z_1\Fil_D\Omega^q_{U_x}) \xrightarrow{(1-C)^*}
  H^2_x(X_x, \Fil_D\Omega^q_{U_x}).
  \end{equation}
  Since $x \in X^{(2)}$ and $\Fil_D\Omega^q_{U_x}$ is a free $R_x$-module, one knows that
  $H^1_x(X_x, \Fil_D\Omega^q_{U_x})=0$ (cf. \cite[Prop.~3.5.4.b]{Bruns-Herzog}). 
  To prove the claim, it suffices therefore to show that $(1-C)^*$ in
  ~\eqref{eqn:H^1-fil-2-0} is injective. 

To prove this injectivity,  we first deduce from Lemma~\ref{lem:Non-complete-3*}
  and a combination of \corref{cor:Non-complete-3} and \lemref{lem:Complete-4}
  that $\Fil_D\Omega^q_{\wh{U}_x}$ (resp. $Z_1\Fil_D\Omega^q_{\wh{U}_x}$) is the $\fm_x$-adic
  completion of the finitely generated (cf. \lemref{lem:Log-fil-4}(2)) $R_x$-module
  $\Fil_D\Omega^q_{U_x}$ (resp. $Z_1\Fil_D\Omega^q_{U_x}$).
 We now apply \cite[Prop~3.5.4.(d)]{Bruns-Herzog} to conclude that
  the vertical arrows in the commutative diagram
  \begin{equation}
            \xymatrix{
              H^2_x(X_x, Z_1\Fil_D\Omega^q_{U_x}) \ar[r]^{\ (1-C)^*}\ar[d]_{\cong}    &
              H^2_x(X_x, \Fil_D\Omega^q_{U_x}) \ar[d]^\cong\\
              H^2_x(\wh{X}_x, Z_1\Fil_D\Omega^q_{\wh U_x}) \ar[r]^{\ (1-C)^*}    &
              H^2_x(\wh{X}_x, \Fil_D\Omega^q_{\wh U_x})}
        \end{equation}
        are bijective. We are therefore reduced to showing that the bottom
        horizontal arrow in this diagram is injective. But this follows from
        \lemref{lem:1-C-inj} because we can write $\wh{R}_x = k(x)[[Y_1, Y_2]], \
  E = V(Y_1) + V(Y_2)$ and  $D = V((Y^n_1 Y^m_2))$ (where $m,n \ge 0$). 

        We next assume that $x$ lies on only one irreducible component of $E$,
        say, on $E_1$. By \lemref{lem:2-C-inj}, we have an exact sequence
        \[
        \H^1_x(X', W_1\Omega^{q-1, \bullet}_{X'|D'_1}) \to
        \H^2_x(X, W_1\Omega^{q, \bullet}_{X|D_1}) \to \H^2_x(X, W_1\Omega^{q, \bullet}_{X|D_2}),
            \]
            where $X'$ is a divisor through $x$ such that
            locally at $x$, the generators of the ideals of $E_1$ and  $X'$ form
            the maximal ideal of $\sO_{X,x}$, $D_1 = n_1E_1, \ D_2 = n_1E_1 + 0 X'$
            and $D'_1 = D_1|_{X'}$.
            The first term in the above exact sequence is zero by
            \lemref{lem:H^1-fil-1} while the third term is zero by the case
            we considered above. It follows that
            $\H^2_x(X, \sF_1) = \H^2_x(X, W_1\Omega^{q, \bullet}_{X|D_1}) =0$.
            This proves the claim.
            
        The claim implies that $E^{2,0}_\infty(\sF_m) =0$ in ~\eqref{eqn:spectral}.
        Since $E^{0,2}_\infty(\sF_m)$ is clearly $0$, we get
        $\H^2_E(X, \sF_m)= E^{1,1}_\infty(\sF_m)$. One also checks that
        $E^{1,1}_\infty(\sF_m) =
        E^{1,1}_3(\sF_m) \inj E^{1,1}_2(\sF_m) \inj E^{1,1}_1(\sF_m) =
        \bigoplus\limits_{x \in X^{(1)}\cap E} \H^2_{x}(X, \sF_m)$.
        This concludes the proof.   
  \end{proof}

We now prove the main theorem of \S~\ref{sec:Kato-global}. This will prove
\thmref{thm:Main-1}.

\begin{thm}\label{thm:H^1-fil}
For $q \ge 0$ and $m \ge 1$, there is a canonical isomorphism
 \[
 \H^1_\et(X, W_m\sF^{q,\bullet}_{X|D}) \xrightarrow{\cong}
 \Fil^\log_DH^{q+1}_{p^m}(U).
 \]
\end{thm}
\begin{proof}
We shall use the exact sequence 
\begin{equation}\label{eqn:H^1-fil-0}
\H^1_E(X, \sF_m) \to \H^1_\et( X, \sF_m) \to H^1_\et(U, W_m\Omega^q_{U,\log}) \to
\H^2_E(X, \sF_m),
\end{equation}
in which $\H^1_\et(U, \sF_m)$ has been identified with $H^1_\et(U, W_m\Omega^q_{U,\log})$
using \thmref{thm:Global-version}(13).
The first term in this exact sequence is zero by \lemref{lem:H^1-fil-1}.
We now let $x \in E^{(0)}$ and consider the commutative diagram of exact sequences 
\begin{equation}\label{eqn:loc-2}
    \xymatrix@C.8pc{
      0\ar[r]& \H^1_\et( X, \sF_m) \ar[r] \ar[d]&
      H^1_\et(U, W_m\Omega^q_{U,\log}) \ar[r] \ar[d] &
      \H^2_E(X, \sF_m) \ar[d] & \\
      &\H^1_\et(X_x, \sF_m) \ar[r]^-{(1)} &
      H^1_\et(U_x, W_m\Omega^q_{U_x,\log}) \ar[r]^-{(2)}&
      \H^2_x(X_x, \sF_m) &}
\end{equation}
where the vertical arrows are the restriction maps.

The arrow (2) in the bottom row is surjective because the next term
$\H^2(X_x, \sF_m)$ is zero by \lemref{lem:hyp}. On the other hand, 
the arrow (1) in the bottom row of
\eqref{eqn:loc-2} is identified with the inclusion
$\H^1_\et(X_x, \sF_m) \cong \Fil^{\bk}_{n_x}H^{q+1}(K_x)[p^m] \inj
H^{q+1}_{p^m}(K_x)$ by \thmref{thm:Kato-fil-10} if we let $n_x$ denote the multiplicity of
$D$ at $x$. It follows that the composition of the maps
$H^1_\et(U, W_m\Omega^q_{U,\log}) \to \H^2_E(X, \sF_m) \to \H^2_x(X_x, \sF_m)$ coincides
with the map $H^{q+1}_{p^m}(U) \to \frac{H^{q+1}_{p^m}(K_x)}
{\Fil^{\bk}_{n_x}H^{q+1}(K_x)[p^m]}$, induced by the pull-back via the canonical map
$X_x \to X$.

 Hence, we get 
\begin{equation*}
    \begin{array}{cl}
      \H^1_\et( X, \sF_m)&\cong \Ker \left(H^1_\et(U, W_m\Omega^q_{U,\log}) \to
      \H^2_E(X, \sF_m)\right)    \\
      &\cong  \Ker \left(H^1_\et(U, W_m\Omega^q_{U,\log})  \to
      \bigoplus\limits_{x \in X^{(1)}\cap E} \H^2_{x}(X, \sF_m)\right) \\
      &= \Ker \left(H^{q+1}_{p^m}(U) \to
      \bigoplus\limits_{x \in X^{(1)}\cap E}
      \frac{H^{q+1}_{p^m}(K_x)}{\Fil^{\bk}_{n_x}H^{q+1}(K_x)[p^m]}\right) \\
      & =  \Fil^\log_DH^{q+1}_{p^m}(U),
\end{array}
\end{equation*}
where the second isomorphism follows from \lemref{lem:H^1-fil-2}.
This finishes the proof of the theorem.
\end{proof}

An immediate, yet important, consequence of \thmref{thm:H^1-fil} is the following
functoriality property of the Kato filtration.
This property is not obvious from Definition~\ref{defn:Log-fil-D}.

\begin{cor}\label{cor:Fil-functor}
  Let $f\colon (X',E') \to (X,E)$ be a morphism between two snc-pairs (cf.
  \S~\ref{sec:de-log-0}) and let
  $D \in \Div_E(X)$ be an effective divisor. Then the pull-back
  $f^* \colon H^{q+1}_{p^m}(X \setminus E) \to H^{q+1}_{p^m}(X' \setminus E')$ restricts to a
  map $f^* \colon \Fil^{\log}_D H^{q+1}_{p^m}(X \setminus E) \to
  \Fil^{\log}_{D'} H^{q+1}_{p^m}(X' \setminus E')$ for every $q \ge 0, m \ge 1$
  whenever $D' \ge f^*(D)$ in $\Div_{E'}(X')$.
\end{cor}
\begin{proof}
  It is clear from \thmref{thm:Main-0} and the definitions of
  $Z_1\Fil_D W_m\Omega^q_U$ and the Cartier operator that
  there is a morphism of complexes $f^* \colon W_m\sF^{q, \bullet}_{X|D} \to
  f_*(W_m\sF^{q, \bullet}_{{X'}|{D'}})$, and this induces the map
  $f^* \colon \H^1_\et(X, W_m\sF^{q, \bullet}_{X|D}) \to 
  \H^1_\et(X', W_m\sF^{q, \bullet}_{{X'}|{D'}})$. We now apply \thmref{thm:H^1-fil} to
  conclude the proof.
  \end{proof}

Using \thmref{thm:H^1-fil} in Example~\ref{exm:Examples-fil}(2), we get the
following cohomological description of the Brauer group with modulus.

\begin{cor}\label{cor:F^q}
For $m \ge 1$, there is a canonical exact sequence
    \begin{equation}
      0 \to \Pic(U)/p^m \to \H^1_\et(X, W_m\sF^{1,\bullet}_{X|D}) \to
      \Fil^{\log}_D\Br(U)[p^m] \to 0.
    \end{equation}
\end{cor}

If $n \ge 1$ is any integer and we write $n = p^mr$ where $m \ge 1$ and
$p \nmid r$. We let
${\Z}/n(q)_{X|D} := {\Z}/r(q) \bigoplus W_m\sF^{q-1, \bullet}_{X|D}[q-1]$
for any $q \ge 1$. We then get

\begin{cor}\label{cor:F^q-00}
For any $n, q \ge 1$, there is a canonical isomorphism
\[
\H^q_\et(X, {\Z}/n(q-1)_{X|D}) \xrightarrow{\cong}
 \Fil^\log_DH^{q}_n(U).
 \]
\end{cor}

\section{Refined Swan conductor on schemes}\label{sec:RSC}
In this section, we give our first application of \thmref{thm:H^1-fil}.
The context of this application is the following. If $A$ is an $F$-finite
henselian discrete
valuation ring with maximal ideal $(\pi)$, quotient field $K$ and residue field $\ff$
of characteristic $p > 0$, then Kato constructed a refined Swan
conductor map
\begin{equation}\label{eqn:Kato-SC}
  {\rsw}^{q+1}_{K,n} \colon \gr^{\bk}_nH^{q+1}(K) \to
  \pi^{-n} \Omega^{q+1}_A(\log \pi)\otimes_A \ff
  \end{equation}
for $n \ge 1, q \ge 0$ and showed that it is injective. This is one of the
main results of \cite{Kato-89}. If $X$ is, more generally,
a regular  finite type $\F_p$-scheme with an 
snc divisor $E$ and if $D \in \Div_E(X)$ is effective, then
Kato--Leal--Saito \cite[\S~4]{KLS} showed that Kato's Swan conductor at the
generic points of $E$ glue together to give rise to a refined Swan conductor map
${\Rsw} \colon \frac{\Fil^{\log}_{D} H^{1}(X)}{\Fil^{\log}_{D/p}H^{1}(X)} \to
H^0_\zar(X, \Omega^1_X(\log E) \otimes_{\sO_X} \frac{\sO_X(D)}{\sO_{X(D/p)}})$ which is
injective.

We shall construct a general version of this map using \thmref{thm:H^1-fil}.
The generalized refined Swan conductor map that we construct has many applications.
For instance, Kato's integrality theorem
(cf. \cite[Thm.~7.1]{Kato-89}) follows immediately from this construction.
The generalized refined Swan conductor map also plays the central role
in the proofs of the main results of \cite{KM-2}. We shall use the following definition
from \cite{Kato-89}.

\begin{defn}\label{defn:Swan-conductor}
  Let $X$ be a Noetherian regular $F$-finite $\F_p$-scheme and $U \subset X$ a
  dense open subscheme. We let $E \subset X$ be an irreducible Weil divisor
  disjoint from $U$. Given $q \ge 0$ and $\chi \in H^{q+1}(U)\{p\}$, we let $\Sw_E(\chi)$
  denote Kato's Swan conductor (cf. \cite[Defn.~2.3]{Kato-89})
  of the image of $\chi$ in $H^{q+1}(K_E)$,
  where $K_E$ is the henselization of the total quotient ring of $X$ at $E$.
  \end{defn}

\subsection{The generalized refined Swan conductor}\label{sec:Comp-H1}
Let $k = \F_p$. Let $X$ be a connected Noetherian regular $F$-finite $k$-scheme and let
$E$ be an snc divisor on  $X$ with irreducible components
$E_1, \ldots , E_r$. Let $D \ge D' \ge D/p \ge 0$ be two divisors in $\Div_E(X)$
(note that $|D/p| \subset |D| \subseteq E$).
We let $j \colon U \inj X$ be the inclusion of the complement of $E$.
For any pair of integers $q \ge 0$ and $m \ge 1$, we then have a short exact
sequence of complexes of {\'e}tale sheaves
\begin{equation}\label{eqn:Comp*-0}
  0 \to W_m\sF^{q,\bullet}_{X|D'} \to W_m\sF^{q,\bullet}_{X|D} \to
  \frac{W_m\sF^{q,\bullet}_{X|D}}{W_m\sF^{q,\bullet}_{X|D'}} \to 0
  \end{equation}
on $X$, where the first map is the canonical inclusion.
Letting $\sF^{m,q}_{X|(D, D')} :=
\frac{W_m\sF^{q,\bullet}_{X|D}}{W_m\sF^{q,\bullet}_{X|D'}}$, we get an injection
\begin{equation}\label{eqn:Comp*-1}
  \alpha_{X|(D, D')} \colon
  \frac{\H^1_\et(X, W_m\sF^{q,\bullet}_{X|D})}{\H^1_\et(X, W_m\sF^{q,\bullet}_{X|D'})}
  \inj  \H^1_\et(X, \sF^{m,q}_{X|(D, D')}).
  \end{equation}

The complex $\sF^{m,q}_{X|(D, D')}$ has the form
\begin{equation}\label{eqn:Comp*-2}
  \sF^{m,q}_{X|(D, D')} = \left(\frac{Z_1\Fil_D W_m\Omega^q_U}{Z_1\Fil_{D'} W_m\Omega^q_U}
  \xrightarrow{1 - \ov{C}} \frac{\Fil_D W_m\Omega^q_U}{\Fil_{D'} W_m\Omega^q_U}\right).
  \end{equation}
From our assumption and the definition of $C$ (cf. ~\eqref{eqn:Fil-D-complex}),
we see that $\ov{C} = 0$ and we can write
\begin{equation}\label{eqn:Comp*-3}
  \sF^{m,q}_{X|(D, D')} = \left(\frac{Z_1\Fil_D W_m\Omega^q_U}{Z_1\Fil_{D'} W_m\Omega^q_U}
  \xrightarrow{\can} \frac{\Fil_D W_m\Omega^q_U}{\Fil_{D'} W_m\Omega^q_U}\right),
  \end{equation}
where `$\can$' is the canonical map induced by the inclusion
$Z_1\Fil_D W_m\Omega^q_U \inj \Fil_D W_m\Omega^q_U$.
From the definition of $Z_1\Fil_D W_m\Omega^q_U$, we see that `$\can$' is injective.
In particular,
\begin{equation}\label{eqn:Comp*-4}
  \sF^{m,q}_{X|(D, D')} \xrightarrow{\simeq}
  \left(\frac{\Fil_D W_m\Omega^q_U}{\Fil_{D'} W_m\Omega^q_U +
    Z_1\Fil_D W_m\Omega^q_U}\right)[-1].
\end{equation}
We therefore get an injection
\begin{equation}\label{eqn:Comp*-5}
\alpha^{m,q}_{X|(D, D')} \colon
  \frac{\H^1_\et(X, W_m\sF^{q,\bullet}_{X|D})}{\H^1_\et(X, W_m\sF^{q,\bullet}_{X|D'})}
  \inj  H^0_\zar\left(X, \frac{\Fil_D W_m\Omega^q_U}{\Fil_{D'} W_m\Omega^q_U +
    Z_1\Fil_D W_m\Omega^q_U}\right).
  \end{equation}

By \thmref{thm:Global-version}(3), $F^{m-1}d$ induces a map
\begin{equation}\label{eqn:Comp*-6}
 F^{m-1}d \colon  \frac{\Fil_D W_m\Omega^q_U}{\Fil_{D'} W_m\Omega^q_U +
    Z_1\Fil_D W_m\Omega^q_U} \to \frac{\Fil_D \Omega^{q+1}_U}{\Fil_{D'} \Omega^{q+1}_U}.
\end{equation}

\begin{lem}\label{lem:Comp*-7}
  The above map is injective.
\end{lem}
\begin{proof}
  We assume that $X = \Spec(A)$, and let $\alpha$ be a global section of
  $\Fil_D W_m\Omega^q_U$ such that $\alpha' := F^{m-1}d(\alpha) \in
  \Fil_{D'} \Omega^{q+1}_U$. By \thmref{thm:Global-version}(8), we get that
  $\alpha' \in \Ker(V^{m-1}\colon \Fil_D \Omega^{q+1}_U \to \Fil_D W_m\Omega^{q+1}_U)
  \bigcap  \Fil_{D'} \Omega^{q+1}_U$. Equivalently,
  $\alpha' \in \Ker(V^{m-1} \colon  \Fil_{D'} \Omega^{q+1}_U \to  \Fil_{D'}
  W_m\Omega^{q+1}_U)$.
  By \thmref{thm:Global-version}(8) again, we get that
  $F^{m-1}d(\alpha) = F^{m-1}d(\beta)$ for some $\beta \in \Fil_{D'} W_m\Omega^{q}_U
  \subseteq \Fil_{D} W_m\Omega^{q}_U$. But this implies that
  $\alpha = \beta + \gamma$ for some $\gamma \in  Z_1\Fil_D W_m\Omega^q_U$.
  Since the assertion of the lemma is local, this finishes the proof.
\end{proof}

Letting $\wt{\alpha}^{m,q}_{X|(D, D')} = (-1)^{q+1} (F^{m-1}d) \circ \alpha^{m,q}_{X|(D, D')}$,
we get the following.

\begin{cor}\label{cor:Comp*-8}
  There exists a canonical injection
  \[
\wt{\alpha}^{m,q}_{X|(D, D')} \colon   
\frac{\H^1_\et(X, W_m\sF^{q,\bullet}_{X|D})}{\H^1_\et(X, W_m\sF^{q,\bullet}_{X|D'})}
\inj  H^0_\zar\left(X, \frac{\Fil_D \Omega^{q+1}_U}{\Fil_{D'} \Omega^{q+1}_U}\right).
\]
\end{cor}
Since the operator $V$ commutes with each of $F,C$ and $d$, it follows that 
$V \circ {\alpha}^{m,q}_{X|(D, D')} = {\alpha}^{m+1,q}_{X|(D, D')} \circ V$ and
$\wt{\alpha}^{m,q}_{X|(D, D')} = \wt{\alpha}^{m+1,q}_{X|(D, D')} \circ V$.
Letting $F = D - D'$ and using \lemref{lem:Log-fil-4} and \thmref{thm:H^1-fil}, the
above corollary implies the following.

\vskip.2cm

The main result of this section is the following.

\begin{thm}\label{thm:RSW-gen}
$(1)$ There exists a (strict) map of ind-abelian groups 
  \begin{equation}\label{eqn:RSW-gen-0}
  \left\{ \frac{\Fil^\log_DH^{q+1}_{p^m}(U)}{\Fil^\log_{D'}H^{q+1}_{p^m}(U)}\right\}_{m \ge 1}
      \xrightarrow{\Rsw^{m,q+1}_{X|(D,D')}} 
H^0_\zar\left(F, \Omega^{q+1}_X(\log E)(D) \otimes_{\sO_X} \sO_F\right)
  \end{equation}
  which is a level-wise monomorphism.
  
  $(2)$ Let $f \colon (X',E') \to (X,E)$ be a morphism between snc-pairs with
  $U' = X'\setminus E'$, then letting $F = D - D/p$ and $F' = f^*(D) - {f^*(D)}/p$,
  we have a commutative diagram of ind-abelian groups
\begin{equation}\label{eqn:RSW-gen-1-0}
  \xymatrix@C1pc{
    \left\{ \frac{\Fil^\log_DH^{q+1}_{p^m}(U)}{\Fil^\log_{D/p}H^{q+1}_{p^m}(U)}\right\}_{m \ge 1}
    \ar[r] \ar[d]_-{f^*} &
    H^0_\zar\left(F, \Omega^{q+1}_X(\log E)(D) \otimes_{\sO_X} \sO_{F}\right)
    \ar[d]^-{f^*} \\
    \left\{ \frac{\Fil^\log_{f^*(D)}H^{q+1}_{p^m}(U')}
            {\Fil^\log_{{f^*(D)}/p}H^{q+1}_{p^m}(U')}\right\}_{m \ge 1} \ar[r] &
H^0_\zar\left(F', \Omega^{q+1}_{X'}(\log E')(f^*(D)) \otimes_{\sO_{X'}} \sO_{F'}\right),}
\end{equation}
where the horizontal arrows are the refined Swan conductor maps of
~\eqref{eqn:RSW-gen-0}.

$(3)$ If $f^*(E)$ is reduced, then we can replace $D/p$ (resp. ${f^*(D)}/p$) in
~\eqref{eqn:RSW-gen-1-0} by any $D'$ (resp. $f^*(D')$) such that
$D \ge D' \ge D/p \ge 0$.
\end{thm}

\begin{remk}\label{remk:RSW-gen-3}
As mentioned earlier, if  $X$ is of finite type over $\F_p$ and $q = 0$, part (2) of
\thmref{thm:RSW-gen} was shown by Kato--Leal--Saito in \cite[Thms.~4.2.2, 4.3.1]{KLS}.
\end{remk}

If $X=\Spec(A)$ and $E$ is irreducible, then we shall write $\Rsw^{m,q+1}_{X|(D,D')}$ as $\Rsw^{m,q+1}_{A|(n,n')}$ if $D=nE$ and $D'=n'E$.

Suppose now that $D$ and $D'$ are as in \thmref{thm:RSW-gen}, and let $\chi \in \Fil_DH^{q+1}_{p^m}(U)$. Let $1 \le i \le r$ and let $n_i$ (resp. $n_i'$) be the multiplicity of $D$ (resp. $D'$). Then we have the following straightforward corollary
of \thmref{thm:RSW-gen}(3) (cf. \cite[Rem.~5.2.1]{Matsuda}).

\begin{cor}\label{cor:RSW-gen-2}
  The element $\Rsw^{m,q+1}_{X|(D,D')}(\chi) \in
  H^0_\zar\left(F, \Omega^{q+1}_{X}(\log E)(D) \otimes_{\sO_{X}} \sO_{F}\right)$
  has the property that its stalk at $\eta_i$ is $\Rsw^{m,q+1}_{\sO_{X,\eta_i}|(n,n')}(\chi)$, where $\eta_i$ is the generic point of $E_i$. 
\end{cor}

\subsection{A special case}\label{sec:Special-case}
We now consider a special case of \thmref{thm:RSW-gen}.  We fix an integer
$i \in \{1, \ldots , r\}$, let $D = \stackrel{r}{\underset{j =1}\sum} n_j E_j$
with $n_i \ge 1$ and let $D' = D - E_i$.
For $j \neq i$, we let $\ov{E}_j$ denote the
scheme-theoretic intersection $E_i \bigcap E_j$. We let $E_o = E - E_i$ and let
$F_i$ denote the scheme-theoretic intersection $E_o \cap E_i$. We let
$D_i = \iota^*_i(D)$ and ${D}'_i = \iota^*_i(D')$, where
$\iota_i \colon E_i \inj X$ is the inclusion. 
Using \lemref{lem:Log-fil-4}, \thmref{thm:RSW-gen} and \corref{cor:Fil-functor}, we
get the following.

\begin{cor}\label{cor:RSW-gen-1}
  There exists a monomorphism of ind-abelian groups
  \begin{equation}\label{eqn:RSW-gen-1-2}
  \Rsw^{\star,q+1}_{X|(D,D')} \colon
  \left\{ \frac{\Fil^\log_DH^{q+1}_{p^m}(U)}{\Fil^\log_{D'}H^{q+1}_{p^m}(U)}\right\}_{m \ge 1}
  \to H^0_\zar\left(E_i, \Omega^{q+1}_X(\log E)(D) \otimes_{\sO_X} \sO_{E_i}\right).
  \end{equation}
  If $f \colon (X',E') \to (X,E)$ is a morphism between snc-pairs such that
  $E'_i := f^*(E_i)$ is reduced and $U' = X'\setminus E'$,
  then there is a commutative diagram of ind-abelian groups
\begin{equation}\label{eqn:RSW-gen-1-1}
  \xymatrix@C1pc{
    \left\{ \frac{\Fil^\log_DH^{q+1}_{p^m}(U)}{\Fil^\log_{D'}H^{q+1}_{p^m}(U)}\right\}_{m \ge 1}
    \ar[r] \ar[d]_-{f^*} &
    H^0_\zar\left(E_i, \Omega^{q+1}_X(\log E)(D) \otimes_{\sO_X} \sO_{E_i}\right)
    \ar[d]^-{f^*} \\
    \left\{ \frac{\Fil^\log_{f^*(D)}H^{q+1}_{p^m}(U')}
            {\Fil^\log_{f^*(D')}H^{q+1}_{p^m}(U')}\right\}_{m \ge 1} \ar[r] &
H^0_\zar\left(E'_i, \Omega^{q+1}_{X'}(\log E')(f^*(D)) \otimes_{\sO_{X'}} \sO_{E'_i}\right),}
\end{equation}
where the horizontal arrows are the refined Swan conductor maps of
~\eqref{eqn:RSW-gen-1-2}.
\end{cor}

One can present $\frac{\Fil_D \Omega^{q+1}_U}{\Fil_{D'} \Omega^{q+1}_U}$ in a simplified
form in the special case of this subsection which is helpful in computing the
refined Swan conductor.
We need the next lemma for this purpose.
We let $(X,E)$ be the snc-pair as above and let $D  \in \Div_E(X)$. Let
$u \colon X' \inj X$ be the inclusion of a regular connected divisor. We assume that
$E' := X' \times_X E$ is an snc divisor on  $X'$.

The following is a further simplification of the target of the refined
Swan conductor in the special case of \S~\ref{sec:Special-case}.

\begin{lem}\label{lem:RSW-spl}
 Under the set-up of \corref{cor:RSW-gen-1}, there exists an exact sequence
  \[
  0 \to \Omega^{q+1}_{E_i}(\log F_i)({D}_i) \xrightarrow{\nu'_i}
  \frac{\Fil_D \Omega^{q+1}_U}{\Fil_{D'} \Omega^{q+1}_U} \xrightarrow{\nu_i}  
  \Omega^{q}_{E_i}(\log F_i)({D}_i) \to 0.
  \]
\end{lem}
\begin{proof}
  We look at the commutative diagram
\begin{equation}\label{eqn:RSW-spl-0}
  \xymatrix@C1pc{
    & 0 \ar[d] & 0 \ar[d] & & \\
      & \Omega^{q+1}_{X}(\log E)(D') \ar@{=}[r] \ar[d]_-{\wt{\alpha}} &
      \Omega^{q+1}_{X}(\log E)(D') \ar[d]^-{\alpha} & & \\
      0 \ar[r] & \Omega^{q+1}_{X}(\log E_o)(D) \ar[r] \ar[d] & \Omega^{q+1}_{X}(\log E)(D)
      \ar[r]^-{\res} \ar[d] & \Omega^{q}_{E_i}(\log F_i)({D}_i) \ar[r] \ar@{=}
      [d] & 0 \\
      0 \ar[r] &  \Omega^{q+1}_{E_i}(\log F_i)({D}_i) \ar[r] \ar[d] &
      \coker(\alpha) \ar[r] \ar[d] & \Omega^{q}_{E_i}(\log F_i)({D}_i) \ar[r] &
      0, \\
      & 0 & 0 & &}
\end{equation}
where the middle row is the residue exact sequence and $\alpha$ is the canonical
inclusion whose image lies in $\Omega^{q+1}_{X}(\log E_o)(D)$, as one easily checks.
We let $\wt{\alpha}$ denote the resulting factorization. The left column is exact
by \lemref{lem:Kahler-seq}(2). It follows that the bottom row is also exact, and
this finishes the proof as $\coker(\alpha) =
 \frac{\Fil_D \Omega^{q+1}_U}{\Fil_{D'} \Omega^{q+1}_U}$ by \lemref{lem:Log-fil-4}(3).
\end{proof}

The following variant of \lemref{lem:RSW-spl} will be useful 
in \cite{KM-2}, and its proof is immediate if we replace
$D'$ by $D -2E_i$ in ~\eqref{eqn:RSW-spl-0}.

\begin{cor}\label{cor:RSW-spl-0-0}
  Assume that $n_i \ge 2$ and let $D'' = D -2E_i$. Then
  we have an exact sequence
  \[
 0 \to \frac{\Omega^{q+1}_X(\log E_o)(D)}{\Fil_{D''}\Omega^{q+1}_{U}}
    \xrightarrow{\psi'_i}
  \frac{\Fil_{D}\Omega^{q+1}_{U}}{\Fil_{D''}\Omega^{q+1}_{U}} \xrightarrow{\psi_i}
  \Omega^q_{E_i}(\log F_i)(D_i) \to 0,
  \]
  where $\psi_i$ is the composite map $ \frac{\Fil_{D}\Omega^{q+1}_{U}}{\Fil_{D''}\Omega^{q+1}_{U}} \surj  \frac{\Fil_{D}\Omega^{q+1}_{U}}{\Fil_{D'}\Omega^{q+1}_{U}}\xrightarrow{\nu_i}
  \Omega^q_{E_i}(\log F_i)(D_i)$.
  \end{cor}
  
\vskip.2cm

\begin{remk}\label{remk:Artin-conductor}
  Using \corref{cor:RSW-spl-0-0}, it is straightforward to
  define a non-log version of the filtration $\Fil^\log_DH^{q+1}_{p^m}(U)$
  (cf. \cite{KM-2}) which generalizes Matsuda's filtration for $q =0$
  (cf. \cite{Matsuda}) and derive its functoriality.
    One would like to construct an analogue of the generalized refined Swan conductor
    of \corref{cor:RSW-gen-1} for the non-log version. To our
    knowledge, this problem is yet unsolved and we shall pursue it elsewhere.
    \end{remk}

\vskip.2cm

\subsection{Agreement with Kato's Swan conductor}\label{sec:Agreement-Kato}
Suppose now in \corref{cor:RSW-gen-1} that $X = \Spec(A)$, where $A$ is a henselian
discrete valuation ring with maximal ideal $\fm = (\pi)$, the quotient field $K$ and the
residue field $\ff$. We write $D = D_n = V((\pi^n))$ so that $\Fil_D \Omega^{q+1}_U =
\Fil_n \Omega^{q+1}_K$. \corref{cor:RSW-gen-1} then yields the map
${\Rsw}^{q+1}_{K,n}:= {\Rsw}^{q+1}_{X|(D,D')} \colon \gr^{\bk}_n H^{q+1}(K) \inj
\pi^{-n} \Omega^{q+1}_A(\log \pi)\otimes_A \ff \cong \Omega^{q+1}_{\ff} \oplus \Omega^{q}_{\ff}$.
Recall Kato's refined Swan conductor $ {\rsw}^q_{K,n}$ from ~\eqref{eqn:Kato-SC}.
We refer to \S~\ref{sec:Kato-comp} for various notations that we shall use in this
subsection.
\begin{lem}\label{lem:Rsw-spl-Kato}
  We have ${\Rsw}^{q+1}_{K,n} = {\rsw}^{q+1}_{K,n}$. 
\end{lem}
\begin{proof}
  We let $\chi \in \Fil^{\bk}_n H^{q+1}(K)$. We then have
  $\chi \in \H^1_\et(X, W_m\sF^{q,\bullet}_{X|{D_n}})$ for some $m \ge 1$.
  By definition (cf. \cite[\S~3.9]{Kato-89}),
  $\rsw^{q+1}_{K,n}(\chi) = (\alpha, \beta)$ so that
  $\lambda^{q+2}_{\pi}(T\alpha, T\beta) = \{\chi, 1 + \pi^nT\} \in H^{q+2}(K)$, where
  $\lambda^{q+2}_\pi \colon
  \Omega^{q+1}_{\ff[T]} \oplus \Omega^q_{\ff[T]} \to H^{q+2}((A[T])^h[\pi^{-1}])$
  is as in \cite[\S~1.9, p.~107]{Kato-89}.
  Using the surjectivity of the composite map
  $K^M_q(K) \otimes \Fil_n W_m(K) \to \Fil_n W_m\Omega^q_K
  \xrightarrow{\delta^{m,q}_n} T^{m,q}_n(K)$ (cf. \corref{cor:VR-5}, \thmref{thm:Kato-fil-10}), it suffices
  to prove ${\Rsw}^{m,q+1}_{K,n}(\chi) = {\rsw}^{m,q+1}_{K,n}(\chi)$ in the
  following two cases.

  \vskip .2cm

 {\bf{Case~1:}} $\chi = \delta^{m,q}_n\left(V^i([x]_{m-i})\dlog([u_1]_m)
  \wedge \cdots \wedge \dlog([u_q]_m)\right)$ (cf. ~\eqref{eqn:hyp-1}), where
  $x \in K^{\times}$ such that $x^{p^{m-i-1}} \in \Fil_n K$ for some $i$ and
  $u_j \in A^{\times}$ for each $j$.

  \vskip .2cm

  In this case, we let $\chi' = \delta^{m,0}_n(V^i([x]_{m-i}))$ and $\alpha =
  \dlog([u_1]_m) \wedge \cdots \wedge \dlog([u_q]_m)$
  so that $\chi' \in  \Fil^{\bk}_n H^{1}_m(K)$ and $\chi = \chi' \cup \alpha$ by the
  commutative diagram ~\eqref{eqn:Milnor-1}. In particular,
  \begin{equation}\label{eqn:Rsw-spl-Kato-0}
  \{\chi, 1 + \pi^nT\} =  \{\chi' \cup \alpha, 1 + \pi^nT\} 
  =  (-1)^q \{\chi', 1 + \pi^nT\} \cup \alpha.
  \end{equation}
  Meanwhile, one knows from Kato's definition of his refined Swan conductor
  (cf. \cite[Rem.~3.2.12]{Matsuda}, \cite[(3.2.4)]{KLS}) that
  \begin{equation}\label{eqn:Rsw-spl-Kato-1}
  -(F^{m-1}d(V^i([x]_{m-i}))) = \rsw^1_{K,n}(\chi') = (\alpha_1, \beta_1) \in
    \Omega^1_{\ff} \oplus \ff.
  \end{equation}
  In particular, $\lambda^2_{\pi}(T\alpha_1, T\beta_1) =
  \{\chi', 1 + \pi^nT\} \in H^{2}((A[T])^h[\pi^{-1}])$. We thus get
\[
\{\chi, 1 + \pi^nT\} =   (-1)^q \lambda^2_{\pi}(T\alpha_1, T\beta_1) \cup \alpha
= \lambda^{q+2}_\pi \left((-1)^q T\alpha_1\wedge \ov{\alpha},
T\beta_1 \wedge \ov{\alpha}\right),
\]
where $\ov{\alpha}$ is the image of $\alpha$ under the quotient map
$\Omega^q_A \surj \Omega^q_{\ff}$.
By definition of ${\Rsw}^{m,q+1}_{K,n}(\chi)$, it remains therefore to show that
$(-1)^{q+1} F^{m-1}d(V^i([x]_{m-i})\alpha) = (-1)^q \left(\alpha_1 +
\beta_1 \dlog(\pi)\right) \wedge \ov{\alpha} = (-1)^q \alpha_1 \wedge \ov{\alpha} +
(\beta_1 \ov{\alpha}) \wedge \dlog(\pi)$.

To that end, we note that
\[
\begin{array}{lll}
  (-1)^{q+1} F^{m-1}d(V^i([x]_{m-i})\alpha) & = & (-1)^{q+1} F^{m-1}(dV^i([x]_{m-i}) \wedge
  \alpha) \\
  & = & (-1)^q \left(-F^{m-1}dV^i([x]_{m-i})\right) \wedge \alpha \\
& {=}^1 & (-1)^q(\alpha_1 + \beta_1 \dlog(\pi)) \wedge \ov{\alpha} \\
  & = & ((-1)^q \alpha_1 \wedge \ov{\alpha}) + (\beta_1  \ov{\alpha} \wedge
  \dlog(\pi)), \\
\end{array}
\]
where ${=}^1$ holds by ~\eqref{eqn:Rsw-spl-Kato-1}. This proves the desired
identity.

\vskip .2cm

  {\bf{Case~2:}} $\chi = \delta^{m,q}_n\left(V^i([x]_{m-i})\dlog([u_1]_m)
  \wedge \cdots \wedge \dlog([u_{q-1}]_m) \wedge \dlog(\pi)\right)$, where
  $x \in K^{\times}$ such that $x^{p^{m-i-1}} \in \Fil_n K$ for some $i$ and
  $u_j \in A^{\times}$ for each $j$.

  \vskip .2cm

  In this case, we let $\alpha = \dlog([u_1]_m)
  \wedge \cdots \wedge \dlog([u_{q-1}]_m) \wedge \dlog(\pi)$ and
  $\alpha' = \dlog([u_1]_m) \wedge \cdots \wedge \dlog([u_{q-1}]_m)$.
  Repeating the argument of the previous case, we then get
  \[
  \begin{array}{lll}
    \{\chi, 1 + \pi^nT\} & = & (-1)^q \{\chi', 1 + \pi^nT\} \cup \alpha \\
    & = & (-1)^q \lambda^2_\pi(T\alpha_1, T\beta_1 ) \cup \alpha \\
    & = & (-1)^q \lambda^2_\pi(T\alpha_1,0) \wedge \alpha' \wedge \dlog(\pi) +
    \lambda^2_\pi(0,T\beta_1) \wedge \alpha \wedge \dlog(\pi) \\
     & = &  (-1)^q \lambda^2_\pi(T\alpha_1,0) \wedge \alpha' \wedge \dlog(\pi) \\
& = & (-1)^q \lambda^{q+2}_\pi\left(0, T\alpha_1 \wedge \ov{\alpha'}\right)
\end{array}
  \]
  It remains therefore to show that
  $(-1)^{q+1}F^{m-1}d(V^i([x]_{m-i}) \alpha) = (-1)^{q} \alpha_1
  (\alpha' \wedge \dlog(\pi))$.
  To check this, we note that
  \[
  \begin{array}{lll}
    (-1)^{q+1}F^{m-1}d(V^i([x]_{m-i}) \alpha)  & = & (-1)^{q+1}F^{m-1}d(V^i([x]_{m-i})
    \wedge \alpha' \wedge \dlog(\pi) \\
    & {=}^2 & (-1)^q(\alpha_1 + \beta_1 \dlog(\pi)) \wedge \alpha'
    \wedge \dlog(\pi) \\
    & = & (-1)^q\alpha_1 (\alpha' \wedge \dlog(\pi)),
  \end{array}
  \]
  ${=}^2$ holds by ~\eqref{eqn:Rsw-spl-Kato-1}.
  This shows the desired identity and concludes the proof. 
\end{proof}
    
\vskip.2cm

\begin{remk}\label{remk:Rsw-spl-Kato-2}
  Using \lemref{lem:Rsw-spl-Kato}, it is easy to check that Kato's integrality and
  specialization theorems (cf. \cite[Thms.~7.1, 9.1]{Kato-89})
  in positive characteristics are immediate consequences of \corref{cor:RSW-gen-1}.
  \end{remk}

\section{Lefschetz theorem for Ramification filtration}\label{sec:Lef}
In this section, we prove the Lefschetz theorem  for Kato's global ramification
filtration (cf. \S~\ref{sec:Kato-global}).
We shall apply this in the next section to prove a similar result for the Brauer group.
We begin by proving the key step, \lemref{lem:Complex-6}.

\subsection{The key lemma}\label{sec:Complexes}
We let $k = \F_p$ and let $X$ be a Noetherian regular and connected
$F$-finite $k$-scheme. 
Let $E$ be an snc divisor (this may be empty) on $X$
with irreducible components $E_1, \ldots , E_r$.  Let $j \colon U \inj X$ be the
inclusion of the complement of $E$ in $X$. For
$\un{n} = (n_1, \ldots , n_r) \in \Z^r$, we let
$D_{\un{n}} = \stackrel{r}{\underset{i =1}\sum} n_i E_i \in \Div_E(X)$. 
We shall write $\Fil_{D_{\un{n}}}W_m\Omega^q_U$ as $\Fil_{\un{n}}W_m\Omega^q_U$.
If $t \in \Z$ and $\un{n} \in \Z^r$, we let
$\un{t}:=(t, {\ldots},t)\in \Z^r$ and
$t\un{n}:=(tn_1, \ldots, tn_r)$. We let
$\un{n} \pm \n' := (n_1\pm n'_1, \ldots , n_r \pm n'_r)$ and
$\n \pm t :=\n \pm \un{t}$.
Recall the notation
$\n/{p^i} =(\lfloor {n_1}/{p^i} \rfloor,\ldots, \lfloor {n_r}/{p^i} \rfloor)$
(cf. \S~\ref{sec:CFDRW}).
For $\un{n} \ge {\un{0}}$, we let
$\sI_{\un{n}} = \Ker(\sO_X \surj \sO_{D_{\un{n}}})$.
Recall also the map of Zariski sheaves
$\dlog \colon {\sK^M_{q, X}}/{p^m} \to W_m\Omega^q_X$ from
\S~\ref{sec:katofil} which is injective and whose image is $W_m\Omega^q_{X, \log}$
for $q \ge 0$ and $m \ge 1$. We $\sK^{m}_{q,X} = {\sK^M_{q,X}}/{p^m}$.

For $\un{n} \ge \un{0}$, we let
 $\sK^M_{q, X|D_{\un{n}}}$ be the Zariski (or {\'e}tale) sheaf on $X$ given by the image
of the canonical map $(1 + \sI_{\un{n}})^\times \otimes j_*\sO^\times_U \otimes \cdots
\otimes j_*\sO^\times_U \to j_*\sK^M_{q,U}$.
As $\sK^M_{q, X|D_{\un{n}}} \subset \sK^M_{q,X}$ (cf. \cite[Prop.~1.1.3]{JSZ}), we get
a subsheaf
$\sK^m_{q, X|D_{\un{n}}} = {\sK^M_{q, X|D_{\un{n}}}}/{(p^m\sK^M_{q,X} \cap \sK^M_{q, X|D_{\un{n}}})}$
of $\sK^m_{q,X}$.
Letting $W_m\Omega^q_{X|D_{\un{n}},\log}$ be the image of the map
$\dlog \colon \sK^M_{q, X|D_{\un{n}}} \to W_m\Omega^q_X$ for $\n \ge \un{1}$,
one gets a canonical isomorphism
$\dlog \colon \sK^m_{q, X|D_{\un{n}}} \xrightarrow{\cong} W_m\Omega^q_{X|D_{\un{n}},\log}$
(cf. \cite[Thm.~1.1.5]{JSZ}). Let $d$ denote the rank of the locally free sheaf $\Omega^1_X$.

\begin{lem}\label{lem:Complex-5}
  For $\n \ge \un{1}$,  we have
  $W_m\Omega^q_{X|D_{\un{n}},\log} \subset \Fil_{-\un{n}}W_m\Omega^q_U$.
\end{lem}
\begin{proof}
  Using Lemmas~\ref{lem:Log-fil-4}(5) and ~\ref{lem:Log-fil-Wcom},
  it suffices to show that 
  $\dlog([1+ at]_m) \in \Fil_{-\un{n}}W_m\Omega^1_K$
  in the notations of \S~\ref{sec:com-to-noncom}
  if $A$ is a regular local ring with quotient field $K$,
  $t = \stackrel{r}{\underset{i =1}\prod} x^{n_i}_i$ with each $n_i \ge 1$
  and $a \in A$. To that end, we apply \cite[Lem.~1.2.3]{Geisser-Hesselholt-JAMS} to
  write $[1+ at]_m = (ta_{m-1}, \ldots , ta_0) + [1]_m$, where $a_i \in A$.
  This implies that $\dlog([1+ at]_m) = [1+at]^{-1}_m (d(x) + d([1]_m))$,
  where $x = (ta_{m-1}, \ldots , ta_0)$. Since $d([1]_m) \in W_m\Omega^1_{\F_p} = 0$
  and $x \in \Fil_{-\un{n}}W_m(K)$, we conclude from \lemref{lem:Log-fil-Wcom} that
  $\dlog([1+ at]_m) \in \Fil_{-\un{n}}W_m\Omega^1_K$.
\end{proof}

\begin{lem}\label{lem:Complex-6}
  The sequences 
  \begin{enumerate}
    \item
     \ \ $0 \to W_m\Omega^q_{X|D_{\un{n}},\log} \to \Fil_{-\un{n}}W_m\Omega^q_U
    \xrightarrow{1 -\ov{F}}
    \frac{\Fil_{-\un{n}}W_m\Omega^q_U}{dV^{m-1}(\Fil_{-\un{n}}\Omega^{q-1}_U)}$;
\item
   \ \ $0 \to W_m\Omega^q_{X|D_{\un{1}},\log} \to \Fil_{-\un{1}}W_m\Omega^q_U
    \xrightarrow{1 -\ov{F}}
    \frac{\Fil_{-\un{1}}W_m\Omega^q_U}{dV^{m-1}(\Fil_{-\un{1}}\Omega^{q-1}_U)} \to 0$;
     \item
  \ \ $0 \to W_m\Omega^q_{X|D_{\un{1}},\log} \to Z_1\Fil_{-\un{1}}W_m\Omega^q_U
    \xrightarrow{1 - C}
    \Fil_{-\un{1}}W_m\Omega^q_U \to 0$;
 \item
    \ \ $0 \to j_* W_m\Omega^q_{U,\log} \to Z_1\Fil_{\un{0}}W_m\Omega^q_U
    \xrightarrow{1 - C}
    \Fil_{\un{0}}W_m\Omega^q_U \to 0$;
  \item
\ \ $0 \to j_* W_m\Omega^q_{U,\log} \to \Fil_{\un{0}}W_m\Omega^q_U
    \xrightarrow{1 -\ov{F}}
    \frac{\Fil_{\un{0}}W_m\Omega^q_U}{dV^{m-1}(\Fil_{\un{0}}\Omega^{q-1}_U)} \to 0$;
  \item 
    \ \ $0 \to \Omega^q_{X|D_\n, \log} \to Z_1\Fil_{-\n}\Omega^q_U
    \xrightarrow{1-C}\Fil_{(-\n)/p}\Omega^q_U \to 0$;
    \end{enumerate}
are exact, where $\un{n} \ge 1$ in $(1)$ and $(6)$.
\end{lem}
\begin{proof}
It is clear that $(1- \ov{F})(W_m\Omega^q_{X|D_{\un{n}},\log}) = 0$.
Using \cite[Lem.~2]{CTSS} and \lemref{lem:Complex-5}, (1)
will therefore follow if we show that
\begin{equation}\label{eqn:Complex-4-2}
  \Fil_{-\un{n}}W_m\Omega^q_U \bigcap W_m\Omega^q_{X,\log} \subset W_m\Omega^q_{X|D_{\un{n}},\log}
  \ \ {\rm for} \ \ \un{n} \ge \un{0}.
  \end{equation} 
  We shall prove this inclusion by induction on $m$. The case $m = 1$ follows
  directly from \lemref{lem:Log-fil-4}(3) and \cite[Thm.~1.2.1]{JSZ}. We now let
  $m \ge 2$ and $x \in \Fil_{-\un{n}}W_m\Omega^q_U \bigcap W_m\Omega^q_{X,\log}$.
  \lemref{lem:Log-fil-Wcom} then implies that
  $$R(x) \in \Fil_{({-\un{n})}/p}W_{m-1}\Omega^q_U \bigcap W_{m-1}\Omega^q_{X,\log} \text{ and }
  F(x) \in \Fil_{-\un{n}}W_{m-1}\Omega^q_U \bigcap W_{m-1}\Omega^q_{X,\log}.$$
  Since we also have $R(x) = F(x)$, it follows that
  $R(x) \in \Fil_{-\un{n}}W_{m-1}\Omega^q_U \bigcap W_{m-1}\Omega^q_{X,\log}$.
  In particular, the induction hypothesis implies that
  $R(x) = F(x) \in  W_{m-1}\Omega^q_{X|D_{\un{n}},\log}$.

We now look at the commutative diagram (cf. \cite[Proof of Thm.~2.3.1]{JSZ})
  \begin{equation}\label{eqn:Complex-4-3}
    \xymatrix@C.8pc{
      & & W_m\Omega^q_{X|D_{\un{n}},\log} \ar[r]^-{R} \ar@{^{(}->}[d] & 
      W_{m-1}\Omega^q_{X|D_{\un{n}},\log}  \ar[r] \ar@{^{(}->}[d] & 0 \\
      0 \ar[r] & \Omega^q_{X, \log} \ar[r]^-{\un{p}^{m-1}} &
      W_m\Omega^q_{X,\log} \ar[r]^-{R} &
      W_{m-1}\Omega^q_{X,\log} \ar[r] & 0,}
    \end{equation}
  whose bottom row is exact (cf. \cite[Lem.~3]{CTSS}). Note that the map $R$ on the
  top is surjective by definition of the sheaves $W_m\Omega^q_{X|D_{\un{n}},\log}$.
  Using this diagram, we can
  find $y \in W_m\Omega^q_{X|D_{\un{n}},\log}$ and $z \in \Omega^q_{X, \log}$ such that
  $\un{p}^{m-1}(z) = x-y \in \Fil_{-\un{n}}W_m\Omega^q_U$. We conclude from
  \lemref{lem:Complete-7} that $z \in \Fil_{-\un{n}'}\Omega^q_U \bigcap
  \Omega^q_{X,\log}$, where $\un{n}' = -(-\un{n}/{p^{m-1}})$.
  In particular, $z \in \Omega^q_{X|D_{n'}, \log}$ by induction.
  It follows that $x-y = \un{p}^{m-1}(z) \in W_m\Omega^q_{X|D_{p^{m-1}\un{n}',\log}}
  \subset W_m\Omega^q_{X|D_{\un{n}},\log}$, and hence $x \in W_m\Omega^q_{X|D_{\un{n}},\log}$.
  This proves (1).

To prove (2), we only need to show that $1-\ov{F}$ is surjective.
By \lemref{lem:VR-3}, \remref{remk:VR-3-0} and the fact that
  $$\ov{F}(w \dlog([x_1]_m) \wedge \cdots \wedge \dlog([x_q]_m)) =
  \ov{F}(w)\dlog([x_1]_m) \wedge \cdots \wedge \dlog([x_q]_m)),$$
  it is sufficient to find pre-images of elements of the form $V^i([x]_{m-i})$
  under the map
  $\Fil_{-\un{1}}W_m(K)  \xrightarrow{1 - \ov{F}} \Fil_{-\un{1}}W_m(K)$.
  We shall show that this map is in fact surjective.
  Using the commutative diagram of exact sequences
  \begin{equation}\label{eqn:Complex-4-1}
    \xymatrix@C1.6pc{
      0 \ar[r] & \Fil_{- \un{1}}W_1(K) \ar[r]^-{V^{m-1}} \ar[d]_-{1 -\ov{F}} &
 \Fil_{- \un{1}}W_m(K) \ar[r]^-{R} \ar[d]^-{1 -\ov{F}} &      
 \Fil_{- \un{1}}W_{m-1}(K) \ar[r] \ar[d]^-{1 -\ov{F}} & 0 \\
 0 \ar[r] & \Fil_{- \un{1}}W_1(K) \ar[r]^-{V^{m-1}} &
 \Fil_{- \un{1}}W_m(K) \ar[r]^-{R} & \Fil_{- \un{1}}W_{m-1}(K) \ar[r] & 0}
  \end{equation}
and induction on $m$, it suffices to show that
  $\Fil_{- \un{1}}W_1(K) \xrightarrow{1 - \ov{F}} \Fil_{-\un{1}}W_1(K)$ is surjective
  when the ring $A$ in Lemma~\ref{lem:VR-3} is strict henselian. But this is
  elementary (cf. \cite[Claim~1.2.2]{JSZ}). 

  To prove the surjectivity of $1-C$ in (3), it suffices to show, using
  \lemref{lem:Complete-4}(2) and the commutative diagram  
  \begin{equation}\label{eqn:Complex-4-4}
  \xymatrix@C.8pc{
    0 \ar[r] & Z_1\Fil_{-\un{1}}W_m\Omega^q_U \ar[r] \ar[d]_-{1-C} &
    \Fil_{{-\un{1}}}W_{m}\Omega^q_U \ar[r] \ar[d]^-{1-\ov F} &
    \frac{\Fil_{{-\un{1}}}W_{m}\Omega^q_U}{Z_1\Fil_{-\un{1}}W_m\Omega^q_U
      \ar[r]} \ar@{=}[d] & 0 \\
    0 \ar[r] & \Fil_{{-\un{1}}}W_{m}\Omega^q_U \ar[r]^-{-\ov F} &
    \frac{\Fil_{-\un{1}}W_{m}\Omega^q_U}{dV^{m-1}(\Fil_{-\un{1}}\Omega^{q-1}_U)} \ar[r] &
    \frac{\Fil_{-\un{1}}W_{m}\Omega^q_U}{\ov F(\Fil_{{-\un{1}}}W_{m}\Omega^q_U )}
    \ar[r] & 0,}
   \end{equation}
  that $1 - \ov{F}$ is surjective. The latter claim follows from (2). The
  remaining part of (3) is easily deduced from (1) and the above diagram. 

The proof of the surjectivity of $1-C$ (resp. $1-\ov{F}$) in (4) (resp. (5))
is identical to that in (3) (resp. (2)). Also, one checks using an analogue of
~\eqref{eqn:Complex-4-4} for $\Fil_0W_m\Omega^q_U$
that $\Ker(1-C) = \Ker(1-\ov{F})$. We thus have to only show
that $\Ker(1-C) = j_* W_m\Omega^q_{U, \log}$ to finish the proofs of (4) and (5).

To that end, we look at the commutative diagram of exact sequences
\begin{equation}\label{eqn:Complex-4-5}
  \xymatrix@C.8pc{
    0 \ar[r] & \Ker(1-C) \ar[r] \ar[d] & Z_1\Fil_0 W_m\Omega^q_U \ar[r]^-{1-C}
    \ar@{^{(}->}[d] & \Fil_0 W_m\Omega^q_U \ar[r] \ar@{^{(}->}[d] & 0 \\
    0 \ar[r] & j_* W_m\Omega^q_{U, \log} \ar[r] & j_* Z_1W_m\Omega^q_U \ar[r]^-{1-C} &
    j_*W_m\Omega^q_U. &}
  \end{equation}

The middle and the right vertical arrows in ~\eqref{eqn:Complex-4-5}
are the canonical inclusion maps.
It follows that $\Ker(1-C) \inj  j_* W_m\Omega^q_{U, \log}$.
To prove this inclusion is a bijection, we note that the map
$\dlog \colon j_* \sK^M_{q,U} \to j_* W_m\Omega^q_{U, \log}$ 
has a factorization
\begin{equation}\label{eqn:Complex-4-6}
  \xymatrix@C.8pc{
    j_* \sK^M_{q,U} \ar[r]^-{\dlog} \ar[d] & Z_1W_m\Omega^q_X(\log E) \ar@{^{(}->}[d] \\
    j_* W_m\Omega^q_{U, \log} \ar@{^{(}->}[r] & j_* Z_1W_m\Omega^q_U.}
  \end{equation}
Hence, it suffices to show that $j_* \sK^M_{q,U} \to j_* W_m\Omega^q_{U, \log}$ is
surjective. Since $j_* {\sK^m_{q,U}} \to
j_* W_m\Omega^q_{U, \log}$ is an isomorphism, we are reduced to
showing that the canonical map $j_* \sK^M_{q,U} \to j_* {\sK^m_{q,U}}$ is
surjective in $X_\et$. To prove the latter, it is enough to show that the map
$j_* \sK^M_{q,U} \to j_* {\sK^m_{q,U}}$ is surjective in $X_\zar$.

To prove the surjectivity of $j_* \sK^M_{q,U} \to j_* {\sK^m_{q,U}}$ in
$X_\zar$, we can assume that $X = \Spec(A)$ and
$U = X_\pi$ as in \S~\ref{sec:com-to-noncom}. In this case, the exact sequence
\begin{equation}\label{eqn:Complex-4-7}
0 \to \sK^M_{q, X_\pi} \xrightarrow{p^m}  \sK^M_{q, X_\pi} \to {\sK^m_{q,X_\pi}} \to 0
\end{equation}
reduces the problem to showing that $H^1_\zar(X_\pi,  \sK^M_{q, X_\pi}) = 0$,
which is \lemref{lem:Gersten-0}.

The proof of (6) is identical to that of (3)
using the commutative diagram 
  \begin{equation}
    \xymatrix@C.8pc{ 0 \ar[r] & Z_1\Fil_{-\un{n}}\Omega^q_U
      \ar[r] \ar[d]_-{1-C} & \Fil_{{-\un{n}}}\Omega^q_U \ar[r] \ar[d]^-{1-\ov F} &
      \frac{\Fil_{{-\un{n}}}\Omega^q_U}{Z_1\Fil_{-\un{n}}\Omega^q_U \ar[r]} \ar@{=}[d] & 0 \\
      0 \ar[r] & \Fil_{({-\un{n}})/p}\Omega^q_U \ar[r]^-{-\ov F} &
      \frac{\Fil_{-\un{n}} \Omega^q_U}{d(\Fil_{-\un{n}}\Omega^{q-1}_U)}
      \ar[r] &
      \frac{\Fil_{-\un{n}} \Omega^q_U}{\ov F(\Fil_{({-\un{n}})/p}\Omega^q_U )} \ar[r] & 0}
  \end{equation}
  and \cite[Thm~1.2.1]{JSZ}. The vertical arrow on the left is defined by
  \corref{cor:Complete-6-global}.
\end{proof}

The following is a generalized version of \cite[Thm.~1.1.6]{JSZ} which proved the
result when $r = m-1$.

\begin{cor}\label{cor:Complex-7}
  For $\un{n} \ge 1$ and $1 \le r \le m-1$, one has an exact sequence
  \[
  0 \to W_{r}\Omega^q_{X|D_{\lceil{{\un{n}}/p^{m-r}}\rceil}, \log} \xrightarrow{{\un{p}}^{m-r}}
   W_{m}\Omega^q_{X|D_{{\un{n}}}, \log} \xrightarrow{R^{r}}  W_{m-r}\Omega^q_{X|D_{{\un{n}}}, \log}
   \to 0.
   \]
\end{cor}
\begin{proof}
  The surjectivity of $R^r$ follows directly from the definition of
  $W_{m}\Omega^q_{X|D_{{\un{n}}}, \log}$. The remaining part of the corollary follows
  from its classical non-modulus version (cf. \cite[Lem.~3]{CTSS}) and the
  identity $\Fil_{-\un{n}}W_m\Omega^q_U \bigcap W_m\Omega^q_{X,\log} =
  W_m\Omega^q_{X|D_{\un{n}},\log}$ that we showed in the proof of \lemref{lem:Complex-6}
  (cf. ~\eqref{eqn:Complex-4-2}).
\end{proof}

Let $X_\log$ be the log scheme associated to the snc-pair $(X,E)$ and let
$W_m\Omega^q_{X_\log, \log}$ denote the logarithmic de Rham--Witt sheaf on
$X_\log$ which is defined analogous to the case of ordinary schemes explained
above (cf. \cite[\S~2, p.~257]{Lorenzon}).
Another application of \lemref{lem:Complex-6} is the following.

\begin{cor}\label{cor:Log-log-HW}
  There is a canonical isomorphism
 $W_m\Omega^q_{X_\log, \log} \xrightarrow{\cong} j_* W_m\Omega^q_{U, \log}$.
  \end{cor}
\begin{proof}
Combine \lemref{lem:Complex-6}(5) and \cite[Prop.~2.13]{Lorenzon}.
\end{proof}

We let $k$ now be an arbitrary field and $X$ a connected smooth projective $k$-scheme of
dimension $N \ge 2$. We fix an embedding $X \inj \P^L_k$.
Let $u \colon X' \inj X$ be the inclusion of a smooth hypersurface section
such that $E' = X' \times_X E$ is an snc divisor on $X'$.
(Note that such a hypersurface section always exists in every large degree,
see \cite[Thm.~3.10, Thm.~4.11]{GK-JLMS}).
Under these assumptions, it is easy to check that $E_0:= E \bigcup X'$ is an snc
divisor on $X$. We let $V = X \setminus E_0$ and let $j_0 \colon V \inj X$
denote the inclusion. We let $j' \colon U' \inj X'$ denote the inclusion of the
complement of $E'$ in $X'$. Let $D' := D \times_X X'$ for any
$D \in \Div(X)$. We let $W = X \setminus X'$.

\subsection{Some exact sequences}\label{sec:Ex-seq*}
In this subsection, we assume that $k$ is an $F$-finite field of characteristic $p > 0$
and collect some exact sequences of sheaves. We fix $q \ge 0$ and
refer to Definition~\ref{defn:Log-fil-3} for the definition of $\Fil_DW_m\Omega^q_U$
and $\Fil_D\Omega^q_U$. In the exposition below,
whenever we use the notation $\Fil_DW_m\Omega^q_V$,
where $D$ is supported on $E_0$ (and not necessarily on $E$),
we shall assume the underlying Rham--Witt complex to be $\Div_{E_0}(X)$-filtered
(and not necessarily $\Div_E(X)$-filtered).

\begin{lem}\label{lem:fil D-Y}
   For any $D \in \Div_E(X)$, we have exact sequences
    \begin{equation}\label{eqn:fil D-Y-0}
      0 \to \Fil_{D-X'} \Omega^q_V \to \Fil_D \Omega^q_U \to u_* \Fil_{D'} \Omega^q_{U'}
      \to 0; \ \
      0 \to \Fil_{-X'} \Omega^q_{V} \to \Omega^q_X \to u_*  \Omega^q_{X'} \to 0.
\end{equation}
\end{lem}
\begin{proof}
  The second exact sequence follows from the first by taking $E$ to be the empty divisor. For the first exact sequence, we take the tensor product of the exact sequence in \lemref{lem:Kahler-seq}(2)  with $\sO_X(D)$.
\end{proof}

\begin{cor}\label{cor:fil-D-Y-2}
    For any $D\in \Div_E(X)$, there are exact sequences
    \begin{equation}\label{eqn:fil-D-Y-2-0}
      0 \to Z_1\Fil_{D-X'}\Omega^q_{V} \to Z_1\Fil_{D}\Omega^q_{U}
      \xrightarrow{\theta_{D}} u_*Z_1\Fil_{D'}\Omega^q_{U'} \to 0;
    \end{equation}
    \begin{equation}\label{eqn:no-log-3}
      0\to Z_1\Fil_{-X'}\Omega^q_{V} \to Z_1\Omega^q_X \xrightarrow{\theta}
      u_*Z_1\Omega^q_{X'} \to 0.
    \end{equation}
\end{cor}
\begin{proof}
    By \lemref{lem:fil D-Y}, we have a commutative diagram of exact sequences
    \begin{equation}
        \xymatrix{
          0 \ar[r]& \Fil_{D-X'}\Omega^q_{V} \ar[r] \ar[d]_-{d}& \Fil_{D}\Omega^q_{U}
          \ar[r]^-{\theta_D} \ar[d]^-{d} & u_* \Fil_{D'}\Omega^q_{U'} \ar[r]
          \ar[d]^-{u_*(d)} & 0 \\
          0 \ar[r]& \Fil_{D-X'}\Omega^{q+1}_{V} \ar[r] & \Fil_{D}\Omega^{q+1}_{U}
          \ar[r]^-{\theta_D} & u_* \Fil_{D'}\Omega^{q+1}_{U'} \ar[r] & 0.
        }
    \end{equation}
    Since $Z_1\Fil_D\Omega^q_U= \Ker\ (d:\Fil_D\Omega^q_U \to \Fil_D\Omega^{q+1}_U)$,
    this gives an exact sequence
    \[
    0 \to Z_1\Fil_{D-X'}\Omega^q_{V} \to Z_1\Fil_{D}\Omega^q_{U} \xrightarrow{\theta_{D}}
    u_*Z_1\Fil_{D'}\Omega^q_{U'},
\]
where we note that $u_*Z_1\Fil_{D'}\Omega^q_{U'} = \Ker(u_*(d))$ because $u_*$ is
exact. To show the right exactness of ~\eqref{eqn:fil-D-Y-2-0}, we
consider the commutative diagram
\[
\xymatrix@C1pc{
  \Fil_{D}W_{2}\Omega^q_{U} \ar[r]^-{\theta_D} \ar[d]_-{F} & u_* \Fil_{D'}W_{2}\Omega^q_{U'}
  \ar[d]^-{u_*F}  \\
 Z_1\Fil_{D}\Omega^q_{U} \ar[r]^-{\theta_{D}} & u_*Z_1\Fil_{D'}\Omega^q_{U'}. 
}
\]

We now note that $u_*F$ is surjective by \lemref{lem:Complete-4}(1)
and exactness of $u_*$.
The top $\theta_D$ is surjective because locally this map is
$\Fil_{\n}W_2\Omega^q_{(R_\pi)} \to \Fil_{\n}W_2\Omega^q_{((R/t)_{\ov \pi})}$, where $R$ is a
regular local ring and $t,\pi,\ov \pi$ and $\un n$ correspond to $X',E$, $E'$ and $D$,
respectively. This map is surjective by definition of $\Fil_{\n}W_2\Omega^q_{(R_\pi)}$
(cf. \defref{defn:Log-fil-3}) because $\Fil_{\n}W_2(R_\pi)\surj
\Fil_{\n}W_2((R/t)_{\ov \pi})$. We conclude that the bottom $\theta_D$ is also
  surjective. The proof of (\ref{eqn:no-log-3}) is similar and is left to the reader.
\end{proof}

We shall use the following snc-pairs.
\begin{equation}\label{eqn:snc-pairs}
  \sX = (X,E), \  \sX_0 =  (X, E_0), \ \sX_1 = (X, X') \ \mbox{and} \ \sX' = (X', E').
  \end{equation}

We let (cf. Definition~\ref{defn:C-complex-global})
\[
W_m\sF^{q,\bullet}_{\sX|D} := W_m\sF^{q,\bullet}_{X|D} =
\left[Z_1\Fil_{D} W_m\Omega^q_U \xrightarrow{1 -C}
  \Fil_{D}W_m\Omega^q_U\right];\ W_m\sF^{q,\bullet}_{\sX'|D'} := W_m\sF^{q,\bullet}_{X'|D'}; 
\]
\[
W_m\sF^{q,\bullet}_{\sX_0|F} := \left[Z_1\Fil_{F} W_m\Omega^q_V \xrightarrow{1 -C}
  \Fil_{F}W_m\Omega^q_V\right];
\]
\[
W_m\sF^{q,\bullet}_{\sX_1|F_1} := \left[Z_1\Fil_{F_1} W_m\Omega^q_{W}
  \xrightarrow{1 -C} \Fil_{F_1}W_m\Omega^q_{W}\right];
\]
\[
W_m\sF^{q,\bullet}_X =
\left[Z_1W_m\Omega^q_{X} \xrightarrow{1 -C} W_m\Omega^q_{X}\right];
\]
where $D \in \Div_E(X), \ F \in \Div_{E_0}(X), \ F_1 \in \Div_{X'}(X)$ with
$D+E \ge 0, \ F + E_0 \ge 0$ and $F_1 + X' \ge 0$.
For $D\ge E$ in $\Div_E(X)$, we let $\sN^{q,\bullet}_{\sX|D}:=
\left[Z_1\Fil_{-D}\Omega^q_U \xrightarrow{1-C}\Fil_{(-D)/p}\Omega^q_U
  \right]$ (cf. \corref{cor:Complete-6-global}) and
define $\sN^{q,\bullet}_{\sX'|D'}$ similarly.
By \lemref{lem:Complex-6}(6), $\sN^{q,\bullet}_{X|D}$ is quasi-isomorphic to
$\Omega^q_{X|D,\log}$. We let $\sN^{q,\bullet}_{\sX_0|D-X'}:=
\left[Z_1\Fil_{-D-X'}\Omega^q_V \xrightarrow{1-C}\Fil_{(-D-X')/p}\Omega^q_V \right]$.
\lemref{lem:fil D-Y} and Corollary~\ref{cor:fil-D-Y-2} then yield
the following.

\begin{cor}\label{cor:ker-lef}
  We have the following short exact sequences of complexes in $\sD(X_\et)$ for
  $D \in \Div_E(X)$ with $D+E \ge 0$.
  \begin{enumerate}
  \item
$0 \to W_1\sF^{q,\bullet}_{\sX_0|D-X'} \to W_1\sF^{q,\bullet}_{\sX|D} \to
    u_*W_1\sF^{q,\bullet}_{\sX'|D'} \to 0$.
  \item
    $ 0 \to W_1\sF^{q,\bullet}_{\sX_1|-X'} \to W_1\sF^{q,\bullet}_X  \to
    u_*W_1\sF^{q,\bullet}_{X'} \to 0$.
  \item
   $0 \to \sN^{q,\bullet}_{\sX_0|D-X'} \to \sN^{q,\bullet}_{\sX|D} \to
    u_*\sN^{q,\bullet}_{\sX'|D'} \to 0$.
\end{enumerate}
 In particular, we have exact triangles in $\sD(X_\et)$.
    \begin{enumerate}
  \item
      $W_1\sF^{q,\bullet}_{\sX_1|-X'} \to \Omega^q_{X,\log} \to u_*\Omega^q_{X',\log}$.
    \item
      $\sN^{q,\bullet}_{\sX_0|D-X'} \to \Omega^q_{X|D,\log} \to u_*\Omega^q_{X'|D',\log}$.
    \end{enumerate}
\end{cor}

\subsection{The main result}\label{sec:WL-p}
In this subsection, we shall prove our main result.
We continue to assume that $k$ is $F$-finite of characteristic $p > 0$.
We make the following definition for convenience of our exposition.

\begin{defn}\label{def:ample**}
  Let $r \ge 0$ be an integer and let $D \in \Div_E(X)$.
  We shall say that a hypersurface section $Y \subset X$ satisfies the
  Diff$(r,0)$ condition if $H^{i}_\zar(X,\Omega^q_X(-nY))=0$ for all $n \ge 1$ and
  $i,q \ge 0$ such that $i+q \le r$.
  We shall say that $Y \subset X$ satisfies the Diff$(r,D)$ condition if
  $H^{i}_\zar(X,\Omega^q_X(\log E)(\ov{D}-nY))=0$ for all $n \ge 1$,
  $\ov{D} \in \Div(X)$ with $0 \le \ov{D} \le D$ and
  $i,q \ge 0$ such that $i+q \le r$.
  \end{defn}

Note that the Diff$(r,0)$ condition coincides with the  Diff$(r,D)$ condition
when $E = \emptyset$.
By \cite[Lem.~5.1]{GK-Jussieu}, there exists $M \gg 0$ such that for every
$n \ge M$, every hypersurface section of $X$ of degree $n$ satisfies  the
 Diff$(r,D)$ (hence Diff$(r,0)$) condition if $r \le N-1$.

 \vskip.2cm

 We let $u \colon X' \inj X$ be the inclusion of the hypersurface chosen
 in the beginning of \S~\ref{sec:Lef}.

\begin{lem}\label{lem:ample}
  Let $r \ge 0$ and  $D \in \Div^{\eff}_E(X)$. Suppose that $X' \subset X$
  satisfies the ${\rm Diff}(r,0)$ condition. Then
  $H^i_\zar(X', \Omega^{q}_{X'}(-nX'))= 0$ for all $n \ge 1$ and
  $i,q \ge 0$ such that $i+q \le r-1$.
  If $X'$ satisfies the ${\rm Diff}(r,D)$ condition, then
  $H^i_\zar(X', \Omega^{q}_{X'}(\log E') (\ov{D}-nX'))=0$ for all $n \ge 1$,
  $\ov{D} \in \Div(X)$ with $0 \le \ov{D} \le D$ and
  $i,q \ge 0$ such that $i+q \le r-1$.
  \end{lem}
  \begin{proof}
    We let $A^q_n:=\Omega^q_{X'}(-nX')$ and
    $A^q_n(\ov{D}): =\Omega^q_{X'}(\log E')(\ov{D}-nX')$ for
  $n \ge 1, q \ge 0$. We claim that
    $H^i_\zar(X', A^q_n) = H^i_\zar(X', A_n^q(\ov{D}))=0$ for all
    $i+q \le r-1, \ n \ge 1$ and $0 \le \ov{D} \le D$.
This will prove the lemma.
We shall prove the claim for $A^q_n(\ov{D})$ as the other one is its special
case. We proceed by induction on $q \ge 0$.
  If $q=0$, then $A^q_n(\ov{D})=\sO_{X'}(\ov{D}-nX')$, and the claim is clear from
our assumption using the exact sequence
  \[
  0 \to \sO_X(\ov{D}-(n+1)X') \to \sO_X(\ov{D}-nX') \to u_*\sO_{X'}(\ov{D}-nX') \to 0.
  \]

For $q \ge 1$, consider the exact sequences
\begin{equation}\label{eqn:s.e.s-5}
  0 \to A^{q-1}_{n+1}(\ov{D}) \xrightarrow{d}
  u^*\Omega^q_X(\log E)(\ov{D}-nX') \to A^q_n(\ov{D}) \to 0;
\end{equation}
\[
0 \to \Omega^q_X(\log E)(\ov{D}-(n+1)X')\to \Omega^q_X(\log E)(\ov{D}-nX') \to
u_*u^*\Omega^q_X(\log E)(\ov{D}-nX') \to 0,
\]
where the first one is obtained from ~\eqref{eqn:Kahler-seq-0}.
The second exact sequence and our assumption together
imply that $H^i_\zar(X',u^*\Omega^q_X(\log E)(\ov{D}-nX'))=0$ for all $i+q \le r-1$.
Also, the induction hypothesis implies that
$H^i_\zar(X',A^{q-1}_{n+1}(\ov{D}))=0$ for all $i+q \le r$. From the first exact
sequence of \eqref{eqn:s.e.s-5}, we therefore conclude that
$H^i_\zar(X',A^q_n(\ov{D}))=0$ for all $i+q \le r-1$.
\end{proof}

We now recall the morphisms of complexes 
\[
W_m\sF^{q,\bullet}_{X|D} \to u_*W_m\sF^{q,\bullet}_{X'|D'}, \ W_m\Omega^q_{X|D,\log} \to
u_*W_m\Omega^q_{X'|D',\log}, \ W_m\Omega^q_{X,\log} \to u_*W_m\Omega^q_{X',\log},
\]
where the first map is defined by the functoriality of the complex
$W_m\sF^{q,\bullet}_{X|D}$ (cf. proof of \corref{cor:Fil-functor}).
The following is our main Lefschetz theorem for the $p$-primary part of various
cohomology groups.

\begin{thm}\label{thm:Lefschetz}
  Let $D \in \Div^{\eff}_E(X)$ and let $i, q,r \ge 0$ be integers.
  If $X'$ satisfies the ${\rm Diff}(r,0)$ condition, then for all $m \ge 1$, 
  the map \[
  u^* \colon H^i_\et(X,W_m\Omega^q_{X,\log}) \to H^i_\et(X',W_m\Omega^q_{X',\log})
  \]
  is an isomorphism for $i+q \le r-1$ and injective for $i+q=r$.
  If $X'$ satisfies the ${\rm Diff}(r,D)$ condition,
  then for all $m \ge 1$, the maps
  \begin{enumerate}
  \item
    $u^* \colon \H^i_\et(X,W_m\sF^{q,\bullet}_{X|D}) \to
    \H^i_\et(X',W_m\sF^{q,\bullet}_{X'|D'})$  and
  \item
    $u^* \colon H^i_\et(X,W_m\Omega^q_{X|D,\log}) \to H^i_\et(X',W_m\Omega^q_{X'|D',\log})$
  \end{enumerate}
  are isomorphisms for $i+q \le r-1$ and injective for $i+q=r$.
\end{thm}
\begin{proof}
By \cite[Prop.~2.12]{Shiho}, \thmref{thm:Global-version}(9), 
\corref{cor:Complex-7} and induction on $m$, it suffices to show that the maps 
 \begin{equation}\label{eqn:vanishing-0}
   H^i_\et(X,\Omega^q_{X,\log}) \to H^i_\et(X',\Omega^q_{X',\log}); \
   \H^i_\et(X,W_1\sF^{q,\bullet}_{X|D/p^j}) \to
  \H^i_\et(X',W_1\sF^{q,\bullet}_{X'|D'/p^j});
    \end{equation}
\begin{equation*}
      H^i_\et(X,\Omega^q_{X|\lceil D/p^j \rceil,\log}) \to
      H^i_\et(X',\Omega^q_{X'|\lceil D'/p^j \rceil,\log})
    \end{equation*}
    are isomorphisms (resp. injective) if $i+q \le r-1$ (resp. $i+q \le r$)
    for all $j \ge 0$. Note that as $D \ge 0$ in each of the above cases, one has
    $D/p^j=D/p^{j+1}$ and $\lceil D/p^j \rceil = \lceil D/p^{j+1} \rceil$ for $j \gg 0$.
In view of \corref{cor:ker-lef}, it is enough to show for all $i+q \le r$ and
    $0 \le \ov D \le D$ that
    \begin{equation}\label{eqn:vanishing-1}
      \H^i_\et(X,W_1\sF^{q,\bullet}_{\sX_1|-X'}) = \H^i_\et(X,W_1\sF^{q,\bullet}_{\sX_0|\ov D-X'}) =
      \H^i_\et(X, \sN^{q,\bullet}_{\sX_0|\ov D-X'})= 0.
    \end{equation}
    We shall show that the left and the middle groups are zero
    as the proof for the group
    on the right  is analogous to that of the one in the middle.

    To prove the vanishing of the left and the middle groups
    in ~\eqref{eqn:vanishing-1},
    it is enough to show (cf. \lemref{lem:Log-fil-4}(3))
    for all $i+q \le r$ and  $0 \le \ov D \le D$ that
\begin{equation}\label{eqn:vanishing-4}
  H^i_\zar(X, Z_1\Omega^q_X(\log X')(-X')) = 0 =
  H^i_\zar(X, Z_1\Omega^q_X(\log (E+X'))(\ov D-X'));
 \end{equation}
\begin{equation}\label{eqn:s.e.s-4}
H^i_\zar(X,\Omega^q_X(\log X')(-X')) = 0 =  H^i_\zar(X,\Omega^q_X(\log (E+X'))(\ov D-X')).
\end{equation}

First, we shall prove \eqref{eqn:vanishing-4} by induction on $q$,
assuming that \eqref{eqn:s.e.s-4} holds for all $i+q \le r$ and $0 \le \ov{D} \le D$.
For $q=0$, this is clear by the hypothesis on $X'$ and by noting that
$Z_1\sO_X(\ov D-X') \cong \sO_X((\ov{D}-X')/p) \cong  \sO_X(\ov D/p-X')$
(cf. \ref{lem:Complete-4}(2)), where the last isomorphism holds because
$X'$ is not a component of $E$. 

For $q \ge 1$, we use the exact sequence (cf. \lemref{lem:Complete-9})
\[
0 \to Z_1\Omega^{q-1}_X(\log (E+X'))(\ov D-X')\to \Omega^{q-1}_X(\log (E+X'))(\ov D-X')
\xrightarrow{d} \hspace{1cm}
\]
\[
\hspace{3cm} Z_1\Omega^q_X(\log (E+X'))(\ov D-X') \xrightarrow{C}
\Omega^q_X(\log (E+X'))(\ov D/p-X') \to 0.
\]
Considering the cohomology groups and using \eqref{eqn:s.e.s-4} and
induction on $q$, we conclude the proof of the vanishing of the right side term in
\eqref{eqn:vanishing-4}. The corresponding claim for the left side term is a special
case, where we let $E = \emptyset$.

To prove ~\eqref{eqn:s.e.s-4}, we take the tensor product of the first exact
sequence of \lemref{lem:Kahler-seq} with $\sO_X(\ov{D} - X')$, which yields the
short exact sequence
\[
0 \to \Omega^q_X(\log E)(\ov{D} - X') \to \Omega^q_X(\log(E + X'))(\ov{D} - X')
      \xrightarrow{\res} u_* \Omega^{q-1}_{X'}(\log E')(\ov{D} - X') \to 0.
      \]
    By letting $E$ be empty, this also yields an exact sequence
    \begin{equation}\label{eqn:vanishing-3}
      0 \to \Omega^q_X(- X') \to \Omega^q_X(\log X')(- X')
      \xrightarrow{\res} u_* \Omega^{q-1}_{X'}(- X') \to 0.
    \end{equation}
    The desired vanishing in ~\eqref{eqn:s.e.s-4} now follows by 
    using these exact sequences, our assumption on $X'$ and \lemref{lem:ample}.
\end{proof}

\begin{cor}\label{cor:s.e.s-6}
  Under the hypotheses of \thmref{thm:Lefschetz}, the restriction map
  \[
  u^* \colon \Fil^{\log}_D H^{q+1}_{p^m}(U) \to 
  \Fil^{\log}_{D'} H^{q+1}_{p^m}(U')
  \]
  is an isomorphism for $q \le r-2$ and injective for $q = r-1$.
\end{cor}
\begin{proof}
  Combine Theorems~\ref{thm:Main-1} and ~\ref{thm:Lefschetz}(1).
  \end{proof}

\section{Lefschetz theorem for Brauer group}\label{sec:Lef-Br}
In this section, we apply \thmref{thm:Lefschetz} to prove a Lefschetz theorem for the
Brauer group and give an application.
We continue to work in the setting of \S~\ref{sec:Lef}.

\subsection{Lefschetz for $\ell$-adic cohomology and Picard group}\label{sec:l-adic}
In order to prove the Lefschetz theorem for the
Brauer group, we need to establish some new results of Lefschetz type for
$\ell$-adic {\'e}tale cohomology and Picard groups.
We begin with the case of $\ell$-adic cohomology.
We let $k$ be an arbitrary field.

\begin{thm}\label{thm:ell-lef}
  Assume that $k$ is either finite or 1-local or separably closed, and $\ell$ is a
  prime number invertible in $k$.
  Let $X$, $E$, $U$, $X'$, $E'$ and $U'$ be as in the beginning of \S~\ref{sec:Lef}.
  Then for all $j \in \Z$, we have the following.
    \begin{enumerate}
    \item The map $H^i_\et(X,\Z/\ell^m(j)) \to H^i_\et(X',\Z/\ell^m(j))$ is an
      isomorphism if $i \le N-2$ and injective if $i=N-1$.
    \item The map $H^i_\et(U,\Z/\ell^m(j)) \to H^i_\et(U',\Z/\ell^m(j))$ is an 
      isomorphism if $i \le N-2$ and injective if $i=N-1$.
    \item The map $H^i_E(X,\Z/\ell^m(j)) \to H^i_{E'}(X',\Z/\ell^m(j))$ is an isomorphism
      if $i \le N-1$ and injective if $i=N$.
    \end{enumerate}
\end{thm}
\begin{proof}
  We let $e$ denote the cohomological dimension of $k$ so that
  $e=0$ if $k$ is separably closed,
  $e=1$ if $k$ is a finite field and $e=2$ if $k$ is an 1-local field.
  We let $e' = e-1$.
  First we prove item (1) (which is classical if $k$ is separably closed).
  Let's write $\Wedge=\Z/\ell^m$ and look at the perfect
  pairing of finite abelian groups
  (cf. \cite[Introduction]{JSZ}, \cite[Chap.~VI,\S 11]{Milne-EC}, \cite[Thm.~9.9]{GKR})
 \begin{equation}\label{eqn:duality-00}
   H^i_{c,\et}(W,\Wedge(j)) \times
   H^{2N+e-i}_\et(W,\Wedge(N+e'-j))\to
   H^{2N+e}_{c,\et}(W,\Wedge(N+e')) \cong \Wedge,
    \end{equation}
 where recall that $W = X \setminus X'$.
 Being an affine scheme over $k$, the $\ell$-cohomological dimension of $W$ is
 $\le N+e$ (cf. \cite[VI, Thm~7.2]{Milne-EC}), hence this perfect pairing
 implies that $H^i_{c,\et}(W, \Wedge(j)) =0 $ if $i<N$. This implies item (1) using
 the exact sequence
\[
\cdots \to H^i_{c,\et}(W, \Wedge(j)) \to H^i_\et(X, \Wedge(j)) \to
H^i_\et(X',\Wedge(j)) \to  H^{i+1}_{c,\et}(W, \Wedge(j)) \to \cdots.
\]

We shall prove items (2) and (3) simultaneously using induction on the number of
components of $E$. When $r =1$, we look at the diagram
\begin{equation}\label{eqn:loc-0}
  \xymatrix@C.8pc{
    H_\et^{i-2}(E, \Wedge(j-1)) \ar[r]^-{\cong} \ar[d] &
    H^{i}_{\et, E}(X, \Wedge(j)) \ar[d] \\
     H_\et^{i-2}(E', \Wedge(j-1)) \ar[r]^-{\cong}  &
     H^{i}_{\et, E'}(X', \Wedge(j)),}
\end{equation}
where the horizontal arrows are the purity isomorphisms
(cf. \cite[Chap.~VI, Thm.~5.1]{Milne-EC}) and the vertical arrows are the
restriction maps. This diagram is commutative (cf. 
\cite[Defn.~A.2.9, Exm.~A.2.10]{CD-Comp}). Since item (1) holds even if we replace
$X$ by $E$, ~\eqref{eqn:loc-0} implies that item (3) holds when $r =1$. 

To prove item (2) when $r =1$, we look at the commutative diagram 
    \begin{equation}\label{eqn:loc}
        \xymatrix@C.8pc{
          \cdots \ar[r]&H^{i}_E(X) \ar[r] \ar[d]_-{\phi_i} & H^i_\et(X) \ar[r]
          \ar[d]^-{\theta_i}& H^i_\et(U) \ar[r] \ar[d]^-{\psi_i} & H^{i+1}_E(X) \ar[r]
          \ar[d]^-{\phi_{i+1}} &H^{i+1}_\et(X) \ar[r] \ar[d]^-{\theta_{i+1}}& \cdots \\
          \cdots \ar[r]&H^{i}_{E'}(X') \ar[r] & H^i_\et(X') \ar[r]& H^i_\et(U') \ar[r] &
          H^{i+1}_{E'}(X') \ar[r]&H^{i+1}_\et(X') \ar[r]& \cdots}
    \end{equation}
    of exact sequences, where $H^i_\et(-)$ (resp. $H^i_{\et, *}(-)$) is the shorthand
      for $H^i_\et(-,\Wedge(j))$ (resp. $H^i_{\et, *}(-, \Wedge(j))$), and
      all the vertical arrows are the canonical restriction maps. Since
      $\theta_i$ is an isomorphism (resp. injective) when $i \le N-2$ (resp.
      $i = N-1$) and $\phi_i$ is an isomorphism (resp. injective) when $i \le N-1$
      (resp. $i = N$), a diagram chase implies (2) when $r =1$. 
   
      We now let $r \ge 2$ and set $\ov{E} = \bigcup\limits_{t=2}^{r}E_t$ and
      $\ov{E}'= \ov E \bigcap X'$. Then $\ov{E}'$ is an snc 
      divisor on $X'$ with $r-1$ components by our assumption.
      Furthermore, $E\setminus\ov E$ (resp. $E'\setminus \ov E'$) is a smooth
      divisor on $X\setminus \ov E$ (resp. $X'\setminus \ov E'$). We now look
      at the commutative diagram
      
 \begin{equation}\label{eqn:loc-1}
        \xymatrix@C.6pc{
         \cdots \ar[r] & H^{i-1}_{E\setminus \ov{E}}(X\setminus \ov{E}) \ar[r]
          \ar[d]_-{\gamma_{i-1}} & H^{i}_{\ov E}(X) \ar[r]
          \ar[d]^-{\alpha_i} & H^{i}_{E}(X) \ar[r]
          \ar[d]^-{\phi_i} & H^i_{E\setminus \ov{E}}(X\setminus \ov{E})
          \ar[r]\ar[d]^-{\gamma_i} &  \cdots \\
          \cdots \ar[r] & H^{i-1}_{E'\setminus \ov{E}'}(X' \setminus \ov{E}') \ar[r] &
          H^{i}_{\ov E'}(X') \ar[r] & H^{i}_{ E'}(X') \ar[r] &
          H^{i}_{E'\setminus \ov E'}(X'\setminus \ov E') \ar[r] & \cdots
        }
    \end{equation}     
 with exact rows (cf. \cite[Chap.~III, Rem.~1.26]{Milne-EC}) in which the vertical
 arrows are the restriction maps.

 By using ~\eqref{eqn:loc-0} with the pair $(X,E)$ (resp. $(X', E')$)
 replaced by $(X \setminus \ov{E}, E \setminus \ov{E})$
 (resp. $(X' \setminus \ov{E}', E' \setminus \ov{E}')$ and
 applying the $r =1$ case of item (2) to the inclusion of the hypersurface section
 $E' \setminus \ov{E}' \inj E \setminus \ov{E}$, we get that
 $\gamma_i$ is bijective if $i \le N-1$ and injective if $i=N$.
 Note here that $E \setminus \ov{E} = E_1 \setminus \ov{E}$ and
 $\ov{E} \bigcap E_1$ is an snc divisor on $E_1$.
 Also, $E' \setminus \ov{E}' = E'_1 \setminus \ov{E}'$ and
 $E'_1 := E_1 \bigcap X'$ is a hypersurface section of $E_1$.
 
 The induction hypothesis implies that
 $\alpha_i$ is bijective if $i \le N-1$ and injective if $i=N$. 
 An easy diagram chase in \eqref{eqn:loc-1} shows that $\phi_i$ is bijective if
 $i \le N-1$ and injective if $i=N$. This proves item (3). Using item (1) and
 ~\eqref{eqn:loc} again, we also get (2).
This concludes the proof.
 \end{proof}

The next ingredient is the following version of the Grothendieck--Lefschetz theorem for
Picard group over arbitrary fields. 

\begin{prop}\label{prop:N-lef}
  Let $k$ be any field and let $X$ be a smooth geometrically connected projective
  variety over $k$ such that $\dim(X) \ge 4$.  Let $X' \subset X$ be a smooth
  hypersurface section which satisfies the ${\rm Diff}(3,0)$ condition.
  Then the canonical map $u^* \colon \Pic(X) \to \Pic(X')$
  is bijective if $\Char(k) = 0$. It is injective and
  its cokernel is a $p$-primary torsion abelian group if $\Char(k) = p > 0$.
\end{prop}
\begin{proof}
  When $k$ is algebraically closed, this is the classical Grothendieck-Lefschetz theorem
  (cf. \cite[Chap.~IV, Thm.~3.1]{Hartshorne-Ample}), where we note that the condition
  ${\rm Leff}(X, X')$ holds by Theorem~1.5 in Chapter~IV of op. cit. and other
  conditions hold by our assumption (including the Diff$(3,0)$ condition).

Suppose now that $k$ is separably closed but not algebraically closed. 
Let $\ov k$ be an algebraic closure of $k$. Let  $\ov X=X\times_{\Spec(k)}  \Spec(\ov k)$
and $\ov{X}' = X' \times_{\Spec(k)} \Spec(\ov k)$. Then $\ov{X}'$ is a geometrically
connected smooth hypersurface section of $\ov{X}$
which satisfies the Diff$(3,0)$ condition (cf. \cite[Thm.~III.9.3]{Hartshorne-AG}).
In particular, the pull-back maps $\Pic(X) \to \Pic(\ov{X})$ and
$\Pic(X') \to \Pic(\ov{X}')$ are injective (cf. \cite[Tag~0CC5]{SP}), and
their cokernels are $p$-primary torsion groups. The latter claim can easily be checked
using the long exact cohomology sequence corresponding to 
\[
0 \to \sO^{\times}_Z \to \sO^{\times}_{Z_{k'}} \to {\sO^{\times}_{Z_{k'}}}/{\sO^{\times}_{Z_{k'}}}
\to 0,
\]
where ${k'}/k$ is any finite purely
inseparable extension ${k'}/k$ and $Z \in \{X, X'\}$.
This proves the proposition when $k$ is separably closed.

Suppose now that $k$ is arbitrary and let $k_s$ be a separable closure of $k$.
Let $X_s = X\times_{\Spec(k)} \Spec(k_s)$ and $X'_s = X' \times_{\Spec(k)} \Spec(k_s)$.
By a similar argument as above, we see that $X_s$ and $X'_s$ are geometrically
connected and $X'_s$ is a hypersurface section of $X_s$ which satisfies the
Diff$(3,0)$ condition.
By \cite[Prop.~5.4.2]{CTS}, we have a commutative diagram of exact sequences
    \begin{equation}\label{eqn:N-lef-0}
        \xymatrix@C.8pc{
          0 \ar[r]& \Pic(X) \ar[r] \ar[d]_-{u^*} &
          \Pic(X_s)^G \ar[r]\ar[d]^-{(u^*_s)^G} &
          \Br(k)\ar@{=}[d] \\
        0 \ar[r]& \Pic(X')\ar[r] & \Pic(X'_s)^G\ar[r]& \Br(k),
        }
    \end{equation}
    where $G$ is the absolute Galois group of $k$.
    The desired result now follows using the previous cases shown above.
\end{proof}

\subsection{Lefschetz theorem for Brauer group}\label{sec:Lef-Brauer}
We shall now prove our Lefschetz theorem for the Brauer group. The result is the
following.

\begin{thm}\label{thm:Br-lef}
  Let $k$ be a field and let $X$ be a smooth geometrically connected projective
  variety over $k$ such that $\dim(X) \ge 4$.
  Let $X' \subset X$ be a smooth hypersurface section which satisfies the
  ${\rm Diff}(3,0)$ condition.
  We then have the following.
  \begin{enumerate}
  \item
    If $k$ is $F$-finite of characteristic $p > 0$, then the map
    $u^* \colon \Br(X)\{p\} \to \Br(X')\{p\}$ is an isomorphism.
    \item
  If $k$ is either finite or 1-local  or separably closed, then  the map
  $u^* \colon \Br(X) \xrightarrow{\cong} \Br(X')$
  is an isomorphism.
  \end{enumerate}
\end{thm}
\begin{proof}
If $\Char(k) = p > 0$, we use the commutative diagram of exact sequences
    \begin{equation}\label{eqn:lef-Br}
        \xymatrix@C.8pc{
          0 \ar[r] & \Pic(X)/p^m \ar[r] \ar[d]_-{u^*} &
          H^1_\et(X,W_m\Omega^1_{X,\log}) \ar[r]
          \ar[d]_-{\cong}^-{u^*} & \Br(X)[p^m] \ar[r] \ar[d]^-{u^*} & 0 \\
          0 \ar[r] & \Pic(X')/p^m \ar[r] & H^1_\et(X',W_m\Omega^1_{X',\log}) \ar[r] &
          \Br(X')[p^m] \ar[r] & 0
        }
    \end{equation}
    and Theorem~\ref{thm:Lefschetz} to get an exact sequence
    \[
    0 \to {\rm Coker}(\Pic(X) \to \Pic(X')) \otimes {\Q_p}/{\Z_p} \to
    \Br(X)\{p\} \to \Br(X')\{p\} \to 0.
    \]
    We now use \propref{prop:N-lef} to conclude that the map
    $u^* \colon  \Br(X)\{p\} \to \Br(X')\{p\}$ is an isomorphism.
    This proves item (1). If $k$ is as in item (2), then a similar argument
    (using the Kummer sequence) and \thmref{thm:ell-lef} imply that
    the map $u^* \colon  \Br(X)\{p'\} \to \Br(X')\{p'\}$ is an isomorphism.
    This completes the proof.
    \end{proof}

Let $D = \stackrel{r}{\underset{i =1}\sum} n_iE_i \in \Div^{\eff}_E(X)$.
Recall from \cite[Defn.~8.7]{KRS} that the Brauer group of $X$ with modulus $D$ is
defined to be the subgroup (cf. Definition~\ref{defn:Log-fil-D},
Example~\ref{exm:Examples-fil})
\[
\begin{array}{lll}
\Fil^\log_D \Br(U) & := & 
  \Ker \left(\Br(U) \xrightarrow{\bigoplus j_i} \
  \stackrel{r}{\underset{i =1}\bigoplus} 
  \frac{H^{2}(K_i)}{\Fil^{\bk}_{n_i} H^{2}(K_i)}\right) \\
  & \cong & 
\Ker \left(\Br(U)\{p\} \xrightarrow{\bigoplus j_i} \
  \stackrel{r}{\underset{i =1}\bigoplus} 
  \frac{H^{2}(K_i)}{\Fil^{\bk}_{n_i} H^{2}(K_i)}\right) \bigoplus \Br(U)\{p'\},
\end{array}
\]
  where $K_i$ is the fraction field of the henselization of the function field of $X$
  at the generic point of $E_i$ for $1 \le i \le r$. This is the Brauer group
  analogue of the abelianized fundamental group with modulus \cite{Kerz-Saito-ANT}.
  
Combining Theorem~\ref{thm:Lefschetz}, \thmref{thm:ell-lef},  \propref{prop:N-lef}
and \corref{cor:F^q} with the surjection $\Pic(X') \surj \Pic(U')$, an argument
similar to that of \thmref{thm:Br-lef} yields the
following result. This generalizes the Lefschetz theorem for abelianized
fundamental group with modulus of Kerz--Saito \cite{Kerz-Saito-ANT}
to Brauer group with modulus.

\begin{cor}\label{cor:Brauer-Lef-mod}
  Let $X$ be as in \thmref{thm:Br-lef} and assume that $X' \subset X$
  satisfies the ${\rm Diff}(3,D)$ condition. Then we have the following.
  \begin{enumerate}
  \item
    If $k$ is $F$-finite of characteristic $p > 0$, then the pull-back
    $u^* \colon \Fil^\log_D \Br(U)\{p\} \to \Fil^\log_{D'} \Br(U')\{p\}$
    is an isomorphism.
  \item
    If $k$ is either finite or 1-local or separably closed, then
    the pull-back
  $u^* \colon \Fil^\log_D\Br(U) \to \Fil^\log_{D'}\Br(U')$ is an isomorphism.
  \end{enumerate}
\end{cor}

\subsection{The Brauer group of complete intersection}\label{sec:Br-CI}
As an application of \thmref{thm:Br-lef}, we shall compute the Brauer group
of a complete intersection projective variety over arbitrary fields.

Let $k$ be a field. Recall from \cite[Defn.~5.5]{ABM} that if $X$ is a 
$k$-scheme of dimension $N$ and $\sL$ is an ample line bundle on $X$, then
$(X,\sL)$ is called a Kodaira pair if
$H^i_\zar(X,\Omega^q_{X/k}\otimes_{\sO_X} \sL^{-n})=0$ for all $i+q \le N-1$ and
$n \ge 1$.  One knows that if $\Char(k) = 0$ and $X$ is a smooth and
projective variety over $k$ with an ample line bundle $\sL$, then $(X, \sL)$ is a
Kodaira pair (cf. \cite[Cor.~2.11]{Deligne-Illusie}).

\begin{lem}\label{lem:KP-F}
  Assume that $k$ is an $F$-finite field of characteristic $p > 0$ and $X$ is a smooth
  $k$-scheme with an ample line bundle on $\sL$. If ($X,\sL$) is a Kodaira pair, then
  $H^i_\zar(X,\Omega^q_{X}\otimes_{\sO_X} \sL^{-n})=0$ for all $i+q \le N-1$ and $n \ge 1$.
    \end{lem}
    \begin{proof}
      By our assumptions, there is an exact
      sequence of finitely generated locally free sheaves
      \[
      0 \to \Omega^1_k \otimes_k \sO_X \to \Omega^1_X \to \Omega^1_{X/k} \to 0.
      \]
      There is therefore a decreasing filtration
      $\{F^t\}_{0 \le t\le q}$ on $\Omega^q_X \otimes_{\sO_X} \sL^{-n}$ such that
      $F^t/F^{t+1} \cong \Omega^t_k \otimes_k (\Omega^{q-t}_{X/k}\otimes_{\sO_X} \sL^{-n})$
      (cf. \cite[Exc.~ II.5.16(d)]{Hartshorne-AG}).
      Now, the hypothesis on $(X,\sL)$ implies
      $H^i_\zar(X, F^t/F^{t+1})=0$ for all $i+q \le N-1$ and $t \le q$. This easily
      implies the lemma. 
\end{proof}

\begin{lem}\label{lem:Kodaira}
  Let $X\inj \P^N_k$ be any complete intersection closed subvariety
  (not necessarily smooth).
  Let $X'$ be any hypersurface section of $X$. Then $(X,\sO_X(X'))$ is a Kodaira pair.
    \end{lem}
    \begin{proof}
Since $(\P^N_k, \sO(1))$ is a Kodaira pair, \cite[Cor.~5.8]{ABM} implies
that $(X, \sO_{X}(1))$ is also a Kodaira pair. In particular, so is $(X, \sO_{X}(n))$ 
for all $n \ge 1$. This implies the claim. 
    \end{proof}

\begin{lem}\label{lem:lef}
  Let $k$ be a perfect field and let $X \subset \P^N_k$ be a smooth complete
  intersection  closed subvariety of dimension $d$.
  Then the pull-back map $H^i_\et(\P^N_k,W_m\Omega^q_{\P^N_k,\log})
  \to H^i_\et(X, W_m\Omega^q_{X,\log})$ is an isomorphism for all $m \ge 1$ and all
  $i,q \ge 0$ such that $i+q \le d-1$. 
    \end{lem}
\begin{proof}
  If $d = N-1$, the result follows by applying \lemref{lem:Kodaira} to
  $\P^N_k$ and applying \thmref{thm:Lefschetz} to the inclusion of the hypersurface
  $X \subset \P^N_k$. In general, we write $X = X_1\bigcap H$, where $X_1$
  (which is taken to be $\P^N_k$ when $d = N-1$) is a complete intersection in
  $\P^N_k$ of dim $d+1$ and $H \subset \P^N_k$ is a hypersurface.
      As $X$ is regular, the same must be the case with $X_1$. As $k$ is perfect, it
      follows that $X_1$ is smooth over $k$.
      Moreover, \lemref{lem:Kodaira} says that $(X_1, \sO_{X_1}(X))$ is a Kodaira pair. 
      It follows that $X \subset X_1$ satisfies the Diff$(d,0)$ condition. We now apply
      \thmref{thm:Lefschetz} inductively (on $N-d$) to conclude the proof.
\end{proof}

The following result was shown by Colliot-Th{\'e}l{\`e}ne--Skorobogatov (cf.
\cite[Cor.~5.5.4]{CTS}) for the prime-to-$\Char(k)$ part of the Brauer group.

\begin{thm}\label{thm:Br-ci}
  Let $k$ be any field and let $X \subset \P^N_k$ be a smooth complete intersection
  of dimension $d \ge 3$. Then the pull-back map $\Br(k) \to \Br(X)$ is an
  isomorphism.
    \end{thm}
\begin{proof}
We first assume that $k$ is algebraically closed.
      In view of \cite[Cor.~5.5.4, Thm.~6.1.3]{CTS}, we only need to show that
      $\Br(\P^N_k)\{p\} \to \Br(X)\{p\}$ is an isomorphism if $\Char(k) = p >0$.
      To prove this statement,
      we write $X= X_1\bigcap H$, where $X_1$ is a complete intersection in
      $\P^N_k$ of dim $d+1$ and $H \subset \P^N_k$ is a hypersurface. One then
      checks that $X_1$ and $X$ are geometrically connected.
      Since $k$ is algebraically closed,
      $X_1$ is smooth over $k$ (see the proof of \lemref{lem:lef}).
      Furthermore, $X \subset X_1$ is an ample hypersurface section satisfying the
      Diff$(3,0)$ condition by \lemref{lem:Kodaira}. It follows therefore by
      inductively applying \propref{prop:N-lef} that the pull-back map
      $\Pic(\P^N_k) \to \Pic(X)$ is injective whose cokernel is a $p$-primary
      torsion group. We now apply \lemref{lem:lef} and repeat the argument of
      \thmref{thm:Br-lef} to conclude the proof. 

      We next assume that $k$ is separably closed and let $\ov{X}$ denote the base
      change of $X$ to an algebraic closure of $k$. From what we showed above,
      we have that $\Br(\ov{X}) = 0$. Since $X \subset \P^N_k$ is a normal
      complete intersection, it is projectively normal and
      arithmetically Cohen-Macaulay scheme (i.e., its homogeneous coordinate ring
      is Cohen-Macaulay). This implies in particular that
      $H^1_\zar(X, \sO_X) = 0$ (cf. \cite[Lem.~1.3]{Krishna-Crelle}).
      We conclude from \cite[Thm.~5.2.5]{CTS} that $\Br(X) = 0 = \Br(k)$.
      Finally, if $k$ is arbitrary, we can apply \cite[Cor.~5.5.4]{CTS} to finish the
      proof.
\end{proof}

\end{document}